%% file: GL32zeta_ams01.tex
\documentclass[draft]{memo-l}

\usepackage{amssymb}

\newcommand{\diag}{\mathop{\mathrm {diag}}\nolimits}

\newcommand{\Ad}{\mathop{\mathrm {Ad}}\nolimits}
\newcommand{\adj}{\mathop{\mathrm {ad}}\nolimits}
\newcommand{\Hom}{\mathop{\mathrm {Hom}}\nolimits}
\newcommand{\sgn}{\mathop{\mathrm {sgn}}\nolimits}

\newcommand{\id}{\mathop{\mathrm {id}}\nolimits}

\newcommand{\Res}{\mathop{\mathrm {Res}}\nolimits}

\newcommand{\sI}{\sqrt{-1}}

\newcommand{\cC}{\mathcal{C}}

\newcommand{\cE}{\mathcal{E}}

\newcommand{\cI}{\mathcal{I}}

\newcommand{\cR}{\mathcal{R}}
\newcommand{\cS}{\mathcal{S}}

\newcommand{\cV}{\mathcal{V}}

\newcommand{\bC}{\mathbb{C}}

\newcommand{\bH}{\mathbb{H}}

\newcommand{\bQ}{\mathbb{Q}}
\newcommand{\bR}{\mathbb{R}}
\newcommand{\bZ}{\mathbb{Z}}

\newcommand{\me}{\mathbf{e}}

\newcommand{\mk}{\mathbf{k}}

\newcommand{\mn}{\mathbf{n}}
\newcommand{\mmp}{\mathbf{p}}

\newcommand{\gH}{\mathfrak{H}}

\newcommand{\gS}{\mathfrak{S}}
\newcommand{\ga}{\mathfrak{a}}

\newcommand{\g }{\mathfrak{g}}

\newcommand{\gk}{\mathfrak{k}}
\newcommand{\gl}{\mathfrak{l}}

\newcommand{\gn}{\mathfrak{n}}

\newcommand{\gp}{\mathfrak{p}}
\newcommand{\gs}{\mathfrak{s}}

\newtheorem{thm}{Theorem}[chapter]
\newtheorem{lem}[thm]{Lemma}
\newtheorem{prop}[thm]{Proposition}
\newtheorem{cor}[thm]{Corollary}

\theoremstyle{definition}

\theoremstyle{remark}
\newtheorem{rem}[thm]{Remark}

\numberwithin{section}{chapter}
\numberwithin{equation}{chapter}

\makeindex

\begin{document}

\allowdisplaybreaks 

\frontmatter

\title{Archimedean zeta integrals for $GL(3)\times GL(2)$}


\author{Miki Hirano}
\address{Department of Mathematics, 
Faculty of Science, 
Ehime University,
2-5, Bunkyo-cho, Matsuyama, Ehime, 790-8577, Japan}
\curraddr{}
\email{hirano.miki.mf@ehime-u.ac.jp}
\thanks{}

\author{Taku Ishii}
\address{Faculty of Science and Technology, 
Seikei University, 3-3-1 Kichijojikitamachi, 
Musashino, Tokyo, 180-8633, Japan}
\curraddr{}
\email{ishii@st.seikei.ac.jp}
\thanks{}

\author{Tadashi Miyazaki}
\address{Department of Mathematics, 
College of Liberal Arts and Sciences, Kitasato University, 
1-15-1, Kitasato, Minamiku, Sagamihara, Kanagawa, 252-0373, Japan}
\curraddr{}
\email{miyaza@kitasato-u.ac.jp}
\thanks{}

\date{}




\begin{abstract}
In this article, 
we give explicit formulas of archimedean Whittaker functions 
on $GL(3)$ and $GL(2)$. Moreover, we apply those to the calculation 
of archimedean zeta integrals for $GL(3)\times GL(2)$, 
and show that the zeta integral for 
appropriate Whittaker functions is equal to the associated $L$-factors. 
\end{abstract}

\maketitle 

\tableofcontents

\include{preface_GL32zeta_003}

\mainmatter

\include{main_GL32zeta_004}

\appendix
\include{appendix_GL32zeta_002}

\backmatter
\def\cprime{$'$}

\printindex

\end{document}

%% file: preface_GL32zeta_003.tex
\chapter*{Introduction}

The archimedean theory of automorphic forms on 
a real reductive group $G$ is based on the $(\g ,K)$-module 
structures associated to irreducible admissible representations of $G$. 
In pioneering work by Jacquet-Langlands 
\cite{Jacquet_Langlands_001,Jacquet_003}, 
the $(\g, K)$-module structures for archimedean representations 
of $GL(2)$ are explicitly described and the important properties of 
$L$-functions are derived from 
an evaluation of the zeta integrals by using their description. 
The group $GL(2,\bR)$ is the easiest one to handle for the following reason; 
the dimensions of irreducible representations of 
the maximal compact subgroup $K$ are at most two (and essentially one), 
and the multiplicities of $K$-types for any irreducible $(\g, K)$-modules 
are free. It is not easy to describe the $(\g, K)$-module structures 
for general $G$, even if the case of $G=GL(2,\bC)$. 
For this reason, the explicit study for the 
archimedean theory of $L$-functions has been developed 
with a focus on the class one cases. 
We naturally expect that some of good properties 
for the class one cases also hold for general cases. 
However, precise knowledge is not enough to show good properties of 
$L$-functions by analyzing the $(\g, K)$-module structures. 

Under these background, explicit calculus of $(\g, K)$-module structures 
for real reductive groups has been developed 
in the context of automorphic forms. 
In particular, we have obtained a result of the automorphic $L$-functions 
by applying some explicit harmonic analysis on 
groups with low real rank. 
Here, a technique for analysis of representations of 
the maximal compact group $K$ is important, 
such as a suitable choice of a basis or a system of generators 
for the representation space. 
Also, an advanced treatment for special functions of several variables 
is needed, such as a suitable integral expression of 
the generalized spherical function which is considered 
a model of the associated irreducible $(\g ,K)$-module.

This article consists of two parts together with an appendix. 
The purpose of Part \ref{part:1} is to give explicit formulas of 
archimedean Whittaker functions for the irreducible (generalized) 
principal series representations on $GL(3)$ and $GL(2)$. 
These formulas are the same as in previous papers essentially. 
However we give here a refined form suitably for our application, 
since the relation among formulas for $GL(3)$ and $GL(2)$ is important 
in Part \ref{part:2}. 
Beside our purpose, we also expect these formulas are effective 
in further analytic study of automorphic forms on $GL(3)$. 
In Part \ref{part:2} of this article, the archimedean zeta integrals 
for $GL(3) \times GL(2)$ are discussed by using our explicit formulas 
of Whittaker functions. 
Our explicit evaluation of the archimedean zeta integrals leads 
the following main result in this article. 
\begin{thm}
Let $F$ be $\bR$ or $\bC$. 
Let $\Pi$ and $\Pi'$ be irreducible generalized principal series 
representations of $GL(3,F)$ and $GL(2,F)$, respectively. 
Then there exist Whittaker functions $W$ and 
$W'$ for $\Pi$ and $\Pi'$, respectively, such that 
\[
Z(s,W,W')=L(s,\Pi\times \Pi').
\]
Here $Z(s,W,W')$ is the local zeta integral for $W, W'$, 
and $L(s,\Pi\times \Pi')$ is the local $L$-factor for 
$\Pi \times \Pi'$. 
\end{thm}

It is expected that the archimedean zeta integrals 
for appropriate Whittaker functions equal to the associated $L$-factors 
in the case of general $GL(n+1)\times GL(n)$. 
This expectation is well-grounded by Stade's result \cite{Stade_002} for 
the spherical $GL(n+1,\bR)\times GL(n,\bR)$-case 
and Popa's result \cite{Popa_001} for general $GL(2)\times GL(1)$-case. 
Adding to the above cases, 
our result gives a supporting evidence for the expectation. 
By using our explicit formulas of archimedean Whittaker functions 
on $GL(2)$ in Part \ref{part:1}, we give an explicit evaluation of 
the archimedean zeta integrals for $GL(2) \times GL(1)$ 
and $GL(2) \times GL(2)$. 
The results in this appendix are restatement of wellknown results.

\section*{Acknowledgments}

This work was supported by JSPS KAKENHI
Grant Numbers JP24540022, JP15K04796, JP15K04800, JP18K03252. 

%% file: main_GL32zeta_004.tex
\part{Whittaker functions}
\label{part:1}

\chapter{Basic objects}
\label{sec:Fn_basic_objects}

\section{Notation}
\label{subsec:Fn_notation}

We denote by $\bZ $, $\bQ$, $\bR $ and $\bC $ the ring of rational integers, 
the rational number field, the real number field and 
the complex number field, respectively. 
For $m\in \bZ$, we set 
$\bZ_{\geq m}=\{i\in \bZ \mid i\geq m\}$ and 
$\bZ_{\leq m}=\{i\in \bZ \mid i\leq m\}$. 
Let $\bR_+$ be the set of positive real numbers. 
For $z\in \bC$, 
we denote the real part, the imaginary part and 
the complex conjugate of $z$ 
by $\mathrm{Re}(z)$, $\mathrm{Im}(z)$ and $\overline{z}$, respectively. 
For $i,j\in \bZ$, we denote by $\delta_{i,j}$ the Kronecker delta, that is, 
\[
\delta_{i,j}=\left\{\begin{array}{ll}
1&\text{if}\ i=j,\\
0&\text{otherwise}. 
\end{array}\right.
\]
For $m,i\in \bZ_{\geq 0}$ such that $m\geq i$, 
let $\displaystyle 
\binom{m}{i}$ be the binomial coefficient 
$\dfrac{m!}{i!(m-i)!}$. 
For $t\in \bR$, we set 
\[
\sgn (t)=\left\{\begin{array}{ll}
1&\text{if}\ t\geq 0,\\
-1&\text{if}\ t<0.
\end{array}\right.
\]

Let ${F}$ be $\bR$ or $\bC$. 
We define the standard character $\psi_{F} \colon {F}\to \bC^\times$ by 
\begin{align*}
&\psi_{\bR} (t)=\exp (2\pi \sI t),&
&\psi_{\bC} (t)=\exp (2\pi \sI (t +\overline{t})).
\end{align*}

We denote by $1_n$ the unit matrix of degree $n$. We denote by 
$E_{i,j}=E_{i,j}^{(n)}$ the square matrix of degree $n$ 
with $1$ at the $(i,j)$-th entry and $0$ at other entries. 

For a subset $S$ of a $\bC$-vector space $V$, 
we denote by $\bC \textrm{-span}\,S$ the subspace of $V$ 
consisting of all linear combinations of vectors in $S$.

\section{Groups and algebras}
\label{subsec:Fn_group_algebra}

Let ${G}=G_n$ be the general linear group 
$GL(n,{F} )$ of degree $n$ over ${F}$. 
In this article, we view ${G}$ as a real reductive Lie group 
even if ${F} =\bC$. 
Let ${N}=N_n$ be the group of 
upper triangular matrices in ${G}$ 
with diagonal entries equal to $1$, 
and let ${A}=A_n$ be the group 
of diagonal matrices in ${G}$ with positive diagonal entries. 
Moreover, we fix a maximal compact subgroup ${K}=K_n$ of ${G}$ by 
\[
{K}=\left\{
\begin{array}{ll}
O(n)&\text{if }{F} =\bR,\\
U(n)&\text{if }{F} =\bC,
\end{array}
\right. 
\]
where $O(n)$ and $U(n)$ are the orthogonal group and the unitary group 
of degree $n$, respectively. 
Then we have 
an Iwasawa decomposition ${G}={N}{A}{K}$ of ${G}$. 
It is convenient to fix 
the coordinates on ${N}$ and ${A}$ as follows:
\begin{align*}
&x=\left(\begin{array}{ccccc}
1&x_{1,2}&x_{1,3}&\cdots &x_{1,n}\\
0&1&x_{2,3}&\cdots &x_{2,n}\\
\vdots &\ddots&\ddots&\ddots &\vdots\\
0&\cdots &0&1&x_{n-1,n}\\
0&\cdots &0&0&1
\end{array}\right)\in {N},\\
&y=\diag (y_1y_2\cdots y_n,\ y_2\cdots y_n,\ \cdots ,\ y_n)\in {A},
\end{align*}
where $x_{i,j}\in {F}\ (1\leq i<j\leq n)$ and 
$y_k\in \bR_+\ (1\leq k\leq n)$. 

Let ${\g}=\g_n$, ${\gn}=\gn_n$, ${\ga}=\ga_n$ and ${\gk}=\gk_n$ be 
the associated Lie algebras of ${G}$, ${N}$, ${A}$ and ${K}$, respectively. 
Let ${\gp}=\gp_n$ be the orthogonal complement of ${\gk}$ in ${\g}$ 
with respect to the Killing form, that is, 
${\gp}=\{ X\in {\g}\mid {}^tX=X\}$. 
We denote by $\gp_{\bC}=\gp_{n\bC}$ the complexification 
$\gp \otimes_{\bR}\bC$ of $\gp$. 
Let $L$ be a Lie subgroup of ${G}$, and $\gl$ 
the associated Lie algebra of $L$. 
We denote by $\gl_{\bC}$ the complexification 
$\gl \otimes_\bR \bC$ of $\gl$, by $U(\gl_{\bC} )$ 
the universal enveloping algebra of $\gl_\bC$, and 
by $Z(\gl_{\bC})$ the center of $U(\gl_\bC)$. 
In this article, $c(X\otimes z)$ denotes $X\otimes (cz)$ 
for $X\otimes z\in \gl_{\bC}$ and $c\in \bC$. 
We identify $\gl$ with the space of left $L$-invariant 
vector fields on $L$ in the usual way, that is, 
\begin{align*}
&(R(X)f)(g)=\frac{d}{dt}f(g\exp (tX))|_{t=0}&&(g\in L)
\end{align*}
for $X\in \gl$ and a differentiable function $f$ on $L$. 
Then $U(\gl_{\bC} )$ is identified with the algebra of 
left $L$-invariant differential operators on $L$. 
For a Fr\'{e}chet space $V$, let $C^\infty (L,V)$ be 
the space of $V$-valued smooth functions on $L$. 
We equip the space $C^\infty (L,V)$ with the topology 
of uniform convergence on compact sets of a function 
and its derivatives. 
We denote $C^\infty (L,\bC)$ simply by $C^\infty (L)$.

\section{Whittaker functions}
\label{subsec:Fn_def_whittaker}

For $(c_1,c_2,\cdots ,c_{n-1})\in {F}^{n-1}$, 
we define a character $\psi_{(c_1,c_2,\cdots ,c_{n-1})}$ of ${N}$ by 
\begin{align*}
&\psi_{(c_1,c_2,\cdots ,c_{n-1})}(x)
=\psi_{F} (c_1x_{1,2}+c_2x_{2,3}+\cdots +c_{n-1}x_{n-1,n})&
&(x=(x_{i,j})\in {N}). 
\end{align*}
Then unitary characters of ${N}$ are exhausted 
by the characters of this form. 
We say that $\psi_{(c_1,c_2,\cdots ,c_{n-1})}$ is non-degenerate 
if $(c_1,c_2,\cdots ,c_{n-1})\in ({F}^\times )^{n-1}$. For $c\in F$, 
we denote $\psi_{(c ,c ,\cdots ,c )}$ simply by $\psi_c$.

We regard $C^\infty ({G})$ as a $G$-module 
via the right translation. 
For a non-degenerate character $\psi$ of $N$, 
let $C^\infty ({N} \backslash {G} ;\psi )$ be 
the subspace of $C^\infty ({G})$ consisting of 
all functions $f$ satisfying 
\begin{align*}
f(xg)&=\psi (x)f(g)
&(x\in {N},\ g\in {G}).
\end{align*}
For an admissible representation $(\Pi ,H_{\Pi})$ of ${G}$, let 
\begin{align*}
&{\cI}_{\Pi,\psi }=\Hom_{({\g_\bC},{K})}
(H_{\Pi ,{K}} ,C^\infty ({N}\backslash {G};\psi )_{{K}}).
\end{align*}
Here $H_{\Pi ,{K}}$ and $C^\infty ({N}\backslash {G};\psi )_{{K}}$ are 
the subspaces of $H_{\Pi}$ and $C^\infty ({N}\backslash {G};\psi )$ 
consisting of all ${K}$-finite vectors, respectively. 
We define the subspace ${\cI}_{\Pi,\psi }^{\mathrm{mg}}$ 
of ${\cI}_{\Pi,\psi }$ consisting of 
all homomorphisms $\Phi$ such that  
$\Phi (f)$ $(f\in H_{\Pi ,{K}})$ are 
moderate growth functions. 
We define the space ${\mathrm{Wh}}(\Pi ,\psi )$ of 
Whittaker functions for $(\Pi ,\psi )$ by 
\begin{align*}
&{\mathrm{Wh}}(\Pi ,\psi )=
\bC \textrm{-span}\{\Phi (f) \mid f\in H_{\Pi ,{K}},\ 
\Phi \in {\cI}_{\Pi,\psi }\}, 
\end{align*}
and define the subspace 
${\mathrm{Wh}}(\Pi ,\psi )^{\mathrm{mg}}$ of 
${\mathrm{Wh}}(\Pi ,\psi )$ by 
\begin{align*}
&{\mathrm{Wh}}(\Pi ,\psi )^{\mathrm{mg}}=
\bC \textrm{-span}\{\Phi (f) \mid f\in H_{\Pi ,{K}},\ 
\Phi \in {\cI}_{\Pi,\psi }^{\mathrm{mg}}\}.
\end{align*}

Let $\varphi \colon V_{\tau}\to {\mathrm{Wh}}(\Pi ,\psi )$ 
be a ${K}$-embedding with a ${K}$-type $(\tau ,V_\tau )$ of $\Pi$. 
By definition, we have 
\begin{align}
\label{eqn:Fn_whitt_ngk}
&\varphi (v)(xgk)=\psi (x)\varphi (\tau (k)v)(g)&
&(v\in V_\tau,\ x\in {N},\ g\in {G},\ k\in {K}).
\end{align}
Because of the Iwasawa decomposition ${G}={N}{A}{K}$, 
we note that $\varphi $ 
is characterized by its restriction $v\mapsto \varphi (v)|_{{A}}$ to ${A}$. 
We call $v\mapsto \varphi (v)|_{{A}}$ 
the radial part of $\varphi$.

Assume that $\Pi$ is irreducible. 
Then the multiplicity one theorem 
(\textit{cf.} \cite{Shalika_001}, \cite{Wallach_001}) 
tells that the intertwining space 
${\cI}_{\Pi,\psi }^{\mathrm{mg}}$ 
is at most one dimensional. 
By the result of Matumoto 
\cite[Corollary 2.2.2, Theorem 6.2.1]{Matumoto_001} 
for $SL(n,F)$, we know that ${\cI}_{\Pi,\psi }\neq 0$ if and only if 
$\Pi$ is large in the sense of Vogan \cite{Vogan_001}.

In Part \ref{part:1}, we consider explicit formulas of the radial parts 
of ${K}$-embeddings $\varphi \colon V_{\tau}\to {\mathrm{Wh}}(\Pi ,\psi )$ 
for an irreducible admissible 
large representation $\Pi$ of $G=GL(n,F)$ ($n=2,3$). 
For $(c_1,c_2,\cdots ,c_{n-1})\in (F^\times)^{n-1}$, 
there is a $G$-isomorphism 
\begin{align}
\label{eqn:Fn_psi_change}
\Xi_{(c_1,c_2,\cdots ,c_{n-1})}\colon C^\infty ({N}\backslash {G};\psi_1)
\to C^\infty ({N}\backslash {G};\psi_{(c_1,c_2,\cdots ,c_{n-1})})
\end{align}
defined by $\Xi_{(c_1,c_2,\cdots ,c_{n-1})}(f)(g)
=f(m_{(c_1,c_2,\cdots ,c_{n-1})}g)$ ($g\in G$) with 
\[
m_{(c_1,c_2,\cdots ,c_{n-1})}=\diag (c_1c_2\cdots c_{n-1},\,
c_2\cdots c_{n-1},\,\cdots \,,\,c_{n-1},\,1). 
\]
Hence, it suffices to consider the case of $\psi_1$. 
In \S \ref{sec:R2_whittaker} and \S \ref{sec:R3_whittaker}, 
we give the explicit formulas for $G=GL(2,\bR )$ and 
$G=GL(3,\bR )$, respectively. 
In \S \ref{sec:C2_whittaker} and \S \ref{sec:C3_whittaker}, 
we give the explicit formulas for $G=GL(2,\bC )$ and 
$G=GL(3,\bC )$, respectively.

\section{Capelli elements}
\label{subsec:Fn_capelli}

We denote by $U(\g \gl (n,\bC ))$ 
the universal enveloping algebra  of the complex Lie algebra 
$\g \gl (n,\bC )$, and by $Z(\g \gl (n,\bC ))$ the center of 
$U(\g \gl (n,\bC ))$. 
The Capelli elements are known as generators of 
$Z(\g \gl (n,\bC ))$ (\textit{cf}. \cite[\S 11]{MR1116239}). 
We define a matrix $\cE=(\cE_{i,j})_{1\leq i,j\leq n}$ of size $n$ 
with entries in $U(\g \gl (n,\bC ))$ by 
\[
\cE_{i,j}=\left\{\begin{array}{ll} 
E_{i,i}-\tfrac{n+1-2i}{2}&\text{ if }i=j,\\[1mm]
E_{i,j}&\text{ if }i\neq j.
\end{array}\right.
\]
Then the Capelli elements $\cC_h\ (1\leq h\leq n)$ are defined by 
the identity 
\begin{align*}
\mathrm{Det}(t1_n+\cE )
=t^n+\sum_{h=1}^n\cC_ht^{n-h}
\end{align*}
in a variable $t$. Here $\mathrm{Det}$ means the vertical determinant 
defined by 
\begin{align*}
&\mathrm{Det}(X)=\sum_{w\in \gS_n}\sgn (w)
X_{1,w(1)}X_{2,w(2)}\cdots X_{n,w(n)},&
&X=(X_{i,j})_{1\leq i,j\leq n}
\end{align*}
with the symmetric group $\gS_n$ of degree $n$. 
For $1\leq h\leq n$, 
the explicit expression of $\cC_h$ is 
given as follows:
\begin{align}
\label{eqn:Fn_capelli}
\cC_h=\underset{w\in \gS_h}
{\sum_{1\leq i_1<i_2<\cdots <i_h\leq n,}}
\sgn (w)\cE_{i_1,i_{w(1)}}\cE_{i_2,i_{w(2)}}\cdots 
\cE_{i_h,i_{w(h)}}. 
\end{align}

\section{The gamma function and the Bessel functions}
\label{subsec:Fn_special1}

For later use, 
we recall basic facts of the gamma function and the modified Bessel functions. 
Proofs of these facts are found in the standard textbooks, 
for example, in \cite{Whittaker_Watson_001} and \cite[\S 1.9]{Bump_004}.  

The gamma function $\Gamma (s)$ is defined by 
\[
\Gamma (s)=\int_{0}^\infty \exp (-t)t^s\frac{dt}{t}
\]
for $\mathrm{Re}(s)>0$, and 
has the meromorphic continuation to $\bC$. 
The Gamma function $\Gamma (s)$ 
is holomorphic on $\bC -\{0,-1,-2,\cdots \}$ and 
has a simple pole at $s=-m$ with the residue 
\[
\mathop{\rm Res}_{s=-m}
\Gamma (s)=\frac{(-1)^m}{m!}
\]
for $m\in \bZ_{\geq 0}$. 
As usual, we set 
\begin{align*}
&\Gamma_\bR (s)=\pi^{-s/2}\Gamma (s/2),&
&\Gamma_\bC (s)=2(2\pi )^{-s}\Gamma (s). 
\end{align*}

The wellknown functional equation $\Gamma (s+1)=s\Gamma (s)$ implies 
\begin{align}
\label{eqn:Fn_FE_GammaRC}
&\Gamma_\bR (s+2)=(2\pi)^{-1}s\Gamma_\bR (s),&
&\Gamma_\bC (s+1)=(2\pi)^{-1}s\Gamma_\bC (s).
\end{align}
Moreover, the duplication formula   
\begin{align}
&\label{eqn:Fn_gamma_duplication}
\Gamma (s)\Gamma \left(s+\tfrac{1}{2}\right)=
2^{1-2s}\sqrt{\pi}\Gamma (2s)
\end{align}
implies 
\begin{align}
\label{eqn:Fn_gammaRC_duplication}
&\Gamma_\bR (s)\Gamma_\bR (s+1)=\Gamma_\bC (s).
\end{align}

For $a\in \bC$ and $i\in \bZ_{\geq 0}$, 
we introduce the Pochhammer symbol $(a)_i$, 
which is defined by 
$(a)_i=a(a+1)\cdots (a+i-1)$. 
It is easy to see that 
\begin{align}
\label{eqn:Fn_Pochhammer_rel}
&(a)_i=\frac{\Gamma (a+i)}{\Gamma (a)}
=(-1)^i\frac{\Gamma (1-a)}{\Gamma (1-a-i)}
\end{align}
for $a\in \bC$ and $i\in \bZ_{\geq 0}$.

Next, we recall the modified Bessel functions. 
Let $r\in \bC$. We call 
\begin{align}
\label{eqn:Fn_bessel_DE}
&\left(z^2\frac{d^2}{dz^2}
+z\frac{d}{dz}-r^2-z^2\right){f}(z)=0
\end{align}
the modified Bessel differential equation. 
The space of smooth solutions of (\ref{eqn:Fn_bessel_DE}) on $\bR_+$ 
is $2$ dimensional, and contains 
the unique moderate growth solution 
\begin{align}
\label{eqn:def_K_Bessel}
&K_r(z)=\frac{1}{2}\int_{0}^\infty 
\exp \left(-\frac{z}{2}(t+t^{-1})\right)
t^r\frac{dt}{t}&
&(z\in \bR_+)
\end{align}
up to scalar multiple. 
The function $K_r(z)$ is called the modified Bessel function 
of the second kind. 
The Mellin transform of $K_r(z)$ is given by
\begin{align}
\label{eqn:Mellin_K_Bessel}
\int^\infty_0K_r(z)z^s\frac{dz}{z}
=2^{s-2}
\Gamma \left(\frac{s+r}{2}\right)
\Gamma \left(\frac{s-r}{2}\right)
\end{align}
for $\mathrm{Re}(s)>|\mathrm{Re}(r)|$. 
Applying the Mellin-inversion formula, we have 
\begin{align}
\label{eqn:Barnes_K_Bessel}
K_r(z)=\frac{1}{4}\cdot \frac{1}{2\pi \sI}
\int_{s}
\Gamma \left(\frac{s+r}{2}\right)
\Gamma \left(\frac{s-r}{2}\right)
\left(\frac{z}{2}\right)^{-s}ds.
\end{align} 
Here the path of integration $\int_{s}$ is the vertical line 
from $\mathrm{Re}(s)-\sI \infty$ to $\mathrm{Re}(s)+\sI \infty$ 
with the real part $\mathrm{Re}(s)>|\mathrm{Re}(r)|$.

Assume $r\not\in \bZ$, and we set  
\begin{align*}
&\hat{I}_r(z)=\frac{-\pi}{\sin (\pi r)}I_r(z)
=\sum_{k=0}^\infty \frac{(-1)^k}{k!}
\Gamma (-r-k)\left(\frac{z}{2}\right)^{r+2k}&
&(z\in \bR_+)
\end{align*}
with the modified Bessel function $I_r(z)$ of the first kind. 
Then $\{\hat{I}_r,\hat{I}_{-r}\}$ is a basis of 
the space of smooth solutions of (\ref{eqn:Fn_bessel_DE}) on $\bR_+$, 
and satisfies 
\begin{align}
\label{eqn:expansion_K_Bessel}
K_r(z)&=\frac{1}{2}\{\hat{I}_{r}(z)+\hat{I}_{-r}(z)\}.
\end{align}

\section{Special functions of two variables}
\label{subsec:Fn_special2}

In order to describe the explicit formulas of Whittaker functions 
on $GL(3,F)$, 
we prepare some special functions of two variables. 
Let $r=(r_1,r_2,r_3)\in \bC^3$. 
In this section, we consider the system of partial differential equations 
\begin{align}
\begin{split}
\label{eqn:Fn_Sol_PDE2}
&\{-\partial_{z_1}^2+\partial_{z_1}\partial_{z_2}-\partial_{z_2}^2
+(r_1+r_2+r_3)(\partial_{z_1}-\partial_{z_2})\\
&-r_1r_2-r_1r_3-r_2r_3
+4z_1^2+4z_2^2\}{f}(z_1,z_2)
=0,
\end{split}\\[1mm]
&\label{eqn:Fn_Sol_PDE3}
\{(\partial_{z_1}-r_1)(\partial_{z_1}-r_2)(\partial_{z_1}-r_3)
-4z_1^2(\partial_{z_1}+\partial_{z_2}+2)\}f(z_1,z_2)=0,
\end{align}
where $\displaystyle 
\partial_{z_i}=z_i\frac{\partial}{\partial z_i}$ ($i=1,2$). 
Let $\mathrm{Sol}(r)$ be 
the space of smooth functions ${f}$ on $(\bR_+)^2$, 
which satisfy 
the system of partial differential equations (\ref{eqn:Fn_Sol_PDE2}) 
and (\ref{eqn:Fn_Sol_PDE3}). 

\begin{lem}
\label{lem:Fn_Sol_dim}
For $r\in \bC^3$, 
it holds that 
$\dim_\bC \mathrm{Sol}(r)\leq 6$.
\end{lem}
\begin{proof}
For $f\in \mathrm{Sol}(r)$, 
we define the functions 
$f_{i}\ (0\leq i\leq 5)$ by 
\begin{align*}
&f_{j+3k}(z_1,z_2)
=\partial_{z_1}^j\partial_{z_2}^kf(z_1,z_2)
&&(0\leq j\leq 2,\ 0\leq k\leq 1). 
\end{align*}
Then we know from (\ref{eqn:Fn_Sol_PDE2}) and (\ref{eqn:Fn_Sol_PDE3}) 
that there exists some polynomial functions 
$m_{i,j}^1(z_1,z_2)$, $m_{i,j}^2(z_1,z_2)$ $(0\leq i,j\leq 5)$ such that 
\begin{align*}
&\partial_{z_k}f_{i}(z_1,z_2)
=\sum_{0\leq j\leq 5}m_{i,j}^k(z_1,z_2)
f_{j}(z_1,z_2)&
&(k=1,2).
\end{align*}
Applying \cite[Theorems B.8 and B.9]{Knapp_002}, 
we have $\dim_\bC \mathrm{Sol}(r)\leq 6$. 
\end{proof}

The holonomic system of partial differential equations 
(\ref{eqn:Fn_Sol_PDE2}) and (\ref{eqn:Fn_Sol_PDE3}) 
has regular singularities along two divisors $z_1=0$ and $z_2=0$ 
which are of simple normal crossing at $(z_1,z_2)=(0,0)$. 
Now we consider the power series solutions of 
this holonomic system at the point $(z_1,z_2)=(0,0)$.

\begin{lem}
\label{lem:Fn_Sol_power_series}
Let $r=(r_1,r_2,r_3)\in \bC^3$ such that 
$r_p-r_q \notin 2\bZ$ for any $1\leq p\neq q\leq 3$. 
For a permutation $(i,j,k)$ of $\{1,2,3\}$, set  
\begin{align*}
&f^{{(i,j,k)}}_r(z_1,z_2) =
\sum_{m_1,m_2 \geq 0}
C_{m_1,m_2}^{(i,j,k)}z_1^{2m_1+r_i} z_2^{2m_2-r_j}&
&(z_1,z_2\in \bR_+)
\end{align*}
with 
\begin{align*}
C_{m_1,m_2}^{(i,j,k)}=
&\frac{(-1)^{m_1+m_2}\Gamma \bigl(-m_1-\tfrac{r_i-r_j}{2}\bigr)
\Gamma \bigl(-m_1-\tfrac{r_i-r_k}{2}\bigr)}
{m_1!m_2!\Gamma \bigl(-m_1-m_2-\tfrac{r_i-r_j}{2}\bigr)}\\
&\times 
\Gamma \bigl(-m_2-\tfrac{r_i-r_j}{2}\bigr)
\Gamma \bigl(-m_2-\tfrac{r_k-r_j}{2}\bigr).
\end{align*}
Then $\{f^{{(i,j,k)}}_r\mid \{i,j,k\}=\{1,2,3\}\, \}$ 
is a basis of $\mathrm{Sol}(r)$. 
\end{lem} 
\begin{proof}
For a power series 
\begin{align*}
f(z_1,z_2)= 
\sum_{m_1,m_2 \geq 0} C_{m_1,m_2}
z_1^{2m_1+\gamma_1} z_2^{2m_2+\gamma_2}&
&(C_{0,0} \neq 0)
\end{align*}
with a characteristic index $(\gamma_1,\gamma_2)$, 
it is easy to see that (\ref{eqn:Fn_Sol_PDE2}) and (\ref{eqn:Fn_Sol_PDE3}) 
can be translated into 
\begin{align}
\label{eqn:Fn_Sol_rec-i1}
\begin{split} 
&\{ -(2m_1+\gamma_1)^2+(2m_1+\gamma_1)(2m_2+\gamma_2)-(2m_2+\gamma_2)^2\\
&+(r_1+r_2+r_3)(2m_1+\gamma_1-2m_2-\gamma_2)-r_1r_2-r_2r_3-r_3r_1\} C_{m_1,m_2}\\
&+4C_{m_1-1,m_2}+ 4C_{m_1,m_2-1}= 0
\end{split} \end{align}
and 
\begin{align}
\label{eqn:Fn_Sol_rec-i2}
\begin{split}
&(2m_1+\gamma_1-r_1)(2m_1+\gamma_1-r_2)(2m_1+\gamma_1-r_3)
C_{m_1,m_2}\\
&-4(2m_1+\gamma_1+2m_2+\gamma_2)C_{m_1-1,m_2}= 0,
\end{split} \end{align}
respectively. Here we put $C_{m_1,m_2}=0$ 
if $m_1<0$ or $m_2<0$. 
Since $ C_{0,0} \neq 0 $, 
these equalities with $(m_1,m_2)= (0,0) $ 
show that 
\[
-\gamma_1^2+\gamma_1\gamma_2-\gamma_2^2+(r_1+r_2+r_3)
(\gamma_1-\gamma_2)-r_1r_2-r_2r_3-r_3r_1= 0, 
\]
and 
\[
(\gamma_1-r_1)(\gamma_1-r_2)(\gamma_1-r_3) = 0. 
\]
We can easily find that the characteristic index $(\gamma_1,\gamma_2)$ 
is of the form 
\begin{align*}
&(\gamma_1,\gamma_2)=(r_i,-r_j)&&(1 \le i \neq j \le 3). 
\end{align*}
It is convenient to take $1\leq k\leq 3$ 
so that $(i,j,k)$ is a permutation of $\{1,2,3\}$. 
Then the recurrence relations (\ref{eqn:Fn_Sol_rec-i1}) and 
(\ref{eqn:Fn_Sol_rec-i2}) can be written as follows:
\begin{align}
\begin{split}\label{eqn:Fn_Sol_rec-i1'}
&\bigl(m_1^2-m_1m_2+m_2^2+\tfrac{r_i-r_k}{2} m_1
+ \tfrac{r_k-r_j}{2} m_2\bigr)
C_{m_1,m_2}\\
&=C_{m_1-1,m_2}+C_{m_1,m_2-1} 
\end{split}
\end{align}
and
\begin{align} 
\begin{split}\label{eqn:Fn_Sol_rec-i2'}
 &m_1\bigl(m_1+\tfrac{r_i-r_j}{2}\bigr)
\bigl(m_1+\tfrac{r_i-r_k}{2}\bigr) C_{m_1,m_2}
= \bigl(m_1+m_2+\tfrac{r_i-r_j}{2}\bigr) C_{m_1-1,m_2}. 
\end{split}
\end{align}
Using the equalities 
\begin{align*}
&\frac{C_{m_1-1,m_2}^{(i,j,k)}}
{C_{m_1,m_2}^{(i,j,k)}}
 = \frac{m_1\bigl(m_1+\frac{r_i-r_j}{2}\bigr)
\bigl(m_1+\frac{r_i-r_k}{2}\bigr)}{m_1+m_2+\frac{r_i-r_j}{2}}, \\
&\frac{C_{m_1,m_2-1}^{(i,j,k)}}
{ C_{m_1,m_2}^{(i,j,k)}}
 = \frac{m_2\bigl(m_2+\frac{r_i-r_j}{2}\bigr)
\bigl(m_2+\frac{r_k-r_j}{2}\bigl)}{m_1+m_2+\frac{r_i-r_j}{2}},
\end{align*}
it is easy to show that 
$C_{m_1,m_2}=C_{m_1,m_2}^{(i,j,k)}$
satisfy (\ref{eqn:Fn_Sol_rec-i1'}) and (\ref{eqn:Fn_Sol_rec-i2'}). 
Hence, we have $f^{(i,j,k)}_r\in \mathrm{Sol}(r)$. 
Since $\{f^{{(i,j,k)}}_r\mid 
\{i,j,k\}=\{1,2,3\}\, \}$ is linearly independent, 
we obtain the assertion by Lemma \ref{lem:Fn_Sol_dim}.
\end{proof}

\begin{lem}
\label{lem:Fn_Sol_mg}
Let $r=(r_1,r_2,r_3)\in \bC^3$. Set  
\begin{align*}
&f^{\mathrm{mg}}_r(z_1,z_2)
=\frac{1}{(4\pi \sqrt{-1})^2} \int_{s_2}\int_{s_1} 
\cV (s_1,s_2) \,z_1^{-s_1} z_2^{-s_2} \,ds_1ds_2&
&(z_1,z_2\in \bR_+)
\end{align*}
with 
\begin{align*}
\cV (s_1,s_2)=&
\frac{\Gamma \bigl(\tfrac{s_1+r_1}{2}\bigr)
\Gamma \bigl(\tfrac{s_1+r_2}{2}\bigr)\Gamma \bigl(\tfrac{s_1+r_3}{2}\bigr) 
\Gamma \bigl(\tfrac{s_2-r_1}{2}\bigr)
\Gamma \bigl(\tfrac{s_2-r_2}{2}\bigr)\Gamma \bigl(\tfrac{s_2-r_3}{2}\bigr)}
{ \Gamma \bigl(\tfrac{s_1+s_2}{2}\bigr)}.
\end{align*}
Here the path of the integration $\int_{s_i}$ is the vertical line 
from $\mathrm{Re}(s_i)-\sI \infty$ to $\mathrm{Re}(s_i)+\sI \infty$ 
with sufficiently large real part to keep the poles of the integrand 
on its left. 
Then $f_r^{\mathrm{mg}}$ is 
a moderate growth function in $\mathrm{Sol}(r)$. 
Moreover, 
if $r_p-r_q \notin 2\bZ$ for any 
$1\leq p\neq q\leq 3$, 
then it holds that 
\begin{align}
\label{eqn:Fn_Sol_sol_expansion}
f_r^{\mathrm{mg}}
& = \sum_{ \{i,j,k\}=\{1,2,3\} } f^{(i,j,k)}_r,
\end{align}
where $f_r^{(i,j,k)}$ are the functions in 
Lemma \ref{lem:Fn_Sol_power_series}. 
\end{lem} 
\begin{proof}
The Stirling formula for the Gamma function 
(\cite[\S 13.6]{Whittaker_Watson_001}) shows that 
$f^{\mathrm{mg}}_{r}$ is 
well-defined and is a moderate growth function. 
Since $\partial_{z_1}^{p}\partial_{z_2}^{q}f^{\mathrm{mg}}_{r}\ 
(p,q\geq 0)$ are entire functions of $(r_1,r_2,r_3)$ 
and $f=f^{(i,j,k)}_{r}$ satisfies the system of partial differential equations 
(\ref{eqn:Fn_Sol_PDE2}) and (\ref{eqn:Fn_Sol_PDE3}), 
in order to complete a proof, 
it suffices to show the expansion formula 
(\ref{eqn:Fn_Sol_sol_expansion}) 
under the assumption 
$r_p-r_q \notin 2\bZ$ $(1\leq p\neq q\leq 3)$. 
By shifting the paths of integrations, we have 
\begin{align*}
&f^{\mathrm{mg}}_{r}(z_1,z_2)\\
&=
\sum_{\{i,j,k\}=\{1,2,3\}}\sum_{\ m_1,m_2\geq 0}
\frac{1}{4}\,{\rm Res}_{(s_1,s_2)=(-2m_1-r_i,-2m_2+r_j)}
\cV (s_1,s_2)z_1^{-s_1}z_2^{-s_2}.
\end{align*}
Since 
$\displaystyle 
\Res_{t=-2m} \Gamma \bigl(\tfrac{t}{2}\bigr) =
\frac{2(-1)^m}{m!}$ for $ m \in \bZ_{\geq 0} $, we find that 
\begin{align*}
& {\rm Res}_{(s_1,s_2)=(-2m_1-r_i,-2m_2+r_j)}
\cV (s_1,s_2) z_1^{-s_1}z_2^{-s_2}
=4C_{m_1,m_2}^{(i,j,k)}z_1^{2m_1+r_i}z_2^{2m_2-r_j},
\end{align*}
and obtain the expansion formula 
(\ref{eqn:Fn_Sol_sol_expansion}). 
\end{proof}

\chapter{Preliminaries for $GL(n,\bR)$}
\label{sec:Rn_rep_theory}

Throughout this chapter, we set $F=\bR $.

\section{Generalized principal series representations}
\label{subsec:Rn_def_gps}
We shall specify certain representations of 
$GL(1,\bR)$ and $GL(2,\bR)$ as follows. 
\begin{itemize}
\item 
For $\nu \in \bC$ and $\delta \in \{0,1\}$, 
we define a one dimensional representation 
$(\chi_{(\nu ,\delta )},\bC_{(\nu ,\delta )})$ of $GL(1,\bR)$ by 
$\bC_{(\nu ,\delta )}=\bC $ and 
\begin{align*}
\hspace{10mm}&\chi_{(\nu ,\delta )}(t)
=\left(\frac{t}{|t|}\right)^{\delta }|t|^{\nu }&
&(t\in GL(1,\bR)=\bR^\times ).
\end{align*}
\item 
For $\nu \in \bC$ and $\kappa \in \bZ_{\geq 2}$, 
let $(D_{(\nu , \kappa )},\gH_{(\nu , \kappa )})$ be 
an irreducible Hilbert representation of $GL(2,\bR)$ 
such that 
$D_{(\nu ,\kappa )}(t1_2)=t^{2\nu }\ (t\in \bR_+)$ and 
$D_{(\nu ,\kappa )}\simeq D^+_{\kappa }\oplus D^-_{\kappa }$ 
as $(\gs \gl (2,\bR ),SO(2))$-modules, where 
$D^\pm_{\kappa }$ is the 
discrete series representations of $SL(2,\bR )$ with 
the minimal $SO(2)$-type 
\[
\hspace{10mm}SO(2)\ni k_\theta^{(2)} =
\left( \begin{array}{cc}
\cos \theta &\sin \theta \\
-\sin \theta &\cos \theta
\end{array}\right) 
\mapsto e^{\pm \sI \kappa \theta }\in \bC^\times .
\]
Such representation $D_{(\nu , \kappa )}$ is unique 
up to infinitesimal equivalence. 
\end{itemize}

Let us define generalized principal series representations of $G=GL(n,\bR)$. 
Let $\mn =(n_1,n_2,\cdots ,n_l)$ be a partition of $n$ into $1$'s and $2$'s, 
that is, $n_i\in \{1,2\}$ ($i=1,2,\cdots ,l$) 
and $n_1+n_2+\cdots +n_l=n$. 
For this partition, we associate the block diagonal subgroup 
\begin{align*}
M_{\mn }&=\left\{\left.
\left(\begin{array}{cccc}
g_1&&& \\
&g_2&& \\
&&\ddots & \\
&&&g_l
\end{array}\right)\in {G}\ 
\right| \ 
g_i\in GL(n_i,\bR)\quad 
(1\leq i\leq l)\ 
\right\},
\end{align*}
the block strictly upper triangular subgroup 
$N_{\mn }$, 
and the block upper triangular subgroup 
$P_{\mn}=N_{\mn}M_{\mn}$. 
For $1\leq i\leq l$, 
let $\sigma_i=\chi_{(\nu_i,\delta_i)}$ if $n_i=1$ 
and $\sigma_i=D_{(\nu_i,\kappa_i)}$ if $n_i=2$. 
Then $(\sigma_1,\sigma_2,\cdots ,\sigma_l)$ defines 
a Hilbert representation $(\sigma ,U_\sigma)$ of 
$M_\mn \simeq GL(n_1,\bR)\times GL(n_2,\bR)\times \cdots 
\times GL(n_l,\bR)$ 
by the tensor product 
$\sigma =\sigma_1\boxtimes \sigma_2\boxtimes \cdots \boxtimes \sigma_l$, 
and we extend this representation 
to $P_\mn =N_\mn M_\mn $ by 
\begin{align*}
&\sigma (xm)=\sigma (m)&(x\in N_{\mn},\ m\in M_\mn). 
\end{align*}
We define the function $\rho_{\mn} $ on $P_{\mn}$ by 
\begin{align*}
&\rho_{\mn} (p)=|\det (\Ad_{\gn_{\mn}} (p))|^{\frac{1}{2}}&&(p\in P_{\mn}), 
\end{align*}
where $\Ad_{\gn_{\mn}} $ means the adjoint action 
on the Lie algebra $\gn_{\mn}$ of $N_{\mn}$.

Let $H(\sigma )^0$ be the space 
of $U_\sigma$-valued continuous functions $f$ on $K$ 
satisfying 
\begin{align*}
&f(mk)=\sigma (m)f(k)&
&(m\in {K}\cap M_{\mn},\ k\in {K}),
\end{align*}
on which ${G}$ acts by 
\begin{align*}
&(\Pi_\sigma (g)f)(k)=\rho_{\mn} (\mmp (kg))
\sigma (\mmp (kg))f(\mk (kg))&
&(g\in {G},\ k\in {K},\ f\in H(\sigma )^0),
\end{align*}
where $kg=\mmp (kg)\mk (kg)$ is the decomposition of $kg$ 
with respect to the decomposition ${G}=P_{\mn}{K}$. 
Although the decomposition $kg=\mmp (kg)\mk (kg)$ is not unique, 
the action $\Pi_\sigma (g)$ does not depend on the choices of 
$\mmp (kg)$ and $\mk (kg)$. 
We define a generalized principal series representation 
$(\Pi_\sigma ,H(\sigma ))$ of $G$ as the completion of 
$(\Pi_\sigma ,H(\sigma )^0)$ 
with respect to the inner product 
\begin{align*}
&\langle f_1,f_2\rangle_{H({\sigma})}
=\int_{{K}}\langle f_1(k),f_2(k)\rangle_{{\sigma}}dk&
&(f_1,f_2\in H(\sigma )^0),
\end{align*}
where $dk$ is the Haar measure on ${K}$, and 
$\langle \cdot ,\cdot \rangle_{{\sigma}}$ is the inner product 
on the Hilbert space $U_\sigma$. 
If $l=n$, that is, $\sigma $ is of the form 
\[
\sigma =\chi_{(\nu_1,\delta_1)}\boxtimes \chi_{(\nu_2,\delta_2)}\boxtimes 
\cdots \boxtimes \chi_{(\nu_n,\delta_n)}, 
\]
then we call $(\Pi_\sigma ,H(\sigma ))$ 
a principal series representation.

Any generalized principal series representation of $G$ 
can be regarded as a subrepresentation 
of a principal series representation (as in \S \ref{subsec:R2_ds_whittaker}, 
\S \ref{subsec:R3_gps}), and the quotient representation 
is not large in the sense of Vogan. 
Hence, combining the results of Kostant 
\cite[Theorems 5.5 and 6.6.2]{Kostant_001} and 
Matumoto \cite[Corollary 2.2.2, Theorem 6.1.6]{Matumoto_001} 
together with the standard arguments, we have
\begin{align}
\label{eqn:Rn_dimWh_gps}
&\dim_\bC  {\cI}_{\Pi_\sigma ,\psi_1}=n!,&
&\dim_\bC  {\cI}_{\Pi_\sigma ,\psi_1}^{\rm mg}=1
\end{align} 
for any generalized principal series 
representation $\Pi_\sigma$ of $G=GL(n,\bR )$. 
Moreover, 
by Vogan's characterization \cite[Theorem 6.2 (f)]{Vogan_001}, 
any irreducible admissible large representation of ${G}$ is 
infinitesimally equivalent 
to some irreducible generalized principal series representation 
$\Pi_\sigma $.

For an element $w$ of the symmetric group $\mathfrak{S}_l$ of degree $l$, 
we set 
\[
w(\sigma)=\sigma_{w^{-1}(1)}\boxtimes \sigma_{w^{-1}(2)}
\boxtimes \cdots \boxtimes \sigma_{w^{-1}(l)}.
\]
When $\Pi_\sigma $ is irreducible, 
it is known that 
$\Pi_\sigma \simeq \Pi_{w(\sigma)}$ 
for any $w\in \mathfrak{S}_l$ 
(\textit{cf.} \cite[Corollary 2.8]{Speh_Vogan_001}). 
Hence, it suffices to consider Whittaker functions for 
irreducible generalized principal series 
representations $\Pi_\sigma $ defined from 
$\sigma$ of the form 
\begin{align*}
D_{(\nu_1,\kappa_1)}\boxtimes D_{(\nu_2,\kappa_2)}\boxtimes \cdots 
\boxtimes D_{(\nu_i,\kappa_i)}\boxtimes \chi_{(\nu_{i+1},\delta_{i+1})}
\boxtimes \chi_{(\nu_{i+2},\delta_{i+2})}\boxtimes \cdots 
\boxtimes \chi_{(\nu_l,\delta_l)}
\end{align*}
with $\kappa_1\geq \kappa_2\geq \cdots \geq \kappa_i$ and 
$\delta_{i+1}\geq \delta_{i+2}\geq \cdots \geq \delta_l$.

We denote by $H(\sigma)_K$ the subspace of $H(\sigma )$ 
consisting of all ${K}$-finite vectors. 
We denote the action of $\g_\bC$ on $H(\sigma)_K$ 
induced from $\Pi_\sigma$ by the same symbol $\Pi_\sigma$. 
The $K$-module structure of $\Pi_\sigma$ is wellknown. 
Let $(\tau ,V_\tau )$ be an irreducible representation of ${K}$. 
By the Frobenius reciprocity law \cite[Theorem 1.14]{Knapp_002}, 
the $\bC$-linear map $\Hom_{{K}\cap M_{\mn}}(V_\tau ,U_\sigma )
\ni \eta \mapsto 
\hat{\eta} \in \Hom_{{K}}(V_\tau ,H(\sigma ))$ defined by 
\begin{align*}
&\hat{\eta}(v)(k)=\eta (\tau (k)v)&&(v\in V_\tau ,\ k\in {K})
\end{align*}
is an isomorphism of $\bC$-vector spaces. 
Since any Hilbert representation of $K$ is completely reducible, we have 
\begin{align}
\label{eqn:Rn_ps_Kstr}
&H({\sigma})_K=\bigoplus_{(\tau ,V_\tau )}
\bC \textrm{-span}
\{\hat{\eta}(v)\mid \eta \in \Hom_{{K}\cap M_{\mn}}(V_\tau ,U_\sigma ),\ 
v\in V_\tau\}. 
\end{align}
Here $(\tau ,V_\tau )$ runs through the set of equivalence classes of 
irreducible representations of ${K}$.

\section{The elements of $\g_\bC$ and $Z(\g_\bC)$}
\label{subsec:Rn_g_gen_Zg}

We define the elements 
$E^{{\g}}_{i,j}=E^{{\g}_n}_{i,j}$, 
$E^{{\gk}}_{i,j}=E^{{\gk}_n}_{i,j}$ 
and $E^{{\gp}}_{i,j}=E^{{\gp}_n}_{i,j}$ of $\g_\bC$ by 
\begin{align*}
&E^{{\g}}_{i,j}=E_{i,j}\otimes 1,&
&E^{{\gk}}_{i,j}=E^{{\g}}_{i,j}-E^{{\g}}_{j,i},&
&E^{{\gp}}_{i,j}=E^{{\g}}_{i,j}+E^{{\g}}_{j,i}
\end{align*}
for $1\leq i,j\leq n$. Then we have 
\begin{align*}
&\g_\bC =\bigoplus_{1\leq i,j\leq n}\bC\,E^{{\g}}_{i,j},&
&\gn_\bC =\bigoplus_{1\leq i<j\leq n}\bC\,E^{{\g}}_{i,j},&
&\ga_\bC =\bigoplus_{i=1}^n\bC\,E^{{\gp}}_{i,i},\\
&\gk_\bC =\bigoplus_{1\leq i<j\leq n}\bC\,E^{{\gk}}_{i,j},&
&\gp_\bC =\bigoplus_{1\leq i\leq j\leq n}\bC\,E^{{\gp}}_{i,j}.
\end{align*}

\begin{lem}
\label{lem:Rn_g_act_Cpsi}
Let $f$ be a function in $C^\infty (N\backslash G;\psi_\varepsilon )$ with 
$\varepsilon \in \{\pm 1\}$. 
Then, for $1\leq i\leq j\leq n$ and 
$y=\diag (y_1y_2\cdots y_n,\,y_2\cdots y_n,\,\cdots ,\,y_n)\in A$, 
it holds that 
\begin{align*}
&(R(E_{i,j}^{{\g}})f)(y)
=\left\{\begin{array}{ll}
(-\partial_{i-1}+\partial_i)f(y)&\text{if}\ j=i,\\[2pt]
2\pi \varepsilon \sI y_if(y)&\text{if}\ j=i+1,\\[2pt]
0&\text{otherwise},
\end{array}\right.
\end{align*}
where $\partial_{0}=0$ and 
$\displaystyle \partial_i=y_i\frac{\partial}{\partial y_i}$ 
$(1\leq i\leq n)$. 
\end{lem}
\begin{proof}
Direct computation.
\end{proof}

The $\bC$-algebra $Z(\g \gl (n,\bC))$ is generated by 
the Capelli elements $\cC_h\ (1\leq h\leq n)$, 
which are introduced in \S \ref{subsec:Fn_capelli}. 
Since ${\g_\bC}=\g \gl (n,\bR)\otimes_{\bR}\bC 
\simeq \g \gl (n,\bC)$ via $X\otimes c \mapsto cX$, 
the center $Z({\g_\bC})$ of $U({\g_\bC})$ is generated by elements 
\begin{align*}
&\cC_h^{{\g}}=\underset{w\in \gS_h}
{\sum_{1\leq i_1<i_2<\cdots <i_h\leq n,}}
\sgn (w)\cE_{i_1,i_{w(1)}}^{{\g}}\cE_{i_2,i_{w(2)}}^{{\g}}\cdots 
\cE_{i_h,i_{w(h)}}^{{\g}}
&&(1\leq h\leq n)
\end{align*}
with 
\[
\cE_{i,j}^{{\g}}=\left\{\begin{array}{ll} 
E_{i,i}^{{\g}}-\tfrac{n+1-2i}{2}&\text{ if }i=j,\\[1mm]
E_{i,j}^{{\g}}&\text{ if }i\neq j.
\end{array}\right.
\]
For $1\leq h\leq n$, we have 
\begin{align}
\label{eqn:Rn_Ch_modN}
\cC_h^{{\g}}\equiv 
\sum_{1\leq i_1<i_2<\cdots <i_h\leq n}
\cE_{i_1,i_1}^{{\g}}\cE_{i_2,i_2}^{{\g}}\cdots \cE_{i_h,i_h}^{{\g}}
\mod \gn_{\bC}U(\g_{\bC}).
\end{align}
Here $X\equiv Y\bmod \cR$ means $X-Y\in \cR$ for $X,Y\in U(\g_{\bC})$ and 
a subspace $\cR$ of $U(\g_{\bC})$.

\section{The eigenvalues of generators of $Z(\g_\bC)$}
\label{subsec:Rn_inf_char_ps}

In this section, let 
$\sigma =\chi_{(\nu_1,\delta_1)}\boxtimes \chi_{(\nu_2,\delta_2)} 
\boxtimes \cdots \boxtimes \chi_{(\nu_n,\delta_n)}$ 
with $\nu_1,\nu_2,\cdots ,\nu_n\in \bC$, 
$\delta_1,\delta_2,\cdots ,\delta_n\in \{0,1\}$, 
and we consider a principal series representation 
$(\Pi_\sigma ,H(\sigma))$. 
We note that $\sigma $ is a character of $P_{(1,1,\cdots ,1)}$, 
and identify the representation space 
\[
U_\sigma =
\bC_{(\nu_1,\delta_1)}\boxtimes_\bC \bC_{(\nu_2,\delta_2)}\boxtimes_\bC 
\cdots \boxtimes_\bC \bC_{(\nu_n,\delta_n)}
\]
of $\sigma$ with $\bC$ via the correspondence 
$c_1\boxtimes c_2\boxtimes \cdots \boxtimes c_n
\leftrightarrow c_1c_2\cdots c_n$. 
It is easy to see that the explicit form of $\rho_{(1,1,\cdots ,1)} $ 
is given by 
\begin{align*}
&\rho_{(1,1,\cdots ,1)}(xm)
=\prod_{i=1}^n|m_i|^{{(n+1-2i)}/{2}}
\end{align*}
for $x\in N_{(1,1,\cdots ,1)}=N$ and 
$m=\diag (m_1,m_2,\cdots ,m_n)\in M_{(1,1,\cdots ,1)}$. 
By the definition of $\Pi_\sigma$, we have 
\begin{align}
\label{eqn:Rn_gact_value1}
&(\Pi_\sigma (X)f)(1_n)=0&&(X\in \gn_{\bC}U(\g_\bC)),\\
\label{eqn:Rn_gact_value2}
&(\Pi_\sigma (E_{i,i}^{{\g}})f)(1_n)
=\bigl(\nu_i+\tfrac{n+1-2i}{2}\bigr)f(1_n)&
&(1\leq i\leq n)
\end{align}
for $f\in H(\sigma)_K$. 
\begin{prop}
\label{prop:Rn_Ch_eigenvalue}
Retain the notation. Then, for $f\in H(\sigma)_K$ and $1\leq h\leq n$, 
it holds that 
\begin{align*}
\Pi_\sigma (\cC_h^{{\g}})f=
\left(\sum_{1\leq i_1<i_2<\cdots <i_h\leq n}
\nu_{i_1}\nu_{i_2}\cdots \nu_{i_h}
\right)f.
\end{align*}
\end{prop}
\begin{proof}
Let $1\leq h\leq n$. 
Let $(\tau ,V_\tau )$ be a $K$-type of $\Pi_\sigma$. 
Let $\{\eta_1,\eta_2,\cdots ,\eta_d\}$ be a basis of 
$\Hom_{{K}\cap M_{(1,1,\cdots ,1)}}(V_\tau ,U_\sigma )$. 
Because of (\ref{eqn:Rn_ps_Kstr}), 
it suffices to show that 
\begin{align*}
&\Pi_\sigma (\cC_h^{{\g}})\circ \hat{\eta}_i=
\left(\sum_{1\leq i_1<i_2<\cdots <i_h\leq n}
\nu_{i_1}\nu_{i_2}\cdots \nu_{i_h}\right)\hat{\eta}_i
&&(1\leq i\leq d)
\end{align*}
with $\hat{\eta}_i(v)(k)=\eta_i(\tau (k)v)\ 
(k\in K,\ v\in V_\tau)$. 
Since 
$\Pi_\sigma (\cC_h^{{\g}})\circ \hat{\eta}_i$ 
is an element of $\Hom_{{K}}(V_\tau ,H(\sigma )_K)$, 
there exist $\gamma_{i,1},\gamma_{i,2},\cdots ,\gamma_{i,d}\in \bC$ such that 
\begin{align*}
\Pi_\sigma (\cC_h^{{\g}})\circ \hat{\eta}_i
=\sum_{j=1}^d\gamma_{i,j}\hat{\eta}_j.
\end{align*}
We take elements $v_{\eta_1},v_{\eta_2},\cdots ,v_{\eta_d}$ of $V_\tau $
such that 
\begin{align*}
\eta_{i}(v_{\eta_j})
=\left\{
\begin{array}{ll}
1&\text{ if }\ i=j,\\
0&\text{ otherwise}
\end{array}
\right.
\end{align*}
for $1\leq i,j\leq d$. 
By the equalities (\ref{eqn:Rn_Ch_modN}), (\ref{eqn:Rn_gact_value1}) 
and (\ref{eqn:Rn_gact_value2}), we have 
\begin{align*}
\gamma_{i,j}=
(\Pi_\sigma (\cC_h^{{\g}})\hat{\eta}_{i}(v_{\eta_j}))(1_n)
=\left\{
\begin{array}{ll}
\displaystyle \sum_{1\leq i_1<i_2<\cdots <i_h\leq n}
\nu_{i_1}\nu_{i_2}\cdots \nu_{i_h}
&\text{ if }\ i=j,\\[5mm]
0&\text{ otherwise}
\end{array}
\right.
\end{align*}
for $1\leq i,j\leq d$, and obtain the assertion. 
\end{proof}

\chapter{Whittaker functions on $GL(2,\bR)$}
\label{sec:R2_whittaker}

Throughout this chapter, we set $n=2$ and $F=\bR $.

\section{Representations of $O(2)$}
\label{subsec:R2_rep_K}
In this section, we discuss the representation theory of ${K}=O(2)$. 

Let $\Lambda_2=\{(0,1)\}\cup \{(\lambda_1,0)\mid \lambda_1\in \bZ_{\geq 0}\}$. 
For $\lambda =(\lambda_1,\lambda_2)\in \Lambda_2$, 
let $\mathcal{P}_\lambda =\mathcal{P}_\lambda^{(2)}$ 
be the $\bC$-vector space of degree $\lambda_1$ 
homogeneous polynomials of two variables ${z}_1,{z}_2$, 
and define the action $T_\lambda $ of $K$ on $\mathcal{P}_\lambda $ by 
\begin{align*}
&(T_{\lambda}(k)p)({z}_1,{z}_2)=(\det k)^{\lambda_2}p(({z}_1,{z}_2)\cdot k)&
&(k\in K,\ p\in \mathcal{P}_\lambda ). 
\end{align*}
Here $({z}_1,{z}_2)\cdot k$ is the ordinal product of matrices. 
Since ${z}_1^2+{z}_2^2$ is ${K}$-invariant, 
we can define the quotient representation 
$\tau_\lambda =\tau_\lambda^{(2)}$ of $T_\lambda$ on the space 
\[
V_{\lambda}=V_{\lambda}^{(2)}=\mathcal{P}_\lambda /
(({z}_1^2+{z}_2^2)\mathcal{P}_{\lambda -(2,0)})
\]
for $\lambda \in \Lambda_2$. 
Here we put $\mathcal{P}_{\lambda -(2,0)}=\{0\}$ 
if $\lambda -(2,0)\not\in \Lambda_2$. 
Then $(\tau_{\lambda },V_{\lambda })$ 
is an irreducible representation of ${K}$, 
and we note that $\dim_\bC  V_\lambda =1$ or $\dim_\bC  V_\lambda =2$ 
according as $\lambda_1=0$ or $\lambda_1>0$. 
The set of equivalence classes of 
irreducible representations of ${K}$ is 
exhausted by $\{\tau_\lambda \mid \lambda \in \Lambda_2\}$. 

For $\lambda=(\lambda_1,\lambda_2)\in \Lambda_2$, 
let $Q_{\lambda}=\{\pm \lambda_1\}$. 
For $q\in Q_{\lambda}$, 
let $v_{\lambda ,q}$ 
be the image of 
\[
(\sgn (q) {z}_1+\sqrt{-1}{z}_2)^{|q|}
\]
under the natural surjection 
$\mathcal{P}_\lambda \to V_{\lambda}$. 
Then $\{v_{\lambda ,q}\mid q\in Q_\lambda \}$ is a basis of $V_{\lambda}$. 
By direct computation, 
for $\theta \in \bR$ and 
$\varepsilon_1,\varepsilon_2\in \{\pm 1\}$, 
we have 
\begin{align}
\label{eqn:R2_Kact}
&\tau_{\lambda }\bigl(k_\theta^{(2)} \bigr)
v_{\lambda,q}=e^{\sI q\theta }v_{\lambda,q},&
&\tau_{\lambda }\bigl(k_{(\varepsilon_1,\varepsilon_2)}^{(1,1)}\bigr)
v_{\lambda,q}
=\varepsilon_1^{\lambda_2}\varepsilon_2^{\lambda_1+\lambda_2} 
v_{\lambda,\,\varepsilon_1\varepsilon_2q}
\end{align}
with 
\begin{align*}
k_\theta^{(2)} &=\left(
\begin{array}{cc}
\cos \theta &\sin \theta \\
-\sin \theta &\cos \theta 
\end{array}
\right) ,&
k_{(\varepsilon_1,\varepsilon_2)}^{(1,1)}&=\left(
\begin{array}{cc}
\varepsilon_1&0\\
0&\varepsilon_2
\end{array}
\right) \in {K}.
\end{align*}
We denote the differential of $\tau_{\lambda }$ 
again by $\tau_{\lambda }$. Then we have 
\begin{align}
\label{eqn:R2_gkact}
&\tau_{\lambda }(E_{1,2}^{\gk} )v_{\lambda,q}=\sI q v_{\lambda,q}.
\end{align}

For $r\in \{\pm 2\}$, we define the element $X_{r}^{\gp}=X_{r}^{\gp_2}$ 
of $\gp_{\bC}$ by 
\begin{align}
&X_{r}^{\gp}=E_{1,1}^{{\gp}}-E_{2,2}^{{\gp}}+r\sI E_{1,2}^{{\gp}}. 
\label{eqn:R2_def_Xr}
\end{align}
Then $\{\cC_1^{\g} ,X_{2}^{\gp},X_{-2}^{\gp}\}$ is 
a basis of $\gp_{\bC}$. 
We regard $\gp_\bC$ as a $K$-module by the adjoint action $\Ad$. 
By direct computation, for $\theta \in \bR$, $\varepsilon_1,\varepsilon_2\in \{\pm 1\}$ 
and $r\in \{\pm 2\}$, we have 
\begin{align}
\label{eqn:R2_Kact_p1}
&\Ad \bigl(k^{(2)}_\theta \bigr)\cC_1^{\g}=
\Ad \bigl(k^{(1,1)}_{(\varepsilon_1,\varepsilon_2)} \bigr)\cC_1^{\g}=
\cC_1^{\g},\\
\label{eqn:R2_Kact_p2}
&\Ad \bigl(k^{(2)}_\theta \bigr)X^{\gp}_{r}=e^{\sI r\theta }X^{\gp}_r,&
&\Ad \bigl(k^{(1,1)}_{(\varepsilon_1,\varepsilon_2)} \bigr)X^{\gp}_{r}
=X^{\gp}_{\varepsilon_1\varepsilon_2r}.&
\end{align}
Comparing (\ref{eqn:R2_Kact_p1}) and (\ref{eqn:R2_Kact_p2}) with 
(\ref{eqn:R2_Kact}), 
we have 
\begin{align*}
\gp_{\bC}=(\bC \, \cC_1^{\g})\oplus 
(\bC \, X^{\gp}_{2}\oplus \bC \, X^{\gp}_{-2})\simeq V_{(0,0)}\oplus V_{(2,0)}
\end{align*}
as $K$-modules, via $\cC_1^{\g}\mapsto v_{(0,0),0}$ and 
$X^{\gp}_{r}\mapsto v_{(2,0),r}$ ($r\in \{\pm 2\}$).

\section{Principal series representations}
\label{subsec:R2_ps}

In this section, 
let $\sigma =\chi_{(\nu_1,\delta_1)}\boxtimes \chi_{(\nu_2,\delta_2)}$ 
with $\nu_1,\nu_2\in \bC$ and $\delta_1,\delta_2\in \{0,1\}$ 
such that $\delta_1\geq \delta_2$. 
We recall the $(\g_{\bC},K)$-module structure 
of a principal series representation 
$(\Pi_\sigma ,H(\sigma ))$ of $G=GL(2,\bR)$. 
We identify the representation space $U_\sigma $ of $\sigma$ 
with $\bC$ as in \S \ref{subsec:Rn_inf_char_ps}.

Let $\Lambda (\sigma)=\{(\delta_1-\delta_2,\delta_2)\}\cup 
\{(\lambda_1,0)\mid \lambda_1\in \delta_1-\delta_2+2\bZ_{\geq 1}\}$. 
Because of 
\[
{K}\cap M_{(1,1)}=
\bigl\{\,k_{(\varepsilon_1,\varepsilon_2)}^{(1,1)}
\,\big| \,
\varepsilon_1,\varepsilon_2\in \{\pm 1\}\,\bigr\} 
\]
and (\ref{eqn:R2_Kact}), 
for $\lambda =(\lambda_1,\lambda_2)\in \Lambda_2$, we have 
\begin{align*}
&\Hom_{{K}\cap M_{(1,1)}}(V_{\lambda},U_{\sigma})
=\left\{\begin{array}{ll}
\bC \,\eta_{(\sigma ;\lambda )}&\text{ if }\ 
\lambda \in \Lambda (\sigma),\\
0&\text{ if }\ 
\lambda \not\in \Lambda (\sigma),
\end{array}\right.
\end{align*} 
where $\eta_{(\sigma ;\lambda )}\colon V_{\lambda }\to U_\sigma $ 
is a $\bC$-linear map defined by 
$\eta_{(\sigma ;\lambda )}(v_{\lambda,q})=\sgn (q)^{\delta_1}$ 
$(q\in Q_\lambda)$. 
Hence, 
we know from (\ref{eqn:Rn_ps_Kstr}) that 
\begin{align*}
&H({\sigma})_K=\bigoplus_{\lambda \in \Lambda (\sigma)}
\{\hat{\eta}_{(\sigma ;\lambda )}(v)\mid 
v\in V_{\lambda }\} 
\end{align*}
with 
\begin{align}
\label{eqn:R2_def_hateta}
&\hat{\eta}_{(\sigma ;\lambda )}(v)(k)=\eta_{(\sigma ;\lambda )}
(\tau_{\lambda} (k)v)&&(v\in V_{\lambda},\ k\in {K}). 
\end{align}
We set 
${\zeta}_{(\sigma;q)}
=\sgn (q)^{\delta_1}\hat{\eta}_{(\sigma ;\lambda )}(v_{\lambda,q})$ 
with $\lambda =(|q|,\lambda_2)\in \Lambda (\sigma)$ 
for $q \in \delta_1-\delta_2+2\bZ$. 
Then we have 
\begin{align}
\label{eqn:R2_ps_str}
&H({\sigma})_K=\bigoplus_{q \in \delta_1-\delta_2+2\bZ}
\bC\,{\zeta}_{(\sigma;q)},
\end{align}
and 
\begin{align}
\label{eqn:R2_Kact1_ps}
&\Pi_{\sigma}(k_\theta^{(2)} ){\zeta}_{(\sigma ;q)}
=e^{\sI q\theta }{\zeta}_{(\sigma ;q)}&&(\theta \in \bR),\\
\label{eqn:R2_Kact2_ps}
&\Pi_\sigma 
(k_{(\varepsilon_1,\varepsilon_2)}^{(1,1)}){\zeta}_{(\sigma ;q)}
=
\varepsilon_1^{\delta_1} \varepsilon_2^{\delta_2} 
{\zeta}_{(\sigma ;\varepsilon_1\varepsilon_2q)}&
&(\varepsilon_1,\varepsilon_2\in \{\pm 1\}),\\
\label{eqn:R2_basis_ps}
&{\zeta}_{(\sigma ;q)}(1_2)=1.
\end{align}
We consider the action of the basis 
$\{E^{\gk}_{1,2},\cC_1^{{\g}},X_2^{\gp},X_{-2}^{\gp}\}$ 
of $\g_\bC$ on $H(\sigma)_K$. We have already given the action 
of $\cC_1^{{\g}}$ in Proposition \ref{prop:Rn_Ch_eigenvalue}. 
By (\ref{eqn:R2_Kact1_ps}), we have 
\begin{align}
\label{eqn:R2_gk_act_ps}
&\Pi_\sigma (E^{\gk}_{1,2}){\zeta}_{(\sigma ;q)}=\sI q\,{\zeta}_{(\sigma ;q)}.
\end{align}
We give the actions of $X_2^\gp$ and $X_{-2}^\gp$ 
in the following lemma.

\begin{lem}
\label{lem:R2_g_act_ps}
Retain the notation. For $q \in \delta_1-\delta_2+2\bZ$, 
it holds that 
\begin{align*}
&\Pi_\sigma (X_{r}^{\gp}){\zeta}_{(\sigma;q)}
=(2\nu_1-2\nu_2+2+rq){\zeta}_{(\sigma;q+r)}&&(r\in \{\pm 2\}).
\end{align*}
\end{lem}
\begin{proof}
Let $r\in \{\pm 2\}$. 
By (\ref{eqn:R2_Kact_p2}) and (\ref{eqn:R2_Kact1_ps}), we have 
\begin{align*}
&\Pi_\sigma (k^{(2)}_\theta )
\Pi_\sigma (X_{r}^{\gp})
{\zeta}_{(\sigma;q)}
=e^{\sI (q+r)\theta }\,\Pi_\sigma (X_{r}^{\gp}){\zeta}_{(\sigma;q)}&
&(\theta \in \bR ).
\end{align*}
From (\ref{eqn:R2_ps_str}), (\ref{eqn:R2_Kact1_ps}) and 
this equality, we know that 
$\Pi_\sigma (X_{r}^{\gp}){\zeta}_{(\sigma;q)}$ is a constant multiple 
of ${\zeta}_{(\sigma ;q+r)}$. 
Since ${\zeta}_{(\sigma ;q+r)}(1_2)=1$ and 
\begin{align*}
(\Pi_\sigma (X_{r}^{\gp}){\zeta}_{(\sigma;q)})(1_2)
&=\bigl(\Pi_\sigma \!
\bigl(2E_{1,1}^{{\g}}-2E_{2,2}^{{\g}}
+r\sI (2E_{1,2}^{{\g}}- E_{1,2}^{{\gk}})\bigr)
{\zeta}_{(\sigma;q)}\bigr)(1_2)\\
&=(2\nu_1+1)-(2\nu_2-1)+0+rq
=2\nu_1-2\nu_2+2+rq,
\end{align*}
we obtain the assertion. 
\end{proof}
From the above formulas of the actions of $\g_{\bC}$ and $K$ 
on $H(\sigma)_K$, we know that 
$H(\sigma )_K$ is an irreducible $(\g_\bC ,K)$-module 
if and only if 
\begin{align}
\label{eqn:R2_cdn_ps_irred}
&\nu_1-\nu_2+1 \not\in \delta_1-\delta_2+2\bZ&
&\text{ or }&
&\nu_1-\nu_2=0. 
\end{align}

\section{Principal series Whittaker functions}
\label{subsec:R2_ps_whittaker}

We use the notation in \S \ref{subsec:R2_ps}. 
We will give the explicit formulas of Whittaker functions 
for $\Pi_\sigma$ with 
$\sigma =\chi_{(\nu_1,\delta_1)}\boxtimes \chi_{(\nu_2,\delta_2)}$ 
$(\delta_1\geq \delta_2)$. 
Let $\Phi \in {\cI}_{\Pi_\sigma ,\psi_1}$. 
We note that $\Phi$ is characterized by 
\begin{align*}
&\Phi ({\zeta}_{(\sigma;q)})(y)&&(q \in \delta_1-\delta_2 +2\bZ)
\end{align*}
with $y=\diag(y_1y_2,y_2)\in A$. 
We set $\partial_i = y_i \dfrac{\partial}{\partial y_i}$ for $i=1,2$.

\begin{lem} Retain the notation. 
The function $\Phi ({\zeta}_{(\sigma;q)})(y) $ on $A$ satisfies
\begin{align}
\begin{split} \label{eqn:R2_ps_PDE1_1}
& (\partial_2-\nu_1-\nu_2) \Phi ({\zeta}_{(\sigma;q)})(y)= 0,  
\end{split} \\[1mm]
\begin{split} \label{eqn:R2_ps_PDE1_2}
& \bigl\{\bigl(\partial_1-\tfrac{1}{2}\bigr)
\bigl(-\partial_1+\partial_2+\tfrac{1}{2}\bigr)
+(2\pi y_1)^2-(2\pi y_1)q-\nu_1\nu_2\bigr\}
\Phi ({\zeta}_{(\sigma;q)})(y)=0,
\end{split}
\end{align}
and 
\begin{align} \label{eqn:R2_ps_PDE1_3}
\begin{split}
& (4 \partial_1 -2 \partial_2-4\pi ry_1+rq)\Phi ({\zeta}_{(\sigma;q)})(y)
 =(2\nu_1-2\nu_2+2+rq) \Phi ({\zeta}_{(\sigma;q+r)})(y) 
\end{split}
\end{align}
for $q\in \delta_1-\delta_2+2\bZ$ and $r\in \{\pm 2\}$. 
\end{lem}
\begin{proof}
Recall that the generators $\cC_1^{{\g}},\cC_2^{{\g}}$ 
of $Z({\g_\bC})$ in \S \ref{subsec:Rn_g_gen_Zg} are given by 
\begin{align}
\label{eqn:R2_C1}
\cC_1^{{\g}}&=E_{1,1}^{{\g}}+E_{2,2}^{{\g}}, \\ 
\begin{split}\label{eqn:R2_C2}
\cC_2^{{\g}}&=
\bigl(E_{1,1}^{{\g}}-\tfrac{1}{2}\bigr)
\bigl(E_{2,2}^{{\g}}+\tfrac{1}{2}\bigr)
-E_{1,2}^{{\g}}E_{2,1}^{{\g}} \\
& =\bigl(E_{1,1}^{{\g}}-\tfrac{1}{2}\bigr)
\bigl(E_{2,2}^{{\g}}+\tfrac{1}{2}\bigr)
-(E_{1,2}^{{\g}})^2 +E_{1,2}^{{\g}}E_{1,2}^{{\gk}}.
\end{split}
\end{align}
Moreover, we can see that
\begin{align}
\label{eqn:R2_Xpm}
X_r^{\gp} = 
2(E_{1,1}^{\g}-E_{2,2}^{\g})+r\sqrt{-1}(2E_{1,2}^{\g}-E_{1,2}^{\gk}).
\end{align}
By Proposition \ref{prop:Rn_Ch_eigenvalue} and Lemma \ref{lem:R2_g_act_ps}, 
we have
\begin{align*}
&(R(\cC_1^{{\g}})\Phi ({\zeta}_{(\sigma;q)}))(y)
  -(\nu_1+\nu_2)\Phi ({\zeta}_{(\sigma;q)})(y)= 0, \\
&(R(\cC_2^{{\g}})\Phi ({\zeta}_{(\sigma;q)}))(y)
-\nu_1\nu_2\Phi ({\zeta}_{(\sigma;q)})(y)= 0, 
\end{align*}
and 
\[ 
(R(X_r^{\gp}) \Phi ({\zeta}_{(\sigma;q)}))(y) 
= (2\nu_1-2\nu_2+2+rq)\Phi ({\zeta}_{(\sigma;q+r)})(y).
\]
Applying (\ref{eqn:R2_gk_act_ps}) 
and Lemma \ref{lem:Rn_g_act_Cpsi} 
to these equalities with 
the expressions (\ref{eqn:R2_C1}), (\ref{eqn:R2_C2}) and 
(\ref{eqn:R2_Xpm}), 
we obtain the assertion. 
\end{proof}

By (\ref{eqn:R2_ps_PDE1_1}), 
for $q\in \delta_1-\delta_2+2\bZ$, 
we can define a function $\hat{\varphi}_q $ on $\bR_+$ by
\begin{align}
\label{eqn:R2_def_hat_varphi}
\Phi ({\zeta}_{(\sigma;q)})(y)
=(\sI )^{q}y_1^{1/2}y_2^{\nu_1+\nu_2}\hat{\varphi}_q(y_1)
\end{align}
with $y=\diag (y_1y_2,y_2)\in A$. 
Then the equations (\ref{eqn:R2_ps_PDE1_2}) and (\ref{eqn:R2_ps_PDE1_3}) 
lead that 
\begin{align} \label{eqn:R2_ps_PDE1_2'}
\{(\partial_1-\nu_1)(\partial_1-\nu_2) -(2\pi y_1)^2 + (2\pi y_1)q \} 
\hat{\varphi}_q=0
\end{align}
and 
\begin{align} \label{eqn:R2_ps_PDE1_3'}
\begin{split} 
&(-4\partial_1+4\pi ry_1+2\nu_1+2\nu_2-2-rq) \hat{\varphi}_q\\
&=(2\nu_1-2\nu_2+2+rq) \hat{\varphi}_{q+r},
\end{split}
\end{align}
respectively, for $q\in \delta_1-\delta_2+2\bZ$ and $r\in \{\pm 2\}$. 
Here we understand $\partial_1 = y_1\dfrac{d}{dy_1}$.

Until the end of this section, we assume 
$\nu_1-\nu_2+\delta_1-\delta_2+1\not\in 2\bZ_{\leq 0}$. Then 
the functions $\hat{\varphi}_q$ $(q\in \delta_1-\delta_2+2\bZ )$ 
are determined from 
$\hat{\varphi}_{\pm (\delta_1-\delta_2)}$ by the equation 
(\ref{eqn:R2_ps_PDE1_3'}). 
Hence, we know that 
$\hat{\varphi}_{\delta_1-\delta_2}=\hat{\varphi}_{\delta_2-\delta_1}=0$ 
if and only if $\Phi =0$. 
We first consider explicit formulas of 
Whittaker functions at the minimal $K$-type 
$\tau_{(\delta_1-\delta_2,\delta_2)}$, which are 
corresponding to $\hat{\varphi}_{\pm (\delta_1-\delta_2)}$.


\begin{thm}
\label{thm:R2_ps_Whittaker}
Let $\sigma = \chi_{(\nu_1,\delta_1)}\boxtimes \chi_{(\nu_2,\delta_2)}$ with 
$\nu_1,\nu_2\in \bC$ and $ \delta_1,\delta_2 \in \{0,1\}$ 
$(\delta_1\geq \delta_2)$ 
such that $\nu_1-\nu_2+\delta_1-\delta_2+1\not\in 2\bZ_{\leq 0}$. \\[1mm] 
(1) There exists a homomorphism 
$\Phi_{\sigma}^{\rm mg} \in {\cI}_{\Pi_\sigma ,\psi_1}^{\rm mg}$ 
with the following radial part.
\begin{itemize}
\item The case $\delta_1=\delta_2$: 
\begin{align*}
\hspace{10mm}
\Phi_{\sigma}^{\rm mg}({\zeta}_{(\sigma;0)})(y) 
&=2y_1^{(\nu_1+\nu_2+1)/2}y_2^{\nu_1+\nu_2}
K_{(\nu_1-\nu_2)/2}(2\pi y_1)\\
&=\frac{y_1^{1/2}y_2^{\nu_1+\nu_2}}{4\pi \sI}
\int_s
\Gamma_\bR (s+\nu_1)\Gamma_\bR (s+\nu_2)y_1^{-s}ds.
\end{align*}

\item The case $(\delta_1,\delta_2)=(1,0)$: 
For $q\in \{\pm1\} $, 
\begin{align*}
\hspace{10mm}
\Phi_{\sigma}^{\rm mg}({\zeta}_{(\sigma;q)})(y)=
& \,2(\sI)^qy_1^{(\nu_1+\nu_2+2)/2}y_2^{\nu_1+\nu_2} \\
& \times 
  \bigl(K_{(\nu_1-\nu_2+1)/2}(2\pi y_1)+q 
K_{(\nu_1-\nu_2-1)/2}(2\pi y_1) \bigr)\\
=&\,(\sI )^{q}\frac{y_1^{1/2}y_2^{\nu_1+\nu_2}}{4\pi \sI}
\int_s
\{\Gamma_\bR (s+\nu_1+1)\Gamma_\bR (s+\nu_2)\\
&+q\Gamma_\bR (s+\nu_1)\Gamma_\bR (s+\nu_2+1)\}
y_1^{-s}ds.
\end{align*}
\end{itemize}
Here $y=\diag (y_1y_2,y_2)\in A $, 
and the path of integration $\int_{s}$ is the vertical line 
from $\mathrm{Re}(s)-\sI \infty$ to $\mathrm{Re}(s)+\sI \infty$ 
with the sufficiently large real part to keep the poles of the integrand 
on its left.  \\
\noindent 
(2) Assume $\nu_1-\nu_2 \notin \delta_1-\delta_2+2\bZ$. 
Then there exist homomorphisms $\Phi_{\sigma}^{+}$ and $\Phi_{\sigma}^{-}$ 
in ${\cI}_{\Pi_\sigma ,\psi_1}$ 
with the following radial parts.
\begin{itemize}
\item The case $\delta_1=\delta_2$:
\begin{align*}
\Phi_{\sigma}^{\pm }({\zeta}_{(\sigma;0)})(y) 
= y_1^{(\nu_1+\nu_2+1)/2}y_2^{\nu_1+\nu_2}
\hat{I}_{\pm (\nu_1-\nu_2)/2}(2\pi y_1).
\end{align*}
\item The case $(\delta_1,\delta_2)=(1,0)$: For $q\in \{\pm 1\} $, 
\begin{align*}
\begin{split}
\hspace{10mm}
\Phi_{\sigma}^{\pm }({\zeta}_{(\sigma;q)})(y) = 
\,& (\sI)^qy_1^{(\nu_1+\nu_2+2)/2}y_2^{\nu_1+\nu_2} \\
& \times 
  \bigl(\hat{I}_{\pm(\nu_1-\nu_2+1)/2}(2\pi y_1)+q 
\hat{I}_{\pm(\nu_1-\nu_2-1)/2}(2\pi y_1) \bigr).
\end{split}
\end{align*}
\end{itemize}
Here $y=\diag (y_1y_2,y_2)\in A $. 
Moreover, $\{\Phi_{\sigma}^{+},\Phi_{\sigma}^{-}\}$ forms a basis of
${\cI}_{\Pi_\sigma ,\psi_1}$ and satisfies 
$\Phi_{\sigma}^{\rm mg}  = \Phi_{\sigma}^{+} + \Phi_{\sigma}^{-}$. 
\end{thm}
\begin{proof}
Let us consider the case of $\delta_1=\delta_2$. 
The equation (\ref{eqn:R2_ps_PDE1_2'}) with $q=0$ 
implies that $f(z)=z^{-(\nu_1+\nu_2)/2} \hat{\varphi}_0((2\pi)^{-1}z)$ 
satisfies the modified Bessel differential equation 
(\ref{eqn:Fn_bessel_DE}) with $r=(\nu_1-\nu_2)/2$. 
Since the dimensions of the solution space 
and ${\cI}_{\Pi_\sigma ,\psi_1}$ are both $2$, 
we note that smooth solutions of (\ref{eqn:R2_ps_PDE1_2'}) 
correspond to elements of ${\cI}_{\Pi_\sigma ,\psi_1}$. 
Hence, we obtain the statement in this case 
by means of the formulas in \S \ref{subsec:Fn_special1}.

Let us consider the case of $(\delta_1,\delta_2)=(1,0)$. 
From the equation (\ref{eqn:R2_ps_PDE1_2'}) with $q=\pm 1$, 
we get
\begin{align}
& \begin{split} \label{eqn:R2_ps_PDE1_2a}
 \{(\partial_1-\nu_1)(\partial_1-\nu_2)-(2\pi y_1)^2\}
(\hat{\varphi}_1+\hat{\varphi}_{-1})
  + (2\pi y_1) (\hat{\varphi}_1-\hat{\varphi}_{-1}) = 0.  
\end{split}
\end{align}
Set $ (q,r)=(\pm 1, \mp2) $ in (\ref{eqn:R2_ps_PDE1_3'}) to find that
\begin{align}
&-(\partial_1-\nu_2)(\hat{\varphi}_1+\hat{\varphi}_{-1})
-(2\pi y_1)(\hat{\varphi}_1-\hat{\varphi}_{-1})=0.
\label{eqn:R2_ps_PDE1_3a}
\end{align}
Adding the respective sides of 
(\ref{eqn:R2_ps_PDE1_2a}) and (\ref{eqn:R2_ps_PDE1_3a}), 
we have 
\begin{align}
\label{eqn:R2_ps_PDE1_2aaa}
& \{(\partial_1-\nu_1-1)(\partial_1-\nu_2)-(2\pi y_1)^2\}
(\hat{\varphi}_1+\hat{\varphi}_{-1}) =0.
\end{align}
This equation implies that 
$f(z)=z^{-(\nu_1+\nu_2+1)/2}(\hat{\varphi}_1+\hat{\varphi}_{-1})
((2\pi )^{-1}z)$ satisfies the modified Bessel differential equation 
(\ref{eqn:Fn_bessel_DE}) with $r=(\nu_1-\nu_2+1)/2$. 
Hence, as is the case of $\delta_1=\delta_2$, 
we note that solutions of (\ref{eqn:R2_ps_PDE1_2aaa}) 
correspond to elements of ${\cI}_{\Pi_\sigma ,\psi_1}$. 
By means of the formulas in \S \ref{subsec:Fn_special1}, 
we obtain the 
explicit expression of $\hat{\varphi}_1+\hat{\varphi}_{-1}$. 
Moreover, we can get 
the explicit expression of $\hat{\varphi}_1-\hat{\varphi}_{-1}$ 
from that of $\hat{\varphi}_1+\hat{\varphi}_{-1}$ by 
(\ref{eqn:R2_ps_PDE1_3a}) with $\Gamma (s+1)=s\Gamma (s)$. 
Hence, we obtain the statement in this case. 
\end{proof}

\begin{rem}
Because of (\ref{eqn:R2_cdn_ps_irred}), we note that, 
for $\sigma = \chi_{(\nu_1,\delta_1)}\boxtimes \chi_{(\nu_2,\delta_2)}$ 
with $\nu_1,\nu_2\in \bC$ and $ \delta_1,\delta_2 \in \{0,1\}$ 
$(\delta_1\geq \delta_2)$, 
the condition $\nu_1-\nu_2+\delta_1-\delta_2+1\not\in 2\bZ_{\leq 0}$ 
holds if $\Pi_\sigma$ is irreducible. 
\end{rem}

Now we want to describe explicit formulas of $\hat{\varphi}_q $ 
$(q \in \delta_1-\delta_2+2\bZ)$. 
By means of (\ref{eqn:R2_ps_PDE1_3'}), 
the functions $\hat{\varphi}_q $ 
$(q \in \delta_1-\delta_2+2\bZ)$ can be uniquely determined 
from $\hat{\varphi}_{\pm (\delta_1-\delta_2)}$. 
We start with another Mellin-Barnes integral representation for 
the modified Bessel function. In view of the formula  
\[
 \int_0^{\infty} e^{-z} K_{r}(z) z^s \frac{dz}{z} 
 = \pi^{1/2} 2^{-s} \frac{\Gamma(s+r)\Gamma(s-r)}{\Gamma(s+1/2)} 
\]
for ${\mathrm{Re}}(s)>|{\mathrm{Re}}(r)| $ 
(\cite[6.8. (28)]{Erdelyi_001}), the 
Mellin-inversion formula implies 
\begin{align}
\label{eqn:R2_Bessel_another}
K_{r}(z) = \frac{\pi^{1/2} e^{z}}{2\pi \sqrt{-1}}\int_s
\frac{\Gamma(s+r)\Gamma(s-r)}{\Gamma(s+1/2)} (2z)^{-s} ds.
\end{align}
Here the path of integration $\int_{s}$ is the vertical line 
from $\mathrm{Re}(s)-\sI \infty$ to $\mathrm{Re}(s)+\sI \infty$ 
with the real part $\mathrm{Re}(s)>|\mathrm{Re}(r)|$.


\begin{thm}
\label{thm:R2_ps_Whittaker2}
Let $\sigma = \chi_{(\nu_1,\delta_1)}\boxtimes \chi_{(\nu_2,\delta_2)}$ with 
$\nu_1,\nu_2\in \bC$ and $ \delta_1,\delta_2 \in \{0,1\}$ 
$(\delta_1\geq \delta_2)$ such that 
$\nu_1-\nu_2+\delta_1-\delta_2+1\not\in 2\bZ_{\leq 0}$. 
Let $\Phi_{\sigma}^{\rm mg}\in {\cI}_{\Pi_\sigma ,\psi_1}^{\rm mg}$ 
be the homomorphism in Theorem \ref{thm:R2_ps_Whittaker}. 
For $q \in \delta_1-\delta_2+2\bZ$, it holds that 
\begin{align*}
\begin{split}
&\Phi_{\sigma}^{\rm mg} ({\zeta}_{(\sigma;q)})(y) 
=C_q(\sI )^qy_1^{1/2}y_2^{\nu_1+\nu_2}
\frac{e^{2\pi y_1} }{4\pi \sI}
\int_s\cV_q(s) 
y_1^{-s}ds
\end{split}
\end{align*}
with 
\begin{align*} 
&C_q  =\frac{\Gamma_\bR (\nu_1-\nu_2+1+\delta_1-\delta_2)}
{\Gamma_\bR (\nu_1-\nu_2+1+q)},&
&\cV_q(s)=
\frac{\Gamma_\bC (s+\nu_1)\Gamma_\bC (s+\nu_2)}
{\Gamma_\bR (2s+\nu_1+\nu_2+1-q)}.
\end{align*}
Here $y=\diag (y_1y_2,y_2)\in A $, 
and the path of integration $\int_{s}$ is the vertical line 
from $\mathrm{Re}(s)-\sI \infty$ to $\mathrm{Re}(s)+\sI \infty$ 
with the sufficiently large real part to keep the poles of the integrand 
on its left.  
\end{thm}
\begin{proof} 
For $q\in \delta_1-\delta_2+2\bZ$, we set 
\begin{align}
\label{eqn:R2_pf_psWh2_001}
\hat{\varphi}_{q}^{\rm mg}(y_1) 
&=C_q\frac{e^{2\pi y_1}}{4\pi \sqrt{-1}}
\int_s \cV_q(s) y_1^{-s} ds&
&(y_1\in \bR_+).
\end{align}
By the formulas in Theorem \ref{thm:R2_ps_Whittaker}, 
(\ref{eqn:R2_Bessel_another}) and (\ref{eqn:Fn_FE_GammaRC}), 
we know that 
\begin{align*}
&\Phi_{\sigma}^{\rm mg} ({\zeta}_{(\sigma;q)})(y)
=(\sI)^{q}y_1^{1/2}y_2^{\nu_1+\nu_2}
\hat{\varphi}_{q}^{\rm mg}(y_1)&
&(\,q=\pm (\delta_1-\delta_2)\,)
\end{align*}
holds for $y=\diag (y_1y_2,y_2)\in A$. 
Hence, in order to complete a proof, 
it suffices to show that $\hat{\varphi}_{q}=\hat{\varphi}_{q}^{\rm mg}$ 
($q\in \delta_1-\delta_2+2\bZ$) satisfy 
(\ref{eqn:R2_ps_PDE1_3'}) for $r\in \{\pm 2\}$. 
Our task is to show the equalities 
\begin{align}   
&\hat{\varphi}_{q+2}^{\rm mg}(y_1)
=\frac{-2\partial_1+4\pi y_1+\nu_1+\nu_2-1-q}
{\nu_1-\nu_2+1+q} \hat{\varphi}_q^{\rm mg}(y_1),
\label{eqn:R2_ps_PDE1_3'r2} \\
&\hat{\varphi}_{q-2}^{\rm mg}(y_1)
=\frac{-2\partial_1-4\pi y_1+\nu_1+\nu_2-1+q}
{\nu_1-\nu_2+1-q} \hat{\varphi}_q^{\rm mg}(y_1).
 \label{eqn:R2_ps_PDE1_3'r-2}
\end{align}
Since the right hand side of (\ref{eqn:R2_ps_PDE1_3'r2}) becomes 
\begin{align*}
& \frac{C_q}{\nu_1-\nu_2+1+q}  \frac{e^{2\pi y_1}}{4\pi \sqrt{-1}} 
   \int_s (2s+\nu_1+\nu_2-1-q) \cV_q(s) y_1^{-s} ds\\
& =  C_{q+2}  \frac{e^{2\pi y_1}}{4\pi \sqrt{-1}} 
   \int_s  \cV_{q+2}(s)y_1^{-s} ds
\end{align*}
by (\ref{eqn:Fn_FE_GammaRC}), 
the equality (\ref{eqn:R2_ps_PDE1_3'r2}) follows. 
Similarly the right hand side of (\ref{eqn:R2_ps_PDE1_3'r-2}) 
can be written as 
\begin{align*}
\frac{C_q}{\nu_1-\nu_2+1-q}\frac{e^{2\pi y_1}}{4\pi \sqrt{-1}} 
\int_s \left\{ 
(2s+\nu_1+\nu_2-1+q)\cV_q(s)- 8\pi \cV_q(s+1)\right\}y_1^{-s} ds.
\end{align*}
By (\ref{eqn:Fn_FE_GammaRC}), 
the bracket $\{ \ \ \} $ in the above is equal to 
\begin{align*} 
& \left\{(2s+\nu_1+\nu_2-1+q)- \frac{4(s+\nu_1)(s+\nu_2)}
{2s+\nu_1+\nu_2+1-q}\right\}\cV_q(s)\\
&=
\frac{(\nu_1-\nu_2+1-q)(\nu_1-\nu_2-1+q)}
{2s+\nu_1+\nu_2+1-q} \cV_q(s)  \\
& = (2\pi)^{-1} (\nu_1-\nu_2+1-q)(\nu_1-\nu_2-1+q)\cV_{q-2}(s). 
\end{align*}
Thus we can see (\ref{eqn:R2_ps_PDE1_3'r-2}).
\end{proof}

\section{Essentially discrete series Whittaker functions}
\label{subsec:R2_ds_whittaker}

Let $\nu \in \bC$ and $\kappa \in \bZ_{\geq 2}$. 
For $\sigma =D_{(\nu ,\kappa )}$, the space  
$H(\sigma)$ is isomorphic to $U_\sigma =\gH_{(\nu ,\kappa )}$ 
as $G$-modules, 
via the isomorphism 
\begin{align}
\label{eqn:R2_isom_gps_eds}
H(\sigma)\ni f\mapsto f(1_2)
\in \gH_{(\nu ,\kappa )}. 
\end{align}
In this section, we introduce a convenient realization of 
$(D_{(\nu ,\kappa )},\gH_{(\nu  ,\kappa  )})$. 
We set 
\begin{align}
\label{eqn:R2_param_gps_eds}
\widehat{\sigma}=\chi_{(\nu +(\kappa -1)/2,\delta )}
\boxtimes \chi_{(\nu -(\kappa -1)/2,0)}
\end{align}
with $\delta \in \{0,1\}$ such that 
$\delta \equiv \kappa \bmod 2$. 
Then, from the formulas of the actions of $\g_\bC$ and $K$ 
in \S \ref{subsec:R2_ps}, we know that 
\begin{align}
\label{eqn:R2_ds_in_ps}
&\bigoplus_{q\in \kappa +2\bZ_{\geq 0}}
\{\bC\,{\zeta}_{(\widehat{\sigma},q)}
+\bC\,{\zeta}_{(\widehat{\sigma},-q)}\}
\end{align}
is a $(\g_\bC ,K)$-submodule of $H(\sigma)_K$, 
and this submodule is isomorphic to 
$D^+_{\kappa }\oplus D^-_{\kappa }$ 
as $(\gs \gl (2,\bR ),SO(2))$-modules (\textit{cf.} \cite[\S 2.5]{Bump_004}). 
Since $\Pi_{\widehat{\sigma}}(t1_2)=t^{2\nu }\ (t\in \bR_+)$, 
we note that the subrepresentation of $\Pi_{\widehat{\sigma}}$ on 
the closure of (\ref{eqn:R2_ds_in_ps}) satisfies the definition 
of $(D_{(\nu ,\kappa )},\gH_{(\nu ,\kappa )})$ 
in \S \ref{subsec:Rn_def_gps}. Hereafter, 
we regard $(D_{(\nu ,\kappa )},\gH_{(\nu ,\kappa )})$ as 
this subrepresentation of $\Pi_{\widehat{\sigma}}$. 
We note that the $K$-finite part $\gH_{(\nu ,\kappa ),K}$ 
of $\gH_{(\nu ,\kappa )}$ coincides with 
the space (\ref{eqn:R2_ds_in_ps}). 

For $\widehat{\sigma}=\chi_{(\nu +(\kappa -1)/2,\delta )}
\boxtimes \chi_{(\nu -(\kappa -1)/2,0)}$, 
we take $\Phi_{\widehat{\sigma}}^{\rm mg}\in 
{\cI}_{\Pi_{\widehat{\sigma}},\psi_1}^{\rm mg}$ 
as in Theorem \ref{thm:R2_ps_Whittaker}. Then we note that 
\[
{\cI}_{D_{(\nu,\kappa )},\psi_1}^{\rm mg}
=\bC \cdot \Phi_{\widehat{\sigma}}^{\rm mg}|_{\gH_{(\nu,\kappa )}}. 
\]
As a corollary of Theorem \ref{thm:R2_ps_Whittaker2}, 
we obtain the following.

\begin{cor}
\label{cor:R2_ds_Whittaker}
Retain the notation. 
For $y=\diag (y_1y_2,y_2)\in A$ and $q \in \kappa +2\bZ_{\geq 0}$, 
it holds that 
\begin{align*}
\begin{split}
\Phi_{\widehat{\sigma}}^{\rm mg}
({\zeta}_{(\widehat{\sigma};q)})(y) =
\,&
(\sI )^q
\frac{2^{-\nu +\frac{\kappa +1}{2}}\pi^{\frac{\kappa -\delta }{2}}}
{\bigl(\tfrac{\kappa +\delta }{2}\bigr)_{\frac{q-\delta }{2}}}
\frac{y_1^{1/2}y_2^{2\nu }e^{2\pi y_1}}{4\pi \sI}
\int_s\bigl(s+\nu -\tfrac{q-1}{2}\bigr)_{\frac{q-\kappa}{2}}\\
&\times 
\Gamma_\bC \bigl(s+\nu +\tfrac{\kappa -1}{2}\bigr)
(2y_1)^{-s}ds,\\
\Phi_{\widehat{\sigma}}^{\rm mg}
({\zeta}_{(\widehat{\sigma};-q)})(y)=
\,&0. 
\end{split}
\end{align*}
If $q=\kappa$, these formulas are simplified as  
\begin{align*}
\Phi_{\widehat{\sigma}}^{\rm mg}
({\zeta}_{(\widehat{\sigma};\kappa )})(y) =
\,&
(\sI )^{\kappa }\frac{2^{\kappa }\pi^{\frac{\kappa -\delta }{2}}}
{\bigl(\tfrac{\kappa +\delta }{2}\bigr)_{\frac{\kappa -\delta }{2}}}
y_1^{\nu  + \frac{\kappa}{2}}y_2^{2\nu  }
e^{-2\pi y_1}\\
=\,&
(\sI )^{\kappa }
\frac{\Gamma_{\bR}(\kappa +\delta )}{\Gamma_{\bC}(\kappa )}
\frac{y_1^{1/2}y_2^{2\nu  }}{2\pi \sI}
\int_s \Gamma_\bC \bigl(s+\nu  +\tfrac{\kappa  -1}{2}\bigr)y_1^{-s}ds,\\
\Phi_{\widehat{\sigma}}^{\rm mg}
({\zeta}_{(\widehat{\sigma};-\kappa )})(y)=\,&0.
\end{align*}
Here the path of integration $\int_{s}$ is the vertical line 
from $\mathrm{Re}(s)-\sI \infty$ to $\mathrm{Re}(s)+\sI \infty$ 
with the sufficiently large real part to keep the poles of the integrand 
on its left.  
\end{cor}

\chapter{Whittaker functions on $GL(3,\bR)$}
\label{sec:R3_whittaker}

Throughout this chapter, we set $n=3$ and $F=\bR $. 

\section{Representations of $O(3)$}
\label{subsec:R3_rep_K}

In this section, we discuss 
the representation theory of ${K}=O(3)$. 

We set $\Lambda_3
=\{\mu =(\mu_1,\mu_2)\mid \mu_1\in \bZ_{\geq 0},\ 
\mu_2\in \{0,1\}\}$. 
For $\mu =(\mu_1,\mu_2)\in \Lambda_3$, 
let $\mathcal{P}_\mu =\mathcal{P}_\mu^{(3)}$ 
be the $\bC$-vector space of degree $\mu_1$ 
homogeneous polynomials of three variables ${z}_1,{z}_2,{z}_3$. 
We define the action $T_\mu$ of $K$ on 
$\mathcal{P}_\mu$ by 
\begin{align*}
&(T_{\mu}(k)p)({z}_1,{z}_2,{z}_3)
=(\det k)^{\mu_2}p(({z}_1,{z}_2,{z}_3)\cdot k)&
&(k\in {K},\ p\in \mathcal{P}_{\mu}).
\end{align*}
Here $({z}_1,{z}_2,{z}_3)\cdot k$ is the ordinal product of matrices. 
Since ${z}_1^2+{z}_2^2+{z}_3^2$ is ${K}$-invariant, 
we can define the quotient representation $\tau_\mu =\tau_\mu^{(3)}$ 
of $T_\mu$ on the space 
\[
V_\mu=V_\mu^{(3)}=\mathcal{P}_\mu/
(({z}_1^2+{z}_2^2+{z}_3^2)\mathcal{P}_{\mu -(2,0)})
\] 
for $\mu \in \Lambda_3$. 
Here we put $\mathcal{P}_{\mu -(2,0)}=\{0\}$ 
if $\mu -(2,0)\not\in \Lambda_3$. 
Then the representation 
$(\tau_{\mu },V_{\mu })$ 
is irreducible and $\dim_\bC  V_\mu =2\mu_1+1$. 
The set of equivalence classes of 
irreducible representations of ${K}$ is 
exhausted by $\{\tau_{\mu }\mid \mu \in \Lambda_3\}$.

For $\mu =(\mu_1,\mu_2)\in \Lambda_3$ 
and $-\mu_1\leq q\leq \mu_1$, let 
$v_{q}^{\mu}$ be the image of 
\[
(\sgn (q) z_1+\sqrt{-1}z_2)^{|q|}z_3^{\mu_1-|q|}
\]
under the natural surjection 
$\mathcal{P}_\mu \to V_{\mu}$. 
Then $\{v_{q}^{\mu}\mid -\mu_1\leq q\leq \mu_1 \}$ 
is a basis of $V_{\mu }$. 
By direct computation, 
for $\theta \in \bR$ and 
$\varepsilon_1,\varepsilon_2,\varepsilon_3\in \{\pm 1\}$, 
we have 
\begin{align}
\label{eqn:R3_M21act_basis}
&\tau_{\mu }\bigl(k_\theta^{(2,1)} \bigr)
v^{\mu }_q=e^{\sI q \theta }v^{\mu }_q,\\
\label{eqn:R3_M111act_basis}
&\tau_{\mu }
\bigl(k_{(\varepsilon_1,\varepsilon_2,\varepsilon_3)}^{(1,1,1)}\bigr)
v^{\mu}_{q}
=\varepsilon_1^{\mu_2} \varepsilon_2^{\mu_2+q}\varepsilon_3^{\mu_1+\mu_2+q}
v^{\mu }_{\varepsilon_1\varepsilon_2 q},
\end{align}
with 
\begin{align*}
&k_\theta^{(2,1)} =
\left(\begin{array}{ccc}
\cos \theta &\sin \theta &0\\
-\sin \theta &\cos \theta &0\\
0&0&1
\end{array}\right),&
&k_{(\varepsilon_1,\varepsilon_2,\varepsilon_3)
}^{(1,1,1)}=
\left(\begin{array}{ccc}
\varepsilon_1&0&0\\
0&\varepsilon_2&0\\
0&0&\varepsilon_3
\end{array}\right)\in {K}.
\end{align*}

For $\mu=(\mu_1,\mu_2)\in \Lambda_3$, we set 
\[
S_\mu =
\{l=(l_1,l_2,l_3)\in (\bZ_{\geq 0})^3\mid l_1+l_2+l_3=\mu_1 \}. 
\]
For $l=(l_1,l_2,l_3)\in S_{\mu}$, 
let $u_{l}$ be the image of 
${z}_1^{l_1}{z}_2^{l_2}{z}_3^{l_3}$ 
under the natural surjection 
$\mathcal{P}_\mu \to V_{\mu}$. 
Then $\{u_{l}\mid l\in S_\mu\}$ is a system 
of generators of $V_\mu$. 
Here we note that, if $\mu_1<2$, the vectors 
$u_{l}$ ($l\in S_\mu$) are linearly independent, 
and if $\mu_1\geq 2$, the relation 
$u_{l+2\me_1}+u_{l+2\me_2}+u_{l+2\me_3}=0$
$(l\in S_{\mu -(2,0)})$ 
holds with 
\begin{align*}
&\me_1=(1,0,0),&
&\me_2=(0,1,0),&
&\me_3=(0,0,1). 
\end{align*}
It is convenient to set $u_l =0$ if $l\not\in (\bZ_{\geq 0})^3$. 
We denote the differential of $\tau_{\mu}$ 
again by $\tau_{\mu }$. 
By direct computation,  we have  
\begin{align}
\label{eqn:R3_LieKact_gen}
&\tau_{\mu } (E_{i,j}^{{\gk}} )u_{l} 
=l_ju_{l-\me_j+\me_i}-l_iu_{l-\me_i+\me_j}&
&(1\leq i, j\leq 3),\\
&\label{eqn:R3_M111act_gen}
\tau_{\mu} 
\bigl(k_{(\varepsilon_1,\varepsilon_2,\varepsilon_3)}^{(1,1,1)}\bigr)u_l 
=
\varepsilon_1^{l_1+\mu_2}
\varepsilon_2^{l_2+\mu_2}
\varepsilon_3^{l_3+\mu_2}
u_{l}&&(\varepsilon_1,\varepsilon_2,\varepsilon_3\in \{\pm 1\}) 
\end{align}
for  $l=(l_1,l_2,l_3)\in S_\mu $.

We regard $\gp_\bC$ as a $K$-module via the adjoint action $\Ad$. 
We denote by $\adj$ the differential of $\Ad$. 
For later use, we prepare the following lemma. 
\begin{lem}
\label{lem:R3_tensor}
(1) Let $\mu =(\mu_1,\mu_2)\in \Lambda_3$ such that 
$\mu_1\geq 1$. It holds that 
\begin{align}
\label{eqn:R3_tensor11}
\Hom_K(V_{(1,0)}\otimes_{\bC}V_{\mu -(1,0)}, 
V_{\mu })
=\bC \,\mathrm{B}_{\mu }, 
\end{align}
where $\mathrm{B}_{\mu }\colon 
V_{(1,0)}\otimes_{\bC}V_{\mu -(1,0)}\to 
V_{\mu}$ is a surjective $\bC$-linear map characterized by  
$\mathrm{B}_{\mu }(u_{\me_i}\otimes u_l)=u_{l+\me_i}$\ 
$(1\leq i\leq 3,\ l\in S_{\mu -(1,0)})$.
Moreover, for $\mu'=(\mu_1',\mu_2')\in \Lambda_3$, 
it holds that 
\begin{align}
\label{eqn:R3_tensor12}
\Hom_K(V_{(1,0)}\otimes_{\bC}V_{\mu -(1,0)}, 
V_{\mu'})
=\{0\} 
\end{align}
if $\mu_1'>\mu_1 $ or $\mu_1'+\mu_2'\not\equiv 
\mu_1 +\mu_2 \mod 2$. 

\noindent (2) We define a $\bC$-linear map 
$\mathrm{I}^{\gp}_{(1,\mu_2)}\colon V_{(1,\mu_2)}\to 
\gp_{\bC}\otimes_{\bC}V_{(1,\mu_2)}$ by 
\begin{align*}
&\mathrm{I}^{\gp}_{(1,\mu_2)}(u_{\me_i})
=\sum_{j=1}^3E_{i,j}^{\gp}\otimes u_{\me_j}&
&(1\leq i\leq 3).
\end{align*}
Then $\mathrm{I}^{\gp}_{(1,\mu_2)}$ is a $K$-homomorphism.

\end{lem}
\begin{proof}
Let $\mu =(\mu_1,\mu_2),\, \mu'=(\mu_1',\mu_2')
\in \Lambda_3$ with $\mu_1\geq 1$. 
Since 
\begin{align*}
&\tau_{\mu '}(-1_3)=(-1)^{\mu_1' +\mu_2'},&
&(\tau_{(1,0)}\otimes \tau_{\mu -(1,0)})(-1_3)
=(-1)^{\mu_1+\mu_2}, 
\end{align*}
we know that (\ref{eqn:R3_tensor12}) holds 
if $\mu_1'+\mu_2'\not\equiv 
\mu_1+\mu_2 \mod 2$. 
By (\ref{eqn:R3_M21act_basis}), 
we have 
\begin{align*}
&\Bigl\{v\in V_{\mu'}\,\Big|\,
\tau_{\mu '}\bigl(k_\theta^{(2,1)}\bigr)v=e^{\sI \mu_1' \theta }v\quad 
(\theta \in \bR)\Bigr\}=\bC \,v_{\mu_1'}^{\mu'},\\[1mm]
&\Bigl\{v\in V_{(1,0)}\otimes_{\bC}V_{\mu -(1,0)}\, \Big| \, 
(\tau_{(1,0)}\otimes \tau_{\mu -(1,0)})
\bigl(k_\theta^{(2,1)}\bigr)v=e^{\sI \mu_1' \theta }v\quad 
(\theta \in \bR)\Bigr\}\\
&=
\left\{\begin{array}{ll}
\bC \,v^{(1,0)}_{1}\otimes v^{\mu -(1,0)}_{\mu_1-1}&
\text{if}\ \mu_1'=\mu_1,\\[1mm]
\{0\}&\text{if}\ \mu_1'>\mu_1.
\end{array}\right.
\end{align*} 
Hence, we know that (\ref{eqn:R3_tensor12}) holds 
if $\mu_1'>\mu_1 $, and 
\[
\dim_\bC  \Hom_K(V_{(1,0)}\otimes_{\bC}V_{\mu -(1,0)}, 
V_{\mu })\leq 1. 
\]
Since $\mathrm{B}_{\mu}$ 
is a $\bC$-linear map induced from 
\begin{align*}
\mathcal{P}_{(1,0)}\otimes_{\bC}
\mathcal{P}_{\mu -(1,0)}\ni 
p_1\otimes 
p_2\mapsto 
p_1p_2
\in \mathcal{P}_{\mu }, 
\end{align*}
it is obvious that 
$\mathrm{B}_{\mu }$ is a surjective $K$-homomorphism. 
Hence, we obtain (\ref{eqn:R3_tensor11}), 
and complete a proof of the statement (1).

By (\ref{eqn:R3_LieKact_gen}), 
we have 
\begin{align*}
&(\adj \otimes \tau_{(1,\mu_2)})(E_{a,b}^{{\gk}})
\mathrm{I}^{\gp}_{(1,\mu_2)}(u_{\me_i})\\
&=\sum_{j=1}^3
\bigl\{(\adj (E_{a,b}^{{\gk}})E_{i,j}^{\gp})\otimes u_{\me_j}
+E_{i,j}^{\gp}\otimes (\tau_{(1,\mu_2)}(E_{a,b}^{{\gk}})u_{\me_j})\bigr\}\\
&=\sum_{j=1}^3
\bigl\{(\delta_{b,i}E_{j,a}^{{\gp}}
+\delta_{b,j}E_{i,a}^{{\gp}}
-\delta_{a,i}E_{j,b}^{{\gp}}
-\delta_{a,j}E_{i,b}^{{\gp}})\otimes u_{\me_j}\\
&\hspace{11mm}+E_{i,j}^{\gp}\otimes 
(\delta_{b,j}u_{\me_a}-\delta_{a,j}u_{\me_b})\bigr\}\\
&=
\delta_{b,i}\sum_{j=1}^3E_{j,a}^{{\gp}}\otimes u_{\me_j}
+E_{i,a}^{{\gp}}\otimes u_{\me_b}
-\delta_{a,i}\sum_{j=1}^3E_{j,b}^{{\gp}}\otimes u_{\me_j}
-E_{i,b}^{{\gp}}\otimes u_{\me_a}\\
&\hspace{4mm}
+E_{i,b}^{\gp}\otimes u_{\me_a}-E_{i,a}^{\gp}\otimes u_{\me_b}\\[1mm]
&=
\delta_{b,i}\mathrm{I}^{\gp}_{1}(u_{\me_a})
-\delta_{a,i}\mathrm{I}^{\gp}_{1}(u_{\me_b})
\hspace{4.5cm} (1\leq i\leq 3)
\end{align*}
for $1\leq a,b\leq 3$. 
Comparing this equality with (\ref{eqn:R3_LieKact_gen}), 
we know that $\mathrm{I}^{\gp}_{(1,\mu_2)}$ is a $\gk_\bC $-homomorphism. 
Since the identity component of $K=O(3)$ is $SO(3)$, 
we know that $\mathrm{I}^{\gp}_{(1,\mu_2)}$ is an $SO(3)$-homomorphism. 
Since $K=SO(3)\cup ((-1_3)SO(3))$ and 
$(\Ad \otimes \tau_{(1,\mu_2)})(-1_3)
=\tau_{(1,\mu_2)}(-1_3)=(-1)^{1+\mu_2}$, 
we obtain the statement (2). 
\end{proof}

\section{Principal series representations}
\label{subsec:R3_ps}

Let $\sigma =\chi_{(\nu_1,\delta_1)}\boxtimes \chi_{(\nu_2,\delta_2)}
\boxtimes \chi_{(\nu_3,\delta_3)}$ 
with $\nu_1,\nu_2,\nu_3\in \bC$ and $\delta_1,\delta_2,\delta_3\in \{0,1\}$ 
such that $\delta_1\geq \delta_2\geq \delta_3$. 
We identify $U_\sigma $ 
with $\bC$ as in \S \ref{subsec:Rn_inf_char_ps}. 
In this section, 
we consider the action of $\g_\bC$ at the minimal $K$-type of $\Pi_\sigma$. 

Because of (\ref{eqn:R3_M111act_gen}) and 
${K}\cap M_{(1,1,1)}=
\bigl\{k_{(\varepsilon_1,\varepsilon_2,\varepsilon_3)}^{(1,1,1)}\,\big|\,
\varepsilon_1,\varepsilon_2,\varepsilon_3\in \{\pm 1\}\bigr\}$, 
for $\mu =(\mu_1,\mu_2)\in \Lambda_3$, we have 
\begin{align*}
&\Hom_{{K}\cap M_{(1,1,1)}}(V_{\mu },U_{\sigma})\\
&=\left\{\begin{array}{ll}
\bC\, \eta_{\sigma}
&\text{ if }\ \mu =(\delta_1-\delta_3,\delta_2),
\\[1mm]
\{0\}&\text{ if }\ \mu_1 < \delta_1-\delta_3 \ \text{ or }\ 
\mu_1+\mu_2 \not\equiv \delta_1 +\delta_2+\delta_3 \bmod 2. 
\end{array}\right. 
\end{align*}
Here $\eta_{\sigma }\colon V_{(\delta_1-\delta_3 ,\delta_2)}
\to U_\sigma $ is a $\bC$-linear map characterized by 
\begin{align*}
&\eta_{\sigma}(u_l)
=\left\{\begin{array}{ll}
1&\text{ if }\ l=(\delta_1-\delta_2,\,
0,\,\delta_2-\delta_3),\\[1mm]
0&\text{ otherwise}
\end{array}\right.&
&(l\in S_{(\delta_1-\delta_3 ,\delta_2)}). 
\end{align*}
By the Frobenius reciprocity law, 
for $\mu =(\mu_1,\mu_2)\in \Lambda_3$, we have 
\begin{align*}
&\Hom_{{K}}(V_{\mu },H({\sigma}))\\
&=\left\{\begin{array}{ll}
\bC\, \hat{\eta}_{\sigma}
&\text{ if }\ \mu =(\delta_1-\delta_3,\delta_2),
\\[1mm]
\{0\}&\text{ if }\ \mu_1 < \delta_1-\delta_3 \ \text{ or }\ 
\mu_1+\mu_2 \not\equiv \delta_1 +\delta_2+\delta_3 \bmod 2
\end{array}\right. 
\end{align*}
with $\hat{\eta}_{\sigma }(v)(k)=\eta_{\sigma }
(\tau_{(\delta_1-\delta_3 ,\delta_2)}(k)v)\ 
(v\in V_{(\delta_1-\delta_3 ,\delta_2)},\ k\in {K})$. 
We call $\tau_{(\delta_1-\delta_3 ,\delta_2)}$ 
the minimal $K$-type of $\Pi_\sigma$.

\begin{prop}
\label{prop:R3_DSE_ps}
Retain the notation. 
If $(\delta_1,\delta_2,\delta_3)=(1,0,0)$, 
it holds that 
\begin{align*}
&\sum_{j=1}^3\Pi_\sigma (E_{i,j}^{\gp})
\hat{\eta}_{\sigma }(u_{\me_j})
=2\nu_1\hat{\eta}_{\sigma }(u_{\me_i})&
&(1\leq i\leq 3).
\end{align*}
If $(\delta_1,\delta_2,\delta_3)=(1,1,0)$, 
it holds that 
\begin{align*}
&\sum_{j=1}^3\Pi_\sigma (E_{i,j}^{\gp})
\hat{\eta}_{\sigma }(u_{\me_j})
=2\nu_3\hat{\eta}_{\sigma }(u_{\me_i})&
&(1\leq i\leq 3).
\end{align*}
\end{prop}
\begin{proof}
Let $\mathrm{P}_\sigma 
\colon \gp_{\bC}\otimes_{\bC}H(\sigma)_K\to H(\sigma)_K$ 
be a natural $K$-homomorphism defined by 
$X\otimes f\mapsto \Pi_{\sigma}(X)f$. 
Assume $\delta_1>\delta_3$. 
Since the composite 
\[
\mathrm{P}_\sigma \circ 
(\id_{\gp_{\bC}}\otimes \hat{\eta}_{\sigma })\circ 
\mathrm{I}^\gp_{(1,\delta_2)}
\]
is an element of 
$\Hom_K(V_{(1,\delta_2)},H(\sigma)_K)=\bC\,\hat{\eta}_\sigma$, 
there exists $c\in \bC$ such that 
\begin{align*}
c\,\hat{\eta}_{\sigma }=
\mathrm{P}_\sigma \circ 
(\id_{\gp_{\bC}}\otimes \hat{\eta}_{\sigma })\circ 
\mathrm{I}^\gp_{(1,\delta_2)}.
\end{align*}
Considering the image of $u_{\me_i}$ 
under the both sides of this equality, we have 
\begin{align*}
c\,\hat{\eta}_{\sigma }(u_{\me_i})
=\sum_{j=1}^3\Pi_\sigma (E_{i,j}^{\gp})
\hat{\eta}_{\sigma }(u_{\me_j})
\end{align*}
for $1\leq i\leq 3$. 
When $(\delta_1,\delta_2,\delta_3)=(1,0,0)$, we have 
\begin{align*}
c&=c\,\hat{\eta}_{\sigma }(u_{\me_{1}})(1_3)=
\sum_{j=1}^3(\Pi_\sigma (E_{1,j}^{\gp})
\hat{\eta}_{\sigma }(u_{\me_j}))(1_3)\\
&=
2(\Pi_\sigma (E_{1,1}^{\g})\hat{\eta}_{\sigma}(u_{\me_1}))(1_3)\\
&\hphantom{=,}
+2(\Pi_\sigma (E_{1,2}^{\g})\hat{\eta}_{\sigma}(u_{\me_2}))(1_3)
-\hat{\eta}_{\sigma}
(\tau_{(1,\delta_2)}(E_{1,2}^{\gk})u_{\me_2})(1_3)\\
&\hphantom{=,}
+2(\Pi_\sigma (E_{1,3}^{\g})\hat{\eta}_{\sigma}(u_{\me_3}))(1_3)
-\hat{\eta}_{\sigma}
(\tau_{(1,\delta_2)}(E_{1,3}^{\gk})u_{\me_3})(1_3)\\
&=2(\nu_1+1)+0-1+0-1=2\nu_1.
\end{align*}
When $(\delta_1,\delta_2,\delta_3)=(1,1,0)$, we have 
\begin{align*}
c
&=
c\,\hat{\eta}_{\sigma}(u_{\me_{3}})(1_3)
=
\sum_{j=1}^3(\Pi_\sigma (E_{3,j}^{\gp})
\hat{\eta}_{\sigma}(u_{\me_j}))(1_3)\\
&=
2(\Pi_\sigma (E_{1,3}^{\g})\hat{\eta}_{\sigma}(u_{\me_1}))(1_3)
-\hat{\eta}_{\sigma}
(\tau_{(1,\delta_2)}(E_{1,3}^{\gk})u_{\me_1})(1_3)\\
&\hphantom{=,}
+2(\Pi_\sigma (E_{2,3}^{\g})\hat{\eta}_{\sigma}(u_{\me_2}))(1_3)
-\hat{\eta}_{\sigma}
(\tau_{(1,\delta_2 )}(E_{2,3}^{\gk})u_{\me_2})(1_3)\\
&\hphantom{=,}
+2(\Pi_\sigma (E_{3,3}^{\g})\hat{\eta}_{\sigma}(u_{\me_3}))(1_3)\\
&=0-(-1)+0-(-1)+2(\nu_3-1)=2\nu_3.
\end{align*}
Therefore, we obtain the assertion. 
\end{proof}

\section{Principal series Whittaker functions at scalar $K$-types}
\label{subsec:R3_ps1_whittaker}

We use the notation in \S \ref{subsec:R3_ps}. 
We assume $\delta_1-\delta_3=0$, that is, 
\[
\sigma =\chi_{(\nu_1,\delta_2)}\boxtimes \chi_{(\nu_2,\delta_2)}
\boxtimes \chi_{(\nu_3,\delta_2)}
\] 
with $\nu_1,\nu_2,\nu_3\in \bC$ and 
$\delta_2\in \{0,1\}$. 
Then the minimal $K$-type $\tau_{(0,\delta_2)}$ of $\Pi_\sigma$ 
is 1 dimensional. 
Let $\varphi \colon V_{(0,\delta_2)} \to 
{\mathrm{Wh}}(\Pi_\sigma ,\psi_1)$ be a $K$-homomorphism. 
Since $S_{(0,\delta_2)}=\{\mathbf{0}\}$ with $\mathbf{0}=(0,0,0)$, 
we note that $\varphi$ is characterized by 
$\varphi (u_{\mathbf{0}})(y)$ with 
$y=\diag (y_1y_2y_3,y_2y_3,y_3)\in A$. 
We set $\displaystyle \partial_i=y_i\frac{\partial}{\partial y_i}$ 
for $1\leq i\leq 3$.

We will construct the system of 
partial differential equations satisfied by 
the function $\varphi (u_{\mathbf{0}})(y)$, 
from the actions of the generators 
$\cC_1^{{\g}}$, $\cC_2^{{\g}}$ and $\cC_3^{{\g}}$ of $Z(\g_\bC)$. 
The explicit forms of $\cC_1^{{\g}}$, $\cC_2^{{\g}}$ and $\cC_3^{{\g}}$ 
are given by 
\begin{align}
\label{eqn:R3_C1}
\cC_1^{{\g}}=&E_{1,1}^{{\g}}+E_{2,2}^{{\g}}+E_{3,3}^{{\g}},\\ 
\nonumber 
\cC_2^{{\g}}=&(E_{1,1}^{{\g}}-1)E_{2,2}^{{\g}}
+(E_{1,1}^{{\g}}-1)(E_{3,3}^{{\g}}+1)
+E_{2,2}^{{\g}}(E_{3,3}^{{\g}}+1)\\
\nonumber 
&-E_{1,2}^{{\g}}E_{2,1}^{{\g}}-E_{1,3}^{{\g}}E_{3,1}^{{\g}}
-E_{2,3}^{{\g}}E_{3,2}^{{\g}},\\
\nonumber 
\cC_3^{{\g}}
=&(E_{1,1}^{{\g}}-1)E_{2,2}^{{\g}}(E_{3,3}^{{\g}}+1)
+E_{1,2}^{{\g}}E_{2,3}^{{\g}}E_{3,1}^{{\g}}
+E_{1,3}^{{\g}}E_{2,1}^{{\g}}E_{3,2}^{{\g}}\\
\nonumber 
&-E_{1,2}^{{\g}}E_{2,1}^{{\g}}(E_{3,3}^{{\g}}+1)
-E_{1,3}^{{\g}}E_{2,2}^{{\g}}E_{3,1}^{{\g}}
-(E_{1,1}^{{\g}}-1)E_{2,3}^{{\g}}E_{3,2}^{{\g}}.
\end{align}
By the equality $E_{j,i}^{{\g}}=E_{i,j}^{{\g}}-E_{i,j}^{{\gk}}$ 
$(1\leq i,j\leq 3)$, we have 
\begin{align}%
\begin{split} \label{eqn:R3_C2_modE13}
\cC_2^{{\g}}
\equiv 
&(E_{1,1}^{{\g}}-1)E_{2,2}^{{\g}}
+(E_{1,1}^{{\g}}-1)(E_{3,3}^{{\g}}+1)
+E_{2,2}^{{\g}}(E_{3,3}^{{\g}}+1)\\
&-(E_{1,2}^{{\g}})^2-(E_{2,3}^{{\g}})^2
+E_{1,2}^{{\g}}E_{1,2}^{{\gk}}+E_{2,3}^{{\g}}E_{2,3}^{{\gk}}
\hspace{0.5cm} \mod E_{1,3}^{{\g}}U(\g_\bC),
\end{split}\\
\begin{split}\label{eqn:R3_C3_modE13}
\cC_3^{{\g}}
\equiv 
&(E_{1,1}^{{\g}}-1)E_{2,2}^{{\g}}(E_{3,3}^{{\g}}+1)
-(E_{1,2}^{{\g}})^2(E_{3,3}^{{\g}}+1)
-(E_{2,3}^{{\g}})^2(E_{1,1}^{{\g}}-1)\\
&+E_{1,2}^{{\g}}(E_{3,3}^{{\g}}+1)E^{{\gk}}_{1,2}
+E_{2,3}^{{\g}}(E_{1,1}^{{\g}}-1)E^{{\gk}}_{2,3}
-E_{1,2}^{{\g}}E_{2,3}^{{\g}}E^{{\gk}}_{1,3}\\
&
\hspace{7cm}\mod E_{1,3}^{{\g}}U(\g_\bC).
\end{split}
\end{align}

\begin{lem}
\label{lem:R3_ps1_ZgDSE1_original}
Retain the notation. It holds that 
\begin{align}
&\label{eqn:R3_ps1_PDE1_C1}
(\partial_3-\nu_1-\nu_2-\nu_3)\varphi (u_{\mathbf{0}})(y)=0,\\[3pt] 
\begin{split}\label{eqn:R3_ps1_PDE1_C2}
&\{(\partial_1-1)(-\partial_1+\partial_2)
+(\partial_2-1)(-\partial_2+\partial_3+1)\\
&+(2\pi y_1)^2+(2\pi y_2)^2
-\nu_1\nu_2-\nu_1\nu_3-\nu_2\nu_3\}\varphi (u_{\mathbf{0}})(y)
=0,
\end{split}\\[3pt]
\begin{split}\label{eqn:R3_ps1_PDE1_C3}
&\{(\partial_1-1)(-\partial_1+\partial_2)(-\partial_2+\partial_3+1)
+(2\pi y_1)^2(-\partial_2+\partial_3+1)\\
&+(2\pi y_2)^2(\partial_1-1)-\nu_1\nu_2\nu_3\}\varphi (u_{\mathbf{0}})(y)=0.
\end{split}
\end{align}
\end{lem}
\begin{proof}
We note that there is a homomorphism $\Phi \in {\cI}_{\Pi_\sigma ,\psi_1}$ 
such that 
$\varphi =\Phi \circ \hat{\eta}_\sigma$. 
Hence, by Proposition \ref{prop:Rn_Ch_eigenvalue}, we have 
\begin{align*}
&(R(\cC_1^{{\g}})\varphi (u_{\mathbf{0}}))(y)
-(\nu_1+\nu_2+\nu_3)\varphi (u_{\mathbf{0}})(y)=0,\\
&(R(\cC_2^{{\g}})\varphi (u_{\mathbf{0}}))(y)
-(\nu_1\nu_2+\nu_1\nu_3+\nu_2\nu_3)\varphi (u_{\mathbf{0}})(y)=0,\\
&(R(\cC_3^{{\g}})\varphi (u_{\mathbf{0}}))(y)
-\nu_1\nu_2\nu_3\varphi (u_{\mathbf{0}})(y)=0.
\end{align*}
Applying (\ref{eqn:R3_LieKact_gen}) and Lemma \ref{lem:Rn_g_act_Cpsi} 
to these equalities with 
the expressions (\ref{eqn:R3_C1}), (\ref{eqn:R3_C2_modE13}) and 
(\ref{eqn:R3_C3_modE13}), 
we obtain the assertion. 
\end{proof}

By (\ref{eqn:R3_ps1_PDE1_C1}), 
we can define a function $\hat{\varphi}_{\mathbf{0}}$ on $(\bR_+)^2$ by  
\begin{align}
\label{eqn:R3_ps1_def_varphi_l}
&\varphi (u_{\mathbf{0}})(y)=
y_1y_2(y_2y_3)^{\nu_1+\nu_2+\nu_3}\hat{\varphi}_{\mathbf{0}}(y_1,y_2) 
\end{align}
with $y=\diag (y_1y_2y_3,y_2y_3,y_3)\in A$.

\begin{lem}
\label{lem:R3_ps1_PDE3}
Retain the notation. 
Then the function $\hat{\varphi}_{\mathbf{0}}$ satisfies the following system 
of partial differential equations: 
\begin{align}
\begin{split}\label{eqn:R3_ps1_PDE3_2}
&\{-\partial_1^2+\partial_1\partial_2-\partial_2^2
+(\nu_1+\nu_2+\nu_3)(\partial_1-\partial_2)\\
&-\nu_1\nu_2-\nu_1\nu_3-\nu_2\nu_3
+(2\pi y_1)^2+(2\pi y_2)^2\}\hat{\varphi}_{\mathbf{0}}
=0,
\end{split}\\[0.5mm]
&\label{eqn:R3_ps1_PDE3_3}
\{(\partial_1-\nu_1)(\partial_1-\nu_2)(\partial_1-\nu_3)
-(2\pi y_1)^2(\partial_1+\partial_2+2)\}\hat{\varphi}_{\mathbf{0}}=0.
\end{align}
\end{lem}
\begin{proof}
From (\ref{eqn:R3_ps1_PDE1_C2}) and (\ref{eqn:R3_ps1_PDE1_C3}), 
we have (\ref{eqn:R3_ps1_PDE3_2}) and  
\begin{align}
\begin{split}
&\{-\partial_1(-\partial_1+\partial_2+\nu_1+\nu_2+\nu_3)\partial_2
-\nu_1\nu_2\nu_3\\
&-(2\pi y_1)^2\partial_2
+(2\pi y_2)^2\partial_1\}\hat{\varphi}_{\mathbf{0}}=0,
\end{split}\label{eqn:R3_ps1_PDE2_C3}
\end{align}
respectively. 
Multiplying the both sides of (\ref{eqn:R3_ps1_PDE3_2}) 
by $-\partial_1$ from the left, we have 
\begin{align}
\begin{split}\label{eqn:R3_ps1_PDE2_D1C2}
&\{\partial_1^3
-\partial_1^2\partial_2
+\partial_1\partial_2^2
-(\nu_1+\nu_2+\nu_3)(\partial_1^2-\partial_1\partial_2)\\
&+(\nu_1\nu_2+\nu_1\nu_3+\nu_2\nu_3)\partial_1
-(2\pi y_1)^2(\partial_1+2)-(2\pi y_2)^2\partial_1
\}\hat{\varphi}_{\mathbf{0}}=0.
\end{split}
\end{align}
Adding the respective sides of 
(\ref{eqn:R3_ps1_PDE2_C3}) and (\ref{eqn:R3_ps1_PDE2_D1C2}), 
we obtain (\ref{eqn:R3_ps1_PDE3_3}). 
\end{proof}

\begin{thm}[{\cite[\S 1]{Bump_003}}]
\label{thm:R3_ps1_Whittaker}
Let $\sigma = \chi_{(\nu_1,\delta_2)}\boxtimes \chi_{(\nu_2,\delta_2)}
\boxtimes \chi_{(\nu_3,\delta_2)}$ with $\delta_2\in \{0,1\}$ and 
$\nu_1,\nu_2,\nu_3\in \bC$ such that $\Pi_\sigma$ is irreducible. \\
(1) There exists a $K$-homomorphism 
\[
\varphi^{\mathrm{mg}}_\sigma \colon V_{(0,\delta_2)}\to 
{\mathrm{Wh}}(\Pi_{\sigma},\psi_1)^{\mathrm{mg}},
\]
whose radial part is given by 
\begin{align*}
\varphi^{\mathrm{mg}}_\sigma (u_{\mathbf{0}})(y) =&
y_1y_2(y_2y_3)^{\nu_1+\nu_2+\nu_3}\\
&\times \frac{1}{(4\pi \sqrt{-1})^2} \int_{s_2}\int_{s_1}
\frac{\Gamma_{\bR}(s_1+\nu_1)\Gamma_{\bR}(s_1+\nu_2)\Gamma_{\bR}(s_1+\nu_3) }
{ \Gamma_{\bR}(s_1+s_2)}\\
&\times \Gamma_{\bR}(s_2-\nu_1) \Gamma_{\bR}(s_2-\nu_2)\Gamma_{\bR}(s_2-\nu_3)
\,y_1^{-s_1} y_2^{-s_2} \,ds_1ds_2
\end{align*}
with $y=\diag (y_1y_2y_3,y_2y_3,y_3)\in A$. 
Here the path of the integration $\int_{s_i}$ is the vertical line 
from $\mathrm{Re}(s_i)-\sI \infty$ to $\mathrm{Re}(s_i)+\sI \infty$ 
with sufficiently large real part to keep the poles of the integrand 
on its left. \\[2pt]
(2) Assume 
$\nu_p-\nu_q \notin 2\bZ$ for any $1\leq p\neq q\leq 3$. 
For a permutation $(i,j,k)$ of $\{1,2,3\}$, 
there exists a $K$-homomorphism 
\[
\varphi^{(i,j,k)}_\sigma \colon V_{(0,\delta_2)}\to 
{\mathrm{Wh}}(\Pi_{\sigma},\psi_1),
\]
whose radial part is given by the power series 
\begin{align*}
\begin{split}
\varphi^{{(i,j,k)}}_\sigma (u_{\mathbf{0}})(y) &=
y_1y_2(y_2y_3)^{\nu_1+\nu_2+\nu_3}\\
&\hphantom{=}\times \sum_{m_1,m_2 \geq 0}
\frac{(-1)^{m_1+m_2}
\Gamma \bigl(-m_1-\tfrac{\nu_i-\nu_j}{2}\bigr)
\Gamma \bigl(-m_1-\tfrac{\nu_i-\nu_k}{2}\bigr)}
{m_1!m_2!\Gamma \bigl(-m_1-m_2-\tfrac{\nu_i-\nu_j}{2}\bigr)}\\
&\hphantom{=}\times \Gamma \bigl(-m_2-\tfrac{\nu_i-\nu_j}{2}\bigr)
\Gamma \bigl(-m_2-\tfrac{\nu_k-\nu_j}{2}\bigr)
(\pi y_1)^{2m_1+\nu_i} (\pi y_2)^{2m_2-\nu_j}
\end{split}
\end{align*}
with $y=\diag (y_1y_2y_3,y_2y_3,y_3)\in A$. 
Moreover, $\{\varphi^{(i,j,k)}_\sigma \mid \{i,j,k\}=\{1,2,3\}\,\}$ 
forms a basis of 
the $\bC$-vector space $\Hom_K(V_{(0,\delta_2)},
{\mathrm{Wh}}(\Pi_{\sigma},\psi_1))$, and satisfies 
\begin{align*}
\varphi^{\mathrm{mg}}_\sigma 
& = \sum_{ (i,j,k) } \varphi^{(i,j,k)}_\sigma ,
\end{align*}
where $ (i,j,k) $ runs all permutations of $\{1,2,3\}$. 
\end{thm} 
\begin{proof}
By Lemma \ref{lem:R3_ps1_PDE3}, 
we can define an injective homomorphism 
\begin{align}
\label{eqn:R3_ps1_Wh_to_Sol}
\Hom_K(V_{(0,\delta_2)},{\mathrm{Wh}}(\Pi_{\sigma},\psi_1))
\ni \varphi 
\mapsto f_\varphi 
\in \mathrm{Sol}((\nu_1,\nu_2,\nu_3))
\end{align}
of $\bC$-vector spaces by $f_\varphi (z_1,z_2)
=\hat{\varphi}_{\mathbf{0}}(\pi^{-1}z_1,\,\pi^{-1}z_2)$
with (\ref{eqn:R3_ps1_def_varphi_l}). Here, for $r\in \bC^3$,
the space $\mathrm{Sol}(r)$ is defined in \S \ref{subsec:Fn_special2}. 
Since 
\[
\dim_\bC \Hom_K(V_{(0,\delta_2)} ,H(\sigma)_K)=1
\]
and $\Pi_\sigma$ is irreducible, we have 
\begin{align*}
\dim_\bC \Hom_K(V_{(0,\delta_2)} ,{\mathrm{Wh}}(\Pi_{\sigma} ,\psi_1))
&=\dim_\bC {\cI}_{\Pi_{\sigma},\psi_1}=6
\geq 
\dim_{\bC}\mathrm{Sol}((\nu_1,\nu_2,\nu_3)).
\end{align*}
Here the last inequality follows from Lemma \ref{lem:Fn_Sol_dim}.
This inequality implies that (\ref{eqn:R3_ps1_Wh_to_Sol}) 
is bijective. Hence, the assertion follows from 
Lemmas \ref{lem:Fn_Sol_power_series} and \ref{lem:Fn_Sol_mg}. 
\end{proof}

\section{Principal series Whittaker functions at $3$ dimensional $K$-types}
\label{subsec:R3_ps3_whittaker}

We use the notation in \S \ref{subsec:R3_ps}. 
We assume $\delta_1-\delta_3=1$, that is, 
\[
\sigma =\chi_{(\nu_1,1)}\boxtimes \chi_{(\nu_2,\delta_2)}
\boxtimes \chi_{(\nu_3,0)}
\] 
with $\nu_1,\nu_2,\nu_3\in \bC$ and 
$\delta_2\in \{0,1\}$. 
Then the minimal $K$-type $\tau_{(1,\delta_2)}$ of $\Pi_\sigma$ is 
$3$ dimensional. 
Let $\varphi \colon V_{(1,\delta_2)} \to 
{\mathrm{Wh}}(\Pi_\sigma ,\psi_1)$ be a $K$-homomorphism. 
Since $S_{(1,\delta_2)}=\{\me_1,\me_2,\me_3\}$, 
we note that $\varphi$ is characterized by 
$\varphi (u_{\me_i})(y)$ $(1\leq i\leq 3)$ with 
$y=\diag (y_1y_2y_3,y_2y_3,y_3)\in A$. 
We set $\displaystyle \partial_i=y_i\frac{\partial}{\partial y_i}$ 
for $1\leq i\leq 3$.

\begin{lem}
\label{lem:R3_ps3_ZgDSE1_original}
Retain the notation. 

\noindent (1) It holds that 
\begin{align}
&\label{eqn:R3_ps3_PDE1_C1}
(\partial_3-\nu_1-\nu_2-\nu_3)\varphi (u_{\me_i})(y)=0
\hspace{20mm}(1\leq i\leq 3),\\[3pt] 
&\label{eqn:R3_ps3_PDE1_C2}
(\Delta_2-\nu_1\nu_2-\nu_1\nu_3-\nu_2\nu_3)\varphi (u_{\me_1})(y)
-2\pi \sI y_1\varphi (u_{\me_2})(y)
=0,\\[3pt]
&\nonumber 
(\Delta_2-\nu_1\nu_2-\nu_1\nu_3-\nu_2\nu_3)\varphi (u_{\me_2})(y)\\
&\nonumber 
+2\pi \sI y_1\varphi (u_{\me_1})(y)
-2\pi \sI y_2\varphi (u_{\me_3})(y)=0,\\[3pt]
&\nonumber 
(\Delta_2-\nu_1\nu_2-\nu_1\nu_3-\nu_2\nu_3)\varphi (u_{\me_3})(y)
+2\pi \sI y_2\varphi (u_{\me_2})(y)=0,\\[3pt]
\begin{split}\label{eqn:R3_ps3_PDE1_C3}
&(\Delta_3-\nu_1\nu_2\nu_3)\varphi (u_{\me_1})(y)
-2\pi \sI y_1(-\partial_2+\partial_3+1)
\varphi (u_{\me_2})(y)\\
&-(2\pi y_1)(2\pi y_2)
\varphi (u_{\me_3})(y)
=0,
\end{split}\\[3pt]
&\nonumber 
(\Delta_3-\nu_1\nu_2\nu_3)\varphi (u_{\me_2})(y)
+2\pi \sI y_1(-\partial_2+\partial_3+1)
\varphi (u_{\me_1})(y)\\
&\nonumber 
-2\pi \sI y_2(\partial_1-1)\varphi (u_{\me_3})(y)=0,\\[3pt]
&\nonumber 
(\Delta_3-\nu_1\nu_2\nu_3)\varphi (u_{\me_3})(y)
+2\pi \sI y_2(\partial_1-1)
\varphi (u_{\me_2})(y)\\
&\nonumber 
+(2\pi y_1)(2\pi y_2)
\varphi (u_{\me_1})(y)
=0,
\end{align}
where 
\begin{align*}
\Delta_2=&\,
(\partial_1-1)(-\partial_1+\partial_2)
+(\partial_2-1)(-\partial_2+\partial_3+1)
+(2\pi y_1)^2+(2\pi y_2)^2,\\
\Delta_3=&\,
(\partial_1-1)(-\partial_1+\partial_2)(-\partial_2+\partial_3+1)\\
&+(2\pi y_1)^2(-\partial_2+\partial_3+1)
+(2\pi y_2)^2(\partial_1-1).
\end{align*}
\noindent (2) If $\delta_2=0$, it holds that 
\begin{align}
&\label{eqn:R3_ps31_PDE1_DSE1}
(2\partial_1-2-2\nu_1)\varphi (u_{\me_1})(y)
+4\pi \sI y_1\varphi (u_{\me_2})(y)=0,\\[2pt]
\begin{split}
\label{eqn:R3_ps31_PDE1_DSE2}
&(-2\partial_1+2\partial_2-2\nu_1)\varphi (u_{\me_2})(y)
+4\pi \sI y_1\varphi (u_{\me_1})(y)\\
&+4\pi \sI y_2\varphi (u_{\me_3})(y)
=0,
\end{split}\\[2pt]
&\nonumber 
(-2\partial_2+2\partial_3+2-2\nu_1)\varphi (u_{\me_3})(y)
+4\pi \sI y_2\varphi (u_{\me_2})(y)=0. 
\end{align}
If $\delta_2=1$, it holds that 
\begin{align*}
&(2\partial_1-2-2\nu_3)\varphi (u_{\me_1})(y)
+4\pi \sI y_1\varphi (u_{\me_2})(y)=0,\\[2pt]
\begin{split}
&
(-2\partial_1+2\partial_2-2\nu_3)\varphi (u_{\me_2})(y)
+4\pi \sI y_1\varphi (u_{\me_1})(y)\\
&+4\pi \sI y_2\varphi (u_{\me_3})(y)
=0,
\end{split}\\[2pt]
&\nonumber 
(-2\partial_2+2\partial_3+2-2\nu_3)\varphi (u_{\me_3})(y)
+4\pi \sI y_2\varphi (u_{\me_2})(y)=0. 
\end{align*}
\end{lem}
\begin{proof}
We note that there is a homomorphism $\Phi \in {\cI}_{\Pi_\sigma ,\psi_1}$ 
such that 
$\varphi =\Phi \circ \hat{\eta}_\sigma$. 
Hence, by Proposition \ref{prop:Rn_Ch_eigenvalue}, we have 
\begin{align*}
&(R(\cC_1^{{\g}})\varphi (u_{\me_i}))(y)
-(\nu_1+\nu_2+\nu_3)\varphi (u_{\me_i})(y)=0,\\
&(R(\cC_2^{{\g}})\varphi (u_{\me_i}))(y)
-(\nu_1\nu_2+\nu_1\nu_3+\nu_2\nu_3)\varphi (u_{\me_i})(y)=0,\\
&(R(\cC_3^{{\g}})\varphi (u_{\me_i}))(y)
-\nu_1\nu_2\nu_3\,\varphi (u_{\me_i})(y)=0
\end{align*}
for $1\leq i\leq 3$. 
Applying (\ref{eqn:R3_LieKact_gen}) and Lemma \ref{lem:Rn_g_act_Cpsi} 
to these equalities with the expressions (\ref{eqn:R3_C1}), 
(\ref{eqn:R3_C2_modE13}) and (\ref{eqn:R3_C3_modE13}), 
we obtain the statement (1).

By Proposition \ref{prop:R3_DSE_ps}, we have 
\begin{align*}
&\sum_{j=1}^3(R(E_{i,j}^{\gp})\varphi (u_{\me_j}))(y)
-2\nu_1\varphi (u_{\me_i})(y)=0&
&(1\leq i\leq 3)
\end{align*}
if $\delta_2=0$, and 
\begin{align*}
&\sum_{j=1}^3(R(E_{i,j}^{\gp})\varphi (u_{\me_j}))(y)
-2\nu_3\varphi (u_{\me_i})(y)=0&
&(1\leq i\leq 3)
\end{align*}
if $\delta_2=1$. 
Applying (\ref{eqn:R3_LieKact_gen}) and Lemma \ref{lem:Rn_g_act_Cpsi} 
to these equalities with 
\begin{align*}
&E_{i,j}^{\gp}=E_{j,i}^{\gp}=2E_{i,j}^{{\g}}-E_{i,j}^{{\gk}}&
&(1\leq i\leq j\leq 3), 
\end{align*}
we obtain the statement (2). 
\end{proof}

We will give explicit formulas of 
$\varphi (u_{\me_i})(y)$ $(1\leq i\leq 3)$ by 
solving the system of partial differential equations 
in Lemma \ref{lem:R3_ps3_ZgDSE1_original}. 
Here we note that the equations in Lemma \ref{lem:R3_ps3_ZgDSE1_original} 
for $\delta_2=0$ become those for $\delta_2=1$ by 
exchanging $\nu_1$ and $\nu_3$. 
Hence, it suffices to 
consider only the case of either $\delta_2=0$ or 
$\delta_2=1$. 

We set $\delta_2=0$. 
By (\ref{eqn:R3_ps3_PDE1_C1}), 
for $l=(l_1,l_2,l_3)\in S_{(1,0)}=\{\me_1,\me_2,\me_3\}$, 
we can define a function $\hat{\varphi}_l$ on $(\bR_+)^2$ by  
\begin{align}
\label{eqn:R3_ps3_def_varphi}
&\varphi (u_{l})(y)=
(\sI )^{l_1-l_3}y_1y_2(y_2y_3)^{\nu_1+\nu_2+\nu_3}
\hat{\varphi}_{l}(y_1,y_2)
\end{align}
with $y=\diag (y_1y_2y_3,y_2y_3,y_3)\in A$. 

\begin{lem}
\label{lem:R3_ps3_ZgDSE2}
Retain the notation. 
It holds that 
\begin{align}
&\label{eqn:R3_ps31_PDE2_DSE1}
(-\partial_1+\nu_1)\hat{\varphi}_{\me_1}
-2\pi y_1\hat{\varphi}_{\me_2}=0,\\
&\label{eqn:R3_ps31_PDE2_DSE2}
(-\partial_1+\partial_2+\nu_2+\nu_3)\hat{\varphi}_{\me_2}
-2\pi y_1\hat{\varphi}_{\me_1}
+2\pi y_2\hat{\varphi}_{\me_3}
=0. 
\end{align}
\end{lem}
\begin{proof}
The equations (\ref{eqn:R3_ps31_PDE2_DSE1}) and (\ref{eqn:R3_ps31_PDE2_DSE2})
follow immediately 
from (\ref{eqn:R3_ps31_PDE1_DSE1}) and (\ref{eqn:R3_ps31_PDE1_DSE2}), 
respectively. 
\end{proof}
The functions $\hat{\varphi}_{\me_i}$ $(1\leq i\leq 3)$ 
are determined from $\hat{\varphi}_{\me_1}$ by 
the equations in Lemma \ref{lem:R3_ps3_ZgDSE2}. 
Hence, we know that 
$\varphi \colon V_{(1,0)}\to {\mathrm{Wh}}(\Pi_\sigma ,\psi_1)$ 
is uniquely determined from $\hat{\varphi}_{\me_1}$. 

\begin{lem}
\label{lem:R3_ps3_PDE3}
Retain the notation. 
Then $\hat{\varphi}_{\me_1}$ satisfies the following system of partial 
differential equations: 
\begin{align}
\begin{split}
&
\{-\partial_1^2+\partial_1\partial_2-\partial_2^2
+(\nu_1+\nu_2+\nu_3+1)\partial_1
-(\nu_1+\nu_2+\nu_3)\partial_2\\
&\label{eqn:R3_ps31_PDE3_2}
-\nu_1\nu_2-\nu_1\nu_3-\nu_2\nu_3-\nu_1
+(2\pi y_1)^2+(2\pi y_2)^2\}\hat{\varphi}_{\me_1}
=0,
\end{split}\\[3pt]
&\label{eqn:R3_ps31_PDE3_3}
\{(\partial_1-\nu_1)(\partial_1-\nu_2-1)(\partial_1-\nu_3-1)
-(2\pi y_1)^2(\partial_1+\partial_2+1)
\}\hat{\varphi}_{\me_1}
=0. 
\end{align}
\end{lem}
\begin{proof}
From (\ref{eqn:R3_ps3_PDE1_C2}) and (\ref{eqn:R3_ps3_PDE1_C3}),
we have  
\begin{align}
\begin{split}
&
\{-\partial_1^2+\partial_1\partial_2-\partial_2^2
+(\nu_1+\nu_2+\nu_3)\partial_1-(\nu_1+\nu_2+\nu_3)\partial_2\\
&\label{eqn:R3_ps3_PDE2_C2}
-\nu_1\nu_2-\nu_1\nu_3-\nu_2\nu_3
+(2\pi y_1)^2+(2\pi y_2)^2\}\hat{\varphi}_{\me_1}
-(2\pi y_1)\hat{\varphi}_{\me_2}=0,
\end{split}\\[3pt]
\begin{split}
&
\{-\partial_1(-\partial_1+\partial_2+\nu_1+\nu_2+\nu_3)\partial_2
-\nu_1\nu_2\nu_3\\
&\label{eqn:R3_ps3_PDE2_C3}
-(2\pi y_1)^2\partial_2
+(2\pi y_2)^2\partial_1\}\hat{\varphi}_{\me_1}
+(2\pi y_1)\partial_2\hat{\varphi}_{\me_2}
+(2\pi y_1)(2\pi y_2)\hat{\varphi}_{\me_3}
=0,
\end{split}
\end{align} 
respectively. 
Multiplying the both sides of (\ref{eqn:R3_ps31_PDE2_DSE1}) 
by $-1$, we have 
\begin{align*}
&\left(\partial_1-\nu_1\right)\hat{\varphi}_{\me_1}
+(2\pi y_1)\hat{\varphi}_{\me_2}=0.
\end{align*}
Adding the respective sides of 
(\ref{eqn:R3_ps3_PDE2_C2}) and this equation, 
we obtain (\ref{eqn:R3_ps31_PDE3_2}).

Multiplying the both sides of (\ref{eqn:R3_ps3_PDE2_C2}) 
by $-\partial_1$ from the left, we have 
\begin{align}
\begin{split}
& 
\{\partial_1^3
-\partial_1^2\partial_2
+\partial_1\partial_2^2
-(\nu_1+\nu_2+\nu_3)\partial_1^2
+(\nu_1+\nu_2+\nu_3)\partial_1\partial_2\\
&+(\nu_1\nu_2+\nu_1\nu_3+\nu_2\nu_3)\partial_1
-(2\pi y_1)^2(\partial_1+2)-(2\pi y_2)^2\partial_1
\}\hat{\varphi}_{\me_1}\\
&\label{eqn:R3_ps3_PDE2_D1C2}
+(2\pi y_1)(\partial_1+1)\hat{\varphi}_{\me_2}=0.
\end{split}
\end{align}
Adding the respective sides of 
the equations (\ref{eqn:R3_ps3_PDE2_C3}) and (\ref{eqn:R3_ps3_PDE2_D1C2}), 
we have 
\begin{align}
\begin{split}&
\{(\partial_1-\nu_1)(\partial_1-\nu_2)(\partial_1-\nu_3)
-(2\pi y_1)^2(\partial_1+\partial_2+2)
\}\hat{\varphi}_{\me_1}\\
&\label{eqn:R3_ps3_pf1_C3C2}
+(2\pi y_1)(\partial_1+\partial_2+1)\hat{\varphi}_{\me_2}
+(2\pi y_1)(2\pi y_2)\hat{\varphi}_{\me_3}=0.
\end{split}
\end{align}
Multiplying 
$-(-2\partial_1+\nu_2+\nu_3+1)$ 
and $-2\pi y_1$ the both sides of 
(\ref{eqn:R3_ps31_PDE2_DSE1}) and (\ref{eqn:R3_ps31_PDE2_DSE2}) from the left, 
respectively, 
we have 
\begin{align*}
&
(\partial_1-\nu_1)(-2\partial_1+\nu_2+\nu_3+1)
\hat{\varphi}_{\me_1}
+(2\pi y_1)
\left(-2\partial_1+\nu_2+\nu_3-1\right)
\hat{\varphi}_{\me_2}=0,\\[2pt]
&
(2\pi y_1)
\left(\partial_1-\partial_2-\nu_2-\nu_3\right)
\hat{\varphi}_{\me_2}
+(2\pi y_1)^2\hat{\varphi}_{\me_1}
-(2\pi y_1)(2\pi y_2)\hat{\varphi}_{\me_3}=0. 
\end{align*}
Adding up the respective sides of 
(\ref{eqn:R3_ps3_pf1_C3C2}) and these equations, 
we obtain (\ref{eqn:R3_ps31_PDE3_3}). 
\end{proof}

\begin{lem}
\label{lem:R3_ps3_MW_sub1}
Let $\nu_1,\nu_2,\nu_3\in \bC $ and $(\delta_1,\delta_2,\delta_3)=(1,0,0)$. \\
(1) The space of smooth solutions of the system in 
Lemma \ref{lem:R3_ps3_PDE3} on $(\bR_+)^2$ is at most 
$6$ dimensional. \\
\noindent (2) Set  
\[
\hat{\varphi}^{\mathrm{mg}}_{\me_1}(y_1,y_2)
=\frac{1}{(4\pi \sqrt{-1})^2} \int_{s_2}\int_{s_1} 
\cV_{\me_1}(s_1,s_2) \,
y_1^{-s_1} y_2^{-s_2} \,ds_1ds_2
\]
with 
\begin{align*}
\cV_{\me_1}(s_1,s_2)=&
\frac{\Gamma_{\bR}(s_1+\nu_1)
\Gamma_{\bR}(s_1+\nu_2+1)\Gamma_{\bR}(s_1+\nu_3+1)}
{\Gamma_{\bR}(s_1+s_2+1)}\\
&\times \Gamma_{\bR}(s_2-\nu_1+1)
\Gamma_{\bR}(s_2-\nu_2)\Gamma_{\bR}(s_2-\nu_3).
\end{align*}
Here the path of the integration $\int_{s_i}$ is the vertical line 
from $\mathrm{Re}(s_i)-\sI \infty$ to $\mathrm{Re}(s_i)+\sI \infty$ 
with sufficiently large real part to keep the poles of the integrand 
on its left. 
Then $\hat{\varphi}_{\me_1}=\hat{\varphi}^{\mathrm{mg}}_{\me_1}$ is 
a moderate growth solution of the system in Lemma \ref{lem:R3_ps3_PDE3} 
on $(\bR_+)^2$. 
\\[1mm]
(3) Assume $\nu_p-\nu_q-\delta_p+\delta_q \notin 2\bZ$ 
for any $1\leq p\neq q\leq 3$. 
For a permutation $(i,j,k)$ of $\{1,2,3\}$, set  
\[
\hat{\varphi}^{{(i,j,k)}}_{\me_1}(y_1,y_2) =
\sum_{m_1,m_2 \geq 0}
C_{\me_1,(m_1,m_2)}^{(i,j,k)}(\pi y_1)^{2m_1+\nu_i+1-\delta_i}
(\pi y_2)^{2m_2-\nu_j+\delta_j}
\]
with 
\begin{align*}
C_{\me_1,(m_1,m_2)}^{(i,j,k)}=
&\frac{(-1)^{m_1+m_2} 
\Gamma \bigl(-m_1-\tfrac{\nu_i-\nu_j-\delta_i+\delta_j}{2}\bigr)
\Gamma \bigl(-m_1-\tfrac{\nu_i-\nu_k-\delta_i+\delta_k}{2}\bigr)}
{\pi \, m_1!\,m_2!\,\Gamma 
\bigl(-m_1-m_2-\tfrac{\nu_i-\nu_j-\delta_i+\delta_j}{2}\bigr)} \\
& \times \Gamma \bigl(-m_2-\tfrac{\nu_i-\nu_j-\delta_i+\delta_j}{2}\bigr)
\Gamma \bigl(-m_2-\tfrac{\nu_k-\nu_j-\delta_k+\delta_j}{2}\bigr).
\end{align*}
Then $\{\hat{\varphi}^{{(i,j,k)}}_{\me_1}\mid 
\{i,j,k\}=\{1,2,3\}\,\}$ forms a basis of 
the space of smooth solutions of the system in Lemma \ref{lem:R3_ps3_PDE3} 
on $(\bR_+)^2$. Moreover, 
it holds that 
\begin{align*}
\hat{\varphi}^{\mathrm{mg}}_{\me_1}
& = \sum_{ (i,j,k) } \hat{\varphi}^{(i,j,k)}_{\me_1},
\end{align*}
where $ (i,j,k) $ runs all permutations of $\{1,2,3\}$. 
\end{lem} 
\begin{proof}
If we set 
$f(z_1,z_2)=
z_1^{-1}\hat{\varphi}_{\me_1}(\pi^{-1}z_1,\,\pi^{-1}z_2)$ and 
\begin{align*}
r=(r_1,r_2,r_3)=(\nu_1-1,\nu_2,\nu_3),
\end{align*}
then (\ref{eqn:R3_ps31_PDE3_2}) and (\ref{eqn:R3_ps31_PDE3_3}) 
imply 
(\ref{eqn:Fn_Sol_PDE2}) and (\ref{eqn:Fn_Sol_PDE3}), 
respectively. 
Hence, the assertion follows from Lemmas \ref{lem:Fn_Sol_dim}, 
\ref{lem:Fn_Sol_power_series} and \ref{lem:Fn_Sol_mg}. 
\end{proof}

\begin{lem}
\label{lem:R3_ps3_MW_sub2}
Set $(\delta_1,\delta_2,\delta_3)=(1,0,0)$. 
Let $\nu_1,\nu_2,\nu_3\in \bC $ such that 
$\nu_p-\nu_q-\delta_p+\delta_q \notin 2\bZ$ for any $1\leq p\neq q\leq 3$. 
Let $(i,j,k)$ be a permutation of $\{1,2,3\}$. 
Let $\hat{\varphi}_{\me_2}$ and $\hat{\varphi}_{\me_3}$ 
be the functions determined from the function 
$\hat{\varphi}_{\me_1}=\hat{\varphi}^{{(i,j,k)}}_{\me_1}$ 
in Lemma \ref{lem:R3_ps3_MW_sub1} (3) by the equations 
in Lemma \ref{lem:R3_ps3_ZgDSE2}. 
Then it holds that 
\begin{align*}
\hat{\varphi}_{\me_2}(y_1,y_2) = 
&\sum_{m_1, m_2 \geq 0} \frac{(-1)^{m_1+m_2} 
\Gamma\bigl(-m_1-\tfrac{\nu_i-\nu_j+\delta_i-\delta_j}{2}\bigr)
\Gamma\bigl(-m_1-\tfrac{\nu_i-\nu_k+\delta_i-\delta_k}{2}\bigr)}
{\pi \, m_1!m_2! 
\Gamma\bigl(-m_1-m_2-\frac{\nu_i-\nu_j+\delta_i+\delta_j}{2}\bigr)} \\
& \times \Gamma\bigl(-m_2-\tfrac{\nu_i-\nu_j-\delta_i+\delta_j}{2}\bigr)
\Gamma\bigl(-m_2-\tfrac{\nu_k-\nu_j-\delta_k+\delta_j}{2}\bigr)\\
& \times (\pi y_1)^{2m_1+\nu_i+\delta_i}(\pi y_2)^{2m_2-\nu_j+\delta_j},\\[2pt]
\hat{\varphi}_{\me_3}(y_1,y_2) = &\sum_{m_1,m_2 \geq 0} 
\frac{(-1)^{m_1+m_2} 
\Gamma\bigl(-m_1-\tfrac{\nu_i-\nu_j+\delta_i-\delta_j}{2}\bigr)
\Gamma\bigl(-m_1-\tfrac{\nu_i-\nu_k+\delta_i-\delta_k}{2}\bigr)}
{\pi \, m_1!m_2! 
\Gamma\bigl(-m_1-m_2-\tfrac{\nu_i-\nu_j+\delta_i-\delta_j}{2}\bigr)} \\
& \times \Gamma\bigl(-m_2-\tfrac{\nu_i-\nu_j+\delta_i-\delta_j}{2}\bigr)
\Gamma\bigl(-m_2-\tfrac{\nu_k-\nu_j+\delta_k-\delta_j}{2}\bigr)\\
& \times (\pi y_1)^{2m_1+\nu_i+\delta_i}(\pi y_2)^{2m_2-\nu_j+1-\delta_j}.
\end{align*}
\end{lem}
\begin{proof}
By the differential equation (\ref{eqn:R3_ps31_PDE2_DSE1}), we have 
\begin{align*}
&\hat{\varphi}_{\me_2}(y_1,y_2)
 = (2\pi y_1)^{-1}(-\partial_1+\nu_1)\hat{\varphi}_{\me_1}(y_1,y_2)\\
& = \sum_{m_1,m_2 \geq 0} 
\bigl(-m_1-\tfrac{\nu_i-\nu_1+1-\delta_i}{2}\bigr)
C_{\me_1,(m_1,m_2)}^{(i,j,k)}
(\pi y_1)^{2m_1+\nu_i-\delta_i}(\pi y_2)^{2m_2-\nu_j+\delta_j}.
\end{align*}
Hence, the replacement $m_1 \to m_1+\delta_i$ implies 
the desired expression for $\hat{\varphi}_{\me_2}$, 
in view of 
\begin{align*}
 -m_1-\tfrac{\nu_i-\nu_1+1-\delta_i}{2} 
 = \begin{cases}
-m_1 & \mbox{ if }i=1, \\[1pt]
-m_1-\tfrac{\nu_i-\nu_j-\delta_i+\delta_j}{2} 
& \mbox{ if }j=1,\\[1pt] 
-m_1-\tfrac{\nu_i-\nu_k-\delta_i+\delta_k}{2} 
& \mbox{ if }k=1.
\end{cases}
\end{align*}

From (\ref{eqn:R3_ps31_PDE2_DSE2}), we have 
\begin{align*}
&\hat{\varphi}_{\me_3}(y_1,y_2) 
 = (2\pi y_2)^{-1}\bigl\{
\bigl(\partial_1-\partial_2-\nu_2-\nu_3\bigr)\hat{\varphi}_{\me_2}(y_1,y_2)
+(2\pi y_1)\hat{\varphi}_{\me_1}(y_1,y_2)\}
\\
& = \sum_{m_1,m_2 \ge0} 
\bigl(m_1-m_2+\tfrac{\nu_i+\nu_j-\nu_2-\nu_3-\delta_i-\delta_j}{2}\bigr)
\bigl(-m_1-\tfrac{\nu_i-\nu_1+1-\delta_i}{2}\bigr) 
C_{\me_1,(m_1,m_2)}^{(i,j,k)}\\
&\qquad 
\times  (\pi y_1)^{2m_1+\nu_i-\delta_i} (\pi y_2)^{2m_2-\nu_j-1+\delta_j}\\
&\qquad 
+\sum_{m_1,m_2 \geq 0} C_{\me_1,(m_1,m_2)}^{(i,j,k)} 
(\pi y_1)^{2m_1+\nu_i+2-\delta_i}
(\pi y_2)^{2m_2-\nu_j-1+\delta_j}\\
& = \sum_{m_1,m_2 \geq 0} 
\bigl\{\bigl(m_1-m_2+\tfrac{\nu_i+\nu_j-\nu_2-\nu_3-\delta_i-\delta_j}{2}\bigr)
\bigl(-m_1-\tfrac{\nu_i-\nu_1+1-\delta_i}{2}\bigr)
C_{\me_1,(m_1,m_2)}^{(i,j,k)}  \\
& \qquad +C_{\me_1,(m_1-1,m_2)}^{(i,j,k)}\bigr\}
(\pi y_1)^{2m_1+\nu_i-\delta_i} (\pi y_2)^{2m_2-\nu_j-1+\delta_j}
\\
& = \sum_{m_1,m_2\geq 0}\Biggl\{ 
\bigl(m_1-m_2+\tfrac{\nu_i+\nu_j-\nu_2-\nu_3-\delta_i-\delta_j}{2}\bigr)
\bigl(-m_1-\tfrac{\nu_i-\nu_1+1-\delta_i}{2}\bigr)\\
& \qquad +
\frac{m_1\bigl(m_1+\tfrac{\nu_i-\nu_j-\delta_i+\delta_j}{2}\bigr)
\bigl(m_1+\tfrac{\nu_i-\nu_k-\delta_i+\delta_k}{2}\bigr)}
{m_1+m_2+\tfrac{\nu_i-\nu_j-\delta_i+\delta_j}{2}}\Biggr\} 
C_{\me_1,(m_1,m_2)}^{(i,j,k)}\\
& \qquad  \times 
(\pi y_1)^{2m_1+\nu_i-\delta_i} (\pi y_2)^{2m_2-\nu_j-1+\delta_j}.
\end{align*}
Here the third equality follows from 
the replacement $m_1\to m_1-1$ in the second term and 
setting $C_{\me_1,(-1,m_2)}^{(i,j,k)}=0$. 
Since 
\begin{align*}
&\bigl(m_1-m_2+\tfrac{\nu_i+\nu_j-\nu_2-\nu_3-\delta_i-\delta_j}{2}\bigr)
\bigl(-m_1-\tfrac{\nu_i-\nu_1+1-\delta_i}{2}\bigr)\\
& +
\frac{m_1\bigl(m_1+\tfrac{\nu_i-\nu_j-\delta_i+\delta_j}{2}\bigr)
\bigl(m_1+\tfrac{\nu_i-\nu_k-\delta_i+\delta_k}{2}\bigr)}
{m_1+m_2+\tfrac{\nu_i-\nu_j-\delta_i+\delta_j}{2}}\\
& = 
\frac{1}{m_1+m_2+\tfrac{\nu_i-\nu_j-\delta_i+\delta_j}{2}} \\
&\hphantom{=.}
\times \begin{cases} 
m_1m_2\bigl(m_2+\tfrac{\nu_k-\nu_j-\delta_k+\delta_j}{2}\bigr) 
& \mbox{ if }i=1,\\[1pt]
\bigl(m_1+\tfrac{\nu_i-\nu_j-\delta_i+\delta_j}{2}\bigr)
\bigl(m_2+\tfrac{\nu_i-\nu_j-\delta_i+\delta_j}{2}\bigr)
\bigl(m_2+\tfrac{\nu_k-\nu_j-\delta_k+\delta_j}{2}\bigr) 
& \mbox{ if }j=1,\\[1pt]
m_2\bigl(m_1+\tfrac{\nu_i-\nu_k-\delta_i+\delta_k}{2}\bigr)
\bigl(m_2+\tfrac{\nu_i-\nu_j-\delta_i+\delta_j}{2}\bigr) 
& \mbox{ if }k=1,
\end{cases}
\end{align*}
the replacement $(m_1,m_2) \to (m_1+\delta_i,m_2+1-\delta_j) $ implies 
the desired expression for $\hat{\varphi}_{\me_3}$.
\end{proof}

\begin{lem}
\label{lem:R3_ps3_WW_sub2}
Let $\nu_1,\nu_2,\nu_3\in \bC $ and $(\delta_1,\delta_2,\delta_3)=(1,0,0)$. 
Let $\hat{\varphi}_{\me_2}$ and $\hat{\varphi}_{\me_3}$ 
be the functions determined from the function 
$\hat{\varphi}_{\me_1}=\hat{\varphi}^{\mathrm{mg}}_{\me_1}$ 
in Lemma \ref{lem:R3_ps3_MW_sub1} (2) by the equations 
in Lemma \ref{lem:R3_ps3_ZgDSE2}. 
Then it holds that 
\begin{align*}
\hat{\varphi}_{\me_2}(y_1,y_2) = &
\frac{1}{(4\pi \sqrt{-1})^2} 
\int_{s_2}\int_{s_1} \frac{\Gamma_{\bR}(s_1+\nu_1+1)\Gamma_{\bR}(s_1+\nu_2)
\Gamma_{\bR}(s_1+\nu_3)}{\Gamma_{\bR}(s_1+s_2)}\\
&\times \Gamma_{\bR}(s_2-\nu_1+1)\Gamma_{\bR}(s_2-\nu_2)
\Gamma_{\bR}(s_2-\nu_3) \,y_1^{-s_1} y_2^{-s_2} \,ds_1ds_2,\\[2pt]
\hat{\varphi}_{\me_3}(y_1,y_2) = &
\frac{1}{(4\pi \sqrt{-1})^2} 
\int_{s_2}\int_{s_1} \frac{\Gamma_{\bR}(s_1+\nu_1+1)\Gamma_{\bR}(s_1+\nu_2)
\Gamma_{\bR}(s_1+\nu_3)}{\Gamma_{\bR}(s_1+s_2+1)}\\
&\times \Gamma_{\bR}(s_2-\nu_1)\Gamma_{\bR}(s_2-\nu_2+1)
\Gamma_{\bR}(s_2-\nu_3+1)\,y_1^{-s_1} y_2^{-s_2} \,ds_1ds_2.
\end{align*}
Here the path of the integration $\int_{s_i}$ is the vertical line 
from $\mathrm{Re}(s_i)-\sI \infty$ to $\mathrm{Re}(s_i)+\sI \infty$ 
with sufficiently large real part to keep the poles of the integrand 
on its left. 
\end{lem} 
\begin{proof}
By (\ref{eqn:R3_ps31_PDE2_DSE1}), we have 
\begin{align*}
&\hat{\varphi}_{\me_2}(y_1,y_2)
= (2\pi y_1)^{-1}(-\partial_1+\nu_1)
\hat{\varphi}_{\me_1}^{\mathrm{mg}}(y_1,y_2)\\
&=\frac{1}{(4\pi \sqrt{-1})^2} \int_{s_2}\int_{s_1} 
(2\pi )^{-1}(s_1+\nu_1)\cV_{\me_1}(s_1,s_2) \,
y_1^{-s_1-1} y_2^{-s_2} \,ds_1ds_2.
\end{align*}
Hence, we obtain the desired expression of $\hat{\varphi}_{\me_2}$ by 
the substitution $s_1\to s_1-1$.

From (\ref{eqn:R3_ps31_PDE2_DSE2}), we have 
\begin{align*}
&\hat{\varphi}_{\me_3}(y_1,y_2)= (2\pi y_2)^{-1}
\{(2\pi y_1)\hat{\varphi}_{\me_1}^{\mathrm{mg}}(y_1,y_2)+
(\partial_1-\partial_2-\nu_2-\nu_3)\hat{\varphi}_{\me_2}
(y_1,y_2)\}\\
&=\frac{1}{(4\pi \sqrt{-1})^2} \int_{s_2}\int_{s_1} \cV_{\me_1}(s_1,s_2) 
\,y_1^{-s_1+1} y_2^{-s_2-1}\,ds_1ds_2\\
&\hphantom{=}+\frac{1}{(4\pi \sqrt{-1})^2}
\int_{s_2}\int_{s_1} 
(2\pi )^{-2}(s_1+\nu_1)(-s_1+s_2-\nu_2-\nu_3-1)
\cV_{\me_1}(s_1,s_2) \\
&\qquad \times y_1^{-s_1-1} y_2^{-s_2-1}\,ds_1ds_2\\
&=\frac{1}{(4\pi \sqrt{-1})^2} \int_{s_2}\int_{s_1} 
\bigl\{\cV_{\me_1}(s_1+2,s_2)\\
&\hphantom{=}+(2\pi )^{-2}
(s_1+\nu_1)(-s_1+s_2-\nu_2-\nu_3-1)
\cV_{\me_1}(s_1,s_2) \bigr\}
y_1^{-s_1-1} y_2^{-s_2-1}\,ds_1ds_2.
\end{align*}
Here the last equality follows from the substitution $s_1\to s_1+2$ 
in the first term. 
Since 
\begin{align*}
&\cV_{\me_1}(s_1+2,s_2)+(2\pi )^{-2}
(-s_1+s_2-\nu_2-\nu_3-1)
\cV_{\me_1}(s_1,s_2)\\
&=\biggl\{
\frac{(s_1+\nu_2+1)(s_1+\nu_3+1)}{s_1+s_2+1}
+(-s_1+s_2-\nu_2-\nu_3-1)\biggr\}\\
&\hphantom{=.}\times (2\pi)^{-2}(s_1+\nu_1)\cV_{\me_1}(s_1,s_2)\\
&=(2\pi)^{-2}
\frac{(s_1+\nu_1)(s_2-\nu_2)(s_2-\nu_3)}{s_1+s_2+1}
\cV_{\me_1}(s_1,s_2),
\end{align*}
the substitution $(s_1,s_2) \to (s_1-1,s_2-1) $ implies 
the desired expression for $\hat{\varphi}_{\me_3}$. 
\end{proof}

\begin{thm}[{\cite{Manabe_Ihii_Oda_001}}]
\label{thm:R3_ps31_Whittaker}
Let $\sigma = \chi_{(\nu_1,1)}\boxtimes \chi_{(\nu_2,0)}
\boxtimes \chi_{(\nu_3,0)}$ with 
$\nu_1,\nu_2,\nu_3\in \bC$ such that $\Pi_\sigma$ is irreducible. \\
(1) There exists a $K$-homomorphism 
\[
\varphi^{\mathrm{mg}}_\sigma \colon V_{(1,0)}\to 
{\mathrm{Wh}}(\Pi_{\sigma},\psi_1)^{\mathrm{mg}},
\]
whose radial part is given by 
\begin{align*}
\begin{split}
&\varphi^{\mathrm{mg}}_\sigma (u_{l})(y) =
(\sI )^{l_1-l_3}y_1y_2(y_2y_3)^{\nu_1+\nu_2+\nu_3}\\
&\hphantom{=}
\times \frac{1}{(4\pi \sqrt{-1})^2} \int_{s_2}\int_{s_1}
\frac{\Gamma_{\bR}(s_1+\nu_1+1-l_1)\Gamma_{\bR}(s_1+\nu_2+l_1)
\Gamma_{\bR}(s_1+\nu_3+l_1)}{\Gamma_{\bR}(s_1+s_2+l_1+l_3)}\\
&\hphantom{=}
\times \Gamma_{\bR}(s_2-\nu_1+1-l_3)\Gamma_{\bR}(s_2-\nu_2+l_3)
\Gamma_{\bR}(s_2-\nu_3+l_3)
\,y_1^{-s_1} y_2^{-s_2} \,ds_1ds_2
\end{split}
\end{align*}
with $l=(l_1,l_2,l_3)\in S_{(1,0)}$ 
and $y=\diag (y_1y_2y_3,y_2y_3,y_3)\in A$. 
Here the path of the integration $\int_{s_i}$ is the vertical line 
from $\mathrm{Re}(s_i)-\sI \infty$ to $\mathrm{Re}(s_i)+\sI \infty$ 
with sufficiently large real part to keep the poles of the integrand 
on its left. \\[2pt]
(2) Assume 
$\nu_p-\nu_q-\delta_p+\delta_q \notin 2\bZ$ 
with $(\delta_1,\delta_2,\delta_3)=(1,0,0)$, for any $1\leq p\neq q\leq 3$. 
For a permutation $(i,j,k)$ of $\{1,2,3\}$, 
there is a $K$-homomorphism 
\[
\varphi^{(i,j,k)}_\sigma \colon V_{(1,0)}\to 
{\mathrm{Wh}}(\Pi_{\sigma},\psi_1),
\]
whose radial part is given by the power series 
\begin{align*}
\begin{split}
&\varphi^{{(1,j,k)}}_\sigma (u_{l})(y) =
(\sI )^{l_1-l_3}y_1y_2(y_2y_3)^{\nu_1+\nu_2+\nu_3}\\
&\hphantom{=}\times \sum_{m_1, m_2 \geq 0} 
\frac{(-1)^{m_1+m_2}
\Gamma\bigl(-m_1+l_1-\tfrac{\nu_1-\nu_2+1}{2}\bigr)
\Gamma\bigl(-m_1+l_1-\tfrac{\nu_1-\nu_3+1}{2}\bigr)
}{\pi \, m_1!m_2! \Gamma\bigl(-m_1-m_2+l_1-\tfrac{\nu_1-\nu_j+1}{2}\bigr)} \\
&\hphantom{=}\times \Gamma\bigl(-m_2-l_3-\tfrac{\nu_1-\nu_j-1}{2}\bigr)
\Gamma\bigl(-m_2-\tfrac{\nu_k-\nu_j}{2}\bigr)
(\pi y_1)^{2m_1+\nu_1+1-l_1}(\pi y_2)^{2m_2-\nu_j+l_3},
\end{split}\\[3pt]
\begin{split}
&\varphi^{{(i,1,k)}}_\sigma (u_{l})(y) =
(\sI )^{l_1-l_3}y_1y_2(y_2y_3)^{\nu_1+\nu_2+\nu_3}\\
&\hphantom{=}\times \sum_{m_1, m_2 \geq 0} 
\frac{(-1)^{m_1+m_2} 
\Gamma\bigl(-m_2+l_3-\tfrac{\nu_2-\nu_1+1}{2}\bigr)
\Gamma\bigl(-m_2+l_3-\tfrac{\nu_3-\nu_1+1}{2}\bigr)}
{\pi \, m_1!m_2!\Gamma\bigl(-m_1-m_2+l_3-\tfrac{\nu_i-\nu_1+1}{2}\bigr)} \\
&\hphantom{=}\times 
\Gamma\bigl(-m_1-l_1-\tfrac{\nu_i-\nu_1-1}{2}\bigr)
\Gamma\bigl(-m_1-\tfrac{\nu_i-\nu_k}{2}\bigr)
(\pi y_1)^{2m_1+\nu_i+l_1}(\pi y_2)^{2m_2-\nu_1+1-l_3},
\end{split}\\[3pt]
\begin{split}
&\varphi^{{(i,j,1)}}_\sigma (u_{l})(y) =
(\sI )^{l_1-l_3}y_1y_2(y_2y_3)^{\nu_1+\nu_2+\nu_3}\\
&\hphantom{=}\times \sum_{m_1, m_2 \geq 0} 
\frac{(-1)^{m_1+m_2} 
\Gamma\bigl(-m_1-l_1-\tfrac{\nu_i-\nu_1-1}{2}\bigr)
\Gamma\bigl(-m_1-\tfrac{\nu_i-\nu_j}{2}\bigr)}
{\pi \, m_1!m_2! \Gamma\bigl(-m_1-m_2-\tfrac{\nu_i-\nu_j}{2}\bigr)} \\
&\hphantom{=}\times 
\Gamma\bigl(-m_2-l_3-\tfrac{\nu_1-\nu_j-1}{2}\bigr)
\Gamma\bigl(-m_2-\tfrac{\nu_i-\nu_j}{2}\bigr)
(\pi y_1)^{2m_1+\nu_i+l_1}(\pi y_2)^{2m_2-\nu_j+l_3}
\end{split}
\end{align*}
with $l=(l_1,l_2,l_3)\in S_{(1,0)}$ and 
$y=\diag (y_1y_2y_3,y_2y_3,y_3)\in A$. 
Moreover, $\{\varphi^{(i,j,k)}_\sigma \mid \{i,j,k\}=\{1,2,3\}\,\}$ 
forms a basis of  $\Hom_K(V_{(1,0)},
{\mathrm{Wh}}(\Pi_{\sigma},\psi_1))$, and satisfies 
\begin{align*}
\varphi^{\mathrm{mg}}_\sigma 
& = \sum_{ (i,j,k) } \varphi^{(i,j,k)}_\sigma ,
\end{align*}
where $(i,j,k)$ runs all permutations of $\{1,2,3\}$. 
\end{thm} 
\begin{proof}
We denote by $\mathrm{Sol}(\Pi_{\sigma};\psi_1)$ 
the space of smooth solutions of 
the system in Lemma \ref{lem:R3_ps3_PDE3} on $(\bR_+)^2$. 
We define an injective homomorphism 
\begin{align}
\label{eqn:R3_ps3_Wh_to_Sol}
\Hom_K(V_{(1,0)},{\mathrm{Wh}}(\Pi_{\sigma},\psi_1))
\ni \varphi 
\mapsto \hat{\varphi}_{\me_1}
\in \mathrm{Sol}(\Pi_{\sigma};\psi_1)
\end{align}
of $\bC$-vector spaces 
by (\ref{eqn:R3_ps3_def_varphi}). 

Since $\dim_\bC \Hom_K(V_{(1,0)} ,H(\sigma)_K)=1$ and 
$\Pi_\sigma$ is irreducible, we have 
\begin{align*}
\dim_\bC \Hom_K(V_{(1,0)} ,{\mathrm{Wh}}(\Pi_{\sigma} ,\psi_1))
&=\dim_\bC {\cI}_{\Pi_{\sigma},\psi_1}=6
\geq 
\dim_{\bC}\mathrm{Sol}(\Pi_{\sigma};\psi_1).
\end{align*}
Here the last inequality follows from Lemma \ref{lem:R3_ps3_MW_sub1} (1).
This inequality implies that (\ref{eqn:R3_ps3_Wh_to_Sol}) 
is bijective. Hence, the assertion follows from 
Lemmas \ref{lem:R3_ps3_MW_sub1}, 
\ref{lem:R3_ps3_MW_sub2} and \ref{lem:R3_ps3_WW_sub2}. 
\end{proof}

Since the equations in Lemma \ref{lem:R3_ps3_ZgDSE1_original} 
for $\delta_2=0$ become those for $\delta_2=1$ by 
exchanging $\nu_1$ and $\nu_3$, 
we note that the following theorem also holds.

\begin{thm}[{\cite{Manabe_Ihii_Oda_001}}]
\label{thm:R3_ps33_Whittaker}
Let $\sigma = \chi_{(\nu_1,1)}\boxtimes \chi_{(\nu_2,1)}
\boxtimes \chi_{(\nu_3,0)}$ with 
$\nu_1,\nu_2,\nu_3\in \bC$ such that $\Pi_\sigma$ is irreducible. \\
(1) There exists a $K$-homomorphism 
\[
\varphi^{\mathrm{mg}}_\sigma \colon V_{(1,1)}\to 
{\mathrm{Wh}}(\Pi_{\sigma},\psi_1)^{\mathrm{mg}},
\]
whose radial part is given by 
\begin{align*}
\begin{split}
&\varphi^{\mathrm{mg}}_\sigma (u_{l})(y) =
(\sI )^{l_1-l_3}y_1y_2(y_2y_3)^{\nu_1+\nu_2+\nu_3}\\
&\hphantom{=}
\times \frac{1}{(4\pi \sqrt{-1})^2} \int_{s_2}\int_{s_1}
\frac{\Gamma_{\bR}(s_1+\nu_1+l_1)\Gamma_{\bR}(s_1+\nu_2+l_1)
\Gamma_{\bR}(s_1+\nu_3+1-l_1)}{\Gamma_{\bR}(s_1+s_2+l_1+l_3)}\\
&\hphantom{=}
\times \Gamma_{\bR}(s_2-\nu_1+l_3)\Gamma_{\bR}(s_2-\nu_2+l_3)
\Gamma_{\bR}(s_2-\nu_3+1-l_3)
\,y_1^{-s_1} y_2^{-s_2} \,ds_1ds_2
\end{split}
\end{align*}
with $l=(l_1,l_2,l_3)\in S_{(1,1)}$ 
and $y=\diag (y_1y_2y_3,y_2y_3,y_3)\in A$. 
Here the path of the integration $\int_{s_i}$ is the vertical line 
from $\mathrm{Re}(s_i)-\sI \infty$ to $\mathrm{Re}(s_i)+\sI \infty$ 
with sufficiently large real part to keep the poles of the integrand 
on its left. \\[2pt]
(2) Assume 
$\nu_p-\nu_q-\delta_p+\delta_q \notin 2\bZ$ 
with $(\delta_1,\delta_2,\delta_3)=(1,1,0)$, for any $1\leq p\neq q\leq 3$. 
For a permutation $(i,j,k)$ of $\{1,2,3\}$, 
there is a $K$-homomorphism 
\[
\varphi^{(i,j,k)}_\sigma \colon V_{(1,1)}\to 
{\mathrm{Wh}}(\Pi_{\sigma},\psi_1),
\]
whose radial part is given by the power series 
\begin{align*}
\begin{split}
&\varphi^{{(3,j,k)}}_\sigma (u_{l})(y) =
(\sI )^{l_1-l_3}y_1y_2(y_2y_3)^{\nu_3+\nu_2+\nu_1}\\
&\hphantom{=}\times \sum_{m_1, m_2 \geq 0} 
\frac{(-1)^{m_1+m_2}
\Gamma\bigl(-m_1+l_1-\tfrac{\nu_3-\nu_1+1}{2}\bigr)
\Gamma\bigl(-m_1+l_1-\tfrac{\nu_3-\nu_2+1}{2}\bigr)
}{\pi \, m_1!m_2! \Gamma\bigl(-m_1-m_2+l_1-\tfrac{\nu_3-\nu_j+1}{2}\bigr)} \\
&\hphantom{=}\times \Gamma\bigl(-m_2-\tfrac{\nu_k-\nu_j}{2}\bigr)
\Gamma\bigl(-m_2-l_3-\tfrac{\nu_3-\nu_j-1}{2}\bigr)
(\pi y_1)^{2m_1+\nu_3+1-l_1}(\pi y_2)^{2m_2-\nu_j+l_3},
\end{split}\\[3pt]
\begin{split}
&\varphi^{{(i,3,k)}}_\sigma (u_{l})(y) =
(\sI )^{l_1-l_3}y_1y_2(y_2y_3)^{\nu_1+\nu_2+\nu_3}\\
&\hphantom{=}\times \sum_{m_1, m_2 \geq 0} 
\frac{(-1)^{m_1+m_2} 
\Gamma\bigl(-m_2+l_3-\tfrac{\nu_1-\nu_3+1}{2}\bigr)
\Gamma\bigl(-m_2+l_3-\tfrac{\nu_2-\nu_3+1}{2}\bigr)}
{\pi \, m_1!m_2!\Gamma\bigl(-m_1-m_2+l_3-\tfrac{\nu_i-\nu_3+1}{2}\bigr)} \\
&\hphantom{=}\times 
\Gamma\bigl(-m_1-\tfrac{\nu_i-\nu_k}{2}\bigr)
\Gamma\bigl(-m_1-l_1-\tfrac{\nu_i-\nu_3-1}{2}\bigr)
(\pi y_1)^{2m_1+\nu_i+l_1}(\pi y_2)^{2m_2-\nu_3+1-l_3},
\end{split}\\[3pt]
\begin{split}
&\varphi^{{(i,j,3)}}_\sigma (u_{l})(y) =
(\sI )^{l_1-l_3}y_1y_2(y_2y_3)^{\nu_1+\nu_2+\nu_3}\\
&\hphantom{=}\times \sum_{m_1, m_2 \geq 0} 
\frac{(-1)^{m_1+m_2} \Gamma\bigl(-m_1-\tfrac{\nu_i-\nu_j}{2}\bigr)
\Gamma\bigl(-m_1-l_1-\tfrac{\nu_i-\nu_3-1}{2}\bigr)}
{\pi \, m_1!m_2! \Gamma\bigl(-m_1-m_2-\tfrac{\nu_i-\nu_j}{2}\bigr)} \\
&\hphantom{=}\times \Gamma\bigl(-m_2-\tfrac{\nu_i-\nu_j}{2}\bigr)
\Gamma\bigl(-m_2-l_3-\tfrac{\nu_3-\nu_j-1}{2}\bigr)
(\pi y_1)^{2m_1+\nu_i+l_1}(\pi y_2)^{2m_2-\nu_j+l_3}
\end{split}
\end{align*}
with $l=(l_1,l_2,l_3)\in S_{(1,1)}$ and 
$y=\diag (y_1y_2y_3,y_2y_3,y_3)\in A$. 
Moreover, $\{\varphi^{(i,j,k)}_\sigma \mid \{i,j,k\}=\{1,2,3\}\,\}$ 
forms a basis of $\Hom_K(V_{(1,1)},
{\mathrm{Wh}}(\Pi_{\sigma},\psi_1))$, and satisfies 
\begin{align*}
\varphi^{\mathrm{mg}}_\sigma 
& = \sum_{ (i,j,k) } \varphi^{(i,j,k)}_\sigma ,
\end{align*}
where $(i,j,k)$ runs all permutations of $\{1,2,3\}$. 
\end{thm}

\section{Generalized principal series representations}
\label{subsec:R3_gps}

Let $\sigma =\sigma_1\boxtimes \sigma_2
=D_{(\nu_1,\kappa_1 )}\boxtimes \chi_{(\nu_2,\delta_2 )}$ 
with $\nu_1,\nu_2\in \bC$, $\kappa_1 \in \bZ_{\geq 2}$ 
and $\delta_2 \in \{0,1\}$. 
We identify the representation space 
$U_\sigma =\gH_{(\nu_1,\kappa_1 )}\boxtimes_\bC \bC_{(\nu_2,\delta_2 )}$ 
of $\sigma $ with $\gH_{(\nu_1,\kappa_1 )}$ via 
the correspondence $f\boxtimes c\leftrightarrow cf$. 
We use the realization of 
$(D_{(\nu_1,\kappa_1)},\gH_{(\nu_1,\kappa_1)})$ 
introduced in \S \ref{subsec:R2_ds_whittaker}, that is, 
we regard $U_\sigma =\gH_{(\nu_1,\kappa_1)}$ as the closure of 
\begin{align}
\label{eqn:R3_target_ds}
&\bigoplus_{q\in \kappa_1+2\bZ_{\geq 0}}
\{\bC\,{\zeta}_{(\widehat{\sigma}_1;q)}
+\bC\,{\zeta}_{(\widehat{\sigma}_1;-q)}\}
\end{align}
in $H(\widehat{\sigma}_1)$, where 
$\widehat{\sigma}_1=\chi_{(\nu_1+(\kappa_1-1)/2,\delta_1)}
\boxtimes \chi_{(\nu_1-(\kappa_1-1)/2,0)}$
with $\delta_1\in \{0,1\}$ such that 
$\delta_1\equiv \kappa_1\bmod 2$. 
In this section, 
we consider the action of $\g_{\bC}$ at the minimal $K$-type 
of $\Pi_\sigma$. 

The group $K\cap M_{(2,1)}$ is generated by the elements 
$k_\theta^{(2,1)}$ ($\theta \in \bR$) and 
$k_{(\varepsilon_1,\varepsilon_2,\varepsilon_3)}^{(1,1,1)}$ 
($\varepsilon_1,\varepsilon_2,\varepsilon_3\in \{\pm 1\}$). 
Since $\sigma =D_{(\nu_1,\kappa_1)}\boxtimes \chi_{(\nu_2,\delta_2)}$, 
we note that 
\begin{align*} 
\sigma \bigl(k_\theta^{(2,1)} \bigr){\zeta}_{(\widehat{\sigma}_1;q)}
&=
\Pi_{\widehat{\sigma}_1}\bigl(k_\theta^{(2)} \bigr)
{\zeta}_{(\widehat{\sigma}_1;q)}
=e^{\sI q\theta }{\zeta}_{(\widehat{\sigma}_1;q)},\\
\sigma \bigl(k_{(\varepsilon_1,\varepsilon_2,\varepsilon_3)}^{(1,1,1)} \bigr)
{\zeta}_{(\widehat{\sigma}_1;q)}
&=\chi_{(\nu_2,\delta_2)}(\varepsilon_3)
\Pi_{\widehat{\sigma}_1}
\bigl(k_{(\varepsilon_1,\varepsilon_2)}^{(1,1)}\bigr)
{\zeta}_{(\widehat{\sigma}_1;q)}
=\varepsilon_1^{\kappa_1}\varepsilon_3^{\delta_2} 
{\zeta}_{(\widehat{\sigma}_1;\varepsilon_1\varepsilon_2q)}
\end{align*}
for $\theta \in \bR$, $\varepsilon_1,\varepsilon_2,\varepsilon_3\in \{\pm 1\}$ 
and $q\in \kappa_1+2\bZ$ such that $|q|\geq \kappa_1$. 
Comparing (\ref{eqn:R3_M21act_basis}) and (\ref{eqn:R3_M111act_basis}) 
with these equalities, 
for $\mu =(\mu_1,\mu_2)\in \Lambda_3$, 
we have 
\begin{align*}
&\Hom_{K\cap M_{(2,1)}}(V_{\mu},U_{\sigma})
=\left\{\!\begin{array}{ll}
\displaystyle 
\bC\,\eta_{\sigma}
&\text{if}\,\ \mu =(\kappa_1,\delta_2),\\
\{0\}&\text{if}\,\ \mu_1 < \kappa_1 \ \,\text{or}\,\ 
\mu_1 +\mu_2 \not\equiv \kappa_1+\delta_2\bmod 2,
\end{array}\!\right.
\end{align*}
where $\eta_{\sigma}\colon 
V_{(\kappa_1,\delta_2)}\to U_\sigma$ is a $\bC$-linear map characterized by 
\begin{align*}
&\eta_{\sigma}(v^{(\kappa_1,\delta_2)}_{q})
=\left\{
\begin{array}{ll}
\sgn (q)^{\kappa_1+\delta_2}  {\zeta}_{(\widehat{\sigma}_1;q)}&\text{ if }\ 
q\in \{\pm \kappa_1\},\\[1mm]
0&\text{ otherwise}
\end{array}
\right.&
&(-\kappa_1\leq q\leq \kappa_1). 
\end{align*}
By the Frobenius reciprocity law, 
for $\mu =(\mu_1,\mu_2)\in \Lambda_3$, we have 
\begin{align}
&\label{eqn:R3_gps_minKtype} 
\Hom_{K}(V_{\mu },H({\sigma}))
=\left\{\!\begin{array}{ll}
\displaystyle 
\bC\,\hat{\eta}_{\sigma}\!&\!\text{if}\,\ \mu =(\kappa_1,\delta_2),\\[1mm]
\{0\}\!&\!\text{if}\,\ \mu_1 < \kappa_1 \ \,\text{or}\,\ 
\mu_1 +\mu_2 \not\equiv \kappa_1+\delta_2\bmod 2,
\end{array}\!\right.
\end{align}
with $\hat{\eta}_{\sigma}(v)(k)
=\eta_{\sigma}
(\tau_{(\kappa_1,\delta_2)}(k)v)\ (v\in V_{(\kappa_1,\delta_2)},\ k\in {K})$. 
We call $\tau_{(\kappa_1,\delta_2)}$ 
the minimal $K$-type of $\Pi_\sigma$. 

\begin{prop}
\label{prop:R3_ZgDSE_gps}
Retain the notation. For 
$f\in H(\sigma)_K$, it holds that 
\begin{align}
\label{eqn:R3_Zg_act_gps1}
&\Pi_\sigma (\cC_1^{{\g}})f
=(2\nu_1+\nu_2)f,\\
\label{eqn:R3_Zg_act_gps2}
&\Pi_\sigma (\cC_2^{{\g}})f
=\bigl(\nu_1^2+2\nu_1\nu_2-\tfrac{(\kappa_1-1)^2}{4}\bigr)f,\\
\label{eqn:R3_Zg_act_gps3}
&\Pi_\sigma (\cC_3^{{\g}})f
=\bigl(\nu_1^2\nu_2-\tfrac{(\kappa_1-1)^2\nu_2}{4}\bigr)f. 
\end{align}
Moreover, for $l\in S_{(\kappa_1-1,\delta_2)}$ and $1\leq i\leq 3$, 
it holds that 
\begin{align*}
&\sum_{j=1}^3\Pi_\sigma (E_{i,j}^{\gp})
\hat{\eta}_{\sigma }(u_{l+\me_j})
=2\nu_{1}\hat{\eta}_{\sigma }(u_{l+\me_i}).
\end{align*}
\end{prop}
\begin{proof}
We set 
$\widehat{\sigma}=\widehat{\sigma}_1\boxtimes \sigma_2
=\chi_{(\nu_1+(\kappa_1-1)/2,\delta_1)}
\boxtimes \chi_{(\nu_1-(\kappa_1-1)/2,0)}
\boxtimes \chi_{(\nu_2,\delta_2)}$. 
Since there is an injective $G$-homomorphism 
$\mathrm{I}_{\sigma}\colon 
H(\sigma)\to H(\widehat{\sigma})$ 
defined by 
\begin{align*}
&\mathrm{I}_{\sigma}(f)(k)
=f(k)(1_2)
&&(k\in K,\ f\in H(\sigma)), 
\end{align*}
we obtain the equalities 
(\ref{eqn:R3_Zg_act_gps1}), (\ref{eqn:R3_Zg_act_gps2}) 
and (\ref{eqn:R3_Zg_act_gps3}) from Proposition \ref{prop:Rn_Ch_eigenvalue}.

By (\ref{eqn:R3_gps_minKtype}) and 
Lemma \ref{lem:R3_tensor} (1), we have 
\[
\Hom_K(V_{(1,0)}\otimes_{\bC}V_{(\kappa_1-1,\delta_2)}, H(\sigma)_K)
=\bC\,\hat{\eta}_{\sigma }\circ 
\mathrm{B}_{(\kappa_1,\delta_2)}.
\]
Let $\mathrm{P}_\sigma \colon \gp_{\bC}\otimes_{\bC}H(\sigma)_K
\to H(\sigma)_K$ be a natural $K$-homomorphism 
defined by $X\otimes f\mapsto \Pi_{\sigma}(X)f$. 
Since the composite 
\[
\mathrm{P}_\sigma \circ 
(\id_{\gp_{\bC}}\otimes (\hat{\eta}_{\sigma }
\circ \mathrm{B}_{(\kappa_1,\delta_2)}))\circ 
(\mathrm{I}^\gp_{(1,0)}\otimes \id_{V_{(\kappa_1-1,\delta_2)}})
\]
is an element of 
$\Hom_K(V_{(1,0)}\otimes_{\bC}V_{(\kappa_1-1,\delta_2)}, 
H(\sigma)_K)$, 
there is a constant $c$ such that 
\begin{align*}
c\,\hat{\eta}_{\sigma }\circ \mathrm{B}_{(\kappa_1,\delta_2)}=
\mathrm{P}_\sigma \circ 
(\id_{\gp_{\bC}}\otimes (\hat{\eta}_{\sigma }
\circ \mathrm{B}_{(\kappa_1,\delta_2)}))\circ 
(\mathrm{I}^\gp_{(1,0)}\otimes \id_{V_{(\kappa_1-1,\delta_2)}}).
\end{align*}
Considering the image of $u_{\me_i}\otimes u_l$ 
under the both sides of this equality, we have 
\begin{align*}
c\,\hat{\eta}_{\sigma }(u_{l+\me_i})
=\sum_{j=1}^3\Pi_\sigma (E_{i,j}^{\gp})
\hat{\eta}_{\sigma }(u_{l+\me_j})
\end{align*}
for $1\leq i\leq 3$ and $l\in S_{(\kappa_1-1,\delta_2)}$. 
For $l=(l_1,l_2,l_3)\in S_{(\kappa_1-1,\delta_2)}$, we have 
\begin{align*}
&c\,{\eta}_{\sigma }(u_{l+\me_{1}})(1_2)
=\mathrm{I}_{\sigma}(c\,\hat{\eta}_{\sigma }(u_{l+\me_{1}}))(1_3)\\
&=\mathrm{I}_{\sigma}\!\left(
\sum_{j=1}^3\Pi_\sigma (E_{i,j}^{\gp})
\hat{\eta}_{\sigma }(u_{l+\me_j})\right)\!(1_3)
=\sum_{j=1}^3\bigl(
\Pi_{\widehat{\sigma}} (E_{1,j}^{\gp})\mathrm{I}_{\sigma}(
\hat{\eta}_{\sigma }(u_{l+\me_j}))
\bigr)(1_3)\\
&=2\bigl(
\Pi_{\widehat{\sigma}} (E_{1,1}^{\g})
\mathrm{I}_{\sigma}(\hat{\eta}_{\sigma}(u_{l}))
\bigr)(1_3)\\
&\hphantom{=,}
+2\bigl(
\Pi_{\widehat{\sigma}}(E_{1,2}^{\g})
\mathrm{I}_{\sigma}(\hat{\eta}_{\sigma}(u_{l+\me_2}))
\bigr)(1_3)
-\mathrm{I}_{\sigma}(\hat{\eta}_{\sigma}(
\tau_{(\kappa_1,\delta_2)}(E_{1,2}^{\gk})u_{l+\me_2}))(1_3)\\
&\hphantom{=,}
+2\bigl(\Pi_{\widehat{\sigma}}(E_{1,3}^{\g})
\mathrm{I}_{\sigma}(\hat{\eta}_{\sigma}(u_{l+\me_3}))
\bigr)(1_3)
-\mathrm{I}_{\sigma}(\hat{\eta}_{\sigma}(
\tau_{(\kappa_1,\delta_2)}(E_{1,3}^{\gk})u_{l+\me_3}))(1_3)\\
&=
(2\nu_1+\kappa_1+1)\eta_{\sigma}(u_{l+\me_1})(1_2)\\
&\hphantom{=,}
+0
-\eta_{\sigma}((l_2+1)u_{l+\me_1}-l_1u_{l-\me_1+2\me_2})(1_2)\\
&\hphantom{=,}
+0
-\eta_{\sigma}((l_3+1)u_{l+\me_1}-l_1u_{(\kappa_1-2)\me_1+2\me_3})(1_2)\\
&=2\nu_1\,\eta_{\sigma}(u_{l+\me_1})(1_2)
+l_1\eta_{\sigma}(u_{l+\me_1}+u_{l-\me_1+2\me_2}+u_{l-\me_1+2\me_3})(1_2)\\
&=2\nu_1\,\eta_{\sigma }(u_{l+\me_{1}})(1_2).
\end{align*}
Hence, in order to complete a proof, 
it suffices to show the existence of $l\in S_{(\kappa_1-1,\delta_2)}$ 
such that $\eta_{\sigma }(u_{l+\me_{1}})(1_2)\neq 0$. 
Since $u_{(\kappa_1,0,0)}$, $u_{(\kappa_1-1,1,0)}$ and 
$v^{(\kappa_1,\delta_2)}_{\kappa_1-2i}$ are the images of $z_1^{\kappa_1}$, 
$z_1^{\kappa_1-1}z_2$ and $(z_1+\sI z_2)^{\kappa_1-i}(-z_1+\sI z_2)^{i}$ 
under the natural surjection 
$\mathcal{P}_{(\kappa_1,\delta_2)}\to V_{(\kappa_1,\delta_2)}$, 
respectively, the equality 
\begin{align*}
&z_1^{\kappa_1}+\sI z_1^{\kappa_1-1}z_2\\
&=2^{-\kappa_1+1}
\{(z_1+\sI z_2)-(-z_1+\sI z_2)\}^{\kappa_1-1}(z_1+\sI z_2)\\
&=2^{-\kappa_1+1}\sum_{i=0}^{\kappa_1-1}
(-1)^i\binom{\kappa_1-1}{i}
(z_1+\sI z_2)^{\kappa_1-i}(-z_1+\sI z_2)^i
\end{align*}
implies 
\begin{align*}
u_{(\kappa_1,0,0)}+\sI u_{(\kappa_1-1,1,0)}
=2^{-\kappa_1+1}\sum_{i=0}^{\kappa_1-1}
(-1)^i\binom{\kappa_1-1}{i}
v^{(\kappa_1,\delta_2)}_{\kappa_1-2i}.
\end{align*}
By the definition of $\eta_\sigma$, we have 
\begin{align*}
\eta_{\sigma}(u_{(\kappa_1,0,0)})(1_2)
+\sI \eta_{\sigma}(u_{(\kappa_1-1,1,0)})(1_2)
=2^{-\kappa_1+1}
{\zeta}_{(\widehat{\sigma}_1;\kappa_1)}(1_2)
=2^{-\kappa_1+1}\neq 0,
\end{align*}
and complete a proof. 
\end{proof}

\section{Generalized principal series Whittaker functions}
\label{subsec:R3_gps_whittaker}

We use the notation in \S \ref{subsec:R3_gps}. 
Then the minimal $K$-type of $\Pi_\sigma$ is $\tau_{(\kappa_1,\delta_2)}$. 
Let $\varphi \colon V_{(\kappa_1,\delta_2)} \to 
{\mathrm{Wh}}(\Pi_\sigma ,\psi_1)$ be a $K$-homomorphism. 
We note that $\varphi$ is characterized by 
$\varphi (u_{l})(y)$ $(l\in S_{(\kappa_1,\delta_2)} )$ with 
$y=\diag (y_1y_2y_3,y_2y_3,y_3)\in A$. 
We set $\displaystyle \partial_i=y_i\frac{\partial}{\partial y_i}$ 
for $1\leq i\leq 3$.

\begin{lem}
\label{lem:R3_gps_ZgDSE1_original}
Retain the notation. 

\noindent (1) For $l=(l_1,l_2,l_3)\in S_{(\kappa_1,\delta_2)}$, it holds that 
\begin{align}
&\label{eqn:R3_gps_PDE1_C1}
(\partial_3-2\nu_1-\nu_2)\varphi (u_l)(y)=0,\\[3pt] 
\begin{split}
&
\bigl\{(\partial_1-1)(-\partial_1+\partial_2)
+(\partial_2-1)(-\partial_2+\partial_3+1)\\
&
-\nu_1^2-2\nu_1\nu_2+\tfrac{(\kappa_1-1)^2}{4}
+(2\pi y_1)^2+(2\pi y_2)^2\bigr\}\varphi (u_l)(y)\\
&
+2\pi \sI y_1\{l_2\varphi (u_{l-\me_2+\me_1})(y)
-l_1\varphi (u_{l-\me_1+\me_2})(y)\}\\
&\label{eqn:R3_gps_PDE1_C2}
+2\pi \sI y_2\{l_3\varphi (u_{l-\me_3+\me_2})(y)
-l_2\varphi (u_{l-\me_2+\me_3})(y)\}=0,
\end{split}\\[3pt]
\begin{split}
&
\bigl\{(\partial_1-1)(-\partial_1+\partial_2)(-\partial_2+\partial_3+1)
+(2\pi y_1)^2(-\partial_2+\partial_3+1)\\
&
+(2\pi y_2)^2(\partial_1-1)
-\nu_1^2\nu_2+\tfrac{(\kappa_1-1)^2\nu_2}{4}\bigr\}\varphi (u_l)(y)\\
&
+2\pi \sI y_1(-\partial_2+\partial_3+1)
\{l_2\varphi (u_{l-\me_2+\me_1})(y)-l_1\varphi (u_{l-\me_1+\me_2})(y)\}\\
&
+2\pi \sI y_2(\partial_1-1)
\{l_3\varphi (u_{l-\me_3+\me_2})(y)-l_2\varphi (u_{l-\me_2+\me_3})(y)\}\\
&\label{eqn:R3_gps_PDE1_C3}
+(2\pi y_1)(2\pi y_2)
\{l_3\varphi (u_{l-\me_3+\me_1})(y)-l_1\varphi (u_{l-\me_1+\me_3})(y)\}
=0.
\end{split}
\end{align}
\noindent (2) For $l=(l_1,l_2,l_3)\in S_{(\kappa_1-1,\delta_2)}$, it holds that 
\begin{align}
&\label{eqn:R3_gps_PDE1_DSE1}
(2\partial_1-\kappa_1-1-2\nu_1)\varphi (u_{l+\me_1})(y)
+4\pi \sI y_1\varphi (u_{l+\me_2})(y)=0,\\[2pt]
\begin{split}
&
(-2\partial_1+2\partial_2+l_1-l_3-2\nu_1)\varphi (u_{l+\me_2})(y)
+4\pi \sI y_1\varphi (u_{l+\me_1})(y)\\
&\label{eqn:R3_gps_PDE1_DSE2}
+4\pi \sI y_2\varphi (u_{l+\me_3})(y)
-l_2\{\varphi (u_{l-\me_2+2\me_1})(y)
-\varphi (u_{l-\me_2+2\me_3})(y)\}
=0,
\end{split}\\[2pt]
&\nonumber 
(-2\partial_2+2\partial_3+\kappa_1 +1-2\nu_1)\varphi (u_{l+\me_3})(y)
+4\pi \sI y_2\varphi (u_{l+\me_2})(y)=0. 
\end{align}
\end{lem}
\begin{proof}
We note that there is a homomorphism $\Phi \in {\cI}_{\Pi_\sigma ,\psi_1}$ 
such that 
$\varphi =\Phi \circ \hat{\eta}_\sigma$. 
Hence, by Proposition \ref{prop:R3_ZgDSE_gps}, we have 
\begin{align*}
&(R(\cC_1^{{\g}})\varphi (u_{l}))(y)
-(2\nu_1+\nu_2)\varphi (u_{l})(y)=0,\\
&(R(\cC_2^{{\g}})\varphi (u_{l}))(y)
-\bigl(\nu_1^2+2\nu_1\nu_2-\tfrac{(\kappa_1-1)^2}{4}\bigr)
\varphi (u_{l})(y)=0,\\
&(R(\cC_3^{{\g}})\varphi (u_{l}))(y)
-\bigl(\nu_1^2\nu_2-\tfrac{(\kappa_1-1)^2\nu_2}{4}\bigr)
\varphi (u_{l})(y)=0
\end{align*}
for $l=(l_1,l_2,l_3)\in S_{(\kappa_1,\delta_2)}$. 
Applying (\ref{eqn:R3_LieKact_gen}) and Lemma \ref{lem:Rn_g_act_Cpsi} 
to these equalities with 
the expressions (\ref{eqn:R3_C1}), (\ref{eqn:R3_C2_modE13}) and 
(\ref{eqn:R3_C3_modE13}), we obtain the statement (1).

By Proposition \ref{prop:R3_ZgDSE_gps}, 
we have 
\begin{align*}
&\sum_{j=1}^3(R(E_{i,j}^{\gp})\varphi (u_{l+\me_j}))(y)
-2\nu_1\varphi (u_{l+\me_i})(y)=0&
&(1\leq i\leq 3)
\end{align*}
for $l=(l_1,l_2,l_3)\in S_{(\kappa_1-1,\delta_2)}$. 
Applying (\ref{eqn:R3_LieKact_gen}) and Lemma \ref{lem:Rn_g_act_Cpsi} 
to this equality with the decomposition 
$E_{i,j}^{\gp}=E_{j,i}^{\gp}=2E_{i,j}^{{\g}}-E_{i,j}^{{\gk}}$ 
$(1\leq i\leq j\leq 3)$, 
we obtain the statement (2). 
\end{proof}

By the equation (\ref{eqn:R3_gps_PDE1_C1}), 
for $l=(l_1,l_2,l_3)\in S_{(\kappa_1,\delta_2)}$, 
we can define a function $\hat{\varphi}_l$ on $(\bR_+)^2$ by  
\begin{align}
\label{eqn:R3_gps_def_varphi_l}
&\varphi (u_l)(y)=
(\sI)^{l_1-l_3}y_1y_2(y_2y_3)^{2\nu_1+\nu_2}\hat{\varphi}_l(y_1,y_2) 
\end{align}
with $y=\diag (y_1y_2y_3,y_2y_3,y_3)\in A$. 
It is convenient to set $\hat{\varphi}_l=0$ 
if $l\not\in (\bZ_{\geq 0})^3$. 

\begin{lem}
\label{lem:R3_gps_ZgDSE2}
Retain the notation. 
Let $l=(l_1,l_2,l_3)\in S_{(\kappa_1-1,\delta_2)}$. 
Then it holds that 
\begin{align}
&\label{eqn:R3_gps_PDE2_DSE1}
(-2\partial_1+2\nu_1+\kappa_1-1)\hat{\varphi}_{l+\me_1}
-4\pi y_1\hat{\varphi}_{l+\me_2}=0.
\end{align}
If $l_2=0$, it holds that 
\begin{align}
&\label{eqn:R3_gps_PDE2_DSE2}
(-2\partial_1+2\partial_2+l_1-l_3+2\nu_1+2\nu_2)\hat{\varphi}_{l+\me_2}
-4\pi y_1\hat{\varphi}_{l+\me_1}
+4\pi y_2\hat{\varphi}_{l+\me_3}=0. 
\end{align}
\end{lem}
\begin{proof}
The equations (\ref{eqn:R3_gps_PDE2_DSE1}) and 
(\ref{eqn:R3_gps_PDE2_DSE2}) follow immediately 
from (\ref{eqn:R3_gps_PDE1_DSE1}) 
and (\ref{eqn:R3_gps_PDE1_DSE2}), respectively. 
\end{proof}

The functions $\hat{\varphi}_l$ $(l\in S_{(\kappa_1,\delta_2)})$ 
are determined from $\hat{\varphi}_{(\kappa_1,0,0)}$ by 
the equations in Lemma \ref{lem:R3_gps_ZgDSE2}. 
Hence, we know that 
$\varphi \colon V_{(\kappa_1,\delta_2)} \to 
{\mathrm{Wh}}(\Pi_\sigma ,\psi_1)$ 
is uniquely determined from $\hat{\varphi}_{(\kappa_1,0,0)}$. 

\begin{lem}
\label{lem:R3_gps_PDE3}
Retain the notation. 
Then $\hat{\varphi}_{(\kappa_1,0,0)}$ satisfies the following system 
of partial differential equations: 
\begin{align}
\begin{split}
&
\bigl\{-\partial_1^2+\partial_1\partial_2-\partial_2^2
+(2\nu_1+\nu_2+\kappa_1)\partial_1
-(2\nu_1+\nu_2)\partial_2\\
&\label{eqn:R3_gps_PDE3_2}
-\nu_1^2-2\nu_1\nu_2-\nu_1\kappa_1
-\tfrac{\kappa_1^2-1}{4}+(2\pi y_1)^2+(2\pi y_2)^2
\bigr\}\hat{\varphi}_{(\kappa_1,0,0)}
=0,
\end{split}\\[3pt]
\begin{split}
&\bigl\{\bigl(\partial_1-\nu_1-\tfrac{\kappa_1+1}{2}\bigr)
\bigl(\partial_1-\nu_1-\tfrac{\kappa_1-1}{2}\bigr)
(\partial_1-\nu_2-\kappa_1)\\
&\label{eqn:R3_gps_PDE3_3}
-(2\pi y_1)^2(\partial_1+\partial_2-\kappa_1+2)
\bigr\}\hat{\varphi}_{(\kappa_1,0,0)}=0. 
\end{split}
\end{align}
\end{lem}
\begin{proof}
We have  
\begin{align}
\begin{split}
&
\bigl\{-\partial_1^2+\partial_1\partial_2-\partial_2^2
+(2\nu_1+\nu_2)\partial_1-(2\nu_1+\nu_2)\partial_2
-\nu_1^2-2\nu_1\nu_2\\
&\label{eqn:R3_gps_PDE2_C2}
+\tfrac{(\kappa_1-1)^2}{4}
+(2\pi y_1)^2+(2\pi y_2)^2\bigr\}\hat{\varphi}_{(\kappa_1,0,0)}
-\kappa_1(2\pi y_1)\hat{\varphi}_{(\kappa_1-1,1,0)}=0,
\end{split}\\[3pt]
\begin{split}
&
\bigl\{-\partial_1(-\partial_1+\partial_2+2\nu_1+\nu_2)\partial_2
-\nu_1^2\nu_2+\tfrac{(\kappa_1-1)^2\nu_2}{4}\\
&\label{eqn:R3_gps_PDE2_C3}
-(2\pi y_1)^2\partial_2
+(2\pi y_2)^2\partial_1\bigr\}\hat{\varphi}_{(\kappa_1,0,0)}
+\kappa_1(2\pi y_1)\partial_2\hat{\varphi}_{(\kappa_1-1,1,0)}\\
&+\kappa_1(2\pi y_1)(2\pi y_2)\hat{\varphi}_{(\kappa_1-1,0,1)}
=0
\end{split}
\end{align}
from (\ref{eqn:R3_gps_PDE1_C2}) and (\ref{eqn:R3_gps_PDE1_C3}) 
with $l=(\kappa_1,0,0)$, respectively. 
Multiplying the both sides of (\ref{eqn:R3_gps_PDE2_C2}) 
by $-\partial_1$ from the left, we have 
\begin{align}
\begin{split}
&
\bigl\{\partial_1^3
-\partial_1^2\partial_2
+\partial_1\partial_2^2
-(2\nu_1+\nu_2)\partial_1^2
+(2\nu_1+\nu_2)\partial_1\partial_2\\
&
+\bigl(\nu_1^2+2\nu_1\nu_2-\tfrac{(\kappa_1-1)^2}{4}\bigr)
\partial_1-(2\pi y_1)^2(\partial_1+2)
-(2\pi y_2)^2\partial_1\bigr\}\hat{\varphi}_{(\kappa_1,0,0)}\\
&\label{eqn:R3_gps_PDE2_D1C2}
+\kappa_1(2\pi y_1)(\partial_1+1)\hat{\varphi}_{(\kappa_1-1,1,0)}=0.
\end{split}
\end{align}
Adding the respective sides of 
the equations (\ref{eqn:R3_gps_PDE2_C3}) and (\ref{eqn:R3_gps_PDE2_D1C2}), 
we have 
\begin{align}
\begin{split}
&\bigl\{\partial_1^3-(2\nu_1+\nu_2)\partial_1^2
+\bigl(\nu_1^2+2\nu_1\nu_2-\tfrac{(\kappa_1-1)^2}{4}\bigr)\partial_1
-\nu_1^2\nu_2+\tfrac{(\kappa_1-1)^2\nu_2}{4}\\
&
-(2\pi y_1)^2(\partial_1+\partial_2+2)
\bigr\}\hat{\varphi}_{(\kappa_1,0,0)}
+\kappa_1(2\pi y_1)(\partial_1+\partial_2+1)\hat{\varphi}_{(\kappa_1-1,1,0)}
\\
&\label{eqn:R3_gps_pf1_C3C2}
+\kappa_1(2\pi y_1)(2\pi y_2)\hat{\varphi}_{(\kappa_1-1,0,1)}=0.
\end{split}
\end{align}
By Lemma \ref{lem:R3_gps_ZgDSE2}, we have 
\begin{align}
&\label{eqn:R3_gps_pf1_DSE1}
(-2\partial_1+2\nu_1+\kappa_1-1)\hat{\varphi}_{(\kappa_1,0,0)}
-4\pi y_1\hat{\varphi}_{(\kappa_1-1,1,0)}=0,\\[2pt]
\begin{split}
&
(-2\partial_1+2\partial_2+2\nu_1+2\nu_2+\kappa_1-1)
\hat{\varphi}_{(\kappa_1-1,1,0)}\\
&\label{eqn:R3_gps_pf1_DSE2}
-4\pi y_1\hat{\varphi}_{(\kappa_1,0,0)}
+4\pi y_2\hat{\varphi}_{(\kappa_1-1,0,1)}=0. 
\end{split}
\end{align}
Multiplying the both sides of (\ref{eqn:R3_gps_pf1_DSE1}) 
by $-\kappa_1/2$, we have 
\begin{align*}
&\kappa_1\bigl(\partial_1-\nu_1
-\tfrac{\kappa_1-1}{2}\bigr)\hat{\varphi}_{(\kappa_1,0,0)}
+\kappa_1(2\pi y_1)\hat{\varphi}_{(\kappa_1-1,1,0)}=0.
\end{align*}
Adding the respective sides of 
(\ref{eqn:R3_gps_PDE2_C2}) and this equation, 
we obtain (\ref{eqn:R3_gps_PDE3_2}). 
Multiplying the both sides of 
(\ref{eqn:R3_gps_pf1_DSE1}) and (\ref{eqn:R3_gps_pf1_DSE2}) 
by $\kappa_1(4\partial_1-2-2\nu_1-2\nu_2-\kappa_1+1)/4$ 
and $-\kappa_1\pi y_1$ from the left, 
respectively, we have 
\begin{align*}
&
\kappa_1\bigl(2\partial_1-1-\nu_1-\nu_2-\tfrac{\kappa_1-1}{2}\bigr)
\bigl(-\partial_1+\nu_1+\tfrac{\kappa_1-1}{2}\bigr)
\hat{\varphi}_{(\kappa_1,0,0)}\\
&+\kappa_1(2\pi y_1)
\bigl(-2\partial_1-1+\nu_1+\nu_2+\tfrac{\kappa_1-1}{2}\bigr)
\hat{\varphi}_{(\kappa_1-1,1,0)}=0,\\[2pt]
&
\kappa_1(2\pi y_1)
\bigl(\partial_1-\partial_2-\nu_1-\nu_2-\tfrac{\kappa_1-1}{2}\bigr)
\hat{\varphi}_{(\kappa_1-1,1,0)}\\
&+\kappa_1(2\pi y_1)^2\hat{\varphi}_{(\kappa_1,0,0)}
-\kappa_1(2\pi y_1)(2\pi y_2)\hat{\varphi}_{(\kappa_1-1,0,1)}=0. 
\end{align*}
Adding up the respective sides of 
(\ref{eqn:R3_gps_pf1_C3C2}) and these equations, 
we obtain (\ref{eqn:R3_gps_PDE3_3}). 
\end{proof}

\begin{lem}
\label{lem:R3_gps_MW_sub1}
Let $\nu_1,\nu_2\in \bC $, $\kappa_1\in \bZ_{\geq 2}$ and 
\begin{align*}
r=(r_1,r_2,r_3)=
\bigl(\nu_1-\tfrac{\kappa_1-1}{2},\,
\nu_1-\tfrac{\kappa_1+1}{2},\,
\nu_2\bigr).
\end{align*}
(1) The space of smooth solutions of the system in 
Lemma \ref{lem:R3_gps_PDE3} on $(\bR_+)^2$ is at most 
$6$ dimensional. \\
\noindent (2) Set  
\[
\hat{\varphi}^{\mathrm{mg}}_{(\kappa_1,0,0)}(y_1,y_2)
=\frac{1}{(4\pi \sqrt{-1})^2} \int_{s_2}\int_{s_1} 
\cV_{(\kappa_1,0,0)}(s_1,s_2) \,
y_1^{-s_1} y_2^{-s_2} \,ds_1ds_2
\]
with 
\begin{align*}
\cV_{(\kappa_1,0,0)}(s_1,s_2)=&
\frac{\Gamma_{\bR}(s_1+r_1+\kappa_1)\Gamma_{\bR}(s_1+r_2+\kappa_1)
\Gamma_{\bR}(s_1+r_3+\kappa_1)}
{\Gamma_{\bR}(s_1+s_2+\kappa_1)}\\
&\times \Gamma_{\bR}(s_2-r_1)\Gamma_{\bR}(s_2-r_2)\Gamma_{\bR}(s_2-r_3).
\end{align*}
Here the path of the integration $\int_{s_i}$ is the vertical line 
from $\mathrm{Re}(s_i)-\sI \infty$ to $\mathrm{Re}(s_i)+\sI \infty$ 
with sufficiently large real part to keep the poles of the integrand 
on its left. 
Then $\hat{\varphi}_{(\kappa_1,0,0)}
=\hat{\varphi}^{\mathrm{mg}}_{(\kappa_1,0,0)}$ is 
a moderate growth solution of the system in Lemma \ref{lem:R3_gps_PDE3} 
on $(\bR_+)^2$. 
\\[1mm]
(3) Assume $2\nu_1-2\nu_2-\kappa_1+1\not\in 2\bZ$. 
For a permutation $(i,j,k)$ of $\{1,2,3\}$, set  
\[
\hat{\varphi}^{{(i,j,k)}}_{(\kappa_1,0,0)}(y_1,y_2) =
\sum_{m_1,m_2 \geq 0}
C_{(\kappa_1,0,0),(m_1,m_2)}^{(i,j,k)}(\pi y_1)^{2m_1+r_i+\kappa_1}
(\pi y_2)^{2m_2-r_j}
\]
with 
\begin{align*}
C_{(\kappa_1,0,0),(m_1,m_2)}^{(i,j,k)}=
&\frac{(-1)^{m_1+m_2} 
\Gamma \bigl(-m_1-\tfrac{r_i-r_j}{2}\bigr)
\Gamma \bigl(-m_1-\tfrac{r_i-r_k}{2}\bigr)}
{\pi^{\kappa_1}\,m_1!\,m_2!\,\Gamma 
\bigl(-m_1-m_2-\tfrac{r_i-r_j}{2}\bigr)} \\
& \times \Gamma \bigl(-m_2-\tfrac{r_i-r_j}{2}\bigr)
\Gamma \bigl(-m_2-\tfrac{r_k-r_j}{2}\bigr).
\end{align*}
Then $\{\hat{\varphi}^{{(i,j,k)}}_{(\kappa_1,0,0)}\mid 
\{i,j,k\}=\{1,2,3\}\,\}$ forms a basis of 
the space of smooth solutions of the system in Lemma \ref{lem:R3_gps_PDE3} 
on $(\bR_+)^2$. Moreover, 
it holds that 
\begin{align*}
\hat{\varphi}^{\mathrm{mg}}_{(\kappa_1,0,0)}
& = \sum_{ (i,j,k) } \hat{\varphi}^{(i,j,k)}_{(\kappa_1,0,0)},
\end{align*}
where $ (i,j,k) $ runs all permutations of $\{1,2,3\}$. 
\end{lem} 
\begin{proof}
If we take 
$r=(r_1,r_2,r_3)$ as the statement, and set 
\[
f(z_1,z_2)=
z_1^{-\kappa_1}\hat{\varphi}_{(\kappa_1,0,0)}(\pi^{-1}z_1,\,\pi^{-1}z_2), 
\]
then (\ref{eqn:R3_gps_PDE3_2}) and (\ref{eqn:R3_gps_PDE3_3})
imply (\ref{eqn:Fn_Sol_PDE2}) and (\ref{eqn:Fn_Sol_PDE3}), 
respectively. 
Hence, the assertion follows from Lemmas \ref{lem:Fn_Sol_dim}, 
\ref{lem:Fn_Sol_power_series} and \ref{lem:Fn_Sol_mg}. 
\end{proof}

\begin{lem}
\label{lem:R3_gps_MW_sub2}
We use the notation in Lemma \ref{lem:R3_gps_MW_sub1}, 
and let $\delta_2\in \{0,1\}$. Assume $2\nu_1-2\nu_2-\kappa_1+1\not\in 2\bZ$. 
Let $(i,j,k)$ be a permutation of $\{1,2,3\}$. 
Let $\hat{\varphi}_{l}\ (l\in S_{(\kappa_1,\delta_2)})$  
be the functions determined from the function 
$\hat{\varphi}_{(\kappa_1,0,0)}=\hat{\varphi}^{{(i,j,k)}}_{(\kappa_1,0,0)}$ 
in Lemma \ref{lem:R3_gps_MW_sub1} (3) by the equations 
in Lemma \ref{lem:R3_gps_ZgDSE2}. 
Then it holds that 
\begin{align*}
\hat{\varphi}_{l}(y_1,y_2) = 
&\sum_{m_1, m_2 \geq 0}
C_{l,(m_1,m_2)}^{(i,j,k)}
(\pi y_1)^{2m_1+r_i+l_1}(\pi y_2)^{2m_2-r_j-l_3},
\end{align*}
where $l=(l_1,l_2,l_3)\in S_{(\kappa_1,\delta_2)}$ and 
\begin{align*}
\begin{split}
&C_{l,(m_1,m_2)}^{(i,j,k)}
=(-1)^{m_1+m_2}\\
&\times \frac{\Gamma \bigl(-m_1-\tfrac{r_i-r_j}{2}\bigr)
\Gamma \bigl(-m_1-\tfrac{r_i-r_k}{2}\bigr)
\Gamma \bigl(-m_2-\tfrac{r_i-r_j}{2}\bigr)
\Gamma \bigl(-m_2-\tfrac{r_k-r_j}{2}\bigr)}
{\pi^{\kappa_1}\,m_1!\,m_2!\,
\Gamma \bigl(-m_1-m_2-\tfrac{r_i-r_j}{2}+l_3\bigr)}
\\
&\times \!\left(\prod_{p=0}^{l_2+l_3-1}(-m_1-\tfrac{r_i-r_2-p}{2})
\right)\!\left(
\prod_{q=0}^{l_3-1}
\bigl(-m_2-\tfrac{r_1-r_j-q}{2}\bigr)
\bigl(-m_2-\tfrac{r_3-r_j}{2}+q\bigr)
\right)\!.
\end{split}
\end{align*}
\end{lem}
\begin{proof}
We set 
\begin{align*}
\hat{\varphi}_{l}(y_1,y_2) = 
&\sum_{m_1, m_2 \geq 0}
C_{l,(m_1,m_2)}
(\pi y_1)^{2m_1+r_i+l_1}(\pi y_2)^{2m_2-r_j-l_3}
\end{align*}
for $l=(l_1,l_2,l_3)\in S_{(\kappa_1,\delta_2)}$. 
Then (\ref{eqn:R3_gps_PDE2_DSE1}) and (\ref{eqn:R3_gps_PDE2_DSE2}) 
are equivalent to 
\begin{align}
\label{eqn:R3_gps_PDE2_DSE1_coeff}
&C_{l+\me_2,(m_1,m_2)}=
\bigl(-m_1-\tfrac{r_i-r_2-l_2-l_3}{2}\bigr)C_{l+\me_1,(m_1,m_2)}
\end{align}
and 
\begin{align}
\label{eqn:R3_gps_PDE2_DSE2_coeff}
\begin{split}
C_{l+\me_3,(m_1,m_2)}=\,&
C_{l+\me_1,(m_1-1,m_2)}\\
&+\bigl(m_1-m_2+\tfrac{r_i+r_j-r_1-r_3+l_3}{2}\bigr)
C_{l+\me_2,(m_1,m_2)}
\hspace{5mm}(l_2=0), 
\end{split}
\end{align}
respectively, where 
$l=(l_1,l_2,l_3)\in S_{(\kappa_1-1,\delta_2)}$.  
Here we put $C_{l,(m_1,m_2)}=0$ if $m_1<0$ or $m_2<0$. 
Using the equalities 
\begin{align*}
&C_{l+\me_2,(m_1,m_2)}^{(i,j,k)}=
\bigl(-m_1-\tfrac{r_i-r_2-l_2-l_3}{2}\bigr)
C_{l+\me_1,(m_1,m_2)}^{(i,j,k)},\\
&{C_{l+\me_1,(m_1-1,m_2)}^{(i,j,k)}}\\
&=\frac{
\bigl(m_1+\tfrac{r_i-r_1-l_2-l_3}{2}\bigr)
\bigl(m_1+\tfrac{r_i-r_2-l_2-l_3}{2}\bigr)
\bigl(m_1+\tfrac{r_i-r_3}{2}\bigr)}
{m_1+m_2+\tfrac{r_i-r_j}{2}-l_3}
{C_{l+\me_1,(m_1,m_2)}^{(i,j,k)}},\\
&{C_{l+\me_3,(m_1,m_2)}^{(i,j,k)}}\\
&=\frac{
\bigl(m_1+\tfrac{r_i-r_2-l_2-l_3}{2}\bigr)
\bigl(m_2+\tfrac{r_1-r_j-l_3}{2}\bigr)
\bigl(m_2+\tfrac{r_3-r_j}{2}-l_3\bigr)}
{m_1+m_2+\tfrac{r_i-r_j}{2}-l_3}
{C_{l+\me_1,(m_1,m_2)}^{(i,j,k)}}
\end{align*}
with $l=(l_1,l_2,l_3)\in S_{(\kappa_1-1,\delta_2)}$, 
we know that 
$C_{l,(m_1,m_2)}=C_{l,(m_1,m_2)}^{(i,j,k)}$ 
$(l\in S_{(\kappa_1,\delta_2)})$ 
satisfy (\ref{eqn:R3_gps_PDE2_DSE1_coeff}) and 
(\ref{eqn:R3_gps_PDE2_DSE2_coeff}). 
\end{proof}

\begin{lem}
\label{lem:R3_gps_WW_sub2}
Let $\nu_1,\nu_2\in \bC $, $\kappa_1\in \bZ_{\geq 2}$ and 
$\delta_2\in \{0,1\}$. 
Let $\hat{\varphi}_{l}\ (l\in S_{(\kappa_1,\delta_2)})$  
be the functions determined from 
$\hat{\varphi}_{(\kappa_1,0,0)}=\hat{\varphi}^{\mathrm{mg}}_{(\kappa_1,0,0)}$ 
in Lemma \ref{lem:R3_gps_MW_sub1} (2) by the equations 
in Lemma \ref{lem:R3_gps_ZgDSE2}. 
Then it holds that 
\begin{align}
\label{eqn:R3_gps_WW_sub2}
\begin{split}
\hat{\varphi}_{l}(y_1,y_2)
=&\frac{1}{(4\pi \sqrt{-1})^2} \int_{s_2}\int_{s_1} 
\frac{\Gamma_{\bC}\bigl(s_1+\nu_1+\tfrac{\kappa_1-1}{2}\bigr)
\Gamma_{\bR}(s_1+\nu_2+l_1)}
{\Gamma_{\bR}(s_1+s_2+l_1+l_3)}\\
&\times \Gamma_{\bC}\bigl(s_2-\nu_1+\tfrac{\kappa_1-1}{2}\bigr)
\Gamma_{\bR}(s_2-\nu_2+l_3)\,
y_1^{-s_1} y_2^{-s_2} \,ds_1ds_2
\end{split}
\end{align}
for $l=(l_1,l_2,l_3)\in S_{(\kappa_1,\delta_2)}$. 
Here the path of the integration $\int_{s_i}$ is the vertical line 
from $\mathrm{Re}(s_i)-\sI \infty$ to $\mathrm{Re}(s_i)+\sI \infty$ 
with sufficiently large real part to keep the poles of the integrand 
on its left. 
\end{lem} 
\begin{proof}
Let $\hat{\varphi}_{l}\ (l\in S_{(\kappa_1,\delta_2)})$  
be the functions defined by (\ref{eqn:R3_gps_WW_sub2}), 
and let $\hat{\varphi}^{\mathrm{mg}}_{(\kappa_1,0,0)}$ be 
the function in Lemma \ref{lem:R3_gps_MW_sub1} (2). 
Using the duplication formula (\ref{eqn:Fn_gammaRC_duplication}), 
we know that 
$\hat{\varphi}_{(\kappa_1,0,0)}=\hat{\varphi}^{\mathrm{mg}}_{(\kappa_1,0,0)}$. 
Moreover, it is easy to show that 
the functions $\hat{\varphi}_{l}\ (l\in S_{(\kappa_1,\delta_2)})$ satisfy 
the equations in Lemma \ref{lem:R3_gps_ZgDSE2} by direct computation. 
\end{proof}

\begin{thm}[{\cite{Miyazaki_002}}]
\label{thm:R3_gps1_Whittaker}
Let $\sigma = D_{(\nu_1,\kappa_1)}\boxtimes \chi_{(\nu_2,\delta_2)}$ with 
$\nu_1,\nu_2\in \bC$, $\kappa_1\in \bZ_{\geq 2}$ and $\delta_2\in \{0,1\}$ 
such that $\Pi_\sigma$ is irreducible. \\
(1) There exists a $K$-homomorphism 
\[
\varphi^{\mathrm{mg}}_\sigma \colon V_{(\kappa_1,\delta_2)}\to 
{\mathrm{Wh}}(\Pi_{\sigma},\psi_1)^{\mathrm{mg}},
\]
whose radial part is given by 
\begin{align*}
\begin{split}
\varphi^{\mathrm{mg}}_\sigma (u_{l})(y) =&
(\sI )^{l_1-l_3}y_1y_2(y_2y_3)^{2\nu_1+\nu_2}\\
&\times \frac{1}{(4\pi \sqrt{-1})^2} \int_{s_2}\int_{s_1}
\frac{\Gamma_{\bC}\bigl(s_1+\nu_1+\tfrac{\kappa_1-1}{2}\bigr)
\Gamma_{\bR}(s_1+\nu_2+l_1)}
{ \Gamma_{\bR}(s_1+s_2+l_1+l_3)}\\
&\times \Gamma_{\bC}\bigl(s_2-\nu_1+\tfrac{\kappa_1-1}{2}\bigr)
\Gamma_{\bR}(s_2-\nu_2+l_3)
\,y_1^{-s_1} y_2^{-s_2} \,ds_1ds_2
\end{split}
\end{align*}
with $l=(l_1,l_2,l_3)\in S_{(\kappa_1,\delta_2)}$ 
and $y=\diag (y_1y_2y_3,y_2y_3,y_3)\in A$. 
Here the path of the integration $\int_{s_i}$ is the vertical line 
from $\mathrm{Re}(s_i)-\sI \infty$ to $\mathrm{Re}(s_i)+\sI \infty$ 
with sufficiently large real part to keep the poles of the integrand 
on its left. \\[2pt]
(2) Assume 
$2\nu_1-2\nu_2-\kappa_1+1\not\in 2\bZ$, and set 
\begin{align*}
&\xi (m)=\left\{\begin{array}{ll}
0&\text{if }m\equiv 0\bmod 2,\\
1&\text{if }m\equiv 1\bmod 2
\end{array}\right.&
&(m\in \bZ).
\end{align*} 
For a permutation $(i,j,k)$ of $\{1,2,3\}$, 
there exists a $K$-homomorphism 
\[
\varphi^{(i,j,k)}_\sigma \colon V_{(\kappa_1,\delta_2)}\to 
{\mathrm{Wh}}(\Pi_{\sigma},\psi_1),
\]
whose radial part is given by the power series 
\begin{align*}
\begin{split}
\varphi^{(3,j,k)}_\sigma (u_{l})(y) =&
(2\pi)^{-\kappa_1+1}(\sI )^{l_1-l_3}y_1y_2(y_2y_3)^{2\nu_1+\nu_2}\\
&\times \sum_{m_1, m_2 \geq 0}\frac{(-1)^{m_1+l_3+j+1}
2\Gamma (-2m_1+\nu_1-\nu_2+\tfrac{\kappa_1-1}{2}-l_1)}
{m_1!\,(2m_2+\xi (l_3+j+1))!}\\
&\times 
\frac{\Gamma \bigl(-m_2+\tfrac{\nu_1-\nu_2+l_3-\xi (l_3+j+1)}{2}
-\tfrac{\kappa_1-1}{4}\bigr)}
{\Gamma \bigl(-m_1-m_2+\tfrac{\nu_1-\nu_2+l_3-\xi (l_3+j+1)}{2}
-\tfrac{\kappa_1-1}{4}\bigr)}\\
&\times 
(2\pi y_1)^{2m_1+\nu_2+l_1}
(2\pi y_2)^{2m_2-\nu_1+\tfrac{\kappa_1-1}{2}+\xi (l_3+j+1)},
\end{split}\\[5pt]
\begin{split}
\varphi^{(i,3,k)}_\sigma (u_{l})(y) =&
(2\pi)^{-\kappa_1+1}(\sI )^{l_1-l_3}y_1y_2(y_2y_3)^{2\nu_1+\nu_2}\\
&\times \sum_{m_1, m_2 \geq 0}\frac{(-1)^{m_2+\kappa_1-l_1+i}
2\Gamma (-2m_2-\nu_1+\nu_2+\tfrac{\kappa_1-1}{2}-l_3)}
{(2m_1+\xi (\kappa_1-l_1+i))!\,m_2!}\\
&\times \frac{\Gamma \bigl(-m_1-\tfrac{\nu_1-\nu_2-l_1
+\xi (\kappa_1-l_1+i)}{2}-\tfrac{\kappa_1-1}{4}\bigr)}
{\Gamma \bigl(-m_1-m_2
-\tfrac{\nu_1-\nu_2-l_1+\xi (\kappa_1-l_1+i)}{2}
-\tfrac{\kappa_1-1}{4}\bigr)}\\
&\times 
(2\pi y_1)^{2m_1+\nu_1+\tfrac{\kappa_1-1}{2}+\xi (\kappa_1-l_1+i)}
(2\pi y_2)^{2m_2-\nu_2+l_3},
\end{split}\\[5pt]
\begin{split}
\varphi^{(i,j,3)}_\sigma (u_{l})(y) =&
(2\pi)^{-\kappa_1+1}(\sI )^{l_1-l_3}y_1y_2(y_2y_3)^{2\nu_1+\nu_2}\\
&\times 
\sum_{m_1, m_2 \geq 0}\frac{(-1)^{l_2}
\Gamma \bigl(-m_1-\tfrac{\nu_1-\nu_2-l_1+\xi (\kappa_1-l_1+i)}{2}
-\tfrac{\kappa_1-1}{4}\bigr)}
{(2m_1+\xi (\kappa_1-l_1+i))!\,(2m_2+\xi (l_3+j+1))!}\\
&\times 
\frac{\Gamma \bigl(-m_2+\tfrac{\nu_1-\nu_2+l_3-\xi (l_3+j+1)}{2}
-\tfrac{\kappa_1-1}{4}\bigr)}{
\Gamma \bigl(-m_1-m_2-\tfrac{l_2-1+\xi (\kappa_1-l_1+i)+\xi (l_3+j+1)}{2}\bigr)}\\
&\times 
(2\pi y_1)^{2m_1+\nu_1+\tfrac{\kappa_1-1}{2}+\xi (\kappa_1-l_1+i)}
(2\pi y_2)^{2m_2-\nu_1+\tfrac{\kappa_1-1}{2}+\xi (l_3+j+1)}
\end{split}
\end{align*}
with $l=(l_1,l_2,l_3)\in S_{(\kappa_1,\delta_2)}$ and 
$y=\diag (y_1y_2y_3,y_2y_3,y_3)\in A$. 
Moreover, $\{\varphi^{(i,j,k)}_\sigma \mid \{i,j,k\}=\{1,2,3\}\,\}$ 
forms a basis of $\Hom_K(V_{(\kappa_1,\delta_2)},
{\mathrm{Wh}}(\Pi_{\sigma},\psi_1))$, and satisfies 
\begin{align*}
\varphi^{\mathrm{mg}}_\sigma 
& = \sum_{ (i,j,k) } \varphi^{(i,j,k)}_\sigma ,
\end{align*}
where $(i,j,k)$ runs all permutations of $\{1,2,3\}$. 
\end{thm} 
\begin{proof}
We denote by $\mathrm{Sol}(\Pi_{\sigma};\psi_1)$ 
the space of smooth solutions of 
the system in Lemma \ref{lem:R3_gps_PDE3} on $(\bR_+)^2$. 
We define an injective homomorphism 
\begin{align}
\label{eqn:R3_gps_Wh_to_Sol}
\Hom_K(V_{(\kappa_1,\delta_2)},{\mathrm{Wh}}(\Pi_{\sigma},\psi_1))
\ni \varphi 
\mapsto \hat{\varphi}_{(\kappa_1,0,0)}
\in \mathrm{Sol}(\Pi_{\sigma};\psi_1)
\end{align}
of $\bC$-vector spaces 
by (\ref{eqn:R3_gps_def_varphi_l}). 

Since $\dim_\bC \Hom_K(V_{(\kappa_1,\delta_2)} ,H(\sigma)_K)=1$ and 
$\Pi_\sigma$ is irreducible, we have 
\begin{align*}
\dim_\bC \Hom_K(V_{(\kappa_1,\delta_2)},
{\mathrm{Wh}}(\Pi_{\sigma} ,\psi_1))
&=\dim_\bC {\cI}_{\Pi_{\sigma},\psi_1}=6
\geq 
\dim_{\bC}\mathrm{Sol}(\Pi_{\sigma};\psi_1).
\end{align*}
Here the last inequality follows from Lemma \ref{lem:R3_gps_MW_sub1} (1).
This inequality implies that (\ref{eqn:R3_gps_Wh_to_Sol}) 
is bijective. Hence, the assertion follows from 
Lemmas \ref{lem:R3_gps_MW_sub1}, 
\ref{lem:R3_gps_MW_sub2} and \ref{lem:R3_gps_WW_sub2}. 
Here we obtain the expressions of $\varphi^{(i,j,k)}_\sigma (u_{l})(y)$ 
in the statement (2) from Lemma \ref{lem:R3_gps_MW_sub2} 
using the duplication formula (\ref{eqn:Fn_gamma_duplication}) 
and the equalities 
\begin{align*}
&\prod_{p=0}^{q-1}(s+\tfrac{p}{2})
=2^{-q}(2s)_q,\hspace{1cm} 
\prod_{p=0}^{q-1}(s+p)
=(s)_q,\\[2mm]
&\frac{(-1)^{m}}{m!}
\Gamma \bigl(-m+\tfrac{\varepsilon}{2}\bigr)
=\frac{2^{2m+\frac{1-\varepsilon }{2}}
(-1)^{\frac{1-\varepsilon }{2}} \sqrt{\pi}}
{\bigl(2m+\tfrac{1-\varepsilon }{2}\bigr)!},\\
&(-m)_q
=\left\{\begin{array}{ll}
\dfrac{(-1)^{q}m!}{(m-q)!}&
\text{if}\ m\geq q,\\[10pt]
0&\text{if}\ m<q,
\end{array}\right.
\end{align*}
for $s\in \bC$, $q,m\in \bZ_{\geq 0}$ and $\varepsilon \in \{\pm 1\}$.
\end{proof}

\chapter{Preliminaries for $GL(n,\bC)$}
\label{sec:Cn_rep_theory}

Throughout this chapter, we set ${F}=\bC$.

\section{Principal series representations}
\label{subsec:Cn_def_ps}

In this section, 
we define principal series representations of $G=GL(n,\bC)$. 
For $\nu \in \bC$ and $d\in \bZ$, 
we define a character 
$\chi_{(\nu ,d)}\colon GL(1,\bC) \to \bC^\times$ by 
\begin{align*}
\chi_{(\nu ,d)}(t)
&=\left(\frac{t}{|t|}\right)^{d}|t|^{2\nu }
&(t\in GL(1,\bC)=\bC^\times ). 
\end{align*}
We set 
\begin{align*}
&M=\{\diag (m_1,m_2,\cdots ,m_n)\in {G}\}\simeq (\bC^\times )^n,
&&P={N}M.
\end{align*}
Let $(\nu_i,d_i)\in \bC \times \bZ$ for $1\leq i\leq n$. 
Then $(\chi_{(\nu_1,d_1)},\chi_{(\nu_2,d_2)},\cdots ,\chi_{(\nu_n,d_n)})$ 
defines a character $\chi$ of $M$ by 
$\chi =\chi_{(\nu_1,d_1)}\boxtimes \chi_{(\nu_2,d_2)}\boxtimes 
\cdots \boxtimes \chi_{(\nu_n,d_n)}$, that is, 
\begin{align*}
&\chi (m)=\prod_{i=1}^n\chi_{(\nu_i,d_i)}(m_i)
&&(m=\diag (m_1,m_2,\cdots ,m_n)\in M).
\end{align*}
We extend the character $\chi$ to 
$P={N}M$ by $\chi (xm)=\chi (m)\ (x\in {N},\ m\in M)$. 
We define the function $\rho $ on $P$ by  
\begin{align*}
&\rho (xm)
=\prod_{i=1}^n|m_i|^{n+1-2i}&
&(x\in N,\ m=\diag (m_1,m_2,\cdots ,m_n)\in M).
\end{align*}

Let $H(\chi )^0$ be the space 
of $\bC $-valued continuous functions $f$ on $K$ satisfying 
\begin{align*}
&f(mk)=\chi (m)f(k)&
&(m\in {K}\cap M,\ k\in {K}),
\end{align*}
on which ${G}$ acts by 
\begin{align*}
&(\Pi_\chi (g)f)(k)=\rho (\mmp (kg))
\chi (\mmp (kg))f(\mk (kg))&
&(g\in {G},\ k\in {K},\ f\in H(\chi )^0),
\end{align*}
where $kg=\mmp (kg)\mk (kg)$ is the decomposition of $kg$ 
with respect to the decomposition ${G}=PK$. 
Although the decomposition $kg=\mmp (kg)\mk (kg)$ is not unique, 
the action $\Pi_\chi (g)$ does not depend on the choices of 
$\mmp (kg)$ and $\mk (kg)$. 
We define a principal series representation 
$(\Pi_\chi ,H(\chi ))$ of $G$ as the completion of 
$(\Pi_\chi ,H(\chi )^0)$ 
with respect to the inner product 
\begin{align*}
&\langle f_1,f_2\rangle_{H({\chi})}
=\int_{{K}}f_1(k)\,\overline{f_2(k)}\,dk&
&(f_1,f_2\in H(\chi )^0)
\end{align*}
where $dk$ is the Haar measure on ${K}$.

Combining the results of Kostant 
\cite[Theorems 5.5 and 6.6.2]{Kostant_001} and 
Matumoto \cite[Theorem 6.1.6]{Matumoto_001}, we have
\begin{align}
\label{eqn:Rn_dimWh_gps}
&\dim_\bC  {\cI}_{\Pi_\chi ,\psi_1}=n!,&
&\dim_\bC  {\cI}_{\Pi_\chi ,\psi_1}^{\rm mg}=1
\end{align} 
for a principal series representation $\Pi_\chi$ of $G=GL(n,\bC )$. 
Moreover, 
by Vogan's characterization \cite[Theorem 6.2 (f)]{Vogan_001}, 
any irreducible admissible large representation of ${G}$ is 
infinitesimally equivalent 
to some irreducible principal series representation $\Pi_\chi $.

For an element $w$ of the symmetric group $\mathfrak{S}_n$ of degree $n$, 
we set 
\[
w(\chi )=\chi_{(\nu_{w^{-1}(1)} ,d_{w^{-1}(1)})}\boxtimes 
\chi_{(\nu_{w^{-1}(2)} ,d_{w^{-1}(2)})}\boxtimes 
\cdots \boxtimes \chi_{(\nu_{w^{-1}(n)} ,d_{w^{-1}(n)})}.
\]
When $\Pi (\chi)$ is irreducible, 
it is known that $\Pi_\chi \simeq \Pi_{w(\chi)}$ 
for any $w\in \mathfrak{S}_n$ 
(\textit{cf.} \cite[Corollary 2.8]{Speh_Vogan_001}). 
Hence, it suffices to consider Whittaker functions for 
irreducible principal series representations 
$\Pi_\chi$ defined from 
$\chi =\chi_{(\nu_{1} ,d_{1})}\boxtimes 
\chi_{(\nu_{2} ,d_{2})}\boxtimes 
\cdots \boxtimes \chi_{(\nu_{n} ,d_{n})}$
with $d_1\geq d_2\geq \cdots \geq d_n$.

We denote by $H(\chi)_K$ the subspace of $H(\chi )$ 
consisting of all ${K}$-finite vectors. 
We denote the action of $\g_\bC$ on $H(\chi )_K$ 
induced from $\Pi_\chi $ by the same symbol $\Pi_\chi$. 
The $K$-module structure of $\Pi_\chi$ is wellknown. 
Let $(\tau ,V_\tau )$ be an irreducible representation of ${K}$. 
By the Frobenius reciprocity law \cite[Theorem 1.14]{Knapp_002}, 
the $\bC$-linear map $\Hom_{{K}\cap M}(V_\tau ,\bC_\chi)
\ni \eta \mapsto \hat{\eta} \in \Hom_{{K}}(V_\tau ,H(\chi))$ defined by 
$\hat{\eta}(v)(k)=\eta (\tau (k)v)\ (v\in V_\tau ,\ k\in {K})$ is an 
isomorphism of $\bC$-vector spaces. Here $\bC_\chi =\bC$ is 
a $M$-module on which $M$ acts by the character $\chi $. 
Since any Hilbert representation of $K$ is completely reducible, we have 
\begin{align}
\label{eqn:Cn_ps_Kstr}
&H({\chi})_K=\bigoplus_{(\tau ,V_\tau )}
\bC \textrm{-span}
\{\hat{\eta}(v)\mid \eta \in \Hom_{{K}\cap M}(V_\tau ,\bC_\chi ),\ 
v\in V_\tau\}. 
\end{align}
Here $(\tau ,V_\tau )$ runs through the set of equivalence classes of 
irreducible representations of ${K}$.

\section{The elements of $\g_\bC$ and $Z(\g_\bC)$}
\label{subsec:Cn_g_gen_Zg}

The Lie algebra ${\g_\bC}=\g \gl (n,\bC)\otimes_{\bR}\bC$ is isomorphic to 
$\g \gl (n,\bC ) \oplus \g \gl (n,\bC )$, via 
\begin{align}
\label{eqn:Cn_isom_gC_to_gg}
{\g_\bC}\ni X\otimes c\mapsto (cX,c\overline{X})
\in \g \gl (n,\bC ) \oplus \g \gl (n,\bC ).
\end{align}
For $1\leq i,j\leq n$, 
let $E_{i,j}^{\g}=E_{i,j}^{\g_n}$ and 
$\widetilde{E}_{i,j}^{\g}=\widetilde{E}_{i,j}^{\g_n}$ 
be elements of ${\g_\bC}$ corresponding 
to $(E_{i,j},0)$ and $(0,E_{i,j})$, respectively, that is, 
\begin{align*}
&E_{i,j}^{\g}
=\frac{1}{2}\{E_{i,j}\otimes 1-(\sI E_{i,j})\otimes \sI\},\\
&\widetilde{E}_{i,j}^{\g}
=\frac{1}{2}\{E_{i,j}\otimes 1+(\sI E_{i,j})\otimes \sI\}.
\end{align*}
For $1\leq i,j\leq n$, we set 
\begin{align*}
&E_{i,j}^{{\gk}}=E_{i,j}^{\g}-\widetilde{E}_{j,i}^{\g},&
&E_{i,j}^{{\gp}}=E_{i,j}^{\g}+\widetilde{E}_{j,i}^{\g}.
\end{align*}
Then we have 
\begin{align*}
&\g_\bC =\bigoplus_{1\leq i,j\leq n}
(\bC\,E_{i,j}^{\g}\oplus \bC\,\widetilde{E}_{i,j}^{\g}),
\hspace{40pt}
\gn_\bC =\bigoplus_{1\leq i<j\leq n}
(\bC\,E_{i,j}^{\g}\oplus \bC\,\widetilde{E}_{i,j}^{\g}),\\
&\ga_\bC =\bigoplus_{i=1}^n
\bC\,E_{i,i}^{\gp},\hspace{30pt}
\gk_\bC =\bigoplus_{1\leq i,j\leq n}
\bC\,E_{i,j}^{\gk},\hspace{30pt}
\gp_\bC =\bigoplus_{1\leq i,j\leq n}
\bC\,E_{i,j}^{\gp}.
\end{align*}

\begin{lem}
\label{lem:Cn_g_act_Cpsi}
Let $f$ be a function in $C^\infty (N\backslash G;\psi_\varepsilon )$ 
with $\varepsilon \in \{\pm 1\}$. 
Then, for $1\leq i\leq j\leq n$ 
and $y=\diag (y_1y_2\cdots y_n,\,y_2\cdots y_n,\,\cdots ,\,y_n)\in A$, 
it holds that 
\begin{align*}
&2(R(E_{i,j}^{{\g}})f)(y)
=\left\{\begin{array}{ll}
\displaystyle 
(-\partial_{i-1}+\partial_i)f(y)+(R(E_{i,i}^{\gk})f)(y)
&\text{if}\ j=i,\\[2pt]
4\pi \varepsilon \sI y_if(y)&\text{if}\ j=i+1,\\[2pt]
0&\text{otherwise},
\end{array}\right.\\
&2(R(\widetilde{E}_{i,j}^{{\g}})f)(y)
=\left\{\begin{array}{ll}
\displaystyle 
(-\partial_{i-1}+\partial_i)f(y)-(R(E_{i,i}^{\gk})f)(y)
&\text{if}\ j=i,\\[2pt]
4\pi \varepsilon \sI y_if(y)&\text{if}\ j=i+1,\\[2pt]
0&\text{otherwise},
\end{array}\right.
\end{align*}
where $\partial_{0}=0$ and 
$\displaystyle \partial_i=y_i\frac{\partial}{\partial y_i}$ 
$(1\leq i\leq n)$. 
\end{lem}
\begin{proof}
Direct computation.
\end{proof}

The $\bC$-algebra $U({\g_\bC})$ is isomorphic to 
$U(\g \gl (n,\bC )\oplus \g \gl (n,\bC ))$ via the map 
induced from (\ref{eqn:Cn_isom_gC_to_gg}). 
Moreover, the $\bC$-algebra 
$U(\g \gl (n,\bC )\oplus \g \gl (n,\bC ))$ 
is isomorphic to $U(\g \gl (n,\bC ))\otimes_\bC U(\g \gl (n,\bC ))$ 
via the map induced from 
\begin{align*}
\g \gl (n,\bC ) \!\oplus \!\g \gl (n,\bC ) \ni (X_1,X_2)\mapsto 
X_1\otimes 1+1\otimes X_2
\in U(\g \gl (n,\bC ))\otimes_\bC U(\g \gl (n,\bC )). 
\end{align*}
Therefore, 
we may identify $U({\g_\bC} )$ 
with $U(\g \gl (n,\bC ))\otimes_\bC U(\g \gl (n,\bC ))$ 
via the composite of the above isomorphisms. 
For $1\leq h\leq n$, let $\cC_{h}^{\g}$ 
and $\widetilde{\cC}_{h}^{\g}$ be 
the elements of $Z({\g_\bC} )$ corresponding to 
$\cC_h\otimes 1$ and $1\otimes \cC_h$, respectively. 
Here $\cC_h\ (1\leq h\leq n)$ are the Capelli elements 
introduced in \S \ref{subsec:Fn_capelli}.
Since $\cC_h\ (1\leq h\leq n)$ generate $Z(\g \gl (n,\bC))$, 
we know that $\cC_{h}^{\g},\widetilde{\cC}_{h}^{\g}\ 
(1\leq h\leq n)$ generate $Z({\g_\bC})$ as a $\bC$-algebra. 
The explicit forms of 
$\cC_{h}^{\g},\widetilde{\cC}_{h}^{\g}\ 
(1\leq h\leq n)$ are given by 
\begin{align*}
&\cC_h^{\g}=\underset{w\in \gS_h}
{\sum_{1\leq i_1<i_2<\cdots <i_h\leq n,}}
\sgn (w)\cE_{i_1,i_{w(1)}}^{\g} 
\cE_{i_2,i_{w(2)}}^{\g} \cdots 
\cE_{i_h,i_{w(h)}}^{\g} ,\\
&\widetilde{\cC}_h^{\g}=\underset{w\in \gS_h}
{\sum_{1\leq i_1<i_2<\cdots <i_h\leq n,}}
\sgn (w)\widetilde{\cE}_{i_1,i_{w(1)}}^{\g} 
\widetilde{\cE}_{i_2,i_{w(2)}}^{\g} \cdots 
\widetilde{\cE}_{i_h,i_{w(h)}}^{\g},
\end{align*}
with 
\begin{align*}
&\cE_{i,j}^{\g} =\left\{\begin{array}{ll} 
E_{i,i}^{\g} -\tfrac{n+1-2i}{2}&\text{if }\,i=j,\\[1mm]
E_{i,j}^{\g} &\text{if }\,i\neq j,
\end{array}\right.&
&\widetilde{\cE}_{i,j}^{\g} =\left\{\begin{array}{ll} 
\widetilde{E}_{i,i}^{\g} -\tfrac{n+1-2i}{2}&\text{if }\,i=j,\\[1mm]
\widetilde{E}_{i,j}^{\g} &\text{if }\,i\neq j.
\end{array}\right.
\end{align*}

\section{The eigenvalues of generators of $Z(\g_\bC)$}
\label{subsec:Cn_inf_char_ps}

In this section, let 
$\chi =\chi_{(\nu_1,d_1)}\boxtimes \chi_{(\nu_2,d_2)} 
\boxtimes \cdots \boxtimes \chi_{(\nu_n,d_n)}$ 
with $\nu_1,\nu_2,\cdots ,\nu_n\in \bC$ and 
$d_1,d_2,\cdots ,d_n\in \bZ$. 
By the definition of $\Pi_\chi$, we have 
\begin{align}
\label{eqn:Cn_gact_value1}
&(\Pi_\chi (X)f)(1_n)=0&&(X\in \gn_{\bC}U(\g_\bC)),\\
\label{eqn:Cn_gact_value2}
&(\Pi_\chi (E_{i,i}^{{\g}})f)(1_n)
=\bigl(\nu_i+\tfrac{d_i}{2}+\tfrac{n+1-2i}{2}\bigr)f(1_n),\\
\label{eqn:Cn_gact_value3}
&(\Pi_\chi (\widetilde{E}_{i,i}^{{\g}})f)(1_n)
=\bigl(\nu_i-\tfrac{d_i}{2}+\tfrac{n+1-2i}{2}\bigr)f(1_n)&
&(1\leq i\leq n)
\end{align}
for $f\in H(\chi)_K$. 
Using the equalities (\ref{eqn:Cn_gact_value1}), 
(\ref{eqn:Cn_gact_value2}) and (\ref{eqn:Cn_gact_value3}), 
we obtain the following proposition 
as is Proposition \ref{prop:Rn_Ch_eigenvalue}.

\begin{prop}
\label{prop:Cn_Ch_eigenvalue}
Retain the notation. Then, for $f\in H(\chi)_K$ and $1\leq h\leq n$, 
it holds that 
\begin{align*}
&\Pi_\chi (\cC_h^{{\g}})f=
\left\{\sum_{1\leq i_1<i_2<\cdots <i_h\leq n}\!
\Bigl(\nu_{i_1}+\tfrac{d_{i_1}}{2}\Bigr)\!
\Bigl(\nu_{i_2}+\tfrac{d_{i_2}}{2}\Bigr)\cdots 
\Bigl(\nu_{i_h}+\tfrac{d_{i_h}}{2}\Bigr)
\right\}f,\\
&\Pi_\chi (\widetilde{\cC}_h^{{\g}})f=
\left\{\sum_{1\leq i_1<i_2<\cdots <i_h\leq n}\!
\Bigl(\nu_{i_1}-\tfrac{d_{i_1}}{2}\Bigr)\!
\Bigl(\nu_{i_2}-\tfrac{d_{i_2}}{2}\Bigr)\cdots 
\Bigl(\nu_{i_h}-\tfrac{d_{i_h}}{2}\Bigr)
\right\}f.
\end{align*}
\end{prop}

\chapter{Whittaker functions on $GL(2,\bC)$}
\label{sec:C2_whittaker}

Throughout this chapter, we set $n=2$ and $F=\bC $.

\section{Representations of $U(2)$}
\label{subsec:C2_rep_K}

In this section, we discuss 
the representation theory of ${K}=U(2)$. 
Let $\Lambda_2 =\{\lambda 
=(\lambda_1,\lambda_2)\in \bZ^{2}\mid 
\lambda_1\geq \lambda_2\}$. 
For $\lambda =(\lambda_1,\lambda_2)\in \Lambda_2$, 
let $V_{\lambda}=V_{\lambda}^{(2)}$ be the $\bC$-vector space of 
degree $\lambda_1-\lambda_2$ homogeneous polynomials of two variables 
${z}_1,{z}_2$, and define the action 
$\tau_\lambda =\tau_\lambda^{(2)}$ of $K$ on $V_{\lambda}$ by 
\begin{align*}
&(\tau_{\lambda }(k)p)({z}_1,{z}_2)
=(\det k)^{\lambda_2}p(({z}_1,{z}_2)\cdot k)&
&(k\in K,\ p\in V_\lambda ).
\end{align*}
Here $({z}_1,{z}_2)\cdot k$ is the ordinal product of matrices. 
Then $(\tau_{\lambda },V_{\lambda })$ 
is irreducible, and $\dim_\bC  V_{\lambda }=\lambda_1-\lambda_2+1$. 
The set of equivalence classes of 
irreducible representations of ${K}$ is 
exhausted by 
$\{\tau_{\lambda }\mid \lambda \in \Lambda_2 \}$. 

For $\lambda=(\lambda_1,\lambda_2)\in \Lambda_2$, 
we set $Q_{\lambda}=\{q\in \bZ \mid 0\leq q\leq \lambda_1-\lambda_2\}$. 
For $q\in Q_{\lambda}$, 
we define a basis $\{v_{\lambda ,q}\mid q\in Q_\lambda \}$ of $V_{\lambda}$ 
by 
\[
v_{\lambda ,q}(z_1,z_2)=z_1^{\lambda_1-\lambda_2-q}z_2^q. 
\]
It is convenient to set 
$v_{\lambda,q}=0$ if $q\not\in Q_\lambda$. 
The action of 
\[
K\cap M=\{\diag (m_1,m_2)\mid m_1,m_2\in U(1)\} 
\]
on this basis is given by 
\begin{align}
\label{eqn:C2_Mact}
&\tau_{\lambda }(m)v_{\lambda,q}
=m_1^{\lambda_1-q}m_2^{q+\lambda_2}v_{\lambda,q}&
&(m=\diag (m_1,m_2)\in K\cap M).
\end{align}
We denote the differential of $\tau_{\lambda }$ 
again by $\tau_{\lambda }$. 
Then, for $q\in Q_\lambda$, we have 
\begin{align}
\label{eqn:C2_gKact11_22}
&\tau_{\lambda }(E_{1,1}^{\gk} )v_{\lambda,q}=(\lambda_1-q)v_{\lambda,q},&
&\tau_{\lambda }(E_{2,2}^{\gk} )v_{\lambda,q}=(q+\lambda_2)v_{\lambda,q},\\
\label{eqn:C2_gKact12_21}
&\tau_{\lambda }(E_{1,2}^{\gk} )v_{\lambda,q}=qv_{\lambda,q-1},&
&\tau_{\lambda }(E_{2,1}^{\gk} )v_{\lambda,q}
=(\lambda_1-\lambda_2-q)v_{\lambda,q+1}.
\end{align}

For later use, we prepare the following lemma. 
\begin{lem}
\label{lem:C2_tensor}
For $\alpha =(\alpha_1,\alpha_2),
\beta =(\beta_1,\beta_2)\in \Lambda_2$, 
it holds that 
\begin{align}
\label{eqn:C2_tensor1}
\Hom_K(V_{\alpha +\beta},V_{\alpha}\otimes_{\bC}V_{\beta})
=\bC \,\mathrm{I}_{\alpha +\beta}^{\alpha ,\beta}, 
\end{align}
where $\mathrm{I}_{\alpha +\beta}^{\alpha ,\beta}\colon 
V_{\alpha +\beta}\to V_{\alpha}\otimes_{\bC}V_{\beta}$ 
is a $\bC$-linear map characterized by 
\begin{align*}
\mathrm{I}_{\alpha +\beta}^{\alpha ,\beta}
(v_{\alpha +\beta ,q})
=\sum_{i=\max \{0,q-\beta_1 +\beta_2\}}^{\min \{q,\alpha_1-\alpha_2\}}
\binom{q}{i}
\binom{\alpha_1+\beta_1-\alpha_2-\beta_2-q}{\alpha_1-\alpha_2-i}
v_{\alpha ,i}\otimes v_{\beta,q-i}&\\
(0\leq q\leq \alpha_1+\beta_1-\alpha_2-\beta_2)&.
\end{align*}
\end{lem}
\begin{proof}
Since $V_{\alpha}\otimes_{\bC}V_{\beta }$ contains 
a unique highest weight vector 
$v_{\alpha ,0}\otimes v_{\beta,0}$ with weight $\alpha +\beta$ 
up to scalar multiple, we know that 
there exists a $K$-homomorphism 
$\mathrm{I}_{\alpha +\beta}^{\alpha,\beta}$ 
satisfying (\ref{eqn:C2_tensor1}) and $
\mathrm{I}_{\alpha +\beta }^{\alpha,\beta }(v_{\alpha +\beta ,0})
=v_{\alpha,0}\otimes v_{\beta ,0}$. 
Using 
\begin{align*}
\mathrm{I}_{\alpha +\beta }^{\alpha,\beta}(v_{\alpha +\beta ,q+1})
=(\alpha_1+\beta_1-\alpha_2-\beta_2-q)^{-1}
(\tau_{\alpha}\otimes \tau_{\beta})
(E_{2,1}^{\gk} )\mathrm{I}_{\alpha +\beta}^{\alpha,\beta}
(v_{\alpha +\beta ,q})
\end{align*}
for $0\leq q<\alpha_1 +\beta_1-\alpha_2 -\beta_2$, 
we can obtain the explicit expressions of 
$\mathrm{I}_{\alpha +\beta}^{\alpha,\beta}(v_{\alpha +\beta ,q})$ 
by direct computation. 
\end{proof}

We regard $\gp_\bC$ as a $K$-module via the adjoint action $\Ad$, 
and denote the differential of $\Ad$ by $\adj$. 
By direct computation, we have 
\begin{align}
\label{eqn:C2_gKact_gp}
&\adj (E_{a,b}^{{\gk}})E_{i,j}^{{\gp}}
=\delta_{b,i}E_{a,j}^{{\gp}}-\delta_{j,a}E_{i,b}^{{\gp}}
\end{align}
for $1\leq a,b,i,j\leq 2$. Let 
\begin{align*}
&\cC_1^{\gp}=E^{\gp}_{1,1}+E^{\gp}_{2,2},&
&X_{0}^{\gp}=2E_{1,2}^{{\gp}},&
&X_{1}^{\gp}=E_{2,2}^{{\gp}}-E_{1,1}^{{\gp}},&
&X_{2}^{\gp}=-2E_{2,1}^{{\gp}}.
\end{align*}
Then, comparing (\ref{eqn:C2_gKact11_22}) and (\ref{eqn:C2_gKact12_21}) with 
(\ref{eqn:C2_gKact_gp}), we have 
\begin{align}
\label{eqn:C2_gp_Kisom}
\gp_{\bC}=(\bC \, \cC_1^{\gp})\oplus 
(\bC \, X^{\gp}_{0}\oplus \bC \, X^{\gp}_{1}\oplus \bC \, X^{\gp}_{2})
\simeq V_{(0,0)}\oplus V_{(1,-1)}
\end{align}
as $K$-modules, via $\cC_1^{\gp}\mapsto v_{(0,0),0}$ and 
$X^{\gp}_{q}\mapsto v_{(1,-1),q}$ ($0\leq q\leq 2$).

\section{Principal series representations}
\label{subsec:C2_ps}

In this section, 
let $\chi =\chi_{(\nu_1,d_1)}\boxtimes \chi_{(\nu_2,d_2)}$ 
with $(\nu_1,\nu_2)\in \bC^2$ and 
$(d_1,d_2)\in \Lambda_2$. 
We recall the $(\g_{\bC},K)$-module structure 
of a principal series representation 
$(\Pi_\chi ,H(\chi ))$ of $G=GL(2,\bC)$.

Let $\Lambda (\chi)=\{(d_1+l,d_2-l)\in \Lambda_2\mid 
l\in \bZ_{\geq 0}\}$. 
Because of (\ref{eqn:C2_Mact}), 
for $\lambda \in \Lambda_2$, we have 
\begin{align*}
&\Hom_{{K}\cap M}(V_{\lambda},\bC_{\chi})
=\left\{\begin{array}{ll}
\bC \,\eta_{(\chi ;\lambda )}&\text{ if }\ 
\lambda \in \Lambda (\chi),\\
0&\text{ if }\ 
\lambda \not\in \Lambda (\chi), 
\end{array}\right.
\end{align*} 
where $\eta_{(\chi ;\lambda )}\colon V_{\lambda }\to \bC_{\chi }$ 
is a $\bC$-linear map defined by 
\begin{align*}
&\eta_{(\chi ;\lambda )}(v_{\lambda,q})
=\delta_{q,l}&
&(\,\lambda =(d_1+l,d_2-l)\in \Lambda (\chi),\ 
q\in Q_\lambda \,).
\end{align*} 
Hence, we know from (\ref{eqn:Cn_ps_Kstr}) that 
\begin{align*}
&H({\chi})_K=\bigoplus_{\lambda \in \Lambda (\chi)}
\{\hat{\eta}_{(\chi ;\lambda )}(v)\mid 
v\in V_{\lambda }\} 
\end{align*}
with 
\begin{align}
\label{eqn:C2_eta_chi_lambda}
&\hat{\eta}_{(\chi ;\lambda )}(v)(k)=\eta_{(\chi ;\lambda )}
(\tau_{\lambda} (k)v)&&(v\in V_{\lambda},\ k\in {K}). 
\end{align}
We set 
$\zeta_{(\chi;l,q)}=
\hat{\eta}_{(\chi ;\lambda )}(v_{\lambda,q})$ 
for $\lambda =(d_1+l,d_2-l)\in \Lambda (\chi)$ and $q\in Q_\lambda$. 
It is convenient to 
set $\zeta_{(\chi;l,q)}=0$ unless $0\leq q\leq d_1-d_2+2l$. 
Then we have 
\begin{align}
H(\chi)_K=\bigoplus_{l=0}^\infty 
\left(\bigoplus_{q=0}^{d_1-d_2+2l}
\bC \zeta_{(\chi;l,q)}\right)
\end{align}
and 
\begin{align}
\label{eqn:C2_gkact_ps11}
&\Pi_{\chi }(E_{1,1}^{\gk} )\zeta_{(\chi;l,q)}
=(d_1+l-q)\zeta_{(\chi;l,q)},\\
\label{eqn:C2_gkact_ps22}
&\Pi_{\chi }(E_{2,2}^{\gk} )\zeta_{(\chi;l,q)}
=(d_2-l+q)\zeta_{(\chi;l,q)},\\
\label{eqn:C2_gkact_ps12}
&\Pi_{\chi }(E_{1,2}^{\gk} )\zeta_{(\chi;l,q)}
=q\zeta_{(\chi;l,q-1)},\\
\label{eqn:C2_gkact_ps21}
&\Pi_{\chi }(E_{2,1}^{\gk} )\zeta_{(\chi;l,q)}
=(d_1-d_2+2l-q)\zeta_{(\chi;l,q+1)},\\
\label{eqn:C2_value_ps}
&\zeta_{(\chi;l,q)}(1_2)
=\delta_{q,l}.
\end{align}

\begin{lem}
\label{lem:C2_DSE_ps}
Retain the notation. Then, for non-negative integers $l$, $q$ 
such that $q\leq d_1-d_2+2l+2$, 
it holds that 
\begin{align*}
&-(l+1)(d_1-d_2+l+1)(2\nu_1-2\nu_2+d_1-d_2+2l+2)
\zeta_{(\chi;l+1,q)}\\
&=
(d_1-d_2+2l+2-q)(d_1-d_2+2l+1-q)
\Pi_\chi (E^{\gp}_{1,2})\zeta_{(\chi;l,q)}\\
&\phantom{=,}
-q(d_1-d_2+2l+2-q)
\Pi_\chi (E^{\gp}_{1,1}-E^{\gp}_{2,2})\zeta_{(\chi;l,q-1)}\\
&\phantom{=,}
-q(q-1)\Pi_\chi (E^{\gp}_{2,1})\zeta_{(\chi;l,q-2)}.
\end{align*}
\end{lem}
\begin{proof}
Let $\lambda =(d_1+l,d_2-l) \in \Lambda (\chi)$. 
We define a $\bC$-linear map $\mathrm{P}_{\chi}^{\lambda}
\colon V_{(1,-1)}\otimes_{\bC}V_{\lambda}\to H(\chi)_K$ 
by \begin{align*}
&v_{(1,-1),i}\otimes v_{\lambda ,q}\mapsto 
\Pi_{\chi}(X^{\gp}_i)\zeta_{(\chi;l,q)}&
&(0\leq i\leq 2,\ 0\leq q\leq d_1-d_2+2l).
\end{align*} 
Because of the isomorphism (\ref{eqn:C2_gp_Kisom}) and 
the definition of $\zeta_{(\chi;l,q)}$, 
we know that $\mathrm{P}_{\chi}^{\lambda}$ is a $K$-homomorphism. 
Let $\mathrm{I}_{\lambda +(1,-1)}^{(1,-1),\lambda}\colon 
V_{\lambda +(1,-1)}\to V_{(1,-1)}\otimes_{\bC}V_{\lambda}$ 
be the $K$-homomorphism in Lemma \ref{lem:C2_tensor}. 
Since the composite 
$\mathrm{P}_{\chi}^{\lambda} \circ 
\mathrm{I}_{\lambda +(1,-1)}^{(1,-1),\lambda}$ 
is an element of 
$\Hom_K(V_{\lambda +(1,-1)}, 
H(\chi)_K)=\bC \, \hat{\eta}_{(\chi ;\lambda +(1,-1))}$, 
there is a constant $c$ such that 
\begin{align*}
c\,\hat{\eta}_{(\chi ;\lambda +(1,-1))}=
\mathrm{P}_{\chi}^{\lambda} \circ 
\mathrm{I}_{\lambda +(1,-1)}^{(1,-1),\lambda}.
\end{align*}
Considering the image of $v_{\lambda +(1,-1),q}$ 
under the both sides of this equality, we have 
\begin{align*}
c\,\zeta_{(\chi;l+1,q)}
=&\,
(d_1-d_2+2l+2-q)(d_1-d_2+2l+1-q)
\Pi_\chi (E^{\gp}_{1,2})\zeta_{(\chi;l,q)}\\
&-q(d_1-d_2+2l+2-q)
\Pi_\chi (E^{\gp}_{1,1}-E^{\gp}_{2,2})\zeta_{(\chi;l,q-1)}\\
&-q(q-1)\Pi_\chi (E^{\gp}_{2,1})\zeta_{(\chi;l,q-2)}
\end{align*}
for $0\leq q\leq d_1-d_2+2l+2$. 
For $q=l+1$, we have 
\begin{align*}
c&=c\,\zeta_{(\chi;l+1,l+1)}(1_2)\\
&=
(d_1-d_2+l+1)(d_1-d_2+l)
(\Pi_\chi (2E_{1,2}^{\g}-E_{1,2}^{{\gk}})
\zeta_{(\chi;l,l+1)})(1_2)\\
&\phantom{=}\,
-(l+1)(d_1-d_2+l+1)
(\Pi_\chi (E_{1,1}^{\g}+\widetilde{E}_{1,1}^{\g}
-E_{2,2}^{\g}-\widetilde{E}_{2,2}^{\g})\zeta_{(\chi;l,l)})(1_2)\\
&\phantom{=}\,
-(l+1)l
(\Pi_\chi (2\widetilde{E}_{1,2}^{\g}+E_{2,1}^{{\gk}})
\zeta_{(\chi;l,l-1)})(1_2)\\
&=
-(l+1)(d_1-d_2+l+1)(2\nu_1-2\nu_2+d_1-d_2+2l+2),
\end{align*}
using 
(\ref{eqn:Cn_gact_value1}), 
(\ref{eqn:Cn_gact_value2}), 
(\ref{eqn:Cn_gact_value3}), 
(\ref{eqn:C2_gkact_ps12}), 
(\ref{eqn:C2_gkact_ps21}) 
and (\ref{eqn:C2_value_ps}). 
This completes a proof. 
\end{proof}

From the above proof and the irreducible 
decomposition (\textit{cf}. Propsition \ref{prop:C32_tensor_K2}) 
\begin{align*}
&V_{(1,-1)}\otimes_\bC V_\lambda \simeq 
\bigoplus_{i=0}^{\min \{\lambda_1-\lambda_2,2\}}
V_{\lambda +(1-i,i-1)}&
&(\lambda =(\lambda_1,\lambda_2)\in \Lambda_2)
\end{align*}
as a $K$-module, 
we know that, if there exists $m\in \bZ_{\geq 0}$ such that 
$2\nu_1-2\nu_2+d_1-d_2+2=-2m$, 
then the space 
\begin{align}
\bigoplus_{l=0}^m 
\left(\bigoplus_{q=0}^{d_1-d_2+2l}
\bC \zeta_{(\chi;l,q)}\right)
\end{align}
is closed under the action of elements of $\gp_\bC$, 
and is a $(\g_\bC,K)$-submodule of $H(\chi)_K$.  
Hence, if $\Pi_\chi$ is irreducible, it holds that 
\[
2\nu_1-2\nu_2+d_1-d_2+2\not\in 2\bZ_{\leq 0}.
\]

\section{Principal series Whittaker functions}
\label{subsec:C2_ps_whittaker}

We use the notation in \S \ref{subsec:C2_ps}. 
We will give the explicit formulas of Whittaker functions 
of $\Pi_\chi$ with 
$\chi =\chi_{(\nu_1,d_1)}\boxtimes \chi_{(\nu_2,d_2)}$ $(d_1\geq d_2)$. 
Let $\Phi \in {\cI}_{\Pi_\chi ,\psi_1}$. 
We note that $\Phi$ is characterized by 
\begin{align*}
&\Phi ({\zeta}_{(\chi;l,q)} )(y)&&(l \in \bZ_{\geq 0},\ 
0\leq q \leq d_1-d_2+2l )
\end{align*}
with $y=\diag(y_1y_2,y_2)\in A$. 
We set $\partial_i = y_i \dfrac{\partial}{\partial y_i}$ for $i=1,2$.

\begin{lem}
\label{lem:C2_Zg1_original}
Retain the notation. Let $l\in \bZ_{\geq 0}$. 

\noindent (1) 
For $0\leq q\leq d_1-d_2+2l$, it holds that 
\begin{align}
&\label{eqn:C2_PDE1_C1}
(\partial_2-2\nu_1-2\nu_2)
\Phi ({\zeta}_{(\chi;l,q)})(y)=0,\\[1mm]
\begin{split}
&
\{(\partial_1-d_1-l+q-1)
(-\partial_{1}+\partial_2-d_2+l-q+1)\\
&\label{eqn:C2_PDE1_C2}
-(2\nu_1-d_1)(2\nu_2-d_2)+(4\pi y_1)^2\}
\Phi ({\zeta}_{(\chi;l,q)})(y)\\
&+8\pi \sI y_1q
\Phi ({\zeta}_{(\chi;l,q-1)})(y)=0,
\end{split}\\[1mm]
\begin{split}
&
\{(\partial_1+d_1+l-q-1)(-\partial_{1}+\partial_2+d_2-l+q+1)\\
&\label{eqn:C2_PDE1_C2omote}
-(2\nu_1+d_1)(2\nu_2+d_2)+(4\pi y_1)^2\}\Phi ({\zeta}_{(\chi;l,q)})(y)\\
&-8\pi \sI y_1(d_1-d_2+2l-q)
\Phi ({\zeta}_{(\chi;l,q+1)})(y)=0.
\end{split}
\end{align}
\noindent (2) It holds that 
\begin{align}
\label{eqn:C2_PDE1_DSE}
\begin{split}
&-(l+1)(d_1-d_2+l+1)(2\nu_1-2\nu_2+d_1-d_2+2l+2)
\Phi (\zeta_{(\chi;l+1,0)})(y)\\
&=(d_1-d_2+2l+2)(d_1-d_2+2l+1)
(4\pi \sI y_1)\Phi (\zeta_{(\chi;l,0)})(y).
\end{split}
\end{align}
\end{lem}
\begin{proof}
By Proposition \ref{prop:Cn_Ch_eigenvalue}, we have 
\begin{align}
\label{eqn:C2_pf_PDE001}
&(R(\widetilde{\cC}_1^{{\g}})\Phi (\zeta_{(\chi;l,q)}))(y)
-\bigl(\nu_1+\nu_2-\tfrac{d_1+d_2}{2}\bigr)\Phi (\zeta_{(\chi;l,q)})(y)=0,\\
\label{eqn:C2_pf_PDE002}
&(R(\widetilde{\cC}_2^{{\g}})\Phi (\zeta_{(\chi;l,q)}))(y)
-
\bigl(\nu_1-\tfrac{d_1}{2}\bigr)
\bigl(\nu_2-\tfrac{d_2}{2}\bigr)\Phi (\zeta_{(\chi;l,q)})(y)=0,\\
\label{eqn:C2_pf_PDE003}
&(R(\cC_2^{{\g}})\Phi (\zeta_{(\chi;l,q)}))(y)
-\bigl(\nu_1+\tfrac{d_1}{2}\bigr)
\bigl(\nu_2+\tfrac{d_2}{2}\bigr)\Phi (\zeta_{(\chi;l,q)})(y)=0
\end{align}
for $0\leq q\in d_1-d_2+2l$. 
The explicit forms of $\widetilde{\cC}_1^{\g}$, $\widetilde{\cC}_2^{\g}$
and $\cC_2^{\g}$ are given by 
\begin{align}
&\label{eqn:C2_pf_PDE004}
\widetilde{\cC}_1^{\g}=\widetilde{E}_{1,1}^{\g}
+\widetilde{E}_{2,2}^{\g},\\ 
&\nonumber 
\widetilde{\cC}_2^{\g}
=\bigl(\widetilde{E}_{1,1}^{\g}-\tfrac{1}{2}\bigr)
\bigl(\widetilde{E}_{2,2}^{\g}+\tfrac{1}{2}\bigr)
-\widetilde{E}_{1,2}^{\g}\widetilde{E}_{2,1}^{\g},\\
&\nonumber 
\cC_2^{\g}
=\bigl(E_{1,1}^{\g}-\tfrac{1}{2}\bigr)
\bigl(E_{2,2}^{\g}+\tfrac{1}{2}\bigr)
-E_{1,2}^{\g}E_{2,1}^{\g}.
\end{align}
By the equalities $E_{2,1}^{{\g}}=\widetilde{E}_{1,2}^{{\g}}+E_{2,1}^{{\gk}}$ 
and $\widetilde{E}_{2,1}^{{\g}}=E_{1,2}^{{\g}}-E_{1,2}^{{\gk}}$, we have 
\begin{align}
&\label{eqn:C2_pf_PDE005}
\widetilde{\cC}_2^{\g}
=\bigl(\widetilde{E}_{1,1}^{\g}-\tfrac{1}{2}\bigr)
\bigl(\widetilde{E}_{2,2}^{\g}+\tfrac{1}{2}\bigr)
-\widetilde{E}_{1,2}^{\g}E_{1,2}^{{\g}}
+\widetilde{E}_{1,2}^{\g}E_{1,2}^{{\gk}},\\
&\label{eqn:C2_pf_PDE006}
\cC_2^{\g}
=\bigl(E_{1,1}^{\g}-\tfrac{1}{2}\bigr)
\bigl(E_{2,2}^{\g}+\tfrac{1}{2}\bigr)
-E_{1,2}^{\g}\widetilde{E}_{1,2}^{{\g}}-E_{1,2}^{\g}E_{2,1}^{{\gk}}.
\end{align}
Applying (\ref{eqn:C2_gkact_ps11}), (\ref{eqn:C2_gkact_ps22}), 
(\ref{eqn:C2_gkact_ps12}), (\ref{eqn:C2_gkact_ps21}) 
and Lemma \ref{lem:Cn_g_act_Cpsi} 
to the equalities 
(\ref{eqn:C2_pf_PDE001}), 
(\ref{eqn:C2_pf_PDE002}) and (\ref{eqn:C2_pf_PDE003})
with the expressions 
(\ref{eqn:C2_pf_PDE004}), 
(\ref{eqn:C2_pf_PDE005}) and (\ref{eqn:C2_pf_PDE006}), 
we obtain the statement (1).

By Lemma \ref{lem:C2_DSE_ps} with $q=0$, 
we have 
\begin{align*}
&-(l+1)(d_1-d_2+l+1)(2\nu_1-2\nu_2+d_1-d_2+2l+2)
\Phi (\zeta_{(\chi;l+1,0)})(y)\\
&=
(d_1-d_2+2l+2)(d_1-d_2+2l+1)
(R(E^{\gp}_{1,2})\Phi (\zeta_{(\chi;l,0)}))(y).
\end{align*}
Applying (\ref{eqn:C2_gkact_ps12}) and 
Lemma \ref{lem:Cn_g_act_Cpsi} 
to this equality with 
$E_{1,2}^{\gp}=2E_{1,2}^{\g}-E_{1,2}^{{\gk}}$, 
we obtain the statement (2). 
\end{proof}

\begin{rem}
The partial differential equation obtained from the action of 
$\cC_1^{\g}$ coincides with (\ref{eqn:C2_PDE1_C1}), 
which is obtained from the action of $\widetilde{\cC}_1^{\g}$.
\end{rem}

By (\ref{eqn:C2_PDE1_C1}), 
for $l\in \bZ_{\geq 0}$ 
and $0\leq q\leq d_1-d_2+2l$, 
we can define a function $\hat{\varphi}_{l,q}$ on $\bR_+$ by
\begin{align}
\label{eqn:R2_def_hat_varphi}
\Phi ({\zeta}_{(\chi ;l,q)})(y)
=(\sI )^{-q-l}y_1y_2^{2\nu_1+2\nu_2}\hat{\varphi}_{l,q}(y_1)
\end{align}
with $y=\diag (y_1y_2,y_2)\in A$. 
It is convenient to set $\hat{\varphi}_{l,q}=0$ 
unless $0\leq q\leq d_1-d_2+2l$. 
Then, for $l\in \bZ_{\geq 0}$ and 
$0\leq q\leq d_1-d_2+2l$, 
the equations (\ref{eqn:C2_PDE1_C2}), 
(\ref{eqn:C2_PDE1_C2omote}) and (\ref{eqn:C2_PDE1_DSE}) 
lead that 
\begin{align}
\begin{split}
&\label{eqn:C2_PDE2_C2}
\{-(\partial_1-2\nu_1-l+q)(\partial_1-2\nu_2-d_1+d_2-l+q)
+(4\pi y_1)^2\}
\hat{\varphi}_{l,q}\\
&-8\pi y_1q
\hat{\varphi}_{l,q-1}=0,
\end{split}\\[1mm]
\label{eqn:C2_PDE2_C2omote}
\begin{split}
&\{-(\partial_1-2\nu_1+l-q)(\partial_1-2\nu_2+d_1-d_2+l-q)
+(4\pi y_1)^2\}\hat{\varphi}_{l,q}\\
&-8\pi y_1(d_1-d_2+2l-q)\hat{\varphi}_{l,q+1}=0.
\end{split}
\end{align}
and  
\begin{align}
\label{eqn:C2_PDE2_DSE}
\begin{split}
&(l+1)(d_1-d_2+l+1)(2\nu_1-2\nu_2+d_1-d_2+2l+2)\hat{\varphi}_{l+1,0}\\
&=(d_1-d_2+2l+2)(d_1-d_2+2l+1)
(4\pi y_1)\hat{\varphi}_{l,0},
\end{split}
\end{align}
respectively. 
Here we understand $\partial_1 = y_1\dfrac{d}{dy_1}$.

Until the end of this chapter, we assume 
$2\nu_1-2\nu_2+d_1-d_2+2\not\in 2\bZ_{\leq 0}$. Then 
the functions $\hat{\varphi}_{l,q}$ 
$(l\in \bZ_{\geq 0},\,0\leq q\leq d_1-d_2+2l)$ 
are determined from $\hat{\varphi}_{0,0}$ by 
the equations (\ref{eqn:C2_PDE2_C2omote}) 
and (\ref{eqn:C2_PDE2_DSE}). 
Hence, we know that 
$\hat{\varphi}_{0,0}=0$ if and only if $\Phi =0$. 
We first consider explicit formulas of 
Whittaker functions at the minimal $K$-type $\tau_{(d_1,d_2)}$.

\begin{thm}
\label{thm:C2_ps_Whittaker}
Let $\chi = \chi_{(\nu_1,d_1)}\boxtimes \chi_{(\nu_2,d_2)}$ with 
$\nu_1,\nu_2\in \bC$ and $(d_1,d_2) \in \Lambda_2$ 
such that $2\nu_1-2\nu_2+d_1-d_2+2\not\in 2\bZ_{\leq 0}$. 

\noindent  
(1) There exists a homomorphism 
$\Phi_{\chi}^{\rm mg} \in {\cI}_{\Pi_\chi ,\psi_1}^{\rm mg}$ 
with the radial part 
\begin{align*}
&\Phi_{\chi}^{\rm mg}({\zeta}_{(\chi;0,q)})(y) 
=8(\sI)^{-q}y_1^{\nu_1+\nu_2+\tfrac{d_1-d_2}{2}+1}
y_2^{2\nu_1+2\nu_2}
K_{\nu_1-\nu_2+q-(d_1-d_2)/2}(4\pi y_1)\\
&\phantom{====}
=(\sI)^{-q}\frac{y_1y_2^{2\nu_1+2\nu_2}}{2\pi \sI}
\int_{s}
\Gamma_{\bC}\bigl(s+\nu_1+\tfrac{q}{2}\bigr)\Gamma_{\bC}
\bigl(s+\nu_2+\tfrac{d_1-d_2-q}{2}\bigr)y_1^{-2s}ds.
\end{align*}
Here $y=\diag (y_1y_2,y_2)\in A $, 
and the path of integration $\int_{s}$ is the vertical line 
from $\mathrm{Re}(s)-\sI \infty$ to $\mathrm{Re}(s)+\sI \infty$ 
with the sufficiently large real part to keep the poles of the integrand 
on its left.  \\
\noindent 
(2) Assume $2\nu_1-2\nu_2 \notin d_1-d_2+2\bZ$. 
Then there exist homomorphisms 
$\Phi_{\chi}^{+}$ and $\Phi_{\chi}^{-}$ in ${\cI}_{\Pi_\chi ,\psi_1}$ 
with the radial parts 
\begin{align*}
\Phi_{\chi}^{\pm }({\zeta}_{(\chi;0,q)})(y) 
= &\,4(\sI)^{-q}y_1^{\nu_1+\nu_2+\tfrac{d_1-d_2}{2}+1}
y_2^{2\nu_1+2\nu_2}\\
&\times \hat{I}_{\pm (\nu_1-\nu_2+q-(d_1-d_2)/2)}(4\pi y_1).
\end{align*}
Here $y=\diag (y_1y_2,y_2)\in A $. 
Moreover, $\{\Phi_{\chi}^{+},\Phi_{\chi}^{-}\}$ forms a basis of
${\cI}_{\Pi_\chi ,\psi_1}$ and satisfies 
$\Phi_{\chi}^{\rm mg}  = \Phi_{\chi}^{+} + \Phi_{\chi}^{-}$. 
\end{thm}
\begin{proof}
From the equation (\ref{eqn:C2_PDE2_C2}) with $l=q=0$, 
the function $\hat{\varphi}_{0,0}$ satisfies the 
following differential equation 
\begin{align}
&\label{eqn:C2_PDE3_C2}
\{(\partial_1-2\nu_1)(\partial_1-2\nu_2-d_1+d_2)
-(4\pi y_1)^2\}\hat{\varphi}_{0,0}=0.
\end{align}
This implies that 
$f(z)=z^{-\nu_1-\nu_2-\frac{d_1-d_2}{2}}
\hat{\varphi}_{0,0}((4\pi)^{-1}z)$ satisfies 
the Bessel differential equation (\ref{eqn:Fn_bessel_DE}) 
with $r=\nu_1-\nu_2-\tfrac{d_1-d_2}{2}$.  
Since the dimensions of the solution space 
and ${\cI}_{\Pi_\chi ,\psi_1}$ are both $2$, 
we know that smooth solutions of (\ref{eqn:C2_PDE3_C2}) 
correspond to elements of ${\cI}_{\Pi_\chi ,\psi_1}$. 
By means of the formulas in \S \ref{subsec:Fn_special1}, 
we obtain the explicit expression of $\hat{\varphi}_{0,0}$. 
Moreover, we can get 
the explicit expressions of $\hat{\varphi}_{0,q}$ 
($0\leq q\leq d_1-d_2$)  
from that of $\hat{\varphi}_{0,0}$ by 
(\ref{eqn:C2_PDE2_C2omote}) with $\Gamma (s+1)=s\Gamma (s)$. 
Hence, we obtain the assertion. 
\end{proof}

Now we want to describe the radial parts of 
moderate growth Whittaker functions for $\Pi_{\chi}$, 
explicitly. 
It is convenient to set 
\begin{align}
\label{eqn:C2_def_psi}
\begin{split}
\Psi (a_1, a_2)(y_1)
=&\,8y_1^{a_1+a_2}K_{a_1-a_2}(4\pi y_1)\\
=&\,\frac{1}{2\pi \sI}
\int_{s}
\Gamma_{\bC}(s+a_1)\Gamma_{\bC}(s+a_2)y_1^{-2s}ds
\end{split}
\end{align}
for $y_1\in \bR_+$ and $a_1,a_2\in \bC$. 
Here the path of integration $\int_{s}$ is the vertical line 
from $\mathrm{Re}(s)-\sI \infty$ to $\mathrm{Re}(s)+\sI \infty$ 
with the sufficiently large real part to keep the poles of the integrand 
on its left. 
Until the end of this chapter, 
we assume that $\Phi$ is the homomorphism 
$\Phi_\chi^{\rm mg}$ in Theorem \ref{thm:C2_ps_Whittaker}. 
Then we have 
\[
\hat{\varphi}_{0,0}(y_1)=\Psi \bigl(\nu_1,\, 
\nu_2+\tfrac{d_1-d_2}{2}\bigr)(y_1). 
\]
The functions 
$\hat{\varphi}_{l,0}$ ($l\in \bZ_{\geq 0}$) 
can be uniquely determined from $\hat{\varphi}_{0,0}$ 
by (\ref{eqn:C2_PDE2_DSE}), and it is easy to see that  
\begin{align}
\label{eqn:C2_varphi_l0}
\begin{split}
\hat{\varphi}_{l,0}
=&\,\binom{d_1-d_2+2l}{l}
\frac{\Gamma _\bC \bigl(\nu_1-\nu_2+1+\tfrac{d_1-d_2}{2}\bigr)}
{\Gamma_\bC \bigl(\nu_1-\nu_2+1+\tfrac{d_1-d_2}{2}+l\bigr)}\\
&\times \Psi \bigl(\nu_1+\tfrac{l}{2},\,
\nu_2+\tfrac{d_1-d_2+l}{2}\bigr).
\end{split}
\end{align}
For every $l\in \bZ_{\geq 0}$, 
the functions $\hat{\varphi}_{l,q}$ ($0\leq q\leq d_1-d_2+2l$) 
can be uniquely determined from $\hat{\varphi}_{l,0}$ 
by (\ref{eqn:C2_PDE2_C2omote}). 
In order to give the explicit description of $\hat{\varphi}_{l,q}$, 
we prepare the following lemma. 

\begin{lem}
\label{lem:C2_reccurence}
For $a_1,a_2,b_1,b_2\in \bC$ and $y_1\in \bR_{+}$, it holds that 
\begin{align*}
&(8\pi y_1)^{-1}\{(\partial_1-2b_1)(\partial_1-2b_2)-(4\pi y_1)^2\}
\Psi (a_1, a_2)(y_1)\\
&=-(a_1+a_2-b_1-b_2)\Psi (a_1+1/2, a_2-1/2)(y_1)\\
&\phantom{=.}
+(2\pi )^{-1}(a_1-b_1)(a_1-b_2)\Psi (a_1-1/2, a_2-1/2)(y_1).
\end{align*}
\end{lem}
\begin{proof}
We obtain the assertion 
by direct computation using the functional equation 
$\Gamma_\bC (s+1)=(2\pi)^{-1}s\Gamma_\bC (s)$ of the Gamma function. 
\end{proof}

By (\ref{eqn:C2_varphi_l0}), (\ref{eqn:C2_PDE2_C2omote}) and 
Lemma \ref{lem:C2_reccurence}, 
we obtain the formula 
\begin{align*}
\hat{\varphi}_{l,q}
=&\sum_{i=0}^{\min \{q,l\}}
\binom{q}{i}\binom{d_1-d_2+2l-i}{l-i}
\dfrac{(-1)^i\,\Gamma _\bC \bigl(\nu_1-\nu_2+1+\tfrac{d_1-d_2}{2}\bigr)}
{\Gamma_\bC \bigl(\nu_1-\nu_2+1+\tfrac{d_1-d_2}{2}+l-i\bigr)}\\
&\phantom{====}
\times \Psi \bigl(\nu_1+\tfrac{q+l}{2}-i,\,
\nu_2+\tfrac{d_1-d_2-q+l}{2}\bigr)
\end{align*}
for $0\leq q \leq d_1-d_2+2l$, 
recursively. 
Hence, we have the following.

\begin{thm}
\label{thm:C2_ps_Whittaker2}
Let $\chi = \chi_{(\nu_1,d_1)}\boxtimes \chi_{(\nu_2,d_2)}$ with 
$\nu_1,\nu_2\in \bC$ and $(d_1,d_2)\in \Lambda_2$ such that 
$2\nu_1-2\nu_2+d_1-d_2+2\not\in 2\bZ_{\leq 0}$. 
Let $\Phi_{\chi}^{\rm mg}\in {\cI}_{\Pi_\chi ,\psi_1}^{\rm mg}$ 
be the homomorphism in Theorem \ref{thm:C2_ps_Whittaker}. 
For $l\in \bZ_{\geq 0}$, $0\leq q \leq d_1-d_2+2l$ 
and $y=\diag (y_1y_2,y_2)\in A $, 
it holds that 
\begin{align*}
&\Phi_{\chi}^{\rm mg} ({\zeta}_{(\chi ;l,q)})(y) 
=(\sI )^{-q-l}y_1y_2^{2\nu_1+2\nu_2}\\
&\phantom{==}
\times \sum_{i=0}^{\min \{q,l\}}
\binom{q}{i}\binom{d_1-d_2+2l-i}{l-i}
\dfrac{(-1)^i\,
\Gamma _\bC \bigl(\nu_1-\nu_2+1+\tfrac{d_1-d_2}{2}\bigr)}
{\Gamma_\bC \bigl(\nu_1-\nu_2+1+\tfrac{d_1-d_2}{2}+l-i\bigr)}\\
&\phantom{==}
\times 
\frac{1}{2\pi \sI}
\int_{s}
\Gamma_{\bC}\bigl(s+\nu_1+\tfrac{q+l}{2}-i\bigr)
\Gamma_{\bC}\bigl(s+\nu_2+\tfrac{d_1-d_2-q+l}{2}\bigr)
y_1^{-2s}ds.
\end{align*}
Here the path of integration $\int_{s}$ is the vertical line 
from $\mathrm{Re}(s)-\sI \infty$ to $\mathrm{Re}(s)+\sI \infty$ 
with the sufficiently large real part to keep the poles of the integrand 
on its left.  
\end{thm}

\begin{rem}
By the argument in \S \ref{subsec:C2_ps}, we note that, 
for $\chi = \chi_{(\nu_1,d_1)}\boxtimes \chi_{(\nu_2,d_2)}$ 
with $\nu_1,\nu_2\in \bC$ and $ d_1,d_2 \in \bZ$ 
$(d_1\geq d_2)$, 
the condition $2\nu_1-2\nu_2+d_1-d_2+2\not\in 2\bZ_{\leq 0}$ 
holds if $\Pi_\chi$ is irreducible. 
\end{rem}

\chapter{Whittaker functions on $GL(3,\bC)$}
\label{sec:C3_whittaker}

Throughout this chapter, we set $n=3$ and $F=\bC $.

\section{Representations of $U(3)$}
\label{subsec:C3_rep_K}

In this section, we discuss 
the representation theory of ${K}=U(3)$.

Let $\Lambda_3 =\{\mu =(\mu_1,\mu_2,\mu_3)
\in \bZ^{3}\mid \mu_1\geq \mu_2\geq \mu_3\}$. 
For $\mu =(\mu_1,\mu_2,\mu_3)\in \Lambda_3$, 
let $\mathcal{P}_\mu =\mathcal{P}_\mu^{(3)}$ be 
the $\bC$-vector space of 
polynomials of six variables 
${z}_1$, ${z}_2$, ${z}_3$, $\tilde{z}_{1}$, $\tilde{z}_{2}$, $\tilde{z}_{3}$, 
which are 
degree $(\mu_1-\mu_2)$ homogeneous with respect to 
three variables ${z}_1,{z}_2,{z}_3$ 
and are degree $(\mu_2-\mu_3)$ homogeneous 
with respect to three variables $\tilde{z}_{1},\tilde{z}_{2},\tilde{z}_{3}$. 
We define the action $T_\mu$ of $K$ on $\mathcal{P}_\mu$ by 
\begin{align*}
&(T_{\mu}(k)p)({z}_1,{z}_2,{z}_3,
\tilde{z}_{1},\tilde{z}_{2},\tilde{z}_{3})\\
&=(\det k)^{\mu_2}p(({z}_1,{z}_2,{z}_3)\cdot k,
(\tilde{z}_{1},\tilde{z}_{2},\tilde{z}_{3})\cdot \overline{k})&
&(k\in K,\ p\in \mathcal{P}_\mu ).
\end{align*}
Here $({z}_1,{z}_2,{z}_3)\cdot k$, 
$(\tilde{z}_{1},\tilde{z}_{2},\tilde{z}_{3})\cdot \overline{k}$ 
are the ordinal products of matrices, 
and $\overline{k}$ is the complex conjugate of $k$. 
Since ${z}_1\tilde{z}_{1}+{z}_2\tilde{z}_{2}+{z}_3\tilde{z}_{3}$ is 
${K}$-invariant, 
we can define the quotient representation $\tau_\mu =\tau_\mu^{(3)}$ 
of $T_\mu$ on the space 
\[
V_\mu=V_\mu^{(3)}=\mathcal{P}_\mu/
(({z}_1\tilde{z}_{1}+{z}_2\tilde{z}_{2}+{z}_3\tilde{z}_{3})
\mathcal{P}_{\mu -(1,0,-1)})
\] 
for $\mu \in \Lambda_3$. 
Here we put $\mathcal{P}_{\mu -(1,0,-1)}=\{0\}$ 
if $\mu -(1,0,-1)\not\in \Lambda_3$. 
Then the representation 
$(\tau_{\mu },V_{\mu })$ 
is irreducible. The set of equivalence classes of 
irreducible representations of ${K}$ is 
exhausted by $\{\tau_{\mu }\mid \mu \in \Lambda_3\}$.

For $\mu=(\mu_1,\mu_2,\mu_3)\in \Lambda_3$, we set 
\[
S_\mu =
\left\{l=(l_1,l_2,l_3,\tilde{l}_{1},\tilde{l}_{2},\tilde{l}_{3})
\in (\bZ_{\geq 0})^6
\ \left| \ 
\begin{array}{l}
l_1+l_2+l_3=\mu_1-\mu_2,\\[1mm] 
\tilde{l}_{1}+\tilde{l}_{2}+\tilde{l}_{3}=\mu_2-\mu_3
\end{array}
\right\}\right. . 
\]
For $l=(l_1,l_2,l_3,\tilde{l}_{1},\tilde{l}_{2},\tilde{l}_{3})
\in S_{\mu}$, let $u_{l}$ be the image of 
${z}_1^{\,l_1}{z}_2^{\,l_2}{z}_3^{\,l_3}{\tilde{z}_1}^{\,\tilde{l}_{1}}
{\tilde{z}_2}^{\,\tilde{l}_{2}}{\tilde{z}_3}^{\,\tilde{l}_{3}}$ 
under the natural surjection $\mathcal{P}_\mu \to V_{\mu}$. 
It is convenient to set $u_l =0$ if $l\not\in (\bZ_{\geq 0})^6$.

We note that $\{u_{l}\mid l\in S_\mu\}$ is a system 
of generators of $V_\mu$. 
If $\mu_1>\mu_2>\mu_3$, the relation 
\begin{align*}
&u_{l+\me_1+\tilde{\me}_1}+u_{l+\me_2+\tilde{\me}_2}+u_{l+\me_3+\tilde{\me}_3}=0&
&(l\in S_{\mu -(1,0,-1)})
\end{align*}
holds with 
\begin{align*}
&\me_1=(1,0,0,0,0,0),&
&\me_2=(0,1,0,0,0,0),&
&\me_3=(0,0,1,0,0,0),\\
&\tilde{\me}_{1}=(0,0,0,1,0,0),&
&\tilde{\me}_{2}=(0,0,0,0,1,0),&
&\tilde{\me}_{3}=(0,0,0,0,0,1),
\end{align*}
and otherwise, 
$\{u_{l}\mid l\in S_\mu\}$ is linearly independent. 
The action of 
\begin{align*}
K\cap M=\{\diag (m_1,m_2,m_3)\mid m_1,m_2,m_3\in U(1)\}
\end{align*}
on $\{u_{l}\mid l\in S_\mu\}$ is given by 
\begin{align}
\label{eqn:C3_Mact}
&\tau_{\mu }(m)u_{l}
=\left(\prod_{i=1}^3
m_i^{l_i-\tilde{l}_i+\mu_2}\right)\!u_{l}&
&(m=\diag (m_1,m_2,m_3)\in K\cap M).
\end{align}
We denote the differential of $\tau_{\mu }$ 
again by $\tau_{\mu }$. Then we have 
\begin{align}
\label{eqn:C3_gkact}
&\tau_{\mu }(E^{\gk}_{i,i})u_l 
=(l_i-\tilde{l}_{i}+\mu_2)u_l ,&
&\tau_{\mu }(E^{\gk}_{i,j})u_l=
l_ju_{l-\me_j+\me_i} -\tilde{l}_{i}u_{l-\tilde{\me}_{i}+\tilde{\me}_{j}}
\end{align}
for $1\leq i\neq j\leq 3$ and 
$l=(l_1,l_2,l_3,\tilde{l}_{1},\tilde{l}_{2},\tilde{l}_{3})\in S_\mu$.

We regard $\gp_\bC$ as a $K$-module via the adjoint action $\Ad$. 
For later use, we prepare the following lemma. 
\begin{lem}
\label{lem:C3_tensor1}
\noindent (1) Let $\mu =(\mu_1,\mu_2,\mu_3)\in \Lambda_3$ 
with $\mu_1>\mu_2$. Then it holds that 
\begin{align}
\label{eqn:C3_tensor100_1}
\Hom_K(V_{(1,0,0)}\otimes_{\bC}V_{\mu -(1,0,0)}, 
V_{\mu })
=\bC \,\mathrm{B}_{\mu }, 
\end{align}
where $\mathrm{B}_{\mu }\colon 
V_{(1,0,0)}\otimes_{\bC}V_{\mu -(1,0,0)}\to 
V_{\mu}$ is a surjective $\bC$-linear map characterized by  
$\mathrm{B}_{\mu }(u_{\me_i}\otimes u_l)=u_{l+\me_i}$\ 
$(1\leq i\leq 3,\ l\in S_{\mu -(1,0,0)})$. 
For $\mu'=(\mu_1',\mu_2',\mu_3')\in \Lambda_3$, 
it holds that 
\begin{align}
\label{eqn:C3_tensor100_2}
\Hom_K(V_{(1,0,0)}\otimes_{\bC}V_{\mu -(1,0,0)}, 
V_{\mu'})
=\{0\}
\end{align}
if $\mu_1'+\mu_2'+\mu_3'\neq 
\mu_1 +\mu_2+\mu_3$ or $\mu_1'>\mu_1 $ 
or $\mu_3'<\mu_3$. 
\\[2pt]
\noindent (2) We define a $\bC$-linear map 
$\mathrm{I}^{\gp}_{(1,0,0)}\colon 
V_{(1,0,0)}\to \gp_{\bC}\otimes_{\bC}V_{(1,0,0)}$ by 
\begin{align*}
&\mathrm{I}^{\gp}_{(1,0,0)}(u_{\me_i})
=\sum_{j=1}^3E_{i,j}^{\gp}\otimes u_{\me_j}&
&(1\leq i\leq 3).
\end{align*}
Then $\mathrm{I}^{\gp}_{(1,0,0)}$ is a $K$-homomorphism. 
\end{lem}
\begin{proof}
Let $\mu =(\mu_1,\mu_2,\mu_3),\,
\mu'=(\mu_1',\mu_2',\mu_3)
\in \Lambda_3$ with $\mu_1>\mu_2$. 
Since 
\begin{align*}
&\tau_{\mu '}(E_{1,1}^{\gk}+E_{2,2}^{\gk}+E_{3,3}^{\gk})
=\mu_1'+\mu_2'+\mu_3',\\
&(\tau_{(1,0,0)}\otimes \tau_{\mu -(1,0,0)})
(E_{1,1}^{\gk}+E_{2,2}^{\gk}+E_{3,3}^{\gk})
=\mu_1+\mu_2+\mu_3,
\end{align*}
we know that (\ref{eqn:C3_tensor100_2}) holds 
if $\mu_1'+\mu_2'+\mu_3'\neq 
\mu_1+\mu_2+\mu_3$. 
By (\ref{eqn:C3_gkact}), 
we know that there is a unique element $v$ of $V_{\mu'}$ such that 
$\tau_{\mu'}(E_{i,i}^{\gk})v=\mu_i'v$ $(1\leq i\leq 3)$, 
up to scalar multiple, and 
there is no element $v$ of 
$V_{(1,0,0)}\otimes_{\bC}V_{\mu -(1,0,0)}$ such that 
\begin{align*}
&(\tau_{(1,0,0)}\otimes \tau_{\mu -(1,0,0)})(E_{i,i}^{\gk})v
=\mu_i'v&&(1\leq i\leq 3)
\end{align*} 
if $\mu_1'>\mu_1 $ or $\mu_3'<\mu_3$. 
Therefore, (\ref{eqn:C3_tensor100_2}) holds 
if $\mu_1'>\mu_1 $ or $\mu_3'<\mu_3$. 
Moreover, by (\ref{eqn:C3_gkact}), there is a unique element $v$ of 
$V_{(1,0,0)}\otimes_{\bC}V_{\mu -(1,0,0)}$ such that 
\begin{align*}
&(\tau_{(1,0,0)}\otimes \tau_{\mu -(1,0,0)})(E_{i,i}^{\gk})v
=\mu_iv&&(1\leq i\leq 3),
\end{align*}
up to scalar multiple. 
Hence we note that 
\[
\dim_\bC  \Hom_K(V_{(1,0,0)}\otimes_{\bC}V_{\mu -(1,0,0)}, 
V_{\mu })\leq 1. 
\]
Since $\mathrm{B}_{\mu}$ 
is a $\bC$-linear map induced from 
\begin{align*}
\mathcal{P}_{(1,0,0)}\otimes_{\bC}
\mathcal{P}_{\mu -(1,0,0)}\ni 
p_1\otimes 
p_2\mapsto 
p_1p_2
\in \mathcal{P}_{\mu }, 
\end{align*}
it is obvious that 
$\mathrm{B}_{\mu }$ is a surjective $K$-homomorphism. 
Hence, we obtain (\ref{eqn:C3_tensor100_1}), 
and complete a proof of the statement (1). 

By (\ref{eqn:C3_gkact}), we have 
\begin{align*}
&(\adj \otimes \tau_{(1,0,0)})(E_{a,b}^{\gk})
\mathrm{I}^{\gp}_{(1,0,0)}(u_{\me_i})\\
&=\sum_{j=1}^3\bigl\{(\adj (E_{a,b}^{\gk})E_{i,j}^{\gp})\otimes u_{\me_j}
+E_{i,j}^{\gp}\otimes (\tau_{(1,0,0)}(E_{a,b}^{\gk})u_{\me_j})\bigr\}\\
&=\sum_{j=1}^3\bigl\{(\delta_{b,i}E_{a,j}^{\gp}-\delta_{a,j}E_{i,b}^{\gp})
\otimes u_{\me_j}
+E_{i,j}^{\gp}\otimes (\delta_{b,j}u_{\me_a})\bigr\}\\
&=\delta_{b,i}\sum_{j=1}^3E_{a,j}^{\gp}\otimes u_{\me_j}
-E_{i,b}^{\gp}\otimes u_{\me_a}
+E_{i,b}^{\gp}\otimes u_{\me_a}
=\delta_{b,i}\mathrm{I}^{\gp}_{(1,0,0)}(u_{\me_a})
\end{align*}
for $1\leq a,b,i\leq 3$. 
Comparing this equality with (\ref{eqn:C3_gkact}), 
we know that $\mathrm{I}^{\gp}_{(1,0,0)}$ is a $\gk_\bC $-homomorphism. 
Since $K=U(3)$ is connected, 
we know that $\mathrm{I}^{\gp}_{(1,0,0)}$ is an $K$-homomorphism, 
and obtain the statement (2). 
\end{proof}

\begin{lem}
\label{lem:C3_tensor2}
\noindent (1) Let $\mu =(\mu_1,\mu_2,\mu_3)\in \Lambda_3$ 
with $\mu_2>\mu_3$. Then it holds that 
\begin{align}
\label{eqn:C3_tensor001_1}
\Hom_K(V_{(0,0,-1)}\otimes_{\bC}V_{\mu +(0,0,1)}, 
V_{\mu })
=\bC \,\widetilde{\mathrm{B}}_{\mu }, 
\end{align}
where $\widetilde{\mathrm{B}}_{\mu }\colon 
V_{(0,0,-1)}\otimes_{\bC}V_{\mu +(0,0,1)}\to 
V_{\mu}$ is a surjective $\bC$-linear map characterized by  
$\widetilde{\mathrm{B}}_{\mu }(u_{\tilde{\me}_i}\otimes u_l)
=u_{l+\tilde{\me}_i}$\ 
$(1\leq i\leq 3,\ l\in S_{\mu +(0,0,1)})$. 
For $\mu'=(\mu_1',\mu_2',\mu_3')\in \Lambda_3$, 
it holds that 
\begin{align}
\label{eqn:C3_tensor001_2}
\Hom_K(V_{(0,0,-1)}\otimes_{\bC}V_{\mu +(0,0,1)}, 
V_{\mu'})
=\{0\}
\end{align}
if $\mu_1'+\mu_2'+\mu_3'\neq 
\mu_1 +\mu_2+\mu_3$ or 
$\mu_1'>\mu_1 $ or $\mu_3'<\mu_3$. \\[2pt]
\noindent (2) We define a $\bC$-linear map 
$\mathrm{I}^{\gp}_{(0,0,-1)}\colon V_{(0,0,-1)}\to 
\gp_{\bC}\otimes_{\bC}V_{(0,0,-1)}$ by 
\begin{align*}
&\mathrm{I}^{\gp}_{(0,0,-1)}(u_{\tilde{\me}_i})
=\sum_{j=1}^3E_{j,i}^{\gp}\otimes u_{\tilde{\me}_j}&
&(1\leq i\leq 3).
\end{align*}
Then $\mathrm{I}^{\gp}_{(0,0,-1)}$ is a $K$-homomorphism. 
\end{lem}
\begin{proof}
Similar to the proof of Lemma \ref{lem:C3_tensor1}, 
we obtain the assertion.  
\end{proof}

\section{Principal series representations}
\label{subsec:C3_ps}

Let $\chi =\chi_{(\nu_1,d_1)}\boxtimes \chi_{(\nu_2,d_2)}
\boxtimes \chi_{(\nu_3,d_3)}$ 
with $(\nu_1,\nu_2,\nu_3)\in \bC^3$ and 
$(d_1,d_2,d_3)\in \Lambda_3$. 
In this section, 
we consider the action of $\g_\bC$ at the minimal $K$-type of $\Pi_\chi$. 
We define a $\bC$-linear map $\eta_{\chi }\colon V_{(d_1,d_2,d_3)}
\to \bC_\chi $ by 
\begin{align*}
&\eta_{\chi}(u_l)
=\left\{\begin{array}{ll}
1&\text{ if }\ l=(d_1-d_2,0,0,0,0,d_2-d_3),\\[1mm]
0&\text{ otherwise}
\end{array}\right.&
&(l\in S_{(d_1,d_2,d_3)}). 
\end{align*}
Let $\mu =(\mu_1,\mu_2,\mu_3)\in \Lambda_3$. 
Because of (\ref{eqn:C3_Mact}), we have 
\begin{align*}
&\Hom_{{K}\cap M}(V_{(d_1,d_2,d_3)},\bC_{\chi})=\bC\, \eta_{\chi},&
&\Hom_{{K}\cap M}(V_{\mu },\bC_{\chi})=\{0\}
\end{align*}
if $d_1+d_2+d_3\neq \mu_1 +\mu_2+\mu_3$ or 
$d_1>\mu_1$ or $d_3<\mu_3$. We set 
\begin{align*}
&\hat{\eta}_{\chi }(v)(k)=\eta_{\chi }
(\tau_{(d_1,d_2,d_3)}(k)v)&
&(v\in V_{(d_1,d_2,d_3)},\ k\in {K}). 
\end{align*}
Then, by the Frobenius reciprocity law, we have 
\begin{align}
\label{eqn:C3_minKtype}
&\Hom_{{K}}(V_{(d_1,d_2,d_3)},H({\chi})_K)=\bC\, \hat{\eta}_{\chi},&
&\Hom_{{K}}(V_{\mu },H({\chi})_K)=\{0\}
\end{align}
if $d_1+d_2+d_3\neq \mu_1 +\mu_2+\mu_3$ or 
$d_1>\mu_1$ or $d_3<\mu_3$. 
We call $\tau_{(d_1,d_2,d_3)}$ 
the minimal $K$-type of $\Pi_\chi$. 
\begin{prop}
\label{prop:C3_DSE_ps}
Retain the notation. 
If $d_1>d_2$, then it holds that 
\begin{align}
\label{eqn:C3_DSE_ps1}
&\sum_{j=1}^3\Pi_\chi (E_{i,j}^{\gp})
\hat{\eta}_{\chi }(u_{l+\me_j})
=2\nu_1\hat{\eta}_{\chi }(u_{l+\me_i})&
&(1\leq i\leq 3,\ l\in S_{(d_1-1,d_2,d_3)}).
\end{align}
If $d_2>d_3$, then it holds that 
\begin{align}
\label{eqn:C3_DSE_ps2}
&\sum_{j=1}^3\Pi_\chi (E_{j,i}^{\gp})
\hat{\eta}_{\chi }(u_{l+\tilde{\me}_j})
=2\nu_3\hat{\eta}_{\chi }(u_{l+\tilde{\me}_i})&
&(1\leq i\leq 3,\ l\in S_{(d_1,d_2,d_3+1)}).
\end{align}
\end{prop}
\begin{proof}
Let $\mathrm{P}_\chi 
\colon \gp_{\bC}\otimes_{\bC}H(\chi)_K\to H(\chi)_K$ 
be a natural $K$-homomorphism defined by 
$X\otimes f\mapsto \Pi_{\chi}(X)f$. 
Assume $d_1>d_2$. 
By (\ref{eqn:C3_minKtype}) and 
Lemma \ref{lem:C3_tensor1} (1), we have 
\[
\Hom_K(V_{(1,0,0)}\otimes_{\bC}V_{(d_1-1,d_2,d_3)}, 
H(\chi)_K)
=\bC\,\hat{\eta}_{\chi }\circ 
\mathrm{B}_{(d_1,d_2,d_3)}.
\]
Since this space contains the composite 
\[
\mathrm{P}_\chi \circ 
(\id_{\gp_{\bC}}\otimes (\hat{\eta}_{\chi}
\circ \mathrm{B}_{(d_1,d_2,d_3)}))\circ 
(\mathrm{I}^\gp_{(1,0,0)}\otimes \id_{V_{(d_1-1,d_2,d_3)}}),
\]
there is a constant $c\in \bC $ such that 
\begin{align*}
c\,\hat{\eta}_{\chi }\circ \mathrm{B}_{(d_1,d_2,d_3)}=
\mathrm{P}_\chi \circ 
(\id_{\gp_{\bC}}\otimes (\hat{\eta}_{\chi }
\circ \mathrm{B}_{(d_1,d_2,d_3)}))\circ 
(\mathrm{I}^\gp_{(1,0,0)}\otimes \id_{V_{(d_1-1,d_2,d_3)}}).
\end{align*}
For $1\leq i\leq 3$ and $l\in S_{(d_1-1,d_2,d_3)}$, 
considering the image of $u_{\me_i}\otimes u_l$ 
under the both sides of this equality, we have 
\begin{align*}
c\,\hat{\eta}_{\chi }(u_{l+\me_i})
=\sum_{j=1}^3\Pi_\chi (E_{i,j}^{\gp})
\hat{\eta}_{\chi }(u_{l+\me_j}).
\end{align*}
For $l=(d_1-d_2-1,0,0,0,0,d_2-d_3)$, 
we have 
\begin{align*}
c&=c\,\hat{\eta}_{\chi }(u_{l+\me_1})(1_3)
=\sum_{j=1}^3(\Pi_\chi (E_{1,j}^{\gp})\hat{\eta}_{\chi }(u_{l+\me_j}))(1_3)\\
&=
(\Pi_\chi (E_{1,1}^{\gp})\hat{\eta}_{\chi}(u_{l+\me_1}))(1_3)\\
&\hphantom{=,}
+2(\Pi_\chi (E_{1,2}^{\g})\hat{\eta}_{\chi}(u_{l+\me_2}))(1_3)
-\hat{\eta}_{\chi}
(\tau_{(d_1,d_2,d_3)}(E_{1,2}^{\gk})u_{l+\me_2})(1_3)\\
&\hphantom{=,}
+2(\Pi_\chi (E_{1,3}^{\g})\hat{\eta}_{\chi}(u_{l+\me_3}))(1_3)
-\hat{\eta}_{\chi}
(\tau_{(d_1,d_2,d_3)}(E_{1,3}^{\gk})u_{l+\me_3})(1_3)\\
&=(2\nu_1+2)+0-1+0-1=2\nu_1,
\end{align*}
and obtain (\ref{eqn:C3_DSE_ps1}). 
By the similar argument using Lemma \ref{lem:C3_tensor2} 
instead of Lemma \ref{lem:C3_tensor1}, 
we obtain (\ref{eqn:C3_DSE_ps2}). 
\end{proof}

\section{Principal series Whittaker functions}
\label{subsec:C3_ps_whittaker}

We use the notation in \S \ref{subsec:C3_ps}. 
In this section, 
we will give the explicit formulas of Whittaker functions 
for an irreducible principal series representation $\Pi_\chi$ of $G$ 
at the minimal $K$-type $\tau_{(d_1,d_2,d_3)}$. 
Let $\varphi \colon V_{(d_1,d_2,d_3)} \to 
{\mathrm{Wh}}(\Pi_\chi ,\psi_1)$ be a $K$-homomorphism. 
We note that $\varphi$ is characterized by 
\begin{align*}
&\varphi (u_l)(y)&
&(\,l\in S_{(d_1,d_2,d_3)}\,)
\end{align*}
with $y=\diag (y_1y_2y_3,y_2y_3,y_3)\in A$. 
We set $\partial_i=y_i\dfrac{\partial}{\partial y_i}$ 
for $1\leq i\leq 3$. 
 
\begin{lem}
\label{lem:C3_ZgDSE1_original}
Retain the notation. 

\noindent (1) For $l\in S_{(d_1,d_2,d_3)}$, it holds that 
\begin{align}
&\label{eqn:C3_PDE1_C1}
(\partial_3-2\nu_1-2\nu_2-2\nu_3)\varphi (u_l)(y)=0,\\[3pt] 
\begin{split}
&
\{(\partial_1-d_1-2)(-\partial_1+\partial_2-d_2)
+(\partial_2-d_1-d_2-2)(-\partial_2+\partial_3-d_3+2)\\
&
+(4\pi y_1)^2+(4\pi y_2)^2
-(2\nu_1-d_1)(2\nu_2-d_2)-(2\nu_1-d_1)(2\nu_3-d_3)\\
&\label{eqn:C3_PDE1_C2}
-(2\nu_2-d_2)(2\nu_3-d_3)\}\varphi (u_{(d_1-d_2,0,0,0,0,d_2-d_3)})(y)=0,
\end{split}\\[3pt]
\begin{split}
&
\{(\partial_1-d_1-2)(-\partial_1+\partial_2-d_2)
(-\partial_2+\partial_3-d_3+2)\\
&
+(4\pi y_1)^2(-\partial_2+\partial_3-d_3+2)
+(4\pi y_2)^2(\partial_1-d_1-2)\\
&\label{eqn:C3_PDE1_C3}
-(2\nu_1-d_1)(2\nu_2-d_2)(2\nu_3-d_3)\}
\varphi (u_{(d_1-d_2,0,0,0,0,d_2-d_3)})(y)=0.
\end{split}
\end{align}

\noindent (2) If $d_1>d_2$, 
for $l=(l_1,l_2,l_3,\tilde{l}_1,\tilde{l}_2,\tilde{l}_3)
\in S_{(d_1-1,d_2,d_3)}$, it holds that 
\begin{align}
&\label{eqn:C3_PDE1_DSE1}
(\partial_1-\tilde{l}_1-l_2-l_3-2-2\nu_1)\varphi (u_{l+\me_1})(y)
+4\pi \sI y_1\varphi (u_{l+\me_2})(y)=0,\\[2pt]
\begin{split}
&
(-\partial_1+\partial_2+l_1-l_3-2\nu_1)\varphi (u_{l+\me_2})(y)\\
&
+4\pi \sI y_2\varphi (u_{l+\me_3})(y)+4\pi \sI y_1\varphi (u_{l+\me_1})(y)\\
&\label{eqn:C3_PDE1_DSE2}
-\tilde{l}_2\{\varphi (u_{l+\me_1-\tilde{\me}_2+\tilde{\me}_1})(y)
-\varphi (u_{l+\me_3-\tilde{\me}_2+\tilde{\me}_3})(y)\}=0,
\end{split}\\[2pt]
\begin{split}
&(-\partial_2+\partial_3+l_1+l_2+\tilde{l}_3+2-2\nu_1)
\varphi (u_{l+\me_3})(y)\\
&\nonumber 
+4\pi \sI y_2\varphi (u_{l+\me_2})(y)=0.
\end{split}
\end{align}
If $d_2>d_3$, 
for $l=(l_1,l_2,l_3,\tilde{l}_1,\tilde{l}_2,\tilde{l}_3)
\in S_{(d_1,d_2,d_3+1)}$, it holds that 
\begin{align}
&\nonumber 
(\partial_1-l_1-\tilde{l}_2-\tilde{l}_3-2-2\nu_3)
\varphi (u_{l+\tilde{\me}_1})(y)
+4\pi \sI y_1\varphi (u_{l+\tilde{\me}_2})(y)=0,\\
\begin{split}
&(-\partial_1+\partial_2+\tilde{l}_1-\tilde{l}_3-2\nu_3)
\varphi (u_{l+\tilde{\me}_2})(y)\\
&+4\pi \sI y_1\varphi (u_{l+\tilde{\me}_1})(y)
+4\pi \sI y_2\varphi (u_{l+\tilde{\me}_3})(y)\\
&\label{eqn:C3_PDE1_DSE5}
-l_2\{\varphi (u_{l-\me_2+\me_1+\tilde{\me}_1})(y)
-\varphi (u_{l-\me_2+\me_3+\tilde{\me}_3})(y)\}=0,
\end{split}\\
\begin{split}
&(-\partial_2+\partial_3+\tilde{l}_1+\tilde{l}_2+l_3+2-2\nu_3)
\varphi (u_{l+\tilde{\me}_3})(y)\\
&\label{eqn:C3_PDE1_DSE6}
+4\pi \sI y_2\varphi (u_{l+\tilde{\me}_2})(y)=0.
\end{split}
\end{align}
\end{lem}
\begin{proof}
We note that there is a homomorphism $\Phi \in {\cI}_{\Pi_\chi ,\psi_1}$ 
such that 
$\varphi =\Phi \circ \hat{\eta}_\chi$. 
Hence, by Proposition \ref{prop:Cn_Ch_eigenvalue}, we have 
\begin{align}
\label{eqn:C3_pf_PDE001}
&(R(\widetilde{\cC}_1^{{\g}})\varphi (u_l))(y)
-\bigl(\nu_1+\nu_2+\nu_3-\tfrac{d_1+d_2+d_3}{2}\bigr)\varphi (u_l)(y)=0,\\
\label{eqn:C3_pf_PDE002}
&(R(\widetilde{\cC}_2^{{\g}})\varphi (u_l))(y)-\left\{\sum_{1\leq i<j\leq 3}
\bigl(\nu_i-\tfrac{d_i}{2}\bigr)
\bigl(\nu_j-\tfrac{d_j}{2}\bigr)\right\}\varphi (u_l)(y)=0,\\
\label{eqn:C3_pf_PDE003}
&(R(\widetilde{\cC}_3^{{\g}})\varphi (u_l))(y)-
\bigl(\nu_1-\tfrac{d_1}{2}\bigr)
\bigl(\nu_2-\tfrac{d_2}{2}\bigr)
\bigl(\nu_3-\tfrac{d_3}{2}\bigr)\varphi (u_l)(y)=0
\end{align}
for $l\in S_{(d_1,d_2,d_3)}$. 
The explicit forms of $\widetilde{\cC}_1^{{\g}}$, $\widetilde{\cC}_2^{{\g}}$ 
and $\widetilde{\cC}_3^{{\g}}$ 
are given by 
\begin{align}
\label{eqn:C3_C1}
\widetilde{\cC}_1^{\g}
=&\widetilde{E}_{1,1}^{\g}+\widetilde{E}_{2,2}^{\g}
+\widetilde{E}_{3,3}^{\g},\\ 
\nonumber 
\widetilde{\cC}_2^{\g}
=&(\widetilde{E}_{1,1}^{\g}-1)\widetilde{E}_{2,2}^{\g}
+(\widetilde{E}_{1,1}^{\g}-1)(\widetilde{E}_{3,3}^{\g}+1)
+\widetilde{E}_{2,2}^{\g}(\widetilde{E}_{3,3}^{\g}+1)\\
&\nonumber 
-\widetilde{E}_{1,2}^{\g}\widetilde{E}_{2,1}^{\g}
-\widetilde{E}_{1,3}^{\g}\widetilde{E}_{3,1}^{\g}
-\widetilde{E}_{2,3}^{\g}\widetilde{E}_{3,2}^{\g},\\
\nonumber \widetilde{\cC}_3^{\g}=&
(\widetilde{E}_{1,1}^{\g}-1)\widetilde{E}_{2,2}^{\g}
(\widetilde{E}_{3,3}^{\g}+1)
+\widetilde{E}_{1,2}^{\g}\widetilde{E}_{2,3}^{\g}
\widetilde{E}_{3,1}^{\g}
+\widetilde{E}_{1,3}^{\g}\widetilde{E}_{2,1}^{\g}
\widetilde{E}_{3,2}^{\g}\\
&\nonumber 
-\widetilde{E}_{1,2}^{\g}\widetilde{E}_{2,1}^{\g}
(\widetilde{E}_{3,3}^{\g}+1)
-\widetilde{E}_{1,3}^{\g}\widetilde{E}_{2,2}^{\g}\widetilde{E}_{3,1}^{\g}
-(\widetilde{E}_{1,1}^{\g}-1)\widetilde{E}_{2,3}^{\g}
\widetilde{E}_{3,2}^{\g}.
\end{align}
By the equality $\widetilde{E}_{j,i}^{{\g}}=E_{i,j}^{{\g}}-E_{i,j}^{{\gk}}$ 
$(1\leq i,j\leq 3)$, we have 
\begin{align}
\begin{split}
\widetilde{\cC}_2^{\g}
\equiv \,
&(\widetilde{E}_{1,1}^{\g}-1)\widetilde{E}_{2,2}^{\g}
+(\widetilde{E}_{1,1}^{\g}-1)(\widetilde{E}_{3,3}^{\g}+1)
+\widetilde{E}_{2,2}^{\g}(\widetilde{E}_{3,3}^{\g}+1)\\
&\label{eqn:C3_C2_mod}
-\widetilde{E}_{1,2}^{\g}{E}_{1,2}^{\g}
-\widetilde{E}_{2,3}^{\g}{E}_{2,3}^{\g}
+\widetilde{E}_{1,2}^{\g}{E}_{1,2}^{\gk}
+\widetilde{E}_{2,3}^{\g}{E}_{2,3}^{\gk}\\
&\hspace{122pt}\mod E_{1,3}^{\g}U(\g_\bC)+\widetilde{E}_{1,3}^{\g}U(\g_\bC),
\end{split}\\
\begin{split}
\label{eqn:C3_C3_mod}
\widetilde{\cC}_3^{\g}
\equiv \,
&
(\widetilde{E}_{1,1}^{\g}-1)\widetilde{E}_{2,2}^{\g}
(\widetilde{E}_{3,3}^{\g}+1)
-\widetilde{E}_{1,2}^{\g}{E}_{1,2}^{\g}
(\widetilde{E}_{3,3}^{\g}+1)\\
&
-\widetilde{E}_{2,3}^{\g}
{E}_{2,3}^{\g}(\widetilde{E}_{1,1}^{\g}-1)
+\widetilde{E}_{1,2}^{\g}
(\widetilde{E}_{3,3}^{\g}+1){E}_{1,2}^{\gk}
-\widetilde{E}_{1,2}^{\g}\widetilde{E}_{2,3}^{\g}
\widetilde{E}_{1,3}^{\gk}\\
&+\widetilde{E}_{2,3}^{\g}
(\widetilde{E}_{1,1}^{\g}-1){E}_{2,3}^{\gk}
\hspace{30pt}\mod E_{1,3}^{\g}U(\g_\bC)+\widetilde{E}_{1,3}^{\g}U(\g_\bC).
\end{split}
\end{align}
Applying Lemma \ref{lem:Cn_g_act_Cpsi} and (\ref{eqn:C3_gkact}) 
to (\ref{eqn:C3_pf_PDE001}), 
(\ref{eqn:C3_pf_PDE002}) and (\ref{eqn:C3_pf_PDE003}) 
with the expressions (\ref{eqn:C3_C1}), (\ref{eqn:C3_C2_mod}) and 
(\ref{eqn:C3_C3_mod}), we obtain the statement (1).

By Proposition \ref{prop:C3_DSE_ps}, if $d_1>d_2$, it holds that 
\begin{align*}
&\sum_{j=1}^3(R(E_{i,j}^{\gp})\varphi (u_{l+\me_j}))(y)
-2\nu_1\varphi (u_{l+\me_i})(y)=0&
&(1\leq i\leq 3,\ l\in S_{(d_1-1,d_2,d_3)}),
\end{align*}
and if $d_2>d_3$, it holds that 
\begin{align*}
&\sum_{j=1}^3(R(E_{j,i}^{\gp})\varphi (u_{l+\tilde{\me}_j}))(y)
-2\nu_3\varphi (u_{l+\tilde{\me}_i})(y)=0&
&(1\leq i\leq 3,\ l\in S_{(d_1,d_2,d_3+1)}).
\end{align*}
Applying Lemma \ref{lem:Cn_g_act_Cpsi} and (\ref{eqn:C3_gkact}) 
to these equalities 
with 
\begin{align*}
&E_{i,j}^{{\gp}}=2E_{i,j}^{\g}-E_{i,j}^{{\gk}}
=2\widetilde{E}_{j,i}^{\g}+E_{i,j}^{{\gk}}&
&(1\leq i,j\leq 3), 
\end{align*}
we obtain the statement (2). 
\end{proof}

By the equation (\ref{eqn:C3_PDE1_C1}), 
for $l=(l_1,l_2,l_3,\tilde{l}_1,\tilde{l}_2,\tilde{l}_3)
\in S_{(d_1,d_2,d_3)}$, 
we can define a function $\hat{\varphi}_l$ on $(\bR_+)^2$ by  
\begin{align}
\label{eqn:C3_def_varphi_l}
\varphi (u_l)(y)=(\sI)^{l_1+\tilde{l}_1-l_3-\tilde{l}_3}
y_1^2y_2^2(y_2y_3)^{2\nu_1+2\nu_2+2\nu_3}\hat{\varphi}_l(y_1,y_2). 
\end{align}
with $y=\diag (y_1y_2y_3,y_2y_3,y_3)\in A$. 
It is convenient to set $\hat{\varphi}_l=0$ 
if $l\not\in (\bZ_{\geq 0})^6$. 

\begin{lem}
\label{lem:C3_ZgDSE2}
Retain the notation. 

\noindent (1) 
Assume $d_1>d_2$, and 
let $l=(l_1,l_2,l_3,\tilde{l}_1,\tilde{l}_2,\tilde{l}_3)
\in S_{(d_1-1,d_2,d_3)}$. 
It holds that 
\begin{align}
&\label{eqn:C3_PDE2_DSE1}
(-\partial_1+\tilde{l}_1+l_2+l_3+2\nu_1)\hat{\varphi}_{l+\me_1}
-4\pi y_1\hat{\varphi}_{l+\me_2}=0.
\end{align}
If $\tilde{l}_2=0$, it holds that 
\begin{align}
&\label{eqn:C3_PDE2_DSE2}
(-\partial_1+\partial_2+l_1-l_3+2\nu_2+2\nu_3)\hat{\varphi}_{l+\me_2}
-4\pi y_1\hat{\varphi}_{l+\me_1}
+4\pi y_2\hat{\varphi}_{l+\me_3}=0.
\end{align}
\noindent (2) 
Assume $d_2>d_3$, and 
let $l=(l_1,l_2,l_3,\tilde{l}_1,\tilde{l}_2,\tilde{l}_3)
\in S_{(d_1,d_2,d_3+1)}$. 
It holds that 
\begin{align}
&\label{eqn:C3_PDE2_DSE6}
(-\partial_2+\tilde{l}_1+\tilde{l}_2+l_3-2\nu_3)
\hat{\varphi}_{l+\tilde{\me}_3}
-4\pi y_2\hat{\varphi}_{l+\tilde{\me}_2}=0.
\end{align}
If ${l}_2=0$, it holds that 
\begin{align}
&\label{eqn:C3_PDE2_DSE5}
(-\partial_1+\partial_2+\tilde{l}_1-\tilde{l}_3+2\nu_1+2\nu_2)
\hat{\varphi}_{l+\tilde{\me}_2}
-4\pi y_1\hat{\varphi}_{l+\tilde{\me}_1}
+4\pi y_2\hat{\varphi}_{l+\tilde{\me}_3}=0.
\end{align}
\end{lem}
\begin{proof}
The equations (\ref{eqn:C3_PDE2_DSE1}), (\ref{eqn:C3_PDE2_DSE2}), 
(\ref{eqn:C3_PDE2_DSE6}) and (\ref{eqn:C3_PDE2_DSE5}) follow 
immediately from (\ref{eqn:C3_PDE1_DSE1}), (\ref{eqn:C3_PDE1_DSE2}), 
(\ref{eqn:C3_PDE1_DSE6}) and (\ref{eqn:C3_PDE1_DSE5}), respectively. 
\end{proof}

Let $l_{\rm ht}=(d_1-d_2,0,0,0,0,d_2-d_3)$. 
The functions $\hat{\varphi}_{l}$ 
$(l\in S_{(d_1,d_2,d_3)})$ 
are determined from $\hat{\varphi}_{l_{\rm ht}}$ by 
the equations in Lemma \ref{lem:C3_ZgDSE2}. 
Therefore, we know that $\varphi $ 
is uniquely determined from $\hat{\varphi}_{l_{\rm ht}}$.

\begin{lem}
\label{lem:C3_PDE3}
Retain the notation. 
Then it holds that 
\begin{align}
\begin{split}
&
\{
-\partial_1^2+\partial_1\partial_2-\partial_2^2
+(2\nu_1+2\nu_2+2\nu_3+d_1-d_2)\partial_1\\
&
-(2\nu_1+2\nu_2+2\nu_3-d_2+d_3)\partial_2
-4(\nu_1\nu_2+\nu_1\nu_3+\nu_2\nu_3)
\\
&\label{eqn:C3_PDE3_2}
-2\nu_1(d_1-d_2)+2\nu_3(d_2-d_3)
+(4\pi y_1)^2+(4\pi y_2)^2\}
\hat{\varphi}_{l_{\rm ht}}=0,
\end{split}\\[3pt]
\begin{split}
&
\{(\partial_1-2\nu_1)
(\partial_1-2\nu_2-d_1+d_2)
(\partial_1-2\nu_3-d_1+d_3)\\
&\label{eqn:C3_PDE3_3}
-(4\pi y_1)^2(\partial_1+\partial_2-d_1+d_3+2)
\}
\hat{\varphi}_{l_{\rm ht}}=0.
\end{split}
\end{align}
\end{lem}
\begin{proof}
The equation (\ref{eqn:C3_PDE3_2}) follows immediately 
from (\ref{eqn:C3_PDE1_C2}). 
We have  
\begin{align}
\begin{split}
&
\{(\partial_1-d_1)(-\partial_1+\partial_2+2\nu_1+2\nu_2+2\nu_3-d_2)
(-\partial_2-d_3)\\
&
-(4\pi y_1)^2(\partial_2+d_3)
+(4\pi y_2)^2(\partial_1-d_1)\\
&\label{eqn:C3_PDE2_C3}
-(2\nu_1-d_1)(2\nu_2-d_2)(2\nu_3-d_3)\}
\hat{\varphi}_{l_{\rm ht}}=0
\end{split}
\end{align}
from (\ref{eqn:C3_PDE1_C3}). 
Multiplying the both sides of (\ref{eqn:C3_PDE3_2}) 
by $-(\partial_1-d_1)$ from the left, we have 
\begin{align}
\begin{split}
&\bigl\{(\partial_1-d_1)\{
\partial_1^2-\partial_1\partial_2+\partial_2^2
-(2\nu_1+2\nu_2+2\nu_3+d_1-d_2)\partial_1\\
&+(2\nu_1+2\nu_2+2\nu_3-d_2+d_3)\partial_2
+4(\nu_1\nu_2+\nu_1\nu_3+\nu_2\nu_3)\\
&+2\nu_1(d_1-d_2)-2\nu_3(d_2-d_3)\}
-(4\pi y_1)^2(\partial_1-d_1+2)\\
&\label{eqn:C3_PDE2_D1C2}
-(4\pi y_2)^2(\partial_1-d_1)\bigr\}
\hat{\varphi}_{l_{\rm ht}}=0.
\end{split}
\end{align}
Adding the respective sides of 
(\ref{eqn:C3_PDE2_C3}) and (\ref{eqn:C3_PDE2_D1C2}), 
we obtain (\ref{eqn:C3_PDE3_3}). 
\end{proof}

\begin{lem}
\label{lem:C3_MW_sub1}
Let $\nu_1,\nu_2,\nu_3\in \bC $ and $(d_1,d_2,d_3)\in \Lambda_3$. Set 
\begin{align*}
r=(r_1,r_2,r_3)=(2\nu_1-d_1+d_2,2\nu_2,2\nu_3+d_2-d_3).
\end{align*}
(1) The space of smooth solutions of the system in 
Lemma \ref{lem:C3_PDE3} on $(\bR_+)^2$ is at most $6$ dimensional. \\
\noindent (2) Set  
\[
\hat{\varphi}^{\mathrm{mg}}_{l_{\rm ht}}(y_1,y_2)
=\frac{1}{(2\pi \sqrt{-1})^2} \int_{s_2}\int_{s_1} 
\cV_{l_{\rm ht}}(s_1,s_2) \,
y_1^{-2s_1} y_2^{-2s_2} \,ds_1ds_2
\]
with 
\begin{align*}
\cV_{l_{\rm ht}}(s_1,s_2)=&
\frac{
\Gamma_{\bC}\bigl(s_1+\tfrac{r_1+d_1-d_2}{2}\bigr)
\Gamma_{\bC}\bigl(s_1+\tfrac{r_2+d_1-d_2}{2}\bigr)
\Gamma_{\bC}\bigl(s_1+\tfrac{r_3+d_1-d_2}{2}\bigr)}
{\Gamma_{\bC}\bigl(s_1+s_2+\tfrac{d_1-d_3}{2}\bigr)}\\
&\times 
\Gamma_{\bC}\bigl(s_2+\tfrac{-r_1+d_2-d_3}{2}\bigr)
\Gamma_{\bC}\bigl(s_2+\tfrac{-r_2+d_2-d_3}{2}\bigr)
\Gamma_{\bC}\bigl(s_2+\tfrac{-r_3+d_2-d_3}{2}\bigr).
\end{align*}
Here the path of the integration $\int_{s_i}$ is the vertical line 
from $\mathrm{Re}(s_i)-\sI \infty$ to $\mathrm{Re}(s_i)+\sI \infty$ 
with sufficiently large real part to keep the poles of the integrand 
on its left. 
Then $\hat{\varphi}_{l_{\rm ht}}
=\hat{\varphi}^{\mathrm{mg}}_{l_{\rm ht}}$ is 
a moderate growth solution of the system in Lemma \ref{lem:C3_PDE3} 
on $(\bR_+)^2$. 
\\[1mm]
(3) Assume $r_p-r_q \notin 2\bZ$ for any $1\leq p\neq q\leq 3$. 
For a permutation $(i,j,k)$ of $\{1,2,3\}$, set  
\begin{align*}
&\hat{\varphi}^{{(i,j,k)}}_{l_{\rm ht}}(y_1,y_2)
=
\sum_{m_1,m_2 \geq 0}
C_{l_{\rm ht},(m_1,m_2)}^{(i,j,k)}
(2\pi y_1)^{2m_1+d_1-d_2+r_i} (2\pi y_2)^{2m_2+d_2-d_3-r_j}
\end{align*}
with 
\begin{align*}
C_{l_{\rm ht},(m_1,m_2)}^{(i,j,k)}=
&\frac{(-1)^{m_1+m_2}\Gamma \bigl(-m_1-\tfrac{r_i-r_j}{2}\bigr)
\Gamma \bigl(-m_1-\tfrac{r_i-r_k}{2}\bigr)}
{2^{-5}(2\pi)^{d_1-d_3}
m_1!m_2!\Gamma \bigl(-m_1-m_2-\tfrac{r_i-r_j}{2}\bigr)}\\
&\times 
\Gamma \bigl(-m_2-\tfrac{r_i-r_j}{2}\bigr)
\Gamma \bigl(-m_2-\tfrac{r_k-r_j}{2}\bigr).
\end{align*}
Then $\{\hat{\varphi}^{{(i,j,k)}}_{l_{\rm ht}}\mid 
\{i,j,k\}=\{1,2,3\}\,\}$ forms a basis of 
the space of smooth solutions of the system in Lemma \ref{lem:C3_PDE3} 
on $(\bR_+)^2$. Moreover, it holds that 
\begin{align*}
\hat{\varphi}^{\mathrm{mg}}_{l_{\rm ht}}
& = \sum_{ (i,j,k) } \hat{\varphi}^{(i,j,k)}_{l_{\rm ht}},
\end{align*}
where $ (i,j,k) $ runs all permutations of $\{1,2,3\}$. 
\end{lem}
\begin{proof}
If we set 
$f(z_1,z_2)=z_1^{-d_1+d_2}z_2^{-d_2+d_3}
\hat{\varphi}_{l_{\rm ht}}((2\pi)^{-1}z_1,(2\pi )^{-1}z_2)$, 
then (\ref{eqn:C3_PDE3_2}) and (\ref{eqn:C3_PDE3_3}) 
imply (\ref{eqn:Fn_Sol_PDE2}) and (\ref{eqn:Fn_Sol_PDE3}), 
respectively. 
Hence, the assertion follows from Lemmas \ref{lem:Fn_Sol_dim}, 
\ref{lem:Fn_Sol_power_series} and \ref{lem:Fn_Sol_mg}. 
\end{proof}

\begin{thm}[{\cite{Hirano_Oda_001}}]
\label{thm:C3_Whittaker}
Let $\chi = \chi_{(\nu_1,d_1)}\boxtimes \chi_{(\nu_2,d_2)}\boxtimes 
\chi_{(\nu_3,d_3)}$ with 
$\nu_1,\nu_2,\nu_3\in \bC$ and $(d_1,d_2,d_3)\in \Lambda_3$ 
such that $\Pi_\chi$ is irreducible. 
For $l=(l_1,l_2,l_3,\tilde{l}_1,\tilde{l}_2,\tilde{l}_3)
\in S_{(d_1,d_2,d_3)}$, let 
\begin{align*}
&\xi_1(l)=\tilde{l}_1+l_2+l_3,&
&\xi_2(l)=l_1+\tilde{l}_1,&
&\xi_3(l)=l_1+\tilde{l}_2+\tilde{l}_3,\\
&\tilde{\xi}_1(l)=l_1+l_2+\tilde{l}_3,&
&\tilde{\xi}_2(l)=l_3+\tilde{l}_3,&
&\tilde{\xi}_3(l)=\tilde{l}_1+\tilde{l}_2+l_3.
\end{align*}
\noindent 
(1) There exists a $K$-homomorphism 
\[
\varphi^{\mathrm{mg}}_\chi \colon V_{(d_1,d_2,d_3)}\to 
{\mathrm{Wh}}(\Pi_{\chi},\psi_1)^{\mathrm{mg}},
\]
whose radial part is given by 
\begin{align*}
\begin{split}
&\varphi^{\mathrm{mg}}_\chi (u_{l})(y) =
(\sI)^{l_1+\tilde{l}_1-l_3-\tilde{l}_3}
y_1^2y_2^2(y_2y_3)^{2\nu_1+2\nu_2+2\nu_3}\\
&\times \frac{1}{(2\pi \sqrt{-1})^2} \int_{s_2}\int_{s_1}
\frac{\Gamma_{\bC}\bigl(s_1+\nu_1+\tfrac{\xi_1(l)}{2}\bigr)
\Gamma_{\bC}\bigl(s_1+\nu_2+\tfrac{\xi_2(l)}{2}\bigr)
\Gamma_{\bC}\bigl(s_1+\nu_3+\tfrac{\xi_3(l)}{2}\bigr)}
{ \Gamma_{\bC}\bigl(s_1+s_2+\tfrac{\xi_2(l)+\tilde{\xi}_2(l)}{2}\bigr)}\\
&\times 
\Gamma_{\bC}\bigl(s_2-\nu_1+\tfrac{\tilde{\xi}_1(l)}{2}\bigr)
\Gamma_{\bC}\bigl(s_2-\nu_2+\tfrac{\tilde{\xi}_2(l)}{2}\bigr)
\Gamma_{\bC}\bigl(s_2-\nu_3+\tfrac{\tilde{\xi}_3(l)}{2}\bigr)
\,y_1^{-2s_1} y_2^{-2s_2} \,ds_1ds_2
\end{split}
\end{align*}
with $l=(l_1,l_2,l_3,\tilde{l}_1,\tilde{l}_2,\tilde{l}_3)
\in S_{(d_1,d_2,d_3)}$ and $y=\diag (y_1y_2y_3,y_2y_3,y_3)\in A$. 
Here the path of the integration $\int_{s_i}$ is the vertical line 
from $\mathrm{Re}(s_i)-\sI \infty$ to $\mathrm{Re}(s_i)+\sI \infty$ 
with sufficiently large real part to keep the poles of the integrand 
on its left. \\[2pt]
(2) Assume 
$2\nu_p-2\nu_q-d_p+d_q\not\in 2\bZ$ for any $1\leq p\neq q\leq 3$. 
For a permutation $(i,j,k)$ of $\{1,2,3\}$, 
there exists a $K$-homomorphism 
\[
\varphi^{(i,j,k)}_\chi \colon V_{(d_1,d_2,d_3)}\to 
{\mathrm{Wh}}(\Pi_{\chi},\psi_1),
\]
whose radial part is given by the power series 
\begin{align*}
&\varphi^{(i,j,k)}_\chi (u_{l})(y) =
2^5(2\pi)^{-d_1+d_3}(\sI)^{l_1+\tilde{l}_1-l_3-\tilde{l}_3}
y_1^2y_2^2(y_2y_3)^{2\nu_1+2\nu_2+2\nu_3}\\
&\phantom{=}\times \sum_{m_1,m_2 \geq 0}
\frac{
\Gamma \bigl(-m_1-\nu_i+\nu_j-\tfrac{\xi_i(l)-\xi_j(l)}{2}\bigr)
\Gamma \bigl(-m_1-\nu_i+\nu_k-\tfrac{\xi_i(l)-\xi_k(l)}{2}\bigr)}
{m_1!m_2!\Gamma \bigl(-m_1-m_2-\nu_i+\nu_j
-\tfrac{\xi_i(l)+\tilde{\xi}_j(l)-\xi_2(l)-\tilde{\xi}_2(l)}{2}\bigr)}\\
&\phantom{=}\times 
\Gamma \bigl(-m_2-\nu_i+\nu_j+\tfrac{\tilde{\xi}_i(l)-\tilde{\xi}_j(l)}{2}\bigr)
\Gamma \bigl(-m_2-\nu_k+\nu_j+\tfrac{\tilde{\xi}_k(l)-\tilde{\xi}_j(l)}{2}\bigr)\\
&\phantom{=}\times 
(-1)^{m_1+m_2}(2\pi y_1)^{2m_1+2\nu_i+\xi_i(l)}
(2\pi y_2)^{2m_2-2\nu_j+\tilde{\xi}_j(l)}
\end{align*}
with $l=(l_1,l_2,l_3,\tilde{l}_1,\tilde{l}_2,\tilde{l}_3)
\in S_{(d_1,d_2,d_3)}$ and 
$y=\diag (y_1y_2y_3,y_2y_3,y_3)\in A$. 
Moreover, $\{\varphi^{(i,j,k)}_\chi \mid \{i,j,k\}=\{1,2,3\}\,\}$ 
forms a basis of $\Hom_K(V_{(d_1,d_2,d_3)},
{\mathrm{Wh}}(\Pi_{\chi},\psi_1))$, and satisfies 
\begin{align*}
\varphi^{\mathrm{mg}}_\chi 
& = \sum_{ (i,j,k) } \varphi^{(i,j,k)}_\chi ,
\end{align*}
where $(i,j,k)$ runs all permutations of $\{1,2,3\}$. 
\end{thm} 
\begin{proof}
We denote by $\mathrm{Sol}(\Pi_{\chi};\psi_1)$ 
the space of smooth solutions of 
the system in Lemma \ref{lem:C3_PDE3} on $(\bR_+)^2$. 
We define an injective homomorphism 
\begin{align}
\label{eqn:C3_Wh_to_Sol}
\Hom_K(V_{(d_1,d_2,d_3)},{\mathrm{Wh}}(\Pi_{\chi},\psi_1))
\ni \varphi 
\mapsto \hat{\varphi}_{l_{\rm ht}}
\in \mathrm{Sol}(\Pi_{\chi};\psi_1)
\end{align}
of $\bC$-vector spaces by (\ref{eqn:C3_def_varphi_l}). 

Since $\dim_\bC \Hom_K(V_{(d_1,d_2,d_3)} ,H(\chi)_K)=1$ and 
$\Pi_\chi$ is irreducible, we have 
\begin{align*}
\dim_\bC \Hom_K(V_{(d_1,d_2,d_3)},
{\mathrm{Wh}}(\Pi_{\chi} ,\psi_1))
&=\dim_\bC {\cI}_{\Pi_{\chi},\psi_1}=6
\geq 
\dim_{\bC}\mathrm{Sol}(\Pi_{\chi};\psi_1).
\end{align*}
Here the last inequality follows from Lemma \ref{lem:C3_MW_sub1} (1).
This inequality implies that (\ref{eqn:C3_Wh_to_Sol}) 
is bijective. Hence, the assertion follows from 
Lemma \ref{lem:C3_MW_sub1} and the computation 
using Lemma \ref{lem:C3_ZgDSE2}. 
\end{proof}

\part{Archimedean zeta integrals for $GL(3)\times GL(2)$}
\label{part:2}

\chapter{Preliminaries}
\label{sec:F32_zeta}

\section{The aim of Part \ref{part:2}}
\label{subsec:Fmn_main_result}

Let ${F}$ be $\bR$ or $\bC$. 
We define a norm $|\cdot|_F$ on $F$ by 
\begin{align*}
&|t|_\bR =|t|,&
&|t|_\bC =|t|^2,
\end{align*}
where $|\cdot |$ is the ordinary absolute value. 

Let $n$ be a positive integer. 
Let $\Pi$ and ${\Pi'}$ be irreducible admissible large representations 
of $G_{n+1}=GL(n+1,F)$ and $G_{n}=GL(n,F)$, respectively. 
We denote by $L(s,\Pi \times \Pi')$ the local $L$-factor 
for $\Pi \times \Pi'$
defined from the Langlands parameters of $\Pi$ and $\Pi'$
(\textit{cf.} \S \ref{subsec:Rmn_langlands}, \S \ref{subsec:Cmn_landlands}).

Let $\varepsilon \in \{\pm 1\}$. 
For $W\in \mathrm{Wh}(\Pi,\psi_{\varepsilon })^{\mathrm{mg}}$ 
and $W'\in \mathrm{Wh}({\Pi'},\psi_{-\varepsilon })^{\mathrm{mg}}$, 
we define the local zeta integral $Z(s,W,W')$ for $G_{n+1}\times G_n$ by 
\begin{align*}
Z(s,W,W')=&\int_{N_{n}\backslash G_{n}}
W\!\left(
\begin{array}{cc}
g& \\
 &1
\end{array}
\right) 
W'(g)|\det g|_{F}^{s-\frac{1}{2}}
d\dot{g},
\end{align*}
where $d\dot{g}$ is the right $G_{n}$-invariant measure 
on $N_{n}\backslash G_{n}$. 
In this paper, we normalize $d\dot{g}$ so that, 
for any compactly supported continuous function $f$ on $N_n\backslash G_n$, 
\begin{align*}
\int_{N_n\backslash G_n}f(g)\,d\dot{g}
=\int_{(\bR_+)^n}  
\left(\int_{K_n}f(yk)\,dk\right)
\prod_{i=1}^n|y_i|_F^{-(n-i)i}
\frac{2dy_i}{y_i}
\end{align*} 
with $y=\diag (y_1y_2\cdots y_n,\,y_2\cdots y_n,\,\cdots ,\,y_n)\in A_n$ and 
$dk$ is the Haar measure on $K_n$ such that 
$\int_{K_n}dk=1$.

The local zeta integral $Z(s,W,W')$ converges 
if $\mathrm{Re}(s)$ is sufficiently large.  
Jacquet and Shalika \cite{Jacuqet_Shalika_001} show that 
the ratio $Z(s,W,W')/L(s,\Pi \times \Pi')$ 
extends to an entire function of $s$, 
and satisfies the local functional equation. 
Moreover, Jacquet proves the following theorem.

\begin{thm}[{\cite[Theorem 2.7 (i)]{Jacquet_001}}]
Retain the notation. 
There exist 
$W_i\in \mathrm{Wh}(\Pi ,\psi_{\varepsilon })^{\mathrm{mg}}$, 
$W_i'\in \mathrm{Wh}(\Pi',\psi_{-\varepsilon })^{\mathrm{mg}}$ 
$(i=1$, $2$, $\cdots$, $m)$ such that 
\[
\sum_{i=1}^mZ(s,W_{i},W_{i}')=L(s,\Pi\times \Pi'). 
\]
\end{thm}

We expect that there exist Whittaker functions 
$W_0\in \mathrm{Wh}(\Pi,\psi_{\varepsilon })^{\mathrm{mg}}$ 
and $W_0'\in \mathrm{Wh}({\Pi'},\psi_{-\varepsilon })^{\mathrm{mg}}$ 
such that 
\begin{align}
\label{eqn:Fmn_zeta_L_coincide}
Z(s,W_0,W_0')=L(s,\Pi \times \Pi').
\end{align}
There are the following results, 
which support our expectation. 
\begin{thm}[\cite{Stade_002}]
\label{thm:Fmn_stade_zeta}
We set ${F}=\bR$. 
Let $\sigma =\chi_{(\nu_1,0)}\boxtimes 
\chi_{(\nu_2,0)}\boxtimes \cdots \boxtimes \chi_{(\nu_{n+1},0)}$ and 
$\sigma'=\chi_{(\nu_1',0)}\boxtimes 
\chi_{(\nu_2',0)}\boxtimes \cdots \boxtimes \chi_{(\nu_{n}',0)}$ 
such that $\Pi_\sigma $ and $\Pi_{\sigma'}$ are irreducible. 
Let $\varepsilon \in \{\pm 1\}$. Then it holds that 
\[
Z(s,W_{0},W_{0}')=L(s,\Pi_{\sigma }\times \Pi_{\sigma'}), 
\]
where 
$W_0\in \mathrm{Wh}(\Pi_{\sigma},\psi_{\varepsilon })^{\mathrm{mg}}$ and 
$W_0'\in \mathrm{Wh}(\Pi_{\sigma'},\psi_{-\varepsilon })^{\mathrm{mg}}$ 
are the normalized Whittaker functions at the trivial 
$K_{n+1}$- and $K_n$-types, respectively. 
\end{thm}

\begin{thm}[{\cite{Jacquet_Langlands_001}, \cite{Popa_001}, 
Propositions \ref{prop:R21_zeta_L_coincide} and 
\ref{prop:C21_zeta_L_coincide}}]
\label{thm:Fmn_zeta_GL21}
Let $\Pi$ be an irreducible admissible large representation 
of $G_2=GL(2,F)$. 
Let $\chi$ be a character of $G_1=GL(1,F)$. Let $\varepsilon \in \{\pm 1\}$. 
Then there exists 
$W_0\in \mathrm{Wh}(\Pi ,\psi_{\varepsilon })^{\mathrm{mg}}$ such that 
\begin{align*}
&Z(s,W_{0},\chi )=L(s,\Pi \times \chi ).
\end{align*}
\end{thm}

The aim of Part \ref{part:2} is to give Whittaker functions 
$W_0$ and $W_0'$ satisfying (\ref{eqn:Fmn_zeta_L_coincide}) 
for $n=2$, explicitly. 
Our calculation based on the explicit formulas of Whittaker functions 
in Part \ref{part:1}.

\section{Some formulas for the calculation}
\label{subsec:F}

In this section, we prepare some formulas for integrations and 
summations, which play important roles on our calculation.

\begin{lem}
\label{lem:F32_Mellin}
Let $r\in \bR$ and $c\in \bR_{+}$. 
Let $f(z)$ be 
a holomorphic function on $\mathrm{Re}(z)>r$ 
such that, for 
any $x>r$, 
a function $y\mapsto f(x+\sI y)$ is 
contained in $L^1(\bR )\cap L^2(\bR )$. Then it holds that 
\[
\frac{1}{2\pi \sI}
\int_{0}^\infty 
\left\{
\int_t f(t)
y^{-ct}dt
\right\}
y^{cs}\frac{c\,dy}{y}=f(s)
\]
for $s\in \bC$ such that $\mathrm{Re}(s)>r$. 
Here the path of the integration $\int_{t}$ is the vertical line 
from $\mathrm{Re}(t) -\sI \infty$ to $\mathrm{Re}(t)+\sI \infty$ 
with $\mathrm{Re}(t)>r$. 
\end{lem}
\begin{proof}
Substituting $y^c\to y$, we have 
\[
\frac{1}{2\pi \sI}
\int_{0}^\infty 
\left\{
\int_t f(t)
y^{-ct}dt
\right\}
y^{cs}\frac{c\,dy}{y}
=\frac{1}{2\pi \sI}
\int_{0}^\infty 
\left\{
\int_t f(t)
y^{-t}dt
\right\}
y^{s}\frac{dy}{y}.
\]
Hence, it suffices to show the statement for $c=1$. 
It is wellknown that the statement for $c=1$ follows from 
the Fourier inversion formula (See, for example, \cite[\S 1.5]{Bump_004} ). 
\end{proof}

\begin{lem}[{Barnes' lemma}]
\label{lem:F32_Barnes_1st}
For $F=\bR,\bC$ and $a_1,a_2,b_1,b_2\in \bC $ such that 
$\mathrm{Re}(a_i+b_j)>0$ $(1\leq i,j\leq 2)$, it holds that  
\begin{align*}
&\frac{1}{4\pi \sI}\int_z
\Gamma_{F} (z+a_1)\Gamma_{F} (z+a_2)\Gamma_{F} (-z+b_1)\Gamma_{F} (-z+b_2)dz\\
&=\frac{\Gamma_{F} (a_1+b_1)\Gamma_{F} (a_1+b_2)
\Gamma_{F} (a_2+b_1)\Gamma_{F} (a_2+b_2)}{\Gamma_{F} (a_1+a_2+b_1+b_2)}.
\end{align*}
Here the path of integration $\int_{z}$ is the vertical line 
from $\mathrm{Re}(z)-\sI \infty$ to $\mathrm{Re}(z)+\sI \infty$ 
with the real part 
\begin{align*}
&\max \{-\mathrm{Re}(a_1),-\mathrm{Re}(a_2)\}<\mathrm{Re}(z)
<\min \{\mathrm{Re}(b_1),\mathrm{Re}(b_2)\}.
\end{align*}  
\end{lem}
\begin{proof}
The assertion follows immediately from 
Barnes' lemma (\cite[\S 14.52]{Whittaker_Watson_001}): 
\begin{align*}
&\frac{1}{2\pi \sI}\int_z
\Gamma (z+a_1)\Gamma (z+a_2)\Gamma (-z+b_1)\Gamma (-z+b_2)dz\\
&=\frac{\Gamma (a_1+b_1)\Gamma (a_1+b_2)
\Gamma (a_2+b_1)\Gamma (a_2+b_2)}{\Gamma (a_1+a_2+b_1+b_2)}.
\end{align*}
\end{proof}

\begin{lem}
\label{lem:F32_Barnes_eqn}
For $F=\bR,\bC$ and $a_1,a_2,a_3,b_1,b_2,b_3\in \bC $ such that 
$\mathrm{Re}(a_i+b_j)>0$ $(1\leq i,j\leq 3)$, it holds that 
\begin{align*}
&\int_{z}
\frac{\Gamma_{F} (z+a_1)\Gamma_{F} (z+a_2)
\Gamma_{F} (-z+b_1)\Gamma_{F} (-z+b_2)
\Gamma_{F} (-z+b_3)}
{\Gamma_{F} (-z+a_3+b_1+b_2)}dz\\
&=
\frac{\Gamma_{F} (a_1+b_3)\Gamma_{F} (a_2+b_3)}
{\Gamma_{F}(a_3+b_1)\Gamma_{F}(a_3+b_2)}\\
&\phantom{=}\times 
\int_{t}
\frac{\Gamma_{F}(t+b_1)\Gamma_{F}(t+b_2)
\Gamma_{F}(-t+a_1)
\Gamma_{F}(-t+a_2)
\Gamma_{F}(-t+a_3)}
{\Gamma_{F}(-t+a_1+a_2+b_3)}dt.
\end{align*}
Here the paths of integrations $\int_{z}$ and $\int_t$ are 
the vertical lines 
from $\mathrm{Re}(z)-\sI \infty$ to $\mathrm{Re}(z)+\sI \infty$, 
and from $\mathrm{Re}(t)-\sI \infty$ to $\mathrm{Re}(t)+\sI \infty$, 
respectively, 
with the real parts 
\begin{align*}
&\max \{-\mathrm{Re}(a_1),-\mathrm{Re}(a_2)\}<\mathrm{Re}(z)
<\min \{\mathrm{Re}(b_1),\mathrm{Re}(b_2),\mathrm{Re}(b_3)\},\\
&\max \{-\mathrm{Re}(b_1),-\mathrm{Re}(b_2)\}<\mathrm{Re}(t)
<\min \{\mathrm{Re}(a_1),\mathrm{Re}(a_2),\mathrm{Re}(a_3)\}.
\end{align*}
\end{lem}
\begin{proof}
By direct computation, we have 
\begin{align*}
&\int_{z}
\frac{\Gamma_{F} ({z}+a_1)\Gamma_{F} ({z}+a_2)
\Gamma_{F} (-{z}+b_1)\Gamma_{F} (-{z}+b_2)\Gamma_{F} (-{z}+b_3)}
{\Gamma_{F} (-{z}+a_3+b_1+b_2)}dz\\
&=
\frac{1}{4\pi \sI}\int_z
\frac{\Gamma_{F} ({z}+a_1)\Gamma_{F} ({z}+a_2)\Gamma_{F} (-{z}+b_3)}
{\Gamma_{F}(a_3+b_1)\Gamma_{F}(a_3+b_2)}\\
&\phantom{=.}\times 
\left\{
\int_{t}
\Gamma_{F}(t+b_1)\Gamma_{F}(t+b_2)
\Gamma_{F}(-t-{z})\Gamma_{F}(-t+a_3)
dt\right\}d{z}\\
&=
\frac{1}{4\pi \sI}\int_{t}
\frac{\Gamma_{F}(t+b_1)\Gamma_{F}(t+b_2)
\Gamma_{F}(-t+a_3)}
{\Gamma_{F}(a_3+b_1)\Gamma_{F}(a_3+b_2)}\\
&\phantom{=.}\times 
\left\{\int_{z}
\Gamma_{F} ({z}+a_1)\Gamma_{F} ({z}+a_2)
\Gamma_{F}(-{z}-t)\Gamma_{F} (-{z}+b_3)d{z}
\right\}dt\\
&=
\frac{\Gamma_{F} (a_1+b_3)\Gamma_{F} (a_2+b_3)}
{\Gamma_{F}(a_3+b_1)\Gamma_{F}(a_3+b_2)}\\
&\phantom{=.}\times 
\int_{t}
\frac{\Gamma_{F}(t+b_1)\Gamma_{F}(t+b_2)
\Gamma_{F}(-t+a_1)
\Gamma_{F}(-t+a_2)
\Gamma_{F}(-t+a_3)}
{\Gamma_{F}(-t+a_1+a_2+b_3)}dt.
\end{align*}
Here the first and third equalities follow from 
Lemma \ref{lem:F32_Barnes_1st}. 
\end{proof}

\begin{lem}[{Barnes' second lemma}]
\label{lem:F32_Barnes2nd}
For $F=\bR,\bC$ and $a_1,a_2,b_1,b_2,b_3\in \bC $ such that 
$\mathrm{Re}(a_i+b_j)>0$ 
$(1\leq i\leq 2$, $1\leq j\leq 3)$, it holds that 
\begin{align*}
&\nonumber 
\frac{1}{4\pi \sI}\int_{z}
\frac{\Gamma_{F} (z+a_1)\Gamma_{F} (z+a_2)
\Gamma_{F} (-z+b_1)\Gamma_{F} (-z+b_2)
\Gamma_{F} (-z+b_3)}
{\Gamma_{F} (-z+a_1+a_2+b_1+b_2+b_3)}dz\\
&\nonumber 
=\frac{\Gamma_{F}(a_1+b_1)\Gamma_{F}(a_1+b_2)\Gamma_{F} (a_1+b_3)
\Gamma_{F}(a_2+b_1)\Gamma_{F}(a_2+b_2)\Gamma_{F} (a_2+b_3)}
{\Gamma_{F}(a_1+a_2+b_1+b_2)
\Gamma_{F}(a_1+a_2+b_1+b_3)\Gamma_{F}(a_1+a_2+b_2+b_3)}.
\end{align*}
Here the path of integration $\int_{z}$ 
is the vertical line 
from $\mathrm{Re}(z)-\sI \infty$ to $\mathrm{Re}(z)+\sI \infty$ 
with the real part 
\begin{align*}
&\max \{-\mathrm{Re}(a_1),-\mathrm{Re}(a_2)\}<\mathrm{Re}(z)
<\min \{\mathrm{Re}(b_1),\mathrm{Re}(b_2),\mathrm{Re}(b_3)\}.
\end{align*}
\end{lem}
\begin{proof}
By Lemma \ref{lem:F32_Barnes_eqn} with $a_3=a_1+a_2+b_3$, we have 
\begin{align*}
&\frac{1}{4\pi \sI}\int_{z}
\frac{\Gamma_{F} (z+a_1)\Gamma_{F} (z+a_2)
\Gamma_{F} (-z+b_1)\Gamma_{F} (-z+b_2)
\Gamma_{F} (-z+b_3)}
{\Gamma_{F} (-z+a_1+a_2+b_1+b_2+b_3)}dz\\
&=
\frac{\Gamma_{F} (a_1+b_3)\Gamma_{F} (a_2+b_3)}
{\Gamma_{F}(a_1+a_2+b_1+b_3)\Gamma_{F}(a_1+a_2+b_2+b_3)}\\
&\phantom{=}\times 
\frac{1}{4\pi \sI}\int_{t}
\Gamma_{F}(t+b_1)\Gamma_{F}(t+b_2)
\Gamma_{F}(-t+a_1)
\Gamma_{F}(-t+a_2)dt.
\end{align*}
Hence, by Lemma \ref{lem:F32_Barnes_1st}, we obtain the assertion. 
\end{proof}

\begin{lem}
\label{lem:R32_Barnes2nd_sum}
For $a_1,a_2,b_1,b_2,b_3\in \bC$ such that 
$\mathrm{Re}(a_i+b_j)>0$ 
$(1\leq i\leq 2$, $1\leq j\leq 3)$, it holds that 
\begin{align*}
&\nonumber 
\frac{1}{4\pi \sI}
\int_z
\biggl(\frac{
\Gamma_\bR (z+a_1)\Gamma_\bR (z+a_2+1)
\Gamma_{\bR}(-z+b_1)\Gamma_{\bR}(-z+b_2)\Gamma_{\bR}(-z+b_3+1)}
{\Gamma_{\bR}(-z+a_1+a_2+b_1+b_2+b_3)}\\
&\nonumber 
+\frac{\Gamma_\bR (z+a_1+1)\Gamma_\bR (z+a_2)
\Gamma_{\bR}(-z+b_1+1)\Gamma_{\bR}(-z+b_2+1)
\Gamma_{\bR}(-z+b_3)}
{\Gamma_{\bR}(-z+a_1+a_2+b_1+b_2+b_3+1)}
\biggr)dz\\
&\nonumber =
\frac{\Gamma_{\bR}(a_1+b_1)
\Gamma_{\bR}(a_1+b_2)
\Gamma_{\bR} (a_1+b_3+1)}
{\Gamma_{\bR}(a_1+a_2+b_1+b_2+1)
\Gamma_{\bR}(a_1+a_2+b_1+b_3)\Gamma_{\bR}(a_1+a_2+b_2+b_3)} \\
&\phantom{=}\times 
\Gamma_{\bR}(a_2+b_1+1)
\Gamma_{\bR}(a_2+b_2+1)
\Gamma_{\bR} (a_2+b_3).
\end{align*}
Here the path of integration $\int_{z}$ 
is the vertical line 
from $\mathrm{Re}(z)-\sI \infty$ to $\mathrm{Re}(z)+\sI \infty$ 
with the real part 
\begin{align*}
&\max \{-\mathrm{Re}(a_1),-\mathrm{Re}(a_2)\}<\mathrm{Re}(z)
<\min \{\mathrm{Re}(b_1),\mathrm{Re}(b_2),\mathrm{Re}(b_3)\}.
\end{align*}
\end{lem}
\begin{proof}
By Lemma \ref{lem:F32_Barnes_eqn}, we have 
\begin{align*}
&\frac{1}{4\pi \sI}
\int_z
\biggl(\frac{
\Gamma_\bR (z+a_1)\Gamma_\bR (z+a_2+1)
\Gamma_{\bR}(-z+b_1)\Gamma_{\bR}(-z+b_2)\Gamma_{\bR}(-z+b_3+1)}
{\Gamma_{\bR}(-z+a_1+a_2+b_1+b_2+b_3)}\\
&+\frac{\Gamma_\bR (z+a_1+1)\Gamma_\bR (z+a_2)
\Gamma_{\bR}(-z+b_1+1)\Gamma_{\bR}(-z+b_2+1)
\Gamma_{\bR}(-z+b_3)}
{\Gamma_{\bR}(-z+a_1+a_2+b_1+b_2+b_3+1)}
\biggr)dz\\
&=
\frac{\Gamma_{\bR} (a_1+b_3+1)\Gamma_{\bR} (a_2+b_3+2)}
{\Gamma_{\bR}(a_1+a_2+b_1+b_3)\Gamma_{\bR}(a_1+a_2+b_2+b_3)}\\
&\phantom{==}\times 
\frac{1}{4\pi \sI}\int_{t}
\frac{\Gamma_{\bR}(t+b_1)\Gamma_{\bR}(t+b_2)}
{\Gamma_{\bR}(-t+a_1+a_2+b_3+2)}\\
&\phantom{==}\times 
\Gamma_{\bR}(-t+a_1)
\Gamma_{\bR}(-t+a_2+1)
\Gamma_{\bR}(-t+a_1+a_2+b_3)dt\\
&\phantom{=}+
\frac{\Gamma_{\bR} (a_1+b_3+1)\Gamma_{\bR} (a_2+b_3)}
{\Gamma_{\bR}(a_1+a_2+b_1+b_3)\Gamma_{\bR}(a_1+a_2+b_2+b_3)}\\
&\phantom{==}\times 
\frac{1}{4\pi \sI}\int_{t}
\frac{\Gamma_{\bR}(t+b_1+1)\Gamma_{\bR}(t+b_2+1)}
{\Gamma_{\bR}(-t+a_1+a_2+b_3+1)}\\
&\phantom{==}\times 
\Gamma_{\bR}(-t+a_1+1)
\Gamma_{\bR}(-t+a_2)
\Gamma_{\bR}(-t+a_1+a_2+b_3-1)dt\\
&=
\frac{\Gamma_{\bR} (a_1+b_3+1)\Gamma_{\bR} (a_2+b_3)}
{\Gamma_{\bR}(a_1+a_2+b_1+b_3)\Gamma_{\bR}(a_1+a_2+b_2+b_3)}\\
&\phantom{==}\times 
\frac{1}{4\pi \sI}\int_{t}
\Gamma_{\bR}(t+b_1)\Gamma_{\bR}(t+b_2)
\Gamma_{\bR}(-t+a_1)
\Gamma_{\bR}(-t+a_2+1)dt.
\end{align*}
Here the second equality follows from 
the substitution $t\to t-1$ in the second term 
and the equality $\Gamma_{\bR}(s+2)=(2\pi)^{-s}\Gamma_{\bR}(s)$. 
Applying Lemma \ref{lem:F32_Barnes_1st} to the right hand side, 
we obtain the assertion. 
\end{proof}

\begin{lem}
\label{lem:F32_gauss_sum}
For $z_1,z_2\in \bC$, $m\in \bZ_{\geq 0}$ such that 
$\mathrm{Re}(z_1)>0$, $\mathrm{Re}(z_2)>m$, it holds that 
\begin{align*}
&\sum_{j=0}^{m}\binom{m}{j}
\Gamma_{\bC}(z_1+j)\Gamma_{\bC}(z_2-j)
=\frac{\Gamma_{\bC}(z_1)\Gamma_{\bC}(z_1+z_2)\Gamma_{\bC}(z_2-m)}
{\Gamma_{\bC}(z_1+z_2-m)}.
\end{align*}
\end{lem}
\begin{proof}
Using the equalities 
\begin{align}
\label{eqn:F32_Poch}
&(z)_j=\frac{\Gamma_{\bC}(z+j)}{(2\pi)^{-j}\Gamma_{\bC}(z)}
=(-1)^j\frac{(2\pi)^{j}\Gamma_{\bC}(1-z)}{\Gamma_{\bC}(1-z-j)}\qquad 
(z\in \bC,\ j\in \bZ_{\geq 0}),\\ 
\nonumber 
&(m-j)!=(-1)^j\dfrac{m!}{(-m)_j}\hspace{20mm} 
(m,j\in \bZ_{\geq 0}\ \text{such that}\ m\geq j),
\end{align}
we have 
\begin{align}
\label{eqn:F32_pf_gauss}
&\sum_{j=0}^{m}
\binom{m}{j}
\Gamma_{\bC}(z_1+j)\Gamma_{\bC}(z_2-j\bigr)
=\Gamma_{\bC}(z_1)\Gamma_{\bC}(z_2)\,
{}_2F_1\!\left(\begin{array}{c}
-m,\ z_1\\
1-z_2
\end{array};1\right)
\end{align}
with the Gaussian hypergeometric series 
\begin{align*}
{}_2F_1\left(\begin{array}{c}
a_1,\ a_2\\
b_1
\end{array};z\right)=\sum_{j=0}^\infty \frac{(a_1)_j(a_2)_j}{(b_1)_j}
\frac{z^j}{j!}.
\end{align*}
Applying the formula (\cite[\S 14.11]{Whittaker_Watson_001})
\begin{align*}
&{}_2F_1\left(\begin{array}{c}
-m,\ a_2\\
\phantom{-}b_1
\end{array};1
\right) 
=\frac{(b_1-a_2)_m}{(b_1)_m}\\
&(m\in \bZ_{\geq 0},\ 
\text{the both sides are rational functions of 
$a_2,b_1$})
\end{align*}
and (\ref{eqn:F32_Poch}) to (\ref{eqn:F32_pf_gauss}), 
we obtain the assertion.
\end{proof}

\chapter{The local zeta integrals for $GL(3,\bR)\times GL(2,\bR)$}
\label{sec:R32_zeta}

Throughout this chapter, we set $F=\bR$.

\section{The local Langlands correspondence for $GL(n,\bR)$}
\label{subsec:Rmn_langlands}

We recall the theory of finite dimensional 
semisimple representations of the Weil group $W_\bR$. 
The Weil group $W_\bR$ for the real field $\bR$ is given by 
$W_\bR =\bC^\times \cup (\bC^\times \,\mathtt{j})
\subset \bH^\times$. 
Here we regard $W_\bR$ as a subgroup of the multiplicative group 
$\bH^\times$ of the Hamilton quaternion algebra 
$\bH =\bC \oplus ({\bC}\, \mathtt{j})$ ($\mathtt{j}^2=-1$ and 
$\mathtt{j}z\mathtt{j}^{-1}=\overline{z}$ for $z\in \bC$). 
We recall the irreducible representations of $W_\bR$ and 
the corresponding $L$-factor. 
\begin{itemize}
\item {\textit{Characters.}} 
We define characters $\phi^\delta_\nu \ (\nu \in \bC ,\delta \in \{0,1\} )$ 
of $W_\bR$ by 
\begin{align*}
\hspace{10mm}
&\phi^\delta_\nu (z)=|z|^{2\nu }\hspace{20mm}(z\in \bC^\times ),\\
&\phi^\delta_\nu (\mathtt{j})=(-1)^\delta .
\end{align*}
We define the corresponding $L$-factor by 
$L(s,\phi^\delta_\nu )=\Gamma_\bR (s+\nu +\delta )$. 

\item {\textit{Two dimensional representations.}} 
We define two dimensional representations 
$\phi_{\nu ,\kappa }\colon W_\bR \to GL(2,\bC )\ 
(\nu \in \bC ,\kappa \in \bZ_{\geq 0})$ by 
\begin{align*}
\hspace{10mm}&\phi_{\nu ,\kappa }(z)
=|z|^{2\nu -\kappa }\left(
\begin{array}{cc}
\overline{z}^{\,\kappa}&0\\
0&z^{\kappa}
\end{array}
\right) &
(z \in \bC^\times ),\\
&\phi_{\nu ,\kappa }(\mathtt{j})
=\left(
\begin{array}{cc}
0&(-1)^\kappa \\
1&0
\end{array}
\right).
\end{align*}
Here $\phi_{\nu ,\kappa }$ is 
irreducible if $\kappa >0$, and $\phi_{\nu ,0}\simeq \phi_{\nu }^0\oplus 
\phi_{\nu }^1$. For $\phi_{\nu ,\kappa }\ (\kappa >0)$, 
we define the corresponding $L$-factor by 
$L(s,\phi_{\nu ,\kappa } )=\Gamma_\bC \bigl(s+\nu +\tfrac{\kappa}{2})$. 
\end{itemize}
Then the set 
\[
\Sigma_\bR 
=\left\{\phi_{\nu }^{\delta}\mid 
\nu \in \bC ,\ \delta \in \{0,1\} \right\} \cup 
\left\{\phi_{\nu ,\kappa }\mid 
\nu \in \bC ,\ \kappa  \in \bZ_{\geq 1}\right\}
\]
exhausts the equivalence classes of 
irreducible representations of $W_\bR$. 
Moreover, the equivalences 
\begin{align}
\label{eqn:Rmn_weil_tensor}
\begin{split}
&\phi_{\nu}^{\delta}\otimes \phi_{\nu'}^{\delta'}
\simeq \phi_{\nu +\nu'}^{|\delta -\delta'|},\hspace{20mm} 
\phi_{\nu,\kappa }\otimes \phi_{\nu'}^{\delta'}
\simeq \phi_{\nu +\nu',\kappa },\\
&\phi_{\nu ,\kappa }\otimes \phi_{\nu',\kappa '} 
\simeq \phi_{\nu +\nu',\kappa +\kappa'}\oplus 
\phi_{\nu +\nu',|\kappa -\kappa'|}
\end{split}
\end{align}
hold for $\nu,\nu'\in\bC$, $\delta,\delta'\in \{0,1\}$ and 
$\kappa,\kappa'\in \bZ_{\geq 0}$.

For a finite dimensional semisimple representation $\phi$ of $W_\bR$, 
we define the corresponding $L$-factor by 
\begin{align*}
L(s,\phi )&=\prod_{i=1}^mL(s,\phi_i ),
\end{align*}
where 
$\phi \simeq \bigoplus_{i=1}^m\phi_i$ ($\phi_i\in \Sigma_\bR$) 
is the irreducible decomposition of $\phi$.

The Langlands classification of 
irreducible admissible representations of $G_n$ is 
given as follows:
\begin{thm}[{\cite[Theorem 1]{Knapp_003}}]
\label{thm:Rn_langlands_clas}
We use the notation in \S \ref{subsec:Rn_def_gps}. 

\noindent (1) If the parameters $\nu_i$ of $\sigma_i$ satisfy 
\[
{\mathrm{Re}(\nu_1)}\geq 
{\mathrm{Re}(\nu_2)}\geq 
\cdots \geq {\mathrm{Re}(\nu_l)},
\]
then $H(\sigma )$ has a unique irreducible quotient $J(\sigma)$.

\noindent (2) 
Any irreducible admissible representation of $G_n$ is infinitesimally 
equivalent to $J(\sigma )$ for some $\sigma$. 

\noindent (3) A quotient $J(\sigma )$ is infinitesimally equivalent 
to $J(\sigma')$ if and only if $\sigma'=w(\sigma)$ for some $w\in \gS_l$. 
\end{thm}

By Theorem \ref{thm:Rn_langlands_clas} (2), 
for an irreducible admissible representation $\Pi$ of $G_n$, 
there exists some $J(\sigma )$ with 
$\sigma =\sigma_1\boxtimes \sigma_2\boxtimes 
\cdots \boxtimes \sigma_l$, which is infinitesimally equivalent 
to $\Pi$. Then we define the Langlands parameter $\phi [\Pi ]$ of $\Pi$ by 
\[
\phi [\Pi ]=\bigoplus_{i=1}^l\phi [\sigma_i], 
\]
where $\phi [\chi_{(\nu ,\delta )}]
= \phi^{\delta }_{\nu }$ and $\phi [D_{(\nu ,\kappa )}]
= \phi_{\nu ,\kappa -1}$. 
By Theorem \ref{thm:Rn_langlands_clas} (3), this definition is well-defined 
and we obtain the following theorem, which is called 
the local Langlands correspondence over $\bR$. 
\begin{thm}[{\cite[Theorem 2]{Knapp_003}}]
\label{thm:Rn_langlands_corresp}
The correspondence $\Pi \leftrightarrow \phi [\Pi]$ gives 
a bijection between the set of infinitesimal 
equivalence classes of irreducible admissible representations 
of $G_n=GL(n,\bR)$ and the set of equivalence classes of $n$-dimensional 
semisimple representations of $W_\bR$. 
\end{thm}

Let $n$ and $n'$ be positive integers. 
Let $\Pi$ and $\Pi'$ be 
irreducible admissible representations of $G_n$ and $G_{n'}$, respectively. 
Then we define the local $L$-factor $L (s,\Pi \times \Pi')$ by 
\begin{align*}
&L (s,\Pi \times \Pi')
=L(s,\phi [\Pi] \otimes \phi[\Pi']).
\end{align*}
Let $\Pi_\sigma $ and $\Pi_{\sigma'}$ be 
irreducible generalized principal series representations of $G_3$ and $G_2$, 
respectively. 
By (\ref{eqn:Rmn_weil_tensor}), we obtain the explicit form of 
the local $L$-factor for $\Pi_\sigma \times \Pi_{\sigma'}$ as follows:
\begin{itemize}
\item When $\sigma =
\chi_{(\nu_1 ,\delta_1 )}\boxtimes \chi_{(\nu_2 ,\delta_2 )}
\boxtimes \chi_{(\nu_3 ,\delta_3 )}$ 
and $\sigma' =\chi_{(\nu_1',\delta_1')}\boxtimes \chi_{(\nu_2',\delta_2')}$, 
we have 
\begin{align*}
&L (s,\Pi_\sigma \times \Pi_{\sigma'})=
\prod_{i=1}^3\prod_{j=1}^2
\Gamma_\bR (s+\nu_i+\nu_j'+|\delta_i-\delta_j'|).
\end{align*}

\item When $\sigma =
\chi_{(\nu_1 ,\delta_1 )}\boxtimes \chi_{(\nu_2 ,\delta_2 )}
\boxtimes \chi_{(\nu_3 ,\delta_3 )}$ 
and $\sigma' =D_{(\nu',\kappa')}$, 
we have 
\begin{align*}
&L (s,\Pi_\sigma \times \Pi_{\sigma'})=
\prod_{i=1}^3
\Gamma_\bC \bigl(s+\nu_i+\nu '+\tfrac{\kappa '-1}{2}\bigr).
\end{align*}

\item When $\sigma =
D_{(\nu_1 ,\kappa_1 )}\boxtimes \chi_{(\nu_2 ,\delta_2 )}$ 
and $\sigma' =\chi_{(\nu_1',\delta_1')}\boxtimes \chi_{(\nu_2',\delta_2')}$, 
we have 
\begin{align*}
\hspace{15mm}
&L (s,\Pi_\sigma \times \Pi_{\sigma'})=
\prod_{j=1}^2
\Gamma_\bC \bigl(s+\nu_1+\nu_j'+\tfrac{\kappa_1-1}{2}\bigr)
\Gamma_\bR (s+\nu_2+\nu_j'+|\delta_2-\delta_j'|).
\end{align*}

\item When $\sigma =
D_{(\nu_1 ,\kappa_1 )}\boxtimes \chi_{(\nu_2 ,\delta_2 )}$ 
and $\sigma' =D_{(\nu',\kappa')}$, 
we have 
\begin{align*}
\hspace{15mm}
L (s,\Pi_\sigma \times \Pi_{\sigma'})=\,&
\Gamma_\bC \bigl(s+\nu_1+\nu '+\tfrac{\kappa_1+\kappa '-2}{2}\bigr)
\Gamma_\bC \bigl(s+\nu_1+\nu '+\tfrac{|\kappa_1-\kappa '|}{2}\bigr)\\
&\times \Gamma_\bC \bigl(s+\nu_2+\nu '+\tfrac{\kappa '-1}{2}\bigr).
\end{align*}

\end{itemize}

\section{Preparations for $O(2)$-modules}
\label{subsec:R32_rep_of_K2}

We use the notation 
in \S \ref{subsec:R2_rep_K} and \S \ref{subsec:R3_rep_K}. 
Let $\lambda =(\lambda_1,\lambda_2)\in \Lambda_2$. 
We define a $\bC$-bilinear pairing 
$\langle \cdot ,\cdot \rangle $ 
on $V_{\lambda}^{(2)}\otimes_\bC V_{\lambda}^{(2)}$ by 
\begin{align}
\label{eqn:R32_def_K2_inv_pair}
&\langle v_{\lambda ,q'},\,
v_{\lambda,q}\rangle
=\delta_{0,q+q'}&
&(q,q'\in Q_{\lambda}).
\end{align} 
By (\ref{eqn:R2_Kact}), 
we know that the pairing $\langle \cdot ,\cdot \rangle $ is 
$K_2$-invariant. Moreover, we have 
\begin{align}
&\langle v',v\rangle =\langle v,v'\rangle &
&(v,v'\in V_{\lambda}^{(2)}).
\end{align}

We regard $K_2$ as a subgroup of $K_3$ via the embedding 
\begin{align}
\label{eqn:R32_K2_to_K3}
K_2\ni k\mapsto \left(\begin{array}{c|c}
k&\\ \hline &1
\end{array}\right)\in K_3. 
\end{align}
Then the following lemma holds. 
\begin{lem}
\label{lem:R32_K32_restriction}
Let $\mu =(\mu_1,\mu_2)\in \Lambda_3$. 
Then it holds that  
\begin{align}
\label{eqn:R32_K32_restriction}
V^{(3)}_{\mu}
\simeq \bigoplus_{\lambda \in \Omega (\mu )}V^{(2)}_{\lambda } 
\end{align}
as $K_2$-modules, where 
$\Omega (\mu ) =\{(0,\mu_2)\}\cup 
\{(\lambda_1,0)\mid \lambda_1\in \bZ,\ 
1\leq \lambda_1\leq \mu_1\}$. 
For $\lambda =(\lambda_1,\lambda_2)\in \Omega (\mu )$, 
there is a $K_2$-homomorphism 
$\iota_{\lambda}^{\mu}\colon 
V^{(2)}_{\lambda }\to V^{(3)}_{\mu }$ such that 
\begin{align*}
&\iota_{\lambda}^{\mu}(v_{\lambda ,q})=\sgn (q)^{\mu_2}v_{q}^{\mu}& 
&(q\in Q_{\lambda}).
\end{align*}
\end{lem}
\begin{proof}
The assertion follows immediately from (\ref{eqn:R2_Kact}), 
(\ref{eqn:R3_M21act_basis}) and 
(\ref{eqn:R3_M111act_basis}). 
\end{proof}

For later use, we specify the irreducible components of 
the tensor product of irreducible representations of $K_2$.

\begin{lem}
\label{lem:R32_tensor_K2}
Let $\alpha =(\alpha_1,\alpha_2), 
\beta =(\beta_1,\beta_2)\in \Lambda_2$. 
Then it holds that  
\begin{align}
\label{eqn:R32_tensor1}
V^{(2)}_{\alpha }\otimes_{\bC}V^{(2)}_{\beta }
\simeq \bigoplus_{\lambda \in \Omega (\alpha,\beta)}
V^{(2)}_{\lambda } 
\end{align}
as $K_2$-modules, 
where 
\begin{align*}
&\Omega (\alpha ,\beta )=\!
\left\{\!\begin{array}{ll}
\{(0,|\alpha_2-\beta_2|)\}
&\text{if}\ \alpha_1=\beta_1=0,\\
\{(\alpha_1+\beta_1,0),(0,0),(0,1)\}
&\text{if}\ \alpha_1=\beta_1>0,\\
\{(\alpha_1+\beta_1,0)\}
&\text{if}\ \alpha_1> \beta_1=0\ \text{or}\ \beta_1>\alpha_1=0,\\
\{(\alpha_1+\beta_1,0),(|\alpha_1-\beta_1|,0)\}
&\text{if}\ \alpha_1> \beta_1>0\ \text{or}\ \beta_1>\alpha_1>0.
\end{array}\right.
\end{align*}
For each $\lambda \in \Omega (\alpha ,\beta )$, 
there is a $K_2$-homomorphism 
$\mathrm{I}_{\lambda}^{\alpha ,\beta}\colon 
V^{(2)}_{\lambda }\to V^{(2)}_{\alpha }\otimes_{\bC}V^{(2)}_{\beta }$ 
which is characterized by 
the following equalities: 
\begin{itemize}
\item[(i)] The case $\lambda =(0,|\alpha_2-\beta_2|)$ 
$(\alpha_1=\beta_1=0)$: \quad 
\[
\mathrm{I}_{\lambda}^{\alpha ,\beta}(v_{\lambda ,0})
=v_{\alpha ,0}\otimes v_{\beta ,0}.
\]

\item[(ii)] The case $\lambda =(0,\lambda_2)$ $(\alpha_1=\beta_1>0)$:
\begin{align*}
&\mathrm{I}_{\lambda}^{\alpha ,\beta}(v_{\lambda ,0})
=v_{\alpha ,\alpha_1}\otimes v_{\beta ,-\alpha_1}
+(-1)^{\lambda_2}v_{\alpha ,-\alpha_1}\otimes v_{\beta ,\alpha_1}.
\end{align*}

\item[(iii)] The case $\lambda =(\alpha_1+\beta_1,0)$ $(\alpha_1+\beta_1>0)$:
\begin{align*}
&\mathrm{I}_{\lambda}^{\alpha ,\beta}(v_{\lambda ,\alpha_1+\beta_1})
=v_{\alpha ,\alpha_1}\otimes v_{\beta ,\beta_1},\\
&\mathrm{I}_{\lambda}^{\alpha ,\beta}(v_{\lambda ,-\alpha_1-\beta_1})
=(-1)^{\alpha_2+\beta_2}v_{\alpha ,-\alpha_1}\otimes v_{\beta ,-\beta_1}.
\end{align*}

\item[(iv)] The case $\lambda =(|\alpha_1-\beta_1|,0)$ 
$(\alpha_1> \beta_1>0\ \text{or}\ \beta_1>\alpha_1>0)$:
\begin{align*}
&\mathrm{I}_{\lambda}^{\alpha ,\beta}(v_{\lambda ,\alpha_1-\beta_1})
=v_{\alpha ,\alpha_1}\otimes v_{\beta ,-\beta_1},\\
&\mathrm{I}_{\lambda}^{\alpha ,\beta}(v_{\lambda ,-\alpha_1+\beta_1})
=v_{\alpha ,-\alpha_1}\otimes v_{\beta ,\beta_1}.
\end{align*}
\end{itemize}
\end{lem}
\begin{proof}
Using (\ref{eqn:R2_Kact}), it is easy to show that 
the non-zero $\bC $-linear maps 
$\mathrm{I}_{\lambda}^{\alpha ,\beta}$ ($\lambda \in \Omega (\alpha ,\beta )$)
characterized by (i), (ii), (iii) and (iv) are $K_2$-homomorphisms. 
Since 
$\dim_\bC  (V^{(2)}_{\alpha }\otimes_{\bC}V^{(2)}_{\beta })=
\sum_{\lambda \in \Omega (\alpha,\beta)}
\dim_\bC  V^{(2)}_{\lambda }$, we obtain the assertion. 
\end{proof}

\section{Whittaker functions on $GL(2,\bR )$}
\label{subsec:R32_Wh_GL2}

Let $\varepsilon \in \{\pm 1\}$ and take $\Xi_{(\varepsilon )}$ 
as (\ref{eqn:Fn_psi_change}). 
Let $\Pi_\sigma $ be 
an irreducible generalized principal series representation of $G_2$. 
We construct a $K_2$-homomorphism $\varphi_\sigma^{[\varepsilon ]}$ 
from the minimal $K_2$-type of $\Pi_\sigma $ 
to $\mathrm{Wh}(\Pi_\sigma ,\psi_{\varepsilon })^{\mathrm{mg}}$ 
as follows: 
\begin{itemize}
\item The case $\sigma = \chi_{(\nu_1,\delta_1)}\boxtimes 
\chi_{(\nu_2,\delta_2)}$ ($\delta_1\geq \delta_2$): 
We define a $K_2$-homomorphism 
$\varphi_\sigma^{[\varepsilon ]} \colon 
V^{(2)}_{(\delta_1-\delta_2,\delta_2)}\to 
\mathrm{Wh}(\Pi_\sigma ,\psi_{\varepsilon })^{\mathrm{mg}}$ by 
\[
\varphi_{\sigma}^{[\varepsilon ]}
=\varepsilon^{\delta_2}(\sI )^{\delta_2-\delta_1}\, 
\Xi_{(\varepsilon )}\circ \Phi_{\sigma}^{\rm mg}\circ 
\hat{\eta}_{(\sigma ;(\delta_1-\delta_2,\delta_2))}, 
\]
where $\Phi_{\sigma}^{\rm mg}$ and 
$\hat{\eta}_{(\sigma ;(\delta_1-\delta_2,\delta_2))}$ are 
given in Theorem \ref{thm:R2_ps_Whittaker} and 
(\ref{eqn:R2_def_hateta}), respectively. 
Then, for $q\in Q_{(\delta_1-\delta_2,\delta_2)}$ and 
$y=\diag (y_1y_2,y_2)\in A_2$, we have 
\begin{align*}
\hspace{15mm}
\varphi_{\sigma}^{[\varepsilon ]}
(v_{(\delta_1-\delta_2,\delta_2),q})(y)
=&
\frac{y_1^{1/2}y_2^{\nu_1+\nu_2}}{4\pi \sI}
\int_t
\{\Gamma_\bR (t+\nu_1+\delta_1-\delta_2)
\Gamma_\bR (t+\nu_2)\\
&+\varepsilon q\Gamma_\bR (t+\nu_1)\Gamma_\bR (t+\nu_2+\delta_1-\delta_2)
\}y_1^{-t}dt.
\end{align*}

\item The case $\sigma = D_{(\nu ,\kappa )}$: 
We define a $K_2$-homomorphism $
\varphi_\sigma^{[\varepsilon ]} \colon V^{(2)}_{(\kappa ,0)}\to 
\mathrm{Wh}(\Pi_\sigma ,\psi_{\varepsilon })^{\mathrm{mg}}$ by
\[
\varphi_{\sigma}^{[\varepsilon ]}
=(\sI)^{-\kappa }\frac{\Gamma_{\bC}(\kappa )}{\Gamma_{\bR}(\kappa +\delta )}
\Xi_{(\varepsilon )}\circ \Phi_{\widehat{\sigma}}^{\rm mg}\circ 
\hat{\eta}_{(\widehat{\sigma} ;(\kappa ,0))}, 
\]
where $\widehat{\sigma}$, 
$\Phi_{\widehat{\sigma}}^{\rm mg}$ and 
$\hat{\eta}_{(\widehat{\sigma};(\kappa ,0))}$ are 
given in (\ref{eqn:R2_param_gps_eds}), 
Corollary \ref{cor:R2_ds_Whittaker} and (\ref{eqn:R2_def_hateta}), 
respectively. 
Then, for $y=\diag (y_1y_2,y_2)\in A_2$, we have 
\begin{align*}
\varphi_{\sigma}^{[\varepsilon ]}
(v_{(\kappa ,0),\varepsilon \kappa })(y)
\,&=
2y_1^{\nu  +\tfrac{\kappa }{2}}y_2^{2\nu  }
\exp (-2\pi y_1)\\
\,&=
\frac{y_1^{1/2}y_2^{2\nu  }}{2\pi \sI}
\int_t \Gamma_\bC \bigl(t+\nu  +\tfrac{\kappa  -1}{2}\bigr)y_1^{-t}dt,\\
\varphi_{\sigma}^{[\varepsilon ]}
(v_{(\kappa ,0),-\varepsilon \kappa })(y)
\,&=0.
\end{align*}
\end{itemize}
Here the path of integration $\int_{t}$ is the vertical line 
from $\mathrm{Re}(t)-\sI \infty$ to $\mathrm{Re}(t)+\sI \infty$ 
with the sufficiently large real part to keep the poles of the integrand 
on its left.

\section{Whittaker functions on $GL(3,\bR )$}
\label{subsec:R32_Wh_GL3}

Let $\varepsilon \in \{\pm 1\}$, 
and take $\Xi_{(\varepsilon ,\varepsilon )}$ as (\ref{eqn:Fn_psi_change}). 
Let $\Pi_\sigma $ be 
an irreducible generalized principal series representation of $G_3$. 
Let $\tau_\mu^{(3)}$ be the minimal $K_3$-type of $\Pi_\sigma$ with 
$\mu =(\mu_1,\mu_2)\in \Lambda_3$. 
We define 
$\varphi^{[\varepsilon ]}_\sigma \colon V_{\mu}^{(3)}\to 
{\mathrm{Wh}}(\Pi_{\sigma},\psi_{\varepsilon })^{\mathrm{mg}}$ by 
\[
\varphi^{[\varepsilon ]}_\sigma 
=(\sI)^{\mu_1}
\Xi_{(\varepsilon ,\varepsilon )}\circ \varphi^{\mathrm{mg}}_\sigma,
\]
where $\varphi^{\mathrm{mg}}_\sigma $ is 
the $K_3$-homomorphism in Theorems 
\ref{thm:R3_ps1_Whittaker}, 
\ref{thm:R3_ps31_Whittaker}, 
\ref{thm:R3_ps33_Whittaker} 
and \ref{thm:R3_gps1_Whittaker}. 
Then $\mu$ and 
the radial part of $\varphi^{[\varepsilon ]}_\sigma$ 
is given as follows. 
\begin{itemize}
\item The case $\sigma = \chi_{(\nu_1,\delta_1)}\boxtimes 
\chi_{(\nu_2,\delta_2)}\boxtimes \chi_{(\nu_3,\delta_3)}$ 
($\delta_1\geq \delta_2\geq \delta_3$): We have 
$\mu =(\delta_1-\delta_3,\delta_2)$ and 
\begin{align*}
\hspace{15mm}
&\varphi^{[\varepsilon ]}_\sigma (u_{l})(y) =
\varepsilon^{\delta_2}
(-1)^{l_1}(\varepsilon \sI )^{l_2}y_1y_2(y_2y_3)^{\nu_1+\nu_2+\nu_3}\\
&\phantom{.}\times 
\frac{1}{(4\pi \sqrt{-1})^2} \int_{t_2}\int_{t_1}
\frac{\Gamma_{\bR}(t_1+\nu_2+l_1)\Gamma_{\bR}(t_2-\nu_2+l_3)}
{\Gamma_{\bR}(t_1+t_2+l_1+l_3)}\\
&\phantom{.}\times 
\Gamma_{\bR}(t_1+\nu_1+|\delta_1-\delta_2-l_1|)
\Gamma_{\bR}(t_1+\nu_3+|\delta_2-\delta_3-l_1|)\\
&\phantom{.}\times 
\Gamma_{\bR}(t_2-\nu_1+|\delta_1-\delta_2-l_3|)
\Gamma_{\bR}(t_2-\nu_3+|\delta_2-\delta_3-l_3|)y_1^{-t_1} y_2^{-t_2} 
dt_1dt_2.
\end{align*}

\item The case 
$\sigma = D_{(\nu_1,\kappa_1)}\boxtimes \chi_{(\nu_2,\delta_2)}$: 
We have $\mu =(\kappa_1,\delta_2)$ and 
\begin{align*}
\hspace{15mm}
\varphi^{[\varepsilon ]}_\sigma (u_{l})(y) =\,&
\varepsilon^{\delta_2}
(-1)^{l_1}(\varepsilon \sI )^{l_2}y_1y_2
(y_2y_3)^{2\nu_1+\nu_2}\\
&\times \frac{1}{(4\pi \sqrt{-1})^2} \int_{t_2}\int_{t_1}
\frac{\Gamma_{\bC}\bigl(t_1+\nu_1+\tfrac{\kappa_1-1}{2}\bigr)
\Gamma_{\bR}(t_1+\nu_2+l_1)}
{ \Gamma_{\bR}(t_1+t_2+l_1+l_3)}\\
&\times \Gamma_{\bC}\bigl(t_2-\nu_1+\tfrac{\kappa_1-1}{2}\bigr)
\Gamma_{\bR}(t_2-\nu_2+l_3)y_1^{-t_1} y_2^{-t_2} \,dt_1dt_2.
\end{align*}

\end{itemize}
Here $l=(l_1,l_2,l_3)\in S_{\mu}$, 
$y=\diag (y_1y_2y_3,y_2y_3,y_3)\in A_3$ 
and the path of the integration $\int_{t_i}$ is the vertical line 
from $\mathrm{Re}(t_i)-\sI \infty$ to $\mathrm{Re}(t_i)+\sI \infty$ 
with sufficiently large real part to keep the poles of the integrand 
on its left.

We regard $K_2$ as a subgroup of $K_3$ via the embedding 
(\ref{eqn:R32_K2_to_K3}), and regard $U(\g_{3\bC})$ as a $K_2$-module 
via the adjoint action $\Ad$. 
For $\alpha =(\alpha_1,0)\in \Lambda_2$, 
we define a $\bC$-linear map 
$\mathrm{I}^{\gn_3}_{\alpha}\colon V_\alpha^{(2)}\to U(\g_{3\bC})$ by 
\begin{align*}
&\mathrm{I}^{\gn_3}_{\alpha}(v_{\alpha,q})=
(\sgn (q) E^{{\g_3}}_{1,3}+\sI E^{{\g_3}}_{2,3})^{|q|}&
&(\,q \in Q_\alpha =\{\pm \alpha_1\}\,). 
\end{align*}
We define a $\bC$-linear map 
$\mathrm{I}^{\gp_3}_{(2,0)}\colon V_{(2,0)}^{(2)}\to U(\g_{3\bC})$ by 
\begin{align*}
&\mathrm{I}^{\gp_3}_{(2,0)}(v_{(2,0),q})=
E_{1,1}^{\gp_3}-E_{2,2}^{\gp_3}+q\sI E_{1,2}^{\gp_3}&
&(\,q \in Q_{(2,0)}=\{\pm 2\}\,). 
\end{align*}
Then it is easy to see that the $\bC$-linear maps 
$\mathrm{I}^{\gn_3}_{\alpha}$ 
($\alpha =(\alpha_1,0)\in \Lambda_2$) and $\mathrm{I}^{\gp_3}_{(2,0)}$ 
are $K_2$-homomorphisms. 
Moreover, we have 
\begin{align*}
&\bigl(R\bigl(\mathrm{I}^{\gn_3}_{\alpha}(v_{\alpha,q})\bigr)f\bigr)(y)
=(-2\pi \varepsilon y_2)^{\alpha_1}f(y)
\end{align*}
for $q\in Q_\alpha$, $f\in C^\infty (N_3\backslash G_3;\psi_\varepsilon )$ 
and $y=\diag (y_1y_2y_3,y_2y_3,y_3)\in A_3$.

We define a $K_2$-homomorphism 
$\mathrm{I}^{\g_3}_{(0,1)}\colon V_{(0,1)}^{(2)}\to U(\g_{3\bC})$ by 
\begin{align*}
\mathrm{I}^{\g_3}_{(0,1)}=\mathrm{P}_{\g_3}
\circ (\mathrm{I}^{\gn_3}_{(2,0)}\otimes \mathrm{I}^{\gp_3}_{(2,0)})
\circ \mathrm{I}_{(0,1)}^{(2,0) ,(2,0)},
\end{align*}
where $\mathrm{I}_{(0,1)}^{(2,0) ,(2,0)}\colon V_{(0,1)}^{(2)}\to 
V_{(2,0)}^{(2)}\otimes_\bC V_{(2,0)}^{(2)}$ is the $K_2$-homomorphism in 
Lemma \ref{lem:R32_tensor_K2}, and 
$\mathrm{P}_{\g_3}\colon U(\g_{3\bC})\otimes_\bC U(\g_{3\bC})\to U(\g_{3\bC})$ 
is a natural $K_2$-homomorphism defined by $X_1\otimes X_2\mapsto X_1X_2$. 
Then we have 
\begin{align*}
\mathrm{I}_{(0,1)}^{\g_3}(v_{(0,1),0})
=\,&(E^{{\g_3}}_{1,3}+\sI E^{{\g_3}}_{2,3})^{2}
(E_{1,1}^{\gp_3}-E_{2,2}^{\gp_3}-2\sI E_{1,2}^{\gp_3})\\
&-
(-E^{{\g_3}}_{1,3}+\sI E^{{\g_3}}_{2,3})^{2}
(E_{1,1}^{\gp_3}-E_{2,2}^{\gp_3}+2\sI E_{1,2}^{\gp_3})\\
\equiv \,&8\sI (E^{{\g_3}}_{2,3})^{2}E_{1,2}^{\g_3}\quad 
\mod \ E^{{\g_3}}_{1,3}U(\g_{3\bC})+U(\g_{3\bC})E_{1,2}^{\gk_3}.
\end{align*}
Hence, 
for $y=\diag (y_1y_2y_3,y_2y_3,y_3)\in A_3$ and 
$f\in C^\infty (N_3\backslash G_3;\psi_\varepsilon )$ 
such that $R(E_{1,2}^{\gk_3})f=0$, we have 
\begin{align*}
&\bigl(R\bigl(\mathrm{I}_{(0,1)}^{\g_3}(v_{(0,1),0})\bigr)f\bigr)(y)
=\varepsilon (4\pi)^3 y_1y_2^2f(y).
\end{align*}

We define a $K_2$-homomorphism 
$\mathrm{P}_{G_3}\colon U(\g_{3\bC})\otimes_\bC C^\infty (G_3)$ 
by $X\otimes f\mapsto R(X)f$. 
Using the $K_2$-homomorphisms in Lemmas \ref{lem:R32_K32_restriction}, 
\ref{lem:R32_tensor_K2} and the above, 
we define a $K_2$-homomorphism 
\[
\varphi^{[\varepsilon ]}_{\sigma ,\lambda}\colon  V_\lambda^{(2)}\to 
{\mathrm{Wh}}(\Pi_{\sigma},\psi_{\varepsilon })^{\mathrm{mg}}
\]
for each $\lambda =(\lambda_1,\lambda_2)\in \Lambda_2$, as follows: 
\begin{itemize}
\item 
The case $\lambda \in \Omega (\mu )$: 
We define $\varphi^{[\varepsilon ]}_{\sigma ,\lambda}$ by 
$\varphi^{[\varepsilon ]}_{\sigma ,\lambda}
=
\varepsilon^{\lambda_1+\lambda_2}(-1)^{\lambda_1}
\varphi^{[\varepsilon ]}_{\sigma }\circ \iota_{\lambda}^{\mu}$. 
Then, for $q\in Q_{\lambda}$, we have 
\begin{align*}
\hspace{15mm}
&\varphi^{[\varepsilon ]}_{\sigma ,\lambda}(v_{\lambda ,q})
=\varepsilon^{\lambda_1+\lambda_2}(-1)^{\lambda_1}
\sgn (q)^{\mu_2}\varphi^{[\varepsilon ]}_{\sigma }(v_{q}^{\mu})\\
&=\varepsilon^{\lambda_1+\lambda_2}(-1)^{\lambda_1}
\sum_{j=0}^{|q|}\binom{|q|}{j}\sgn (q)^{\mu_2+|q|-j}(\sqrt{-1})^j
\varphi^{[\varepsilon ]}_\sigma (u_{(|q|-j,j,\mu_1-|q|)}).
\end{align*}

\item 
The case $\lambda =(\lambda_1,0)$ with $\lambda_1>\mu_1$: 
We define $\varphi^{[\varepsilon ]}_{\sigma ,\lambda}$ by 
\begin{align*}
\hspace{15mm}
&\varphi^{[\varepsilon ]}_{\sigma ,\lambda}
=(-\varepsilon )^{\lambda_1}(2\pi )^{-\lambda_1+\mu_1}\\
&\times 
\left\{\begin{array}{ll}
\mathrm{P}_{G_3}\circ 
\bigl(\mathrm{I}_{(\lambda_1-\mu_1,0)}^{\gn_3}
\otimes 
\bigl(\varphi^{[\varepsilon ]}_{\sigma }\circ \iota_{(\mu_1,0)}^{\mu}\bigr)
\bigr)
\circ \mathrm{I}_{\lambda }^{(\lambda_1-\mu_1,0) ,(\mu_1,0)}&
\text{if }\ \mu_1>0,\\[1mm]
\mathrm{P}_{G_3}\circ \bigl(\mathrm{I}_{(\lambda_1,0)}^{\gn_3}
\otimes \bigl(\varphi^{[\varepsilon ]}_{\sigma }\circ \iota_{(0,\mu_2)}^{\mu}
\bigr)\bigr)
\circ \mathrm{I}_{\lambda }^{(\lambda_1,0) ,(0,\mu_2)}&\text{if }\ \mu_1=0.
\end{array}\right.
\end{align*}
Then, for $y=\diag (y_1y_2y_3,y_2y_3,y_3)\in A_3$ and 
$q\in Q_{\lambda}$, we have 
\begin{align*}
\hspace{15mm}
&\varphi^{[\varepsilon ]}_{\sigma ,\lambda}(v_{\lambda ,q})(y)
=(-\varepsilon )^{\mu_1}
\sgn (q)^{\mu_2}y_2^{\lambda_1-\mu_1}
\varphi^{[\varepsilon ]}_{\sigma }(v_{\sgn (q)\mu_1}^{\mu})(y)\\
&=(-\varepsilon )^{\mu_1}
\sum_{j=0}^{\mu_1}\binom{\mu_1}{j}\sgn (q)^{\mu_2+\mu_1-j}(\sqrt{-1})^j
y_2^{\lambda_1-\mu_1}
\varphi^{[\varepsilon ]}_\sigma (u_{(\mu_1-j,j,0)})(y). 
\end{align*}

\item 
The case $\lambda =(0,1-\mu_2)$ ($\mu_1>0$): 
We define $\varphi^{[\varepsilon ]}_{\sigma ,\lambda}$ by 
\begin{align*}
\hspace{15mm}
\varphi^{[\varepsilon ]}_{\sigma ,\lambda}
=(-\varepsilon )^{1-\mu_2}(4\pi )^{-1}
\mathrm{P}_{G_3}\circ (\mathrm{I}_{(1,0)}^{\gn_3}
\otimes (\varphi^{[\varepsilon ]}_{\sigma }\circ \iota_{(1,0)}^{\mu}))
\circ \mathrm{I}_{\lambda }^{(1,0) ,(1,0)}.
\end{align*}
Then, for $y=\diag (y_1y_2y_3,y_2y_3,y_3)\in A_3$, we have 
\begin{align*}
&\varphi^{[\varepsilon ]}_{\sigma ,\lambda}(v_{\lambda ,0})(y)
=-\varepsilon^{\mu_2}
y_2\varphi^{[\varepsilon ]}_{\sigma }(u_{(1,0,\mu_1-1)})(y).
\end{align*}

\item 
The case $\lambda =(0,1-\mu_2)$ ($\mu_1=0$): 
We define $\varphi^{[\varepsilon ]}_{\sigma ,\lambda}$ by 
\begin{align*}
\hspace{15mm}
\varphi^{[\varepsilon ]}_{\sigma ,\lambda}=
\varepsilon^{1-\mu_2} (4\pi )^{-3}
\mathrm{P}_{G_3}\circ (\mathrm{I}_{(0,1)}^{\g_3}
\otimes (\varphi^{[\varepsilon ]}_{\sigma }\circ \iota_{(0,\mu_2)}^{\mu}))
\circ \mathrm{I}_{\lambda }^{(0,1) ,(0,\mu_2)}.
\end{align*}
Then, for $y=\diag (y_1y_2y_3,y_2y_3,y_3)\in A_3$, we have 
\begin{align*}
&\varphi^{[\varepsilon ]}_{\sigma ,\lambda}(v_{\lambda ,0})(y)
=\varepsilon^{\mu_2}y_1y_2^2\varphi^{[\varepsilon ]}_{\sigma }(u_{(0,0,0)})(y).
\end{align*}

\end{itemize}

\section{The local zeta integrals for $GL(3,\bR)\times GL(2,\bR)$}

Let $\Pi_\sigma$ and $\Pi_{\sigma'}$ be irreducible generalized principal 
series representations of $G_3$ and $G_2$, respectively. 
Let $\varepsilon \in \{\pm 1\}$. 
We consider 
the local zeta integral 
\begin{align*}
Z(s,W,W')=&\int_{N_{2}\backslash G_{2}}
W\!\left(
\begin{array}{cc}
g& \\
 &1
\end{array}
\right) 
W'(g)|\det g|^{s-\frac{1}{2}}
d\dot{g}
\end{align*}
defined for $W\in \mathrm{Wh}(\Pi_\sigma,\psi_{\varepsilon })^{\mathrm{mg}}$ 
and $W'\in \mathrm{Wh}(\Pi_{\sigma'},\psi_{-\varepsilon })^{\mathrm{mg}}$. 
Here the right $G_2$-invariant measure 
$d\dot{g}$ on $N_2\backslash G_2$ is normalized so that, 
for any compactly supported continuous function $f$ on $N_2\backslash G_2$, 
\begin{align}
\label{eqn:R32_normalize_N2_G2}
\int_{N_2\backslash G_2}f(g)\,d\dot{g}
=\int_{0}^\infty \int_{0}^\infty 
\left(\int_{K_2}f(yk)\,dk\right)
y_1^{-1}
\frac{2dy_1}{y_1}
\frac{2dy_2}{y_2}
\end{align} 
with $y=\diag (y_1y_2,y_2)\in A_2$ and 
the normalized Haar measure $dk$ on $K_2$ such that 
$\int_{K_2}dk=1$.

We regard $K_2$ as a subgroup of $K_3$ via the embedding 
(\ref{eqn:R32_K2_to_K3}). Let $\varphi \colon V_{\lambda}^{(2)}\to 
\mathrm{Wh}(\Pi_{\sigma},\psi_{\varepsilon })^{\mathrm{mg}}$ and 
$\varphi' \colon V_{\lambda'}^{(2)}\to 
\mathrm{Wh}(\Pi_{\sigma'},\psi_{-\varepsilon })^{\mathrm{mg}}$ 
be $K_2$-homomorphisms with $\lambda ,\lambda'\in \Lambda_2$. 
Then the following lemma holds.

\begin{lem}
\label{lem:R32_zeta32_schur}
Retain the notation. For $v\in V_{\lambda}^{(2)}$ and 
$v'\in V_{\lambda'}^{(2)}$, it holds that  
\begin{align}
\label{eqn:R32_zeta_schur}
\begin{split}
&Z(s,\varphi (v),\varphi'(v'))\\
&=\frac{\langle v,v'\rangle}{\dim_\bC  V_\lambda^{(2)}}
\sum_{q\in Q_\lambda }
\int_{0}^\infty \int_{0}^\infty   
\varphi (v_{\lambda,q})(\hat{y})
\varphi'(v_{\lambda,-q})(y)
y_1^{s-\frac{3}{2}}y_2^{2s-1}
\frac{2dy_1}{y_1}\frac{2dy_2}{y_2}
\end{split}
\end{align}
if $\lambda '=\lambda$, 
and $Z(s,\varphi (v),\varphi'(v'))=0$ otherwise. Here 
$y=\diag (y_1y_2,y_2)\in A_2$ and 
$\hat{y}=\diag (y_1y_2,y_2,1)\in A_3$. 
\end{lem}
\begin{proof}
Let $v\in V_{\lambda}^{(2)}$ and $v'\in V_{\lambda'}^{(2)}$. 
By definition, we have 
\begin{align*}
&Z(s,\varphi (v),\varphi'(v'))
=\int_{N_2\backslash G_2}
\varphi (v)\!\left(
\begin{array}{cc}
g& \\
 &1
\end{array}
\right) 
\varphi'(v')(g)|\det g|^{s-\frac{1}{2}}
d\dot{g}\\
&=\int_{0}^\infty \int_{0}^\infty   
\left(\int_{K_2}\varphi (\tau_{\lambda}(k)v)(\hat{y})
\varphi'(\tau_{\lambda'}(k)v')(y)\,dk\right)
y_1^{s-\frac{3}{2}}y_2^{2s-1}
\frac{2dy_1}{y_1}\frac{2dy_2}{y_2}.
\end{align*}
By the decompositions 
\begin{align*}
&\tau_{\lambda}(k)v
=\sum_{q\in Q_\lambda }
\langle \tau_{\lambda}(k)v,v_{\lambda,-q}\rangle v_{\lambda,q},&
&\tau_{\lambda'}(k)v'
=\sum_{q'\in Q_{\lambda'} }
\langle \tau_{\lambda'}(k)v',v_{\lambda',-q'}\rangle v_{\lambda',q'},&
\end{align*}
we have 
\begin{align*}
Z(s,\varphi (v),\varphi'(v'))=
&\sum_{q\in Q_\lambda }\sum_{q'\in Q_{\lambda'} }
\left(\int_{K_2}\langle \tau_{\lambda}(k)v,v_{\lambda,-q}\rangle 
\langle \tau_{\lambda'}(k)v',v_{\lambda',-q'}\rangle \,dk\right)\\
&\times \int_{0}^\infty \int_{0}^\infty   
\varphi (v_{\lambda,q})(\hat{y})
\varphi'(v_{\lambda',q'})(y)
y_1^{s-\frac{3}{2}}y_2^{2s-1}
\frac{2dy_1}{y_1}\frac{2dy_2}{y_2}.
\end{align*}
Applying Schur's orthogonality relation 
(\cite[Proposition 4.4]{Brocker_001}) 
\begin{align}
\begin{split}
\label{eqn:R32_schur}
\int_{K_2}\langle \tau_{\lambda }(k)v,w\rangle 
\langle \tau_{\lambda'}(k)v',w' \rangle \,dk
=\left\{\begin{array}{ll}
\displaystyle 
\frac{\langle v,v'\rangle 
\langle w,w'\rangle }{\dim_\bC V_{\lambda}^{(2)}}
&\text{ if }\lambda'=\lambda ,\\[3mm]
0&\text{ otherwise}
\end{array}\right.&\\
(v,w\in V_{\lambda}^{(2)},\ 
v',w'\in V_{\lambda'}^{(2)})&
\end{split}
\end{align}
to the right hand side of this equality, 
we obtain the assertion. 
\end{proof}

\begin{thm}
\label{thm:R32_zeta_L_coincide}
Let $\Pi_\sigma$ and $\Pi_{\sigma'}$ be irreducible generalized principal 
series representations of $G_3=GL(3,\bR)$ and $G_2=GL(2,\bR)$, respectively. 
Let $\varepsilon \in \{\pm 1\}$. 
Let $\lambda\in \Lambda_2$ such that 
$\tau_{\lambda}^{(2)}$ is the minimal $K_2$-type of $\Pi_{\sigma'}$. 
Then, for $v,v'\in V_{\lambda}^{(2)}$ and 
$s\in \bC$ with sufficiently large real part, it holds that 
\begin{align}
\label{eqn:R32_zeta_L_coincide}
Z\bigl(s,\varphi_{\sigma ,\lambda}^{[\varepsilon ]}(v),
\varphi_{\sigma'}^{[-\varepsilon ]}(v')\bigr)
=\langle v,v'\rangle L(s,\Pi_\sigma \times \Pi_{\sigma'}),
\end{align}
where $\varphi_{\sigma ,\lambda}^{[\varepsilon ]}$ and 
$\varphi_{\sigma'}^{[-\varepsilon ]}$ are 
the $K_2$-homomorphisms in \S \ref{subsec:R32_Wh_GL3}
and \S \ref{subsec:R32_Wh_GL2}, respectively. 
\end{thm}

\section{The calculation for $\sigma'=\chi_{(\nu_1',\delta_2')}\boxtimes \chi_{(\nu_2',\delta_2')}$}

In this section, we give a proof of 
Theorem \ref{thm:R32_zeta_L_coincide} for 
$\sigma'=\chi_{(\nu_1',\delta_2')}\boxtimes \chi_{(\nu_2',\delta_2')}$. 
By Lemma \ref{lem:R32_zeta32_schur} and 
\begin{align*}
\varphi_{\sigma'}^{[-\varepsilon ]}
(v_{(0,\delta_2'),0})(y)
=&\frac{y_1^{1/2}y_2^{\nu_1'+\nu_2'}}{4\pi \sI}
\int_t\Gamma_\bR (t+\nu_1')\Gamma_\bR (t+\nu_2')y_1^{-t}dt
\end{align*}
for $y=\diag (y_1y_2,y_2)\in A_2$, we have 
\begin{align*}
&Z\bigl(s,\varphi_{\sigma ,(0,\delta_2')}^{[\varepsilon ]}(v),
\varphi_{\sigma'}^{[-\varepsilon ]}(v')\bigr)\\
&=\langle v,v'\rangle 
\int_{0}^\infty \int_{0}^\infty   
\varphi_{\sigma ,(0,\delta_2')}^{[\varepsilon ]}(v_{(0,\delta_2'),0})(\hat{y})
\varphi_{\sigma'}^{[-\varepsilon ]}(v_{(0,\delta_2'),0})(y)
y_1^{s-\frac{3}{2}}y_2^{2s-1}
\frac{2dy_1}{y_1}\frac{2dy_2}{y_2}\\
&=\frac{\langle v,v'\rangle }{4\pi \sI}
\int_{0}^\infty \int_{0}^\infty 
\left(\int_{t}
\varphi_{\sigma ,(0,\delta_2')}^{[\varepsilon ]}(v_{(0,\delta_2'),0})
(\hat{y})
\Gamma_\bR (t+\nu_1')\Gamma_\bR (t+\nu_2')y_1^{-t}dt\right)\\
&\phantom{=.}\times 
y_1^{s-1}y_2^{2s+\nu_1'+\nu_2'-1}
\frac{2dy_1}{y_1}\frac{2dy_2}{y_2}
\end{align*}
with $\hat{y}=\diag (y_1y_2,y_2,1)\in A_3$. \vspace{2mm}

\underline{\bf Case 1-1:}  
$\sigma = \chi_{(\nu_1,\delta_1)}\boxtimes 
\chi_{(\nu_2,\delta_2)}\boxtimes \chi_{(\nu_3,\delta_3)}$ 
($\delta_1\geq \delta_2\geq \delta_3$) with $\delta_2'=\delta_2$. 
Since 
\begin{align*}
&\varphi^{[\varepsilon ]}_{\sigma ,(0,\delta_2)}(v_{(0,\delta_2),0})(\hat{y})
=\varepsilon^{\delta_2}
\varphi^{[\varepsilon ]}_\sigma (u_{(0,0,\delta_1-\delta_3)})(\hat{y})\\
&=\frac{y_1y_2^{\nu_1+\nu_2+\nu_3+1}}{(4\pi \sqrt{-1})^2} \int_{t_2}\int_{t_1}
\frac{\Gamma_{\bR}(t_1+\nu_2)\Gamma_{\bR}(t_2-\nu_2+\delta_1-\delta_3)}
{\Gamma_{\bR}(t_1+t_2+\delta_1-\delta_3)}\\
&\phantom{=.}\times 
\Gamma_{\bR}(t_1+\nu_1+\delta_1-\delta_2)
\Gamma_{\bR}(t_1+\nu_3+\delta_2-\delta_3)\\
&\phantom{=.}\times 
\Gamma_{\bR}(t_2-\nu_1+\delta_2-\delta_3)
\Gamma_{\bR}(t_2-\nu_3+\delta_1-\delta_2)y_1^{-t_1} y_2^{-t_2} 
dt_1dt_2
\end{align*}
for $\hat{y}=\diag (y_1y_2,y_2,1)\in A_3$, 
we have 
\begin{align*}
&Z\bigl(s,\varphi_{\sigma ,(0,\delta_2')}^{[\varepsilon ]}(v),
\varphi_{\sigma'}^{[-\varepsilon ]}(v')\bigr)\\
&=\frac{\langle v,v'\rangle }{2(2\pi \sI)^3}
\int_{0}^\infty \int_{0}^\infty 
\biggl(\int_{t}
\int_{t_2}\int_{t_1}
\frac{\Gamma_{\bR}(t_1+\nu_2)\Gamma_{\bR}(t_2-\nu_2+\delta_1-\delta_3)}
{\Gamma_{\bR}(t_1+t_2+\delta_1-\delta_3)}\\
&\phantom{=.}\times 
\Gamma_{\bR}(t_1+\nu_1+\delta_1-\delta_2)
\Gamma_{\bR}(t_1+\nu_3+\delta_2-\delta_3)\\
&\phantom{=.}\times 
\Gamma_{\bR}(t_2-\nu_1+\delta_2-\delta_3)
\Gamma_{\bR}(t_2-\nu_3+\delta_1-\delta_2)
\Gamma_\bR (t+\nu_1')\Gamma_\bR (t+\nu_2')\\
&\phantom{=.}\times 
y_1^{-t-t_1} y_2^{-t_2} 
dt_1dt_2dt\biggr)y_1^{s}y_2^{2s+\nu_1+\nu_2+\nu_3+\nu_1'+\nu_2'}
\frac{dy_1}{y_1}\frac{dy_2}{y_2}.
\end{align*}
Substituting $t_1\to t_1-t$ and applying Lemma \ref{lem:F32_Mellin} twice, 
we have 
\begin{align*}
&Z\bigl(s,\varphi_{\sigma ,(0,\delta_2')}^{[\varepsilon ]}(v),
\varphi_{\sigma'}^{[-\varepsilon ]}(v')\bigr)\\
&=\langle v,v'\rangle \Biggl\{\prod_{1\leq i<j\leq 3}
\Gamma_{\bR}(2s+\nu_i+\nu_j+\nu_1'+\nu_2'+\delta_i-\delta_j)\Biggr\}\\
&\phantom{=.}\times 
\frac{1}{4\pi \sI}
\int_{t}
\frac{\Gamma_\bR (t+\nu_1')\Gamma_\bR (t+\nu_2')
\Gamma_{\bR}(-t+s+\nu_1+\delta_1-\delta_2)}
{\Gamma_{\bR}(-t+3s+\nu_1+\nu_2+\nu_3+\nu_1'+\nu_2'+\delta_1-\delta_3)}\\
&\phantom{=.}\times \Gamma_{\bR}(-t+s+\nu_2)
\Gamma_{\bR}(-t+s+\nu_3+\delta_2-\delta_3)
dt.
\end{align*}
By Lemma \ref{lem:F32_Barnes2nd}, we have 
\begin{align*}
&Z\bigl(s,\varphi_{\sigma ,(0,\delta_2')}^{[\varepsilon ]}(v),
\varphi_{\sigma'}^{[-\varepsilon ]}(v')\bigr)\\
&=\langle v,v'\rangle 
\Gamma_{\bR}(s+\nu_1+\nu_1'+\delta_1-\delta_2)
\Gamma_{\bR}(s+\nu_2+\nu_1')
\Gamma_{\bR}(s+\nu_3+\nu_1'+\delta_2-\delta_3)\\
&\phantom{=}\times 
\Gamma_{\bR}(s+\nu_1+\nu_2'+\delta_1-\delta_2)
\Gamma_{\bR}(s+\nu_2+\nu_2')
\Gamma_{\bR}(s+\nu_3+\nu_2'+\delta_2-\delta_3).
\end{align*}
Hence we obtain the assertion in this case. \vspace{2mm}

\underline{\bf Case 1-2:} $\sigma = \chi_{(\nu_1,\delta_1)}\boxtimes 
\chi_{(\nu_2,\delta_2)}\boxtimes \chi_{(\nu_3,\delta_3)}$ 
($\delta_1\geq \delta_2\geq \delta_3$) with $\delta_2'=1-\delta_2$. 
Our proof in this case is similar to that in Case 1-1. 
Since 
\begin{align*}
&\varphi^{[\varepsilon ]}_{\sigma ,(0,1-\delta_2)}(v_{(0,1-\delta_2),0})
(\hat{y})
=\varepsilon^{\delta_2}(-1)^{\delta_1-\delta_3}
y_1^{1-\delta_1+\delta_3}y_2^{2-\delta_1+\delta_3}
\varphi^{[\varepsilon ]}_{\sigma }(u_{(\delta_1-\delta_3,0,0)})(\hat{y})\\
&=
\frac{y_1y_2^{\nu_1+\nu_2+\nu_3+1}}{(4\pi \sqrt{-1})^2} \int_{t_2}\int_{t_1}
\frac{\Gamma_{\bR}(t_1+\nu_2+1)\Gamma_{\bR}(t_2-\nu_2+2-\delta_1+\delta_3)}
{\Gamma_{\bR}(t_1+t_2+3-\delta_1+\delta_3)}\\
&\phantom{=.}\times 
\Gamma_{\bR}(t_1+\nu_1+1-\delta_1+\delta_2)
\Gamma_{\bR}(t_1+\nu_3+1-\delta_2+\delta_3)\\
&\phantom{=.}\times 
\Gamma_{\bR}(t_2-\nu_1+2-\delta_2+\delta_3)
\Gamma_{\bR}(t_2-\nu_3+2-\delta_1+\delta_2)y_1^{-t_1} y_2^{-t_2} 
dt_1dt_2
\end{align*}
for $\hat{y}=\diag (y_1y_2,y_2,1)\in A_3$, 
we have 
\begin{align*}
&Z\bigl(s,\varphi_{\sigma ,(0,\delta_2')}^{[\varepsilon ]}(v),
\varphi_{\sigma'}^{[-\varepsilon ]}(v')\bigr)\\
&=\langle v,v'\rangle 
\Gamma_{\bR}(s+\nu_1+\nu_1'+1-\delta_1+\delta_2)
\Gamma_{\bR}(s+\nu_2+\nu_1'+1)\\
&\phantom{=}\times 
\Gamma_{\bR}(s+\nu_3+\nu_1'+1-\delta_2+\delta_3)
\Gamma_{\bR}(s+\nu_1+\nu_2'+1-\delta_1+\delta_2)\\
&\phantom{=}\times 
\Gamma_{\bR}(s+\nu_2+\nu_2'+1)
\Gamma_{\bR}(s+\nu_3+\nu_2'+1-\delta_2+\delta_3)
\end{align*}
by Lemmas \ref{lem:F32_Mellin} and \ref{lem:F32_Barnes2nd}. 
Hence we obtain the assertion in this case. \vspace{2mm}

\underline{\bf Case 1-3:} $\sigma = D_{(\nu_1,\kappa_1)}\boxtimes 
\chi_{(\nu_2,\delta_2)}$. Let $\delta =|\delta_2-\delta_2'|$. 
Since  
\begin{align*}
&\varphi^{[\varepsilon ]}_{\sigma ,(0,\delta_2')}
(v_{(0,\delta_2'),0})(\hat{y})
=\varepsilon^{\delta_2}(-y_2)^\delta 
\varphi^{[\varepsilon ]}_\sigma (u_{(\delta ,0,\kappa_1-\delta )})(\hat{y})\\
&=\frac{y_1y_2^{2\nu_1+\nu_2+1}}{(4\pi \sqrt{-1})^2}
\int_{t_2}\int_{t_1}
\frac{\Gamma_{\bC}\bigl(t_1+\nu_1+\tfrac{\kappa_1-1}{2}\bigr)
\Gamma_{\bR}(t_1+\nu_2+\delta )}
{ \Gamma_{\bR}(t_1+t_2+\kappa_1+\delta )}\\
&\phantom{=}\times 
\Gamma_{\bC}\bigl(t_2-\nu_1+\tfrac{\kappa_1-1}{2}+\delta \bigr)
\Gamma_{\bR}(t_2-\nu_2+\kappa_1)y_1^{-t_1} y_2^{-t_2} \,dt_1dt_2
\end{align*}
for $\hat{y}=\diag (y_1y_2,y_2,1)\in A_3$, 
we have 
\begin{align*}
&Z\bigl(s,\varphi_{\sigma ,(0,\delta_2')}^{[\varepsilon ]}(v),
\varphi_{\sigma'}^{[-\varepsilon ]}(v')\bigr)\\
&=\frac{\langle v,v'\rangle }{2(2\pi \sI)^3}
\int_{0}^\infty \int_{0}^\infty 
\biggl(\int_{t}
\int_{t_2}\int_{t_1}
\frac{\Gamma_{\bC}\bigl(t_1+\nu_1+\tfrac{\kappa_1-1}{2}\bigr)
\Gamma_{\bR}(t_1+\nu_2+\delta )}
{\Gamma_{\bR}(t_1+t_2+\kappa_1+\delta )}\\
&\phantom{=.}\times 
\Gamma_{\bC}\bigl(t_2-\nu_1+\tfrac{\kappa_1-1}{2}+\delta \bigr)
\Gamma_{\bR}(t_2-\nu_2+\kappa_1)
\Gamma_\bR (t+\nu_1')\Gamma_\bR (t+\nu_2')\\
&\phantom{=.}\times 
y_1^{-t_1-t} y_2^{-t_2} \,dt_1dt_2dt\biggr)
y_1^{s}y_2^{2s+2\nu_1+\nu_2+\nu_1'+\nu_2'}
\frac{dy_1}{y_1}\frac{dy_2}{y_2}.
\end{align*}
Substituting $t_1\to t_1-t$ and applying Lemma \ref{lem:F32_Mellin} twice, 
we have 
\begin{align*}
&Z\bigl(s,\varphi_{\sigma ,(0,\delta_2')}^{[\varepsilon ]}(v),
\varphi_{\sigma'}^{[-\varepsilon ]}(v')\bigr)\\
&=\langle v,v'\rangle 
\Gamma_{\bC}\bigl(2s+\nu_1+\nu_2+\nu_1'+\nu_2'
+\tfrac{\kappa_1-1}{2}+\delta \bigr)
\Gamma_{\bR}(2s+2\nu_1+\nu_1'+\nu_2'+\kappa_1)\\
&\phantom{=}\times 
\frac{1}{4\pi \sI}\int_{t}
\frac{\Gamma_\bR (t+\nu_1')\Gamma_\bR (t+\nu_2')
\Gamma_{\bC}\bigl(-t+s+\nu_1+\tfrac{\kappa_1-1}{2}\bigr)}
{\Gamma_{\bR}(-t+3s+2\nu_1+\nu_2+\nu_1'+\nu_2'+\kappa_1+\delta )}\\
&\phantom{=}\times \Gamma_{\bR}(-t+s+\nu_2+\delta )dt.
\end{align*}
By Lemma \ref{lem:F32_Barnes2nd} with the duplication formula 
(\ref{eqn:Fn_gammaRC_duplication}), we have 
\begin{align*}
&Z\bigl(s,\varphi_{\sigma ,(0,\delta_2')}^{[\varepsilon ]}(v),
\varphi_{\sigma'}^{[-\varepsilon ]}(v')\bigr)\\
&=\langle v,v'\rangle 
\prod_{i=1}^2
\Gamma_{\bC} \bigl(s+\nu_1+\nu_i'+\tfrac{\kappa_1-1}{2}\bigr)
\Gamma_{\bR} (s+\nu_2+\nu_i'+\delta).
\end{align*}
Hence we obtain the assertion in this case. \vspace{2mm}

\section{The calculation for $\sigma'=\chi_{(\nu_1',1)}\boxtimes \chi_{(\nu_2',0)}$}

In this section, we give a proof of 
Theorem \ref{thm:R32_zeta_L_coincide} for 
$\sigma'=\chi_{(\nu_1',1)}\boxtimes \chi_{(\nu_2',0)}$. 
By Lemma \ref{lem:R32_zeta32_schur} and 
\begin{align*}
\varphi_{\sigma'}^{[-\varepsilon ]}
(v_{(1,0),q})(y)
=&
\frac{y_1^{1/2}y_2^{\nu_1'+\nu_2'}}{4\pi \sI}
\int_t
\{\Gamma_\bR (t+\nu_1'+1)
\Gamma_\bR (t+\nu_2')\\
&-\varepsilon q\Gamma_\bR (t+\nu_1')\Gamma_\bR (t+\nu_2'+1)
\}y_1^{-t}dt
\end{align*}
for $y=\diag (y_1y_2,y_2)\in A_2$ and $q\in Q_{(1,0)}=\{\pm 1\}$, 
we have 
\begin{align*}
&Z\bigl(s,\varphi_{\sigma ,(1,0)}^{[\varepsilon ]}(v),
\varphi_{\sigma'}^{[-\varepsilon ]}(v')\bigr)\\
&=\frac{\langle v,v'\rangle}{2}
\biggl\{
\int_{0}^\infty \int_{0}^\infty   
\varphi_{\sigma ,(1,0)}^{[\varepsilon ]}(v_{(1,0),1})(\hat{y})
\varphi_{\sigma'}^{[-\varepsilon ]}(v_{(1,0),-1})(y)
y_1^{s-\frac{3}{2}}y_2^{2s-1}
\frac{2dy_1}{y_1}\frac{2dy_2}{y_2}\\
&\phantom{=}
+\int_{0}^\infty \int_{0}^\infty   
\varphi_{\sigma ,(1,0)}^{[\varepsilon ]}(v_{(1,0),-1})(\hat{y})
\varphi_{\sigma'}^{[-\varepsilon ]}(v_{(1,0),1})(y)
y_1^{s-\frac{3}{2}}y_2^{2s-1}
\frac{2dy_1}{y_1}\frac{2dy_2}{y_2}\biggr\}\\
&=\frac{\langle v,v'\rangle}{8\pi \sI}
\int_{0}^\infty \int_{0}^\infty 
\biggl\{\int_t
\bigl\{\varphi_{\sigma ,(1,0)}^{[\varepsilon ]}(v_{(1,0),1}+v_{(1,0),-1})
(\hat{y})\Gamma_\bR (t+\nu_1'+1)\Gamma_\bR (t+\nu_2')\\
&\phantom{=}+\varepsilon 
\varphi_{\sigma ,(1,0)}^{[\varepsilon ]}(v_{(1,0),1}-v_{(1,0),-1})
(\hat{y})\Gamma_\bR (t+\nu_1')\Gamma_\bR (t+\nu_2'+1)\bigr\}
\\
&\phantom{=}\times 
y_1^{-t}dt\biggr\}y_1^{s-1}y_2^{2s+\nu_1'+\nu_2'-1}
\frac{2dy_1}{y_1}\frac{2dy_2}{y_2}
\end{align*}
with $\hat{y}=\diag (y_1y_2,y_2,1)\in A_3$. \vspace{2mm}

\underline{\bf Case 2-1:} $\sigma = \chi_{(\nu_1,\delta_2)}\boxtimes 
\chi_{(\nu_2,\delta_2)}\boxtimes \chi_{(\nu_3,\delta_2)}$. 
Since 
\begin{align*}
&\varepsilon^{\delta_2}\varphi_{\sigma ,(1,0)}^{[\varepsilon ]}
(v_{\lambda,1}+(-1)^{\delta_2}v_{\lambda,-1})(\hat{y})
=2\varepsilon^{\delta_2}
y_2\varphi^{[\varepsilon ]}_{\sigma }(u_{(0,0,0)})(\hat{y})\\
&=\frac{y_1y_2^{\nu_1+\nu_2+\nu_3+1}}{2(2\pi \sqrt{-1})^2} \int_{t_2}\int_{t_1}
\frac{\Gamma_{\bR}(t_1+\nu_1)\Gamma_{\bR}(t_1+\nu_2)
\Gamma_{\bR}(t_1+\nu_3)}
{\Gamma_{\bR}(t_1+t_2+1)}\\
&\phantom{=}\times 
\Gamma_{\bR}(t_2-\nu_1+1)\Gamma_{\bR}(t_2-\nu_2+1)
\Gamma_{\bR}(t_2-\nu_3+1)y_1^{-t_1} y_2^{-t_2} 
dt_1dt_2
\end{align*}
and 
\begin{align*}
&\varepsilon^{1-\delta_2}
\varphi_{\sigma ,(1,0)}^{[\varepsilon ]}
(v_{\lambda,1}-(-1)^{\delta_2}v_{\lambda,-1})(\hat{y})=0
\end{align*}
for $\hat{y}=\diag (y_1y_2,y_2,1)\in A_3$, 
we have 
\begin{align*}
&Z\bigl(s,\varphi_{\sigma ,(1,0)}^{[\varepsilon ]}(v),
\varphi_{\sigma'}^{[-\varepsilon ]}(v')\bigr)\\
&=\frac{\langle v,v'\rangle}{2(2\pi \sI)^3}
\int_{0}^\infty \int_{0}^\infty 
\biggl\{\int_t\int_{t_2}\int_{t_1}
\frac{\Gamma_{\bR}(t_1+\nu_1)\Gamma_{\bR}(t_1+\nu_2)
\Gamma_{\bR}(t_1+\nu_3)}
{\Gamma_{\bR}(t_1+t_2+1)}\\
&\phantom{=}\times 
\Gamma_{\bR}(t_2-\nu_1+1)\Gamma_{\bR}(t_2-\nu_2+1)
\Gamma_{\bR}(t_2-\nu_3+1)
\Gamma_\bR (t+\nu_1'+1-\delta_2)\\
&\phantom{=}\times 
\Gamma_\bR (t+\nu_2'+\delta_2)y_1^{-t_1-t} y_2^{-t_2} 
dt_1dt_2dt\biggr\}
y_1^{s}y_2^{2s+\nu_1+\nu_2+\nu_3+\nu_1'+\nu_2'}
\frac{dy_1}{y_1}\frac{dy_2}{y_2}.
\end{align*}
Substituting $t_1\to t_1-t$ and applying Lemma \ref{lem:F32_Mellin} twice, 
we have 
\begin{align*}
&Z\bigl(s,\varphi_{\sigma ,(1,0)}^{[\varepsilon ]}(v),
\varphi_{\sigma'}^{[-\varepsilon ]}(v')\bigr)\\
&=\langle v,v'\rangle 
\Biggl\{\prod_{1\leq i<j\leq 3}
\Gamma_{\bR}(2s+\nu_i+\nu_j+\nu_1'+\nu_2'+1)\Biggr\}\\
&\phantom{=}\times 
\frac{1}{4\pi \sI}\int_t
\frac{\Gamma_\bR (t+\nu_1'+1-\delta_2)
\Gamma_\bR (t+\nu_2'+\delta_2)}
{\Gamma_{\bR}(-t+3s+\nu_1+\nu_2+\nu_3+\nu_1'+\nu_2'+1)}\\
&\phantom{=}\times 
\Gamma_{\bR}(-t+s+\nu_1)\Gamma_{\bR}(-t+s+\nu_2)
\Gamma_{\bR}(-t+s+\nu_3) dt.
\end{align*}
By Lemma \ref{lem:F32_Barnes2nd}, we have 
\begin{align*}
&Z\bigl(s,\varphi_{\sigma ,(0,\delta_2')}^{[\varepsilon ]}(v),
\varphi_{\sigma'}^{[-\varepsilon ]}(v')\bigr)\\
&=\langle v,v'\rangle 
\prod_{i=1}^3
\Gamma_{\bR}(s+\nu_i+\nu_1'+1-\delta_2)
\Gamma_{\bR}(s+\nu_i+\nu_2'+\delta_2).
\end{align*}
Hence we obtain the assertion in this case. \vspace{2mm}

\underline{\bf Case 2-2:} $\sigma = \chi_{(\nu_1,1)}\boxtimes 
\chi_{(\nu_2,\delta_2)}\boxtimes \chi_{(\nu_3,0)}$. 
Since 
\begin{align*}
&\varphi^{[\varepsilon ]}_{\sigma ,(1,0)}(v_{(1,0),1}+v_{(1,0),-1})(\hat{y})
=-2\varepsilon (\sqrt{-1})^{1-\delta_2}
\varphi^{[\varepsilon ]}_\sigma (u_{(\delta_2,1-\delta_2,0)})(\hat{y})\\
&=
\frac{y_1y_2^{\nu_1+\nu_2+\nu_3+1}}{2(2\pi \sqrt{-1})^2} 
\int_{t_2}\int_{t_1}
\frac{\Gamma_{\bR}(t_1+\nu_1+1)\Gamma_{\bR}(t_1+\nu_2+\delta_2)
\Gamma_{\bR}(t_1+\nu_3)}
{\Gamma_{\bR}(t_1+t_2+\delta_2)}\\
&\phantom{=}\times 
\Gamma_{\bR}(t_2-\nu_1+1-\delta_2)\Gamma_{\bR}(t_2-\nu_2)
\Gamma_{\bR}(t_2-\nu_3+\delta_2)y_1^{-t_1} y_2^{-t_2} 
dt_1dt_2
\end{align*}
and 
\begin{align*}
&\varepsilon 
\varphi^{[\varepsilon ]}_{\sigma ,(1,0)}(v_{(1,0),1}-v_{(1,0),-1})(\hat{y})
=-2(\sqrt{-1})^{\delta_2}
\varphi^{[\varepsilon ]}_\sigma (u_{(1-\delta_2,\delta_2,0)})(\hat{y})\\
&=\frac{y_1y_2^{\nu_1+\nu_2+\nu_3+1}}{2(2\pi \sqrt{-1})^2} \int_{t_2}\int_{t_1}
\frac{\Gamma_{\bR}(t_1+\nu_1)\Gamma_{\bR}(t_1+\nu_2+1-\delta_2)
\Gamma_{\bR}(t_1+\nu_3+1)}
{\Gamma_{\bR}(t_1+t_2+1-\delta_2)}\\
&\phantom{=}\times 
\Gamma_{\bR}(t_2-\nu_1+1-\delta_2)\Gamma_{\bR}(t_2-\nu_2)
\Gamma_{\bR}(t_2-\nu_3+\delta_2)y_1^{-t_1} y_2^{-t_2} 
dt_1dt_2
\end{align*}
for $\hat{y}=\diag (y_1y_2,y_2,1)\in A_3$, 
we have 
\begin{align*}
&Z\bigl(s,\varphi_{\sigma ,(1,0)}^{[\varepsilon ]}(v),
\varphi_{\sigma'}^{[-\varepsilon ]}(v')\bigr)\\
&=\frac{\langle v,v'\rangle}{2(2\pi \sI)^3}
\int_{0}^\infty \int_{0}^\infty 
\biggl\{\int_t
\int_{t_2}\int_{t_1}
\biggl(
\frac{\Gamma_{\bR}(t_1+\nu_1+1)\Gamma_{\bR}(t_1+\nu_2+\delta_2)}
{\Gamma_{\bR}(t_1+t_2+\delta_2)}\\
&\phantom{=}\times \Gamma_{\bR}(t_1+\nu_3)
\Gamma_\bR (t+\nu_1'+1)\Gamma_\bR (t+\nu_2')\\
&\phantom{=}+
\frac{\Gamma_{\bR}(t_1+\nu_1)\Gamma_{\bR}(t_1+\nu_2+1-\delta_2)
\Gamma_{\bR}(t_1+\nu_3+1)\Gamma_\bR (t+\nu_1')}
{\Gamma_{\bR}(t_1+t_2+1-\delta_2)}\\
&\phantom{=}\times \Gamma_\bR (t+\nu_2'+1)\biggr)
\Gamma_{\bR}(t_2-\nu_1+1-\delta_2)\Gamma_{\bR}(t_2-\nu_2)
\Gamma_{\bR}(t_2-\nu_3+\delta_2)\\
&\phantom{=}\times 
y_1^{-t_1-t} y_2^{-t_2} 
dt_1dt_2dt\biggr\}y_1^{s}y_2^{2s+\nu_1+\nu_2+\nu_3+\nu_1'+\nu_2'}
\frac{dy_1}{y_1}\frac{dy_2}{y_2}.
\end{align*}
Substituting $t_1\to t_1-t$ and applying Lemma \ref{lem:F32_Mellin} twice, 
we have 
\begin{align*}
&Z\bigl(s,\varphi_{\sigma ,(1,0)}^{[\varepsilon ]}(v),
\varphi_{\sigma'}^{[-\varepsilon ]}(v')\bigr)\\
&=\langle v,v'\rangle \Gamma_{\bR}(2s+\nu_1+\nu_3+\nu_1'+\nu_2')\\
&\phantom{=}\times 
\Gamma_{\bR}(2s+\nu_2+\nu_3+\nu_1'+\nu_2'+1-\delta_2)
\Gamma_{\bR}(2s+\nu_1+\nu_2+\nu_1'+\nu_2'+\delta_2)\\
&\phantom{=}\times 
\frac{1}{4\pi \sI}
\int_t
\biggl(
\frac{\Gamma_\bR (t+\nu_1'+1)\Gamma_\bR (t+\nu_2')
\Gamma_{\bR}(-t+s+\nu_1+1)}
{\Gamma_{\bR}(-t+3s+\nu_1+\nu_2+\nu_3+\nu_1'+\nu_2'+\delta_2)}\\
&\phantom{=}\times \Gamma_{\bR}(-t+s+\nu_2+\delta_2)
\Gamma_{\bR}(-t+s+\nu_3)\\
&\phantom{=}+
\frac{\Gamma_\bR (t+\nu_1')\Gamma_\bR (t+\nu_2'+1)
\Gamma_{\bR}(-t+s+\nu_1)}
{\Gamma_{\bR}(-t+3s+\nu_1+\nu_2+\nu_3+\nu_1'+\nu_2'+1-\delta_2)}\\
&\phantom{=}\times \Gamma_{\bR}(-t+s+\nu_2+1-\delta_2)
\Gamma_{\bR}(-t+s+\nu_3+1)\biggr)dt.
\end{align*}
Applying Lemma \ref{lem:R32_Barnes2nd_sum} for 
\begin{align*}
\left\{\begin{array}{llllll}
a_1=\nu_2',&a_2=\nu_1',&
b_1=s+\nu_3,&b_2=s+\nu_2,&b_3=s+\nu_1&\text{if}\ \delta_2=0,\\
a_1=\nu_1',&a_2=\nu_2',&
b_1=s+\nu_1,&b_2=s+\nu_2,&b_3=s+\nu_3&\text{if}\ \delta_2=1,
\end{array}\right.
\end{align*}
we have 
\begin{align*}
&Z\bigl(s,\varphi_{\sigma ,(1,0)}^{[\varepsilon ]}(v),
\varphi_{\sigma'}^{[-\varepsilon ]}(v')\bigr)\\
&=\langle v,v'\rangle 
\Gamma_{\bR}(s+\nu_1+\nu_1')
\Gamma_{\bR}(s+\nu_2+\nu_1'+1-\delta_2)
\Gamma_{\bR}(s+\nu_3+\nu_1'+1)\\
&\phantom{=}\times 
\Gamma_{\bR}(s+\nu_1+\nu_2'+1)
\Gamma_{\bR}(s+\nu_2+\nu_2'+\delta_2)
\Gamma_{\bR}(s+\nu_3+\nu_2').
\end{align*}
Hence we obtain the assertion in this case. \vspace{2mm}

\underline{\bf Case 2-3:} $\sigma = D_{(\nu_1,\kappa_1)}\boxtimes 
\chi_{(\nu_2,\delta_2)}$. Since 
\begin{align*}
&\varphi^{[\varepsilon ]}_{\sigma ,(1,0)}(v_{(1,0),1}+v_{(1,0),-1})(\hat{y})
=-2\varepsilon (\sqrt{-1})^{1-\delta_2}\varphi^{[\varepsilon ]}_\sigma 
(u_{(\delta_2,1-\delta_2,\kappa_1-1)})(\hat{y})\\
&=\frac{y_1y_2^{2\nu_1+\nu_2+1}}{2(2\pi \sqrt{-1})^2} 
\int_{t_2}\int_{t_1}
\frac{\Gamma_{\bC}\bigl(t_1+\nu_1+\tfrac{\kappa_1-1}{2}\bigr)
\Gamma_{\bR}(t_1+\nu_2+\delta_2)}
{ \Gamma_{\bR}(t_1+t_2+\kappa_1-1+\delta_2)}\\
&\phantom{=}\times \Gamma_{\bC}\bigl(t_2-\nu_1+\tfrac{\kappa_1-1}{2}\bigr)
\Gamma_{\bR}(t_2-\nu_2+\kappa_1-1)y_1^{-t_1} y_2^{-t_2} \,dt_1dt_2
\end{align*}
and 
\begin{align*}
&\varepsilon 
\varphi^{[\varepsilon ]}_{\sigma ,(1,0)}(v_{(1,0),1}-v_{(1,0),-1})(\hat{y})
=-2(\sqrt{-1})^{\delta_2}\varphi^{[\varepsilon ]}_\sigma 
(u_{(1-\delta_2,\delta_2,\kappa_1-1)})(\hat{y})\\
&=\frac{y_1y_2^{2\nu_1+\nu_2+1}}{2(2\pi \sqrt{-1})^2} \int_{t_2}\int_{t_1}
\frac{\Gamma_{\bC}\bigl(t_1+\nu_1+\tfrac{\kappa_1-1}{2}\bigr)
\Gamma_{\bR}(t_1+\nu_2+1-\delta_2)}
{ \Gamma_{\bR}(t_1+t_2+\kappa_1-\delta_2)}\\
&\phantom{=}\times \Gamma_{\bC}\bigl(t_2-\nu_1+\tfrac{\kappa_1-1}{2}\bigr)
\Gamma_{\bR}(t_2-\nu_2+\kappa_1-1)y_1^{-t_1} y_2^{-t_2} \,dt_1dt_2
\end{align*}
for $\hat{y}=\diag (y_1y_2,y_2,1)\in A_3$, 
we have 
\begin{align*}
&Z\bigl(s,\varphi_{\sigma ,(1,0)}^{[\varepsilon ]}(v),
\varphi_{\sigma'}^{[-\varepsilon ]}(v')\bigr)\\
&=\frac{\langle v,v'\rangle}{2(2\pi \sI)^3}
\int_{0}^\infty \int_{0}^\infty 
\biggl\{\int_t\int_{t_2}\int_{t_1}
\biggl(
\frac{
\Gamma_{\bR}(t_1+\nu_2+\delta_2)\Gamma_\bR (t+\nu_1'+1)\Gamma_\bR (t+\nu_2')}
{ \Gamma_{\bR}(t_1+t_2+\kappa_1-1+\delta_2)}\\
&\phantom{=}+
\frac{\Gamma_{\bR}(t_1+\nu_2+1-\delta_2)
\Gamma_\bR (t+\nu_1')\Gamma_\bR (t+\nu_2'+1)}
{ \Gamma_{\bR}(t_1+t_2+\kappa_1-\delta_2)}\biggr)\\
&\phantom{=}\times 
\Gamma_{\bC}\bigl(t_1+\nu_1+\tfrac{\kappa_1-1}{2}\bigr)
\Gamma_{\bC}\bigl(t_2-\nu_1+\tfrac{\kappa_1-1}{2}\bigr)
\Gamma_{\bR}(t_2-\nu_2+\kappa_1-1)\\
&\phantom{=}\times 
y_1^{-t_1-t} y_2^{-t_2} \,dt_1dt_2dt\biggr\}
y_1^{s}y_2^{2s+2\nu_1+\nu_2+\nu_1'+\nu_2'}
\frac{dy_1}{y_1}\frac{dy_2}{y_2}.
\end{align*}
Substituting $t_1\to t_1-t$ and applying Lemma \ref{lem:F32_Mellin} twice, 
we have 
\begin{align*}
&Z\bigl(s,\varphi_{\sigma ,(1,0)}^{[\varepsilon ]}(v),
\varphi_{\sigma'}^{[-\varepsilon ]}(v')\bigr)\\
&=\langle v,v'\rangle
\Gamma_{\bC}\bigl(2s+\nu_1+\nu_2+\nu_1'+\nu_2'+\tfrac{\kappa_1-1}{2}\bigr)
\Gamma_{\bR}(2s+2\nu_1+\nu_1'+\nu_2'+\kappa_1-1)\\
&\phantom{=}\times 
\frac{1}{4\pi \sI}
\int_t\biggl(\frac{\Gamma_\bR (t+\nu_1'+1)\Gamma_\bR (t+\nu_2')
\Gamma_{\bR}(-t+s+\nu_2+\delta_2)}
{\Gamma_{\bR}(-t+3s+2\nu_1+\nu_2+\nu_1'+\nu_2'+\kappa_1-1+\delta_2)}\\
&\phantom{=}+
\frac{\Gamma_\bR (t+\nu_1')\Gamma_\bR (t+\nu_2'+1)
\Gamma_{\bR}(-t+s+\nu_2+1-\delta_2)}
{ \Gamma_{\bR}(-t+3s+2\nu_1+\nu_2+\nu_1'+\nu_2'+\kappa_1-\delta_2)}\biggr)\\
&\phantom{=}\times 
\Gamma_{\bC}\bigl(-t+s+\nu_1+\tfrac{\kappa_1-1}{2}\bigr)dt.
\end{align*}
Applying Lemma \ref{lem:R32_Barnes2nd_sum} for 
\begin{align*}
&\left\{\begin{array}{lll}
a_1=\nu_2',&a_2=\nu_1'
&\text{if}\ \delta_2=0,\\
a_1=\nu_1',&a_2=\nu_2'
&\text{if}\ \delta_2=1,
\end{array}\right.&
&b_1=s+\nu_2,&
&b_2=b_3=s+\nu_1+\tfrac{\kappa_1-1}{2}
\end{align*}
with the duplication formula (\ref{eqn:Fn_gammaRC_duplication}), 
we have 
\begin{align*}
&Z\bigl(s,\varphi_{\sigma ,(1,0)}^{[\varepsilon ]}(v),
\varphi_{\sigma'}^{[-\varepsilon ]}(v')\bigr)\\
&=\langle v,v'\rangle
\Gamma_{\bC}\bigl(s+\nu_1+\nu_1'+\tfrac{\kappa_1-1}{2}\bigr)
\Gamma_{\bR}(s+\nu_2+\nu_1'+1-\delta_2)\\
&\phantom{=}\times 
\Gamma_{\bC}\bigl(s+\nu_1+\nu_2'+\tfrac{\kappa_1-1}{2}\bigr)
\Gamma_{\bR}(s+\nu_2+\nu_2'+\delta_2).
\end{align*}
Hence we obtain the assertion in this case. \vspace{2mm}

\section{The calculation for $\sigma'=D_{(\nu',\kappa')}$}

In this section, we give a proof of 
Theorem \ref{thm:R32_zeta_L_coincide} for 
$\sigma'=D_{(\nu',\kappa')}$. 
By Lemma \ref{lem:R32_zeta32_schur} and 
\begin{align*}
&\varphi_{\sigma}^{[-\varepsilon ]}
(v_{(\kappa',0),-\varepsilon \kappa'})(y)
=
\frac{y_1^{1/2}y_2^{2\nu'}}{2\pi \sI}
\int_s \Gamma_\bC \bigl(t+\nu' +\tfrac{\kappa' -1}{2}\bigr)y_1^{-t}dt,\\
&\varphi_{\sigma}^{[-\varepsilon ]}
(v_{(\kappa',0),\varepsilon \kappa' })(y)=0
\end{align*}
for $y=\diag (y_1y_2,y_2)\in A_2$, 
we have 
\begin{align*}
&Z\bigl(s,\varphi_{\sigma ,(\kappa',0)}^{[\varepsilon ]}(v),
\varphi_{\sigma'}^{[-\varepsilon ]}(v')\bigr)\\
&=\frac{\langle v,v'\rangle}{2}
\sum_{q\in \{\pm \kappa'\}}
\int_{0}^\infty \int_{0}^\infty   
\varphi_{\sigma ,(\kappa',0)}^{[\varepsilon ]}(v_{(\kappa',0),q})(\hat{y})
\varphi_{\sigma'}^{[-\varepsilon ]}(v_{(\kappa',0),-q})(y)\\
&\phantom{=}\times y_1^{s-\frac{3}{2}}y_2^{2s-1}
\frac{2dy_1}{y_1}\frac{2dy_2}{y_2}\\
&=\frac{\langle v,v'\rangle}{4\pi \sI}
\int_{0}^\infty \int_{0}^\infty   
\biggl\{\int_t\varphi_{\sigma ,(\kappa',0)}^{[\varepsilon ]}
(v_{(\kappa',0),\varepsilon \kappa'})(\hat{y})
\Gamma_\bC \bigl(t+\nu' +\tfrac{\kappa' -1}{2}\bigr)y_1^{-t}dt\biggr\}\\
&\phantom{=}\times  
y_1^{s-1}y_2^{2s+2\nu'-1}
\frac{2dy_1}{y_1}\frac{2dy_2}{y_2}
\end{align*}
with $\hat{y}=\diag (y_1y_2,y_2,1)\in A_3$. \vspace{2mm}

\underline{\bf Case 3-1:}  
$\sigma = \chi_{(\nu_1,\delta_2)}\boxtimes 
\chi_{(\nu_2,\delta_2)}\boxtimes \chi_{(\nu_3,\delta_2)}$. 
Since 
\begin{align*}
&\varphi_{\sigma ,(\kappa',0)}^{[\varepsilon ]}
(v_{(\kappa',0),\varepsilon \kappa'})(\hat{y})
=\varepsilon^{\delta_2}y_2^{\kappa'}
\varphi^{[\varepsilon ]}_\sigma (u_{(0,0,0)})(\hat{y})\\
&=
\frac{y_1y_2^{\nu_1+\nu_2+\nu_3+\kappa'+1}}
{(4\pi \sqrt{-1})^2} \int_{t_2}\int_{t_1}
\frac{\Gamma_{\bR}(t_1+\nu_1)\Gamma_{\bR}(t_1+\nu_2)
\Gamma_{\bR}(t_1+\nu_3)}{\Gamma_{\bR}(t_1+t_2)}\\
&\phantom{..}\times 
\Gamma_{\bR}(t_2-\nu_1)\Gamma_{\bR}(t_2-\nu_2)
\Gamma_{\bR}(t_2-\nu_3)y_1^{-t_1} y_2^{-t_2} 
dt_1dt_2
\end{align*}
for $\hat{y}=\diag (y_1y_2,y_2,1)\in A_3$, we have 
\begin{align*}
&Z\bigl(s,\varphi_{\sigma ,(\kappa',0)}^{[\varepsilon ]}(v),
\varphi_{\sigma'}^{[-\varepsilon ]}(v')\bigr)\\
&=\frac{\langle v,v'\rangle}{2(2\pi \sI)^3}
\int_{0}^\infty \int_{0}^\infty   
\biggl\{\int_t\int_{t_2}\int_{t_1}
\frac{\Gamma_{\bR}(t_1+\nu_1)\Gamma_{\bR}(t_1+\nu_2)
\Gamma_{\bR}(t_1+\nu_3)}{\Gamma_{\bR}(t_1+t_2)}\\
&\phantom{=}\times 
\Gamma_{\bR}(t_2-\nu_1)\Gamma_{\bR}(t_2-\nu_2)
\Gamma_{\bR}(t_2-\nu_3)
\Gamma_\bC \bigl(t+\nu' +\tfrac{\kappa' -1}{2}\bigr)\\
&\phantom{=}\times 
y_1^{-t_1-t} y_2^{-t_2}
dt_1dt_2dt\biggr\}
y_1^{s}y_2^{2s+\nu_1+\nu_2+\nu_3+2\nu'+\kappa'}
\frac{dy_1}{y_1}\frac{dy_2}{y_2}.
\end{align*}
Substituting $t_1\to t_1-t$ and applying Lemma \ref{lem:F32_Mellin} twice, 
we have 
\begin{align*}
&Z\bigl(s,\varphi_{\sigma ,(\kappa',0)}^{[\varepsilon ]}(v),
\varphi_{\sigma'}^{[-\varepsilon ]}(v')\bigr)\\
&=\langle v,v'\rangle 
\Biggl\{\prod_{1\leq i<j\leq 3}
\Gamma_{\bR}(2s+\nu_i+\nu_j+2\nu'+\kappa')
\Biggr\}\frac{1}{4\pi \sI}\\
&\phantom{=}\times 
\int_t
\frac{\Gamma_\bC \bigl(t+\nu' +\tfrac{\kappa' -1}{2}\bigr)
\Gamma_{\bR}(-t+s+\nu_1)\Gamma_{\bR}(-t+s+\nu_2)
\Gamma_{\bR}(-t+s+\nu_3)}
{\Gamma_{\bR}(-t+3s+\nu_1+\nu_2+\nu_3+2\nu'+\kappa')}
dt.
\end{align*}
Applying Lemma \ref{lem:F32_Barnes2nd} 
with the duplication formula (\ref{eqn:Fn_gammaRC_duplication}), 
we have 
\begin{align*}
&Z\bigl(s,\varphi_{\sigma ,(\kappa',0)}^{[\varepsilon ]}(v),
\varphi_{\sigma'}^{[-\varepsilon ]}(v')\bigr)
=\langle v,v'\rangle 
\prod_{i=1}^3
\Gamma_\bC \bigl(s+\nu_i+\nu' +\tfrac{\kappa' -1}{2}\bigr).
\end{align*}
Hence we obtain the assertion in this case. \vspace{2mm}

\underline{\bf Case 3-2:}  
$\sigma = \chi_{(\nu_1,1)}\boxtimes 
\chi_{(\nu_2,\delta_2)}\boxtimes \chi_{(\nu_3,0)}$. 
Since 
\begin{align*}
&\varphi_{\sigma ,(\kappa',0)}^{[\varepsilon ]}
(v_{(\kappa',0),\varepsilon \kappa'})(\hat{y})
=-\varepsilon^{\delta_2}y_2^{\kappa'-1}
\{\varphi^{[\varepsilon ]}_\sigma (u_{(1,0,0)})(\hat{y})
+\varepsilon \sqrt{-1}
\varphi^{[\varepsilon ]}_\sigma (u_{(0,1,0)})(\hat{y})\}\\
&=
\frac{y_1y_2^{\nu_1+\nu_2+\nu_3+\kappa'}}{(4\pi \sqrt{-1})^2} 
\int_{t_2}\int_{t_1}
\biggl(
\frac{\Gamma_{\bR}(t_1+\nu_1+\delta_2)\Gamma_{\bR}(t_1+\nu_2+1)
\Gamma_{\bR}(t_1+\nu_3+1-\delta_2)}
{\Gamma_{\bR}(t_1+t_2+1)}\\
&\phantom{=}+\frac{\Gamma_{\bR}(t_1+\nu_1+1-\delta_2)
\Gamma_{\bR}(t_1+\nu_2)
\Gamma_{\bR}(t_1+\nu_3+\delta_2)}
{\Gamma_{\bR}(t_1+t_2)}\biggr)\\
&\phantom{=}\times 
\Gamma_{\bR}(t_2-\nu_1+1-\delta_2)
\Gamma_{\bR}(t_2-\nu_2)
\Gamma_{\bR}(t_2-\nu_3+\delta_2)
y_1^{-t_1} y_2^{-t_2} 
dt_1dt_2
\end{align*}
for $\hat{y}=\diag (y_1y_2,y_2,1)\in A_3$, we have 
\begin{align*}
&Z\bigl(s,\varphi_{\sigma ,(\kappa',0)}^{[\varepsilon ]}(v),
\varphi_{\sigma'}^{[-\varepsilon ]}(v')\bigr)\\
&=\frac{\langle v,v'\rangle}{2(2\pi \sI)^3}
\int_{0}^\infty \int_{0}^\infty   
\biggl\{\int_t
\int_{t_2}\int_{t_1}
\biggl(
\frac{\Gamma_{\bR}(t_1+\nu_1+\delta_2)\Gamma_{\bR}(t_1+\nu_2+1)}
{\Gamma_{\bR}(t_1+t_2+1)}\\
&\phantom{=}\times \Gamma_{\bR}(t_1+\nu_3+1-\delta_2)
+\frac{\Gamma_{\bR}(t_1+\nu_1+1-\delta_2)
\Gamma_{\bR}(t_1+\nu_2)
\Gamma_{\bR}(t_1+\nu_3+\delta_2)}
{\Gamma_{\bR}(t_1+t_2)}\biggr)\\
&\phantom{=}\times 
\Gamma_{\bR}(t_2-\nu_1+1-\delta_2)
\Gamma_{\bR}(t_2-\nu_2)
\Gamma_{\bR}(t_2-\nu_3+\delta_2)
\Gamma_\bC \bigl(t+\nu' +\tfrac{\kappa' -1}{2}\bigr)\\
&\phantom{=}\times  
y_1^{-t_1-t} y_2^{-t_2} dt_1dt_2dt\biggr\}
y_1^{s}y_2^{2s+\nu_1+\nu_2+\nu_3+2\nu'+\kappa'-1}
\frac{dy_1}{y_1}\frac{dy_2}{y_2}.
\end{align*}
Substituting $t_1\to t_1-t$ and applying Lemma \ref{lem:F32_Mellin} twice, 
we have 
\begin{align*}
&Z\bigl(s,\varphi_{\sigma ,(\kappa',0)}^{[\varepsilon ]}(v),
\varphi_{\sigma'}^{[-\varepsilon ]}(v')\bigr)\\
&=\langle v,v'\rangle
\Gamma_{\bR}(2s+\nu_2+\nu_3+2\nu'+\kappa'-\delta_2)
\Gamma_{\bR}(2s+\nu_1+\nu_3+2\nu'+\kappa'-1)\\
&\phantom{=}\times 
\Gamma_{\bR}(2s+\nu_1+\nu_2+2\nu'+\kappa'-1+\delta_2)
\frac{1}{4\pi \sI}
\int_t\Gamma_\bC \bigl(t+\nu' +\tfrac{\kappa' -1}{2}\bigr)\\
&\phantom{=}\times 
\biggl(
\frac{\Gamma_{\bR}(-t+s+\nu_1+\delta_2)\Gamma_{\bR}(-t+s+\nu_2+1)
\Gamma_{\bR}(-t+s+\nu_3+1-\delta_2)}
{\Gamma_{\bR}(-t+3s+\nu_1+\nu_2+\nu_3+2\nu'+\kappa')}\\
&\phantom{=}
+\frac{\Gamma_{\bR}(-t+s+\nu_1+1-\delta_2)
\Gamma_{\bR}(-t+s+\nu_2)
\Gamma_{\bR}(-t+s+\nu_3+\delta_2)}
{\Gamma_{\bR}(-t+3s+\nu_1+\nu_2+\nu_3+2\nu'+\kappa'-1)}\biggr)dt.
\end{align*}
Applying Lemma \ref{lem:R32_Barnes2nd_sum} for 
\begin{align*}
&a_1=a_2=\nu' +\tfrac{\kappa' -1}{2},&
\left\{\begin{array}{llll}
b_1=s+\nu_3,&b_2=s+\nu_2,&b_2=s+\nu_1
&\text{if}\ \delta_2=0,\\
b_1=s+\nu_1,&b_2=s+\nu_2,&b_2=s+\nu_3
&\text{if}\ \delta_2=1
\end{array}\right.
\end{align*}
with the duplication formula (\ref{eqn:Fn_gammaRC_duplication}), 
we have 
\begin{align*}
&Z\bigl(s,\varphi_{\sigma ,(\kappa',0)}^{[\varepsilon ]}(v),
\varphi_{\sigma'}^{[-\varepsilon ]}(v')\bigr)
=\langle v,v'\rangle 
\prod_{i=1}^3
\Gamma_\bC \bigl(s+\nu_i+\nu' +\tfrac{\kappa' -1}{2}\bigr).
\end{align*}
Hence we obtain the assertion in this case. \vspace{2mm}

\underline{\bf Case 3-3:} $\sigma = D_{(\nu_1,\kappa_1)}\boxtimes 
\chi_{(\nu_2,\delta_2)}$. Let $\kappa =\min \{\kappa_1,\kappa'\}$. 
Since 
\begin{align*}
&\varphi_{\sigma ,(\kappa',0)}^{[\varepsilon ]}
(v_{(\kappa',0),\varepsilon \kappa'})(\hat{y})
=\varepsilon^{\delta_2}(-1)^{{\kappa}}
\sum_{j=0}^{{\kappa}}\binom{{\kappa}}{j}(\varepsilon \sqrt{-1})^j
y_2^{\kappa'-{\kappa}}
\varphi^{[\varepsilon ]}_\sigma 
(u_{({\kappa}-j,j,\kappa_1-{\kappa})})(\hat{y})\\
&=
\sum_{j=0}^{{\kappa}}\binom{{\kappa}}{j}
\frac{y_1y_2^{2\nu_1+\nu_2+\kappa'-{\kappa}+1}}{(4\pi \sqrt{-1})^2}
\int_{t_2}\int_{t_1}
\frac{\Gamma_{\bC}\bigl(t_1+\nu_1+\tfrac{\kappa_1-1}{2}\bigr)
\Gamma_{\bR}(t_1+\nu_2+{\kappa}-j)}
{ \Gamma_{\bR}(t_1+t_2+\kappa_1-j)}\\
&\phantom{=}
\times \Gamma_{\bC}\bigl(t_2-\nu_1+\tfrac{\kappa_1-1}{2}\bigr)
\Gamma_{\bR}(t_2-\nu_2+\kappa_1-{\kappa})y_1^{-t_1} y_2^{-t_2} \,dt_1dt_2,
\end{align*}
we have 
\begin{align*}
&Z\bigl(s,\varphi_{\sigma ,(\kappa',0)}^{[\varepsilon ]}(v),
\varphi_{\sigma'}^{[-\varepsilon ]}(v')\bigr)\\
&=\langle v,v'\rangle 
\sum_{j=0}^{{\kappa}}\binom{{\kappa}}{j}
\frac{1}{2(2\pi \sI)^3}
\int_{0}^\infty \int_{0}^\infty   
\biggl\{\int_t
\int_{t_2}\int_{t_1}
\frac{\Gamma_{\bC}\bigl(t_1+\nu_1+\tfrac{\kappa_1-1}{2}\bigr)}
{ \Gamma_{\bR}(t_1+t_2+\kappa_1-j)}\\
&\phantom{=}
\times \Gamma_{\bR}(t_1+\nu_2+{\kappa}-j)
\Gamma_{\bC}\bigl(t_2-\nu_1+\tfrac{\kappa_1-1}{2}\bigr)
\Gamma_{\bR}(t_2-\nu_2+\kappa_1-{\kappa})\\
&\phantom{=}\times 
\Gamma_\bC \bigl(t+\nu' +\tfrac{\kappa' -1}{2}\bigr)
y_1^{-t_1-t} y_2^{-t_2} \,dt_1dt_2dt\biggr\}
y_1^{s}y_2^{2s+2\nu_1+\nu_2+2\nu'+\kappa'-{\kappa}}
\frac{dy_1}{y_1}\frac{dy_2}{y_2}.
\end{align*}
Substituting $t_1\to t_1-t$ and applying Lemma \ref{lem:F32_Mellin} twice, 
we have 
\begin{align*}
&Z\bigl(s,\varphi_{\sigma ,(\kappa',0)}^{[\varepsilon ]}(v),
\varphi_{\sigma'}^{[-\varepsilon ]}(v')\bigr)\\
&=\langle v,v'\rangle 
\Gamma_{\bR}(2s+2\nu_1+2\nu'+\kappa'+\kappa_1-2\kappa )\\
&\phantom{=}\times 
\Gamma_{\bC}\bigl(2s+\nu_1+\nu_2+2\nu'+\kappa'-{\kappa}
+\tfrac{\kappa_1-1}{2}\bigr)\\
&\phantom{=}\times 
\sum_{j=0}^{{\kappa}}\binom{{\kappa}}{j}
\frac{1}{4\pi \sI}
\int_t
\frac{\Gamma_\bC \bigl(t+\nu' +\tfrac{\kappa' -1}{2}\bigr)
\Gamma_{\bC}\bigl(-t+s+\nu_1+\tfrac{\kappa_1-1}{2}\bigr)}
{ \Gamma_{\bR}(-t+3s+2\nu_1+\nu_2+2\nu'+\kappa'+\kappa_1-{\kappa}-j)}\\
&\phantom{=}
\times \Gamma_{\bR}(-t+s+\nu_2+{\kappa}-j)\,dt.
\end{align*}
Substituting $t\to t-j$ and applying Lemma \ref{lem:F32_gauss_sum}, 
we have 
\begin{align*}
&Z\bigl(s,\varphi_{\sigma ,(\kappa',0)}^{[\varepsilon ]}(v),
\varphi_{\sigma'}^{[-\varepsilon ]}(v')\bigr)\\
&=\langle v,v'\rangle 
\Gamma_{\bR}(2s+2\nu_1+2\nu'+\kappa'+\kappa_1-2\kappa )\\
&\phantom{=}\times 
\frac{\Gamma_{\bC}\bigl(2s+\nu_1+\nu_2+2\nu'+\kappa'-{\kappa}
+\tfrac{\kappa_1-1}{2}\bigr)
\Gamma_{\bC}(s+\nu_1+\nu' +\tfrac{\kappa_1+\kappa' -2}{2})}
{\Gamma_{\bC}(s+\nu_1+\nu' +\tfrac{\kappa_1+\kappa' -2\kappa -2}{2})}\\
&\phantom{=}\times 
\frac{1}{4\pi \sI}
\int_t
\frac{\Gamma_{\bC}(t+\nu' +\tfrac{\kappa' -1}{2}-{\kappa})
\Gamma_{\bC}(-t+s+\nu_1+\tfrac{\kappa_1-1}{2})}
{ \Gamma_{\bR}(-t+3s+2\nu_1+\nu_2+2\nu'+\kappa'+\kappa_1-{\kappa})}\\
&\phantom{=}\times 
\Gamma_{\bR}(-t+s+\nu_2+{\kappa})\,dt.
\end{align*}
Applying Lemma \ref{lem:F32_Barnes2nd} 
with the duplication formula (\ref{eqn:Fn_gammaRC_duplication}), 
we have 
\begin{align*}
&Z\bigl(s,\varphi_{\sigma ,(\kappa',0)}^{[\varepsilon ]}(v),
\varphi_{\sigma'}^{[-\varepsilon ]}(v')\bigr)
=\langle v,v'\rangle 
\Gamma_{\bC}\bigl(s+\nu_1+\nu' +\tfrac{\kappa_1+\kappa' -2}{2}\bigr)\\
&\phantom{=======}\times 
\Gamma_{\bC}\bigl(s+\nu_1+\nu' +\tfrac{\kappa_1+\kappa' -2\kappa }{2}\bigr)
\Gamma_{\bC}\bigl(s+\nu_2+\nu' +\tfrac{\kappa' -1}{2}\bigr).
\end{align*}
Hence we obtain the assertion in this case. \vspace{2mm}

\chapter{The local zeta integrals for $GL(3,\bC)\times GL(2,\bC)$}
\label{sec:C32_zeta}

\section{The local Langlands correspondence for $GL(n,\bC)$}
\label{subsec:Cmn_landlands}

We recall the theory of finite dimensional semisimple representations 
of the Weil group $W_\bC$. 
The Weil group $W_\bC$ for the field $\bC$ is given by 
$W_\bC =\bC^\times =GL(1,\bC)$. 
Hence the characters $\chi_{(\nu ,d)}$ ($\nu \in \bC$, $d\in \bZ$) 
exhaust the irreducible representations of $W_\bC$, up to isomorphism. 
Moreover, the equivalence 
\begin{align}
\label{eqn:Cmn_weil_tensor}
&\chi_{(\nu ,d)}\otimes \chi_{(\nu',d')}
\simeq \chi_{(\nu+\nu',d+d')}
\end{align}
holds for $\nu,\nu'\in\bC$ and $d,d'\in \bZ$.

We define the $L$-factor 
corresponding to $\chi_{(\nu ,d)}$ by 
\begin{align*}
&L(s,\chi_{(\nu ,d)})=\Gamma_\bC \bigl(s+\nu +\tfrac{|d|}{2}\bigr).
\end{align*}
For a finite dimensional semisimple representation $\phi$ of $W_\bC$, 
we take an irreducible decomposition 
$\phi \simeq \bigoplus_{i=1}^m\chi_{(\nu_i,d_i)}$. 
We define its $L$-factor by 
\begin{align*}
&L(s,\phi )=\prod_{i=1}^mL(s,\chi_{(\nu_i,d_i)})
=\prod_{i=1}^m\Gamma_\bC \bigl(s+\nu_i+\tfrac{|d_i|}{2}\bigr).
\end{align*}

The Langlands classification of 
irreducible admissible representations of $G_n$ is 
given as follows:
\begin{thm}[{\cite[Theorem 4]{Knapp_003}}]
\label{thm:Cn_landlands_clas}
We use the notation in \S \ref{subsec:Cn_def_ps}. 

\noindent (1) If ${\mathrm{Re}(\nu_1)}\geq {\mathrm{Re}(\nu_2)}\geq 
\cdots \geq {\mathrm{Re}(\nu_n)}$, 
then 
$H(\chi )$ has 
a unique irreducible quotient 
$J(\chi )$. 

\noindent (2) Any irreducible admissible representation of $G_n$ 
is infinitesimally equivalent to $J(\chi )$ for some $\chi$. 

\noindent (3) A quotient $J(\chi)$ 
is infinitesimally equivalent to $J(\chi')$ 
if and only if there exists $w\in \gS_n$ such that 
$\chi'=w(\chi)$. 
\end{thm}

By Theorem \ref{thm:Cn_landlands_clas} (2), 
for an irreducible admissible representation $\Pi$ of $G_n$, 
there exists some $J(\chi)$ with 
$\chi =\chi_{(\nu_{1} ,d_{1})}\boxtimes 
\chi_{(\nu_{2} ,d_{2})}\boxtimes 
\cdots \boxtimes \chi_{(\nu_{n} ,d_{n})}$, 
which is infinitesimally equivalent to $\Pi$. 
Then we define the Langlands parameter $\phi [\Pi ]$ of $\Pi$ by 
$\phi [\Pi ]=\bigoplus_{i=1}^n\chi_{(\nu_i,d_i)}$. 
By Theorem \ref{thm:Cn_landlands_clas} (3), 
this definition is well-defined and 
we obtain the local Langlands correspondence over $\bC$: 
\begin{thm}[{\cite[Theorem 5]{Knapp_003}}]
\label{thm:Cn_landlands_corresp}
The correspondence $\Pi \leftrightarrow \phi [\Pi]$ gives 
a bijection between the set of infinitesimal 
equivalence classes of irreducible admissible representations 
of $G_n=GL(n,\bC)$ and 
the set of equivalence classes of $n$-dimensional 
semisimple representations of $W_\bC$. 
\end{thm}

Let $n$ and $n'$ be positive integers. 
Let $\Pi$ and $\Pi'$ be 
irreducible admissible representations of $G_n$ and $G_{n'}$, respectively. 
Then we define the local $L$-factor $L (s,\Pi \times \Pi')$ by 
\begin{align*}
&L (s,\Pi \times \Pi')
=L(s,\phi [\Pi] \otimes \phi[\Pi']).
\end{align*}
Let $\Pi_\chi $ and $\Pi_{\chi'}$ be 
irreducible principal series representations of $G_3$ and $G_2$, respectively, 
where 
\begin{align*}
&\chi =\chi_{(\nu_1,d_1)}\boxtimes \chi_{(\nu_2,d_2)}
\boxtimes \chi_{(\nu_3 ,d_3)},&
&\chi'=\chi_{(\nu_1',d_1')}\boxtimes \chi_{(\nu_2',d_2')}.
\end{align*}
Then the explicit forms of the local $L$-factor for 
$\Pi_\chi \times \Pi_{\chi'}$ are given by 
\begin{align*}
&L (s,\Pi_{\chi}\times \Pi_{\chi'})=
\prod_{i=1}^3\prod_{j=1}^2
\Gamma_\bC \Bigl(s+\nu_i+\nu_j'+\tfrac{|d_i+d_j'|}{2}\Bigr).
\end{align*}

\section{Preparations for $U(2)$-modules}
\label{subsec:C32_rep_of_K2}

We use the notation in \S \ref{subsec:C2_rep_K} and \S \ref{subsec:C3_rep_K}. 
For $\lambda =(\lambda_1,\lambda_2)\in \Lambda_2$, 
we set $\widetilde{\lambda}=(-\lambda_2,-\lambda_1)$. 
We define a $\bC$-bilinear pairing 
$\langle \cdot ,\cdot \rangle $ 
on $V_{\lambda}^{(2)}\otimes_\bC V_{\widetilde{\lambda}}^{(2)}$ by 
\begin{align}
\label{eqn:C32_def_K2_inv_pair}
&\langle v_{\lambda,q},\,v_{\widetilde{\lambda},q'}\rangle
=\frac{(-1)^{\lambda_1-q}}
{\binom{\lambda_1-\lambda_2}{q}}\,
\delta_{\lambda_1-\lambda_2,q'+q}&
&(q\in Q_{\lambda},\ q'\in Q_{\widetilde{\lambda}}). 
\end{align} 
By the equalities 
(\ref{eqn:C2_gKact11_22}), 
(\ref{eqn:C2_gKact12_21}) and the connectedness of $K_2$, 
we know that 
the pairing $\langle \cdot ,\cdot \rangle $ is $K_2$-invariant. 
Moreover, it holds that 
\begin{align}
&\langle v',v\rangle =\langle v,v'\rangle &
&(v\in V_{\lambda}^{(2)},\ v'\in V_{\widetilde{\lambda}}^{(2)}).
\end{align}

We regard $K_2$ as a subgroup of $K_3$ via the embedding 
\begin{align}
\label{eqn:C32_K2_to_K3}
K_2\ni k\mapsto \left(\begin{array}{c|c}
k&\\ \hline &1
\end{array}\right)\in K_3. 
\end{align}
Then the following lemma holds. 
\begin{lem}
\label{lem:C32_K32_restriction}
Let $\mu =(\mu_1,\mu_2,\mu_3)\in \Lambda_3$. 
Then it holds that  
\begin{align}
\label{eqn:C32_K32_restriction}
V^{(3)}_{\mu}
\simeq \bigoplus_{\lambda \in \Omega (\mu )}V^{(2)}_{\lambda }
\end{align}
as $K_2$-modules, 
where 
$\Omega (\mu ) =\{(\lambda_1,\lambda_2)\in \Lambda_2\mid 
\mu_1 \geq \lambda_1\geq \mu_2 \geq \lambda_2\geq \mu_3 \}$. 
For $\lambda =(\lambda_1,\lambda_2)\in \Omega (\mu )$, 
there is a $K_2$-homomorphism 
$\iota_{\lambda}^{\mu}\colon 
V^{(2)}_{\lambda }\to V^{(3)}_{\mu }$ such that, 
for $q\in Q_{\lambda}$, 
\begin{align}
\label{eqn:C32_K32_res_inj}
&\iota_{\lambda}^{\mu}(v_{\lambda ,q})
=\sum_{i=\max \{0,q-\lambda_1+\mu_2\}
}^{\min \{q,\mu_2-\lambda_2\}}
(-1)^i\binom{q}{i}\binom{\lambda_1-\lambda_2-q}{\mu_2-\lambda_2-i}
u_{l(\lambda ;q-i,i)}
\end{align}
with $l(\lambda;i,j)
=(\lambda_1-\mu_2-i,i,\mu_1-\lambda_1,j,\mu_2-\lambda_2-j,\lambda_2-\mu_3)
\in S_\mu$. 
\end{lem}
\begin{proof}
Let $\lambda =(\lambda_1,\lambda_2)\in \Omega (\mu )$. 
By the equality (\ref{eqn:C3_gkact}), we have 
\begin{align*}
&\tau_{\lambda}^{(3)}(E^{\gk_3}_{i,i})u_{l(\lambda;0,0)}
=\lambda_iu_{l(\lambda;0,0)}
\quad (i=1,2),&
&\tau_{\lambda}^{(3)}(E^{\gk_3}_{1,2})u_{l(\lambda;0,0)}=0,
\end{align*}
and know that $u_{l(\lambda;0,0)}$ is a highest weight vector 
in a $K_2$-module $V_{\mu}^{(3)}$ with highest weight $\lambda$. 
Hence, because of the highest weight theory, there is 
a $K_2$-homomorphism $\iota_{\lambda}^{\mu}\colon 
V^{(2)}_{\lambda }\to V^{(3)}_{\mu }$ such that
\begin{align*}
&\iota_{\lambda}^{\mu}(v_{\lambda ,q})
=\binom{\lambda_1-\lambda_2}{\mu_2-\lambda_2}
u_{l(\lambda ;0,0)}.
\end{align*}
Hence, we obtain (\ref{eqn:C32_K32_restriction}) by 
$\dim_\bC  V^{(3)}_{\mu}
=\sum_{\lambda \in \Omega (\mu )}\dim_\bC V^{(2)}_{\lambda }$. 
Using 
\begin{align}
&\iota^\mu_\lambda (v_{\lambda,q+1})
=(\lambda_1-\lambda_2-q)^{-1}\tau_{\lambda }^{(3)}(E_{2,1}^{\gk_3} )
\iota^\mu_\lambda (v_{\lambda,q})&
&(q\in Q_\lambda ),
\end{align}
we have (\ref{eqn:C32_K32_res_inj}), recursively. 
\end{proof}

For later use, we specify the irreducible components of 
the tensor product of irreducible representations of $K_2$. 
Let $m_1,m_2\in \bZ_{\geq 0}$ and $m_3\in \bZ$. 
Let $\mathcal{P}(m_1,m_2,m_3)$ be 
the $\bC$-vector space of 
polynomials of four variables 
$z_1$, $z_2$, $w_1$, $w_2$, 
which are 
degree $m_1$ homogeneous 
with respect to two variables $z_1$, $z_2$, 
and are degree $m_2$ homogeneous 
with respect to two variables $w_1$, $w_2$. 
We regard $\mathcal{P}(m_1,m_2,m_3)$ as 
a $K_2$-module via the action 
\begin{align*}
&(T_{(m_1,m_2,m_3)}^{(2,2)}(k)p)
(z_1,z_2,w_1,w_2)
=(\det k)^{m_3}
p((z_1,z_2)\cdot k,(w_1,w_2)\cdot k)
\end{align*}
for $k\in K_2$ and 
$p\in \mathcal{P}(m_1,m_2,m_3)$. 
Here $(z_1,z_2)\cdot k$ and $(w_1,w_2)\cdot k$ 
are the ordinal products of matrices. 
For $i\in \bZ_{\geq 0}$, there is a $K_2$-embedding 
\[
\mathrm{I}_{(m_1,m_2,m_3)}^{\mathcal{P},i}
\colon \mathcal{P}(m_1,m_2,m_3)\to 
\mathcal{P}(m_1+i,m_2+i,m_3-i)
\]
defined by $p(z_1,z_2,w_1,w_2)\mapsto 
(z_1w_2-z_2w_1)^ip(z_1,z_2,w_1,w_2)$. 
Moreover, for $\alpha =(\alpha_1,\alpha_2), 
\beta =(\beta_1,\beta_2)\in \Lambda_2$, 
there is a $K_2$-isomorphism 
\[
\iota_{\alpha ,\beta}^{\mathcal{P}} 
\colon V_\alpha^{(2)}\otimes_\bC V_\beta^{(2)}
\to \mathcal{P}(\alpha_1-\alpha_2,\beta_1-\beta_2,\alpha_2+\beta_2)
\]
defined by 
$p_1(z_1,z_2)\otimes p_2(z_1,z_2)\mapsto 
p_1(z_1,z_2)p_2(w_1,w_2)$. 

\begin{prop}
\label{prop:C32_tensor_K2}
Let $\alpha =(\alpha_1,\alpha_2), 
\beta =(\beta_1,\beta_2)\in \Lambda_2$. 
Then it holds that  
\begin{align}
\label{eqn:C32_tensor1}
V^{(2)}_{\alpha }\otimes_{\bC}V^{(2)}_{\beta }
\simeq \bigoplus_{\lambda \in \Omega (\alpha,\beta)}
V^{(2)}_{\lambda } 
\end{align}
as $K_2$-modules, where 
\begin{align*}
&\Omega (\alpha ,\beta )=
\{(\alpha_1+\beta_1-i,\alpha_2+\beta_2+i)\in \Lambda_2\mid 
0\leq i\leq \min \{\alpha_1-\alpha_2,\beta_1-\beta_2\}\}.
\end{align*}
For $\lambda =(\alpha_1+\beta_1-i,\alpha_2+\beta_2+i)
\in \Omega (\alpha ,\beta )$, 
there is a $K_2$-homomorphism 
$\mathrm{I}_{\lambda}^{\alpha ,\beta}\colon 
V^{(2)}_{\lambda }\to V^{(2)}_{\alpha }\otimes_{\bC}V^{(2)}_{\beta }$ 
such that, for $q\in Q_{\lambda}$, 
\begin{align*}
\mathrm{I}_{\lambda }^{\alpha ,\beta}
(v_{\lambda ,q})=
\sum_{j=\max \{0,q+i-\beta_1 +\beta_2\}}^{\min \{q+i,\alpha_1-\alpha_2\}}
C_{(\lambda ;q,j)}^{\alpha ,\beta}
v_{\alpha ,j}\otimes v_{\beta,q+i-j}
\end{align*}
with 
\begin{align*}
C_{(\lambda ;q,j)}^{\alpha ,\beta }
=&\sum_{k=\max \{0,j-q,i+j-\alpha_1+\alpha_2\}}^{\min \{i,j,
\beta_1-\beta_2-q-i+j\}}
(-1)^k\binom{i}{k}\binom{q}{j-k}\\
&\times 
\binom{\alpha_1+\beta_1-\alpha_2-\beta_2-2i-q}{\alpha_1-\alpha_2-i-j+k}.
\end{align*}
\end{prop}
\begin{proof}
For $\lambda =(\alpha_1+\beta_1-i,\alpha_2+\beta_2+i)
\in \Omega (\alpha ,\beta )$, 
the $K_2$-homomorphism $\mathrm{I}_{\lambda}^{\alpha ,\beta}$ 
in the statement is obtained as the composite 
\begin{align*}
(\iota_{\alpha,\beta}^{\mathcal{P}})^{-1}
\circ \mathrm{I}_{(\alpha_1-\alpha_2-i,
\beta_1-\beta_2-i,\alpha_2+\beta_2+i)}^{\mathcal{P},i}
\circ \iota_{(\alpha_1-i,\alpha_2),
(\beta_1,\beta_2+i)}^{\mathcal{P}} 
\circ I_\lambda^{(\alpha_1-i,\alpha_2),(\beta_1,\beta_2+i)}, 
\end{align*}
where $I_\lambda^{(\alpha_1-i,\alpha_2),(\beta_1,\beta_2+i)}$ is 
the $K_2$-homomorphism in Lemma \ref{lem:C2_tensor}. 
Moreover, since 
$\dim_\bC  (V^{(2)}_{\alpha }\otimes_{\bC}V^{(2)}_{\beta })=
\sum_{\lambda \in \Omega (\alpha,\beta)}
\dim_\bC  V^{(2)}_{\lambda }$, we obtain (\ref{eqn:C32_tensor1}). 
\end{proof}

\section{Whittaker functions on $GL(2,\bC )$}
\label{subsec:C32_Wh_GL2}

Let $\varepsilon \in \{\pm 1\}$, 
and take $\Xi_{(\varepsilon )}$ as (\ref{eqn:Fn_psi_change}). 
Let $\Pi_\chi $ be 
an irreducible principal series representation of $G_2$ 
with $\chi = \chi_{(\nu_1,d_1)}\boxtimes \chi_{(\nu_2,d_2)}$ 
($d_1\geq d_2$). 
Let $\Phi_{\chi}^{\rm mg}\in {\cI}_{\Pi_\chi ,\psi_1}^{\rm mg}$ 
be the homomorphism in Theorems \ref{thm:C2_ps_Whittaker} and 
\ref{thm:C2_ps_Whittaker2}. 
For $\lambda =(\lambda_1,\lambda_2)=(d_1+l,d_2-l)\in \Lambda (\chi)$, 
we define a $K_2$-homomorphism $\varphi_{\chi ,\lambda}^{[\varepsilon ]}
\colon V_{\lambda}^{(2)}\to 
\mathrm{Wh}(\Pi_\chi ,\psi_{\varepsilon })^{\mathrm{mg}}$ by 
\begin{align*}
\varphi_{\chi ,\lambda}^{[\varepsilon ]}=
\frac{(\sI)^{\lambda_1+l}}{\binom{d_1-d_2+2l}{l}}
\frac{\Gamma_\bC \bigl(\nu_1-\nu_2+1+\frac{d_1-d_2}{2}+l\bigr)}
{\Gamma _\bC \bigl(\nu_1-\nu_2+1+\frac{d_1-d_2}{2}\bigr)}\,
\Xi_{(\varepsilon )}\circ \Phi_{\chi}^{\rm mg}
\circ \hat{\eta}_{(\chi ;\lambda )},
\end{align*}
where $\hat{\eta}_{(\chi ;\lambda )}$ is defined by 
(\ref{eqn:C2_eta_chi_lambda}). 

\begin{prop}
\label{prop:C32_Wh_GL2}
Retain the notation. 
For $y=\diag (y_1y_2,y_2)\in A_2$ and $q\in Q_\lambda$, 
it holds that  
\begin{align}
\nonumber 
&\varphi_{\chi ,\lambda}^{[\varepsilon ]}(v_{\lambda ,q})(y)\\
\begin{split}
\label{eqn:C32_Wh_GL2_1st}
&\!=(\varepsilon \sI )^{\lambda_1-q}
\sum_{i=0}^{\min \{q,l\}}\binom{q}{i}
\dfrac{(-l)_i\bigl(-\nu_1+\nu_2-\frac{\lambda_1-\lambda_2}{2}\bigr)_i}
{(2\pi )^i(-\lambda_1+\lambda_2)_i}\\
&\phantom{=}
\times \frac{y_1y_2^{2\nu_1+2\nu_2}}{2\pi \sI}\int_t
\Gamma_{\bC}\bigl(t+\nu_1+\tfrac{q+l}{2}-i\bigr)
\Gamma_{\bC}\bigl(t+\nu_2+\tfrac{\lambda_1-\lambda_2-q-l}{2}\bigr)
y_1^{-2t}dt
\end{split}\\
\begin{split}
\label{eqn:C32_Wh_GL2_2nd}
&\!=(\varepsilon \sI )^{\lambda_1-q}\!
\sum_{i=0}^{\min \{\lambda_1-\lambda_2-q,l\}}\!\binom{\lambda_1-\lambda_2-q}{i}
\dfrac{(-l)_i\bigl(-\nu_2+\nu_1-\frac{\lambda_1-\lambda_2}{2}\bigr)_i}
{(2\pi )^i(-\lambda_1+\lambda_2)_i}\\
&\phantom{=}
\times \frac{y_1y_2^{2\nu_1+2\nu_2}}{2\pi \sI}\int_t
\Gamma_{\bC}\bigl(t+\nu_2+\tfrac{\lambda_1-\lambda_2-q+l}{2}-i\bigr)
\Gamma_{\bC}\bigl(t+\nu_1+\tfrac{q-l}{2}\bigr)
y_1^{-2t}dt.
\end{split}
\end{align}
Here the path of integration $\int_t$ is the vertical line 
from $\mathrm{Re}(t)-\sI \infty$ to $\mathrm{Re}(t)+\sI \infty$ 
with the sufficiently large real part to keep the poles of the integrand 
on its left.  
\end{prop}
\begin{proof}
The first expression (\ref{eqn:C32_Wh_GL2_1st}) of 
$\varphi_{\chi ,\lambda}^{[\varepsilon ]}(v_{\lambda ,q})(y)$ 
follows immediately 
from Theorem \ref{thm:C2_ps_Whittaker2}. Using the equalities 
\begin{align}
\label{eqn:C32_pf_EF_Wh2}
&\binom{m}{i}=(-1)^i\frac{(-m)_i}{i!},&
&\Gamma_{\bC}(z-i)=(2\pi)^i\frac{(-1)^i}{(1-z)_i}\Gamma_\bC (z)
\end{align}
for $z\in \bC$ and $m,i\in \bZ$ such that $m\geq i\geq 0$, 
we obtain from (\ref{eqn:C32_Wh_GL2_1st}) that 
\begin{align*}
&\varphi_{\chi ,\lambda}^{[\varepsilon ]}(v_{\lambda ,q})(y)\\
&=(\varepsilon \sI )^{\lambda_1-q}
\frac{y_1y_2^{2\nu_1+2\nu_2}}{2\pi \sI}
\int_t
{}_3F_2\left(\begin{array}{c}
-l,\ -q,\ -\nu_1+\nu_2-\frac{\lambda_1-\lambda_2}{2}\\[1mm]
-\lambda_1+\lambda_2,\ 1-t-\nu_1-\tfrac{q+l}{2}
\end{array};1\right)
\\
&\phantom{=.}\times 
\Gamma_{\bC}\bigl(t+\nu_1+\tfrac{q+l}{2}\bigr)
\Gamma_{\bC}\bigl(t+\nu_2+\tfrac{\lambda_1-\lambda_2-q-l}{2}\bigr)
y_1^{-2t}dt
\end{align*}
with the generalized hypergeometric series 
\[
{}_3F_2\left(\begin{array}{c}
a_1,\ a_2,\ a_3\\
b_1,\ b_2
\end{array};z\right)
=\sum_{i=0}^\infty 
\frac{(a_1)_i(a_2)_i(a_3)_i}{(b_1)_i(b_2)_i}\frac{z^i}{i!}.
\]
Applying the formula \cite[7.4.4.1]{Prudnikov_Brychkov_Marichev_001}
\begin{align*}
{}_3F_2\left(\!\begin{array}{c}
\!-l,\ a_2,\ a_3\\
b_1,\ b_2
\end{array}\!;1\right)
=\frac{(b_1+b_2-a_2-a_3)_l}{(b_2)_l}
{}_3F_2\left(\!\begin{array}{c}
-l,\ b_1-a_2,\ b_1-a_3\\
b_1,\ b_1+b_2-a_2-a_3
\end{array}\!;1\right)&\\
(l\in \bZ_{\geq 0},\ \text{the both sides are rational functions of 
$a_2,a_3,b_1,b_2$})&
\end{align*}
to this equality, we have 
\begin{align*}
\varphi_{\chi ,\lambda}^{[\varepsilon ]}(v_{\lambda ,q})(y)\,
&=(\varepsilon \sI )^{\lambda_1-q}
\frac{y_1y_2^{2\nu_1+2\nu_2}}{2\pi \sI}
\int_t
\frac{\bigl(1-t-\nu_2-\tfrac{\lambda_1-\lambda_2-q+l}{2}\bigr)_l}
{\bigl(1-t-\nu_1-\tfrac{q+l}{2}\bigr)_l}\\
&\phantom{=.}
\times 
{}_3F_2\left(\!\begin{array}{c}
-l,\ -\lambda_1+\lambda_2+q,\ \nu_1-\nu_2-\frac{\lambda_1-\lambda_2}{2}\\
-\lambda_1+\lambda_2,\ 1-t-\nu_2-\tfrac{\lambda_1-\lambda_2-q+l}{2}
\end{array}\!;1\right)\\
&\phantom{=.}
\times \Gamma_{\bC}\bigl(t+\nu_1+\tfrac{q+l}{2}\bigr)
\Gamma_{\bC}\bigl(t+\nu_2+\tfrac{\lambda_1-\lambda_2-q-l}{2}\bigr)y_1^{-2t}dt.
\end{align*}
From this expression, 
we obtain the second expression (\ref{eqn:C32_Wh_GL2_2nd}) 
using (\ref{eqn:C32_pf_EF_Wh2}). 
\end{proof}

We note that the formulas in Proposition \ref{prop:C32_Wh_GL2} 
are simplified in some cases. 
The expressions (\ref{eqn:C32_Wh_GL2_1st}) and (\ref{eqn:C32_Wh_GL2_2nd}) 
are both simplified as 
\begin{align}
\label{eqn:C32_Wh_GL2_min}
\begin{split}
&\varphi_{\chi ,(d_1,d_2)}^{[\varepsilon ]}(v_{(d_1,d_2),q})(y)
=(\varepsilon \sI )^{d_1-q}y_1y_2^{2\nu_1+2\nu_2}\\
&\times \frac{1}{2\pi \sI}\int_t
\Gamma_{\bC}\bigl(t+\nu_1+\tfrac{q}{2}\bigr)
\Gamma_{\bC}\bigl(t+\nu_2+\tfrac{d_1-d_2-q}{2}\bigr)
y_1^{-2t}dt
\end{split}
\end{align}
if $l=0$. 
The expression (\ref{eqn:C32_Wh_GL2_1st}) is simplified as 
\begin{align}
\begin{split}
\label{eqn:C32_Wh_GL2_1st_simple}
&\varphi_{\chi ,\lambda}^{[\varepsilon ]}(v_{\lambda ,0})(y)
=(\varepsilon \sI )^{\lambda_1}y_1y_2^{2\nu_1+2\nu_2}\\
&\times \frac{1}{2\pi \sI}\int_t
\Gamma_{\bC}\bigl(t+\nu_1+\tfrac{l}{2}\bigr)
\Gamma_{\bC}\bigl(t+\nu_2+\tfrac{\lambda_1-\lambda_2-l}{2}\bigr)
y_1^{-2t}dt
\end{split}
\end{align}
if $q=0$. The expression (\ref{eqn:C32_Wh_GL2_2nd}) is simplified as 
\begin{align}
\begin{split}
\label{eqn:C32_Wh_GL2_2nd_simple}
&\varphi_{\chi ,\lambda}^{[\varepsilon ]}
(v_{\lambda ,\lambda_1-\lambda_2})(y)
=(\varepsilon \sI )^{\lambda_2}y_1y_2^{2\nu_1+2\nu_2}\\
&\times 
\frac{1}{2\pi \sI}\int_t
\Gamma_{\bC}\bigl(t+\nu_1+\tfrac{\lambda_1-\lambda_2-l}{2}\bigr)
\Gamma_{\bC}\bigl(t+\nu_2+\tfrac{l}{2}\bigr)
y_1^{-2t}dt
\end{split}
\end{align}
if $q=\lambda_1-\lambda_2$.

\section{Whittaker functions on $GL(3,\bC )$}
\label{subsec:C32_Wh_GL3}

Let $\varepsilon \in \{\pm 1\}$, 
and take $\Xi_{(\varepsilon ,\varepsilon )}$ as (\ref{eqn:Fn_psi_change}). 
Let $\Pi_\chi $ be 
an irreducible principal series representation of $G_3$ with 
$\chi =\chi_{(\nu_1,d_1)}\boxtimes 
\chi_{(\nu_2,d_2)}\boxtimes \chi_{(\nu_3,d_3)}$ 
($d_1\geq d_2\geq d_3$). 
The minimal $K_3$-type of $\Pi_\chi$ is given by 
$\tau^{(3)}_{(d_1,d_2,d_3)}$. 
Let $\varphi^{[\varepsilon ]}_\chi \colon V_{(d_1,d_2,d_3)}^{(3)}\to 
{\mathrm{Wh}}(\Pi_{\chi},\psi_{\varepsilon })^{\mathrm{mg}}$ 
be a $K_3$-homomorphism defined by 
\[
\varphi^{[\varepsilon ]}_\chi 
=(\sI)^{d_1+d_2-d_3}
\Xi_{(\varepsilon ,\varepsilon )}\circ \varphi^{\mathrm{mg}}_\chi,
\]
where $\varphi^{\mathrm{mg}}_\chi $ is 
the $K_3$-homomorphism in Theorem 
\ref{thm:C3_Whittaker}. 
Then the radial part of $\varphi^{[\varepsilon ]}_\chi$ 
is given by 
\begin{align*}
\begin{split}
&\varphi^{[\varepsilon ]}_\chi (u_{l})(y) =
(-1)^{l_1+\tilde{l}_1}
(\varepsilon \sI)^{d_2+l_2+\tilde{l}_2}
y_1^2y_2^2(y_2y_3)^{2\nu_1+2\nu_2+2\nu_3}\\
&\times \frac{1}{(2\pi \sqrt{-1})^2} \int_{t_2}\int_{t_1}
\frac{\Gamma_{\bC}\bigl(t_1+\nu_1+\tfrac{\xi_1(l)}{2}\bigr)
\Gamma_{\bC}\bigl(t_1+\nu_2+\tfrac{\xi_2(l)}{2}\bigr)
\Gamma_{\bC}\bigl(t_1+\nu_3+\tfrac{\xi_3(l)}{2}\bigr)}
{ \Gamma_{\bC}\bigl(t_1+t_2+\tfrac{\xi_2(l)+\tilde{\xi}_2(l)}{2}\bigr)}\\
&\times 
\Gamma_{\bC}\bigl(t_2-\nu_1+\tfrac{\tilde{\xi}_1(l)}{2}\bigr)
\Gamma_{\bC}\bigl(t_2-\nu_2+\tfrac{\tilde{\xi}_2(l)}{2}\bigr)
\Gamma_{\bC}\bigl(t_2-\nu_3+\tfrac{\tilde{\xi}_3(l)}{2}\bigr)
\,y_1^{-2t_1} y_2^{-2t_2} \,dt_1dt_2
\end{split}
\end{align*}
for $l=(l_1,l_2,l_3,\tilde{l}_1,\tilde{l}_2,\tilde{l}_3)
\in S_{(d_1,d_2,d_3)}$ and 
$y=\diag (y_1y_2y_3,y_2y_3,y_3)\in A_3$. 
Here 
\begin{align*}
&\xi_1(l)=\tilde{l}_1+l_2+l_3,&
&\xi_2(l)=l_1+\tilde{l}_1,&
&\xi_3(l)=l_1+\tilde{l}_2+\tilde{l}_3,\\
&\tilde{\xi}_1(l)=l_1+l_2+\tilde{l}_3,&
&\tilde{\xi}_2(l)=l_3+\tilde{l}_3,&
&\tilde{\xi}_3(l)=\tilde{l}_1+\tilde{l}_2+l_3.
\end{align*}
and the path of the integration $\int_{t_i}$ is the vertical line 
from $\mathrm{Re}(t_i)-\sI \infty$ to $\mathrm{Re}(t_i)+\sI \infty$ 
with sufficiently large real part to keep the poles of the integrand 
on its left.

We regard $K_2$ as a subgroup of $K_3$ via the embedding 
(\ref{eqn:C32_K2_to_K3}). 
We  regard $U(\g_{3\bC})$ as a $K_2$-module 
via the adjoint action $\Ad$. 
For $\alpha =(\alpha_1,0)\in \Lambda_2$, 
we define a $\bC$-linear map 
$\mathrm{I}^{\gn_3}_{\alpha}\colon V_\alpha^{(2)}\to U(\g_{3\bC})$ by 
\begin{align*}
&\mathrm{I}^{\gn_3}_{\alpha}(v_{\alpha,q})=
(E_{1,3}^{\g_3})^{\alpha_1-q}(E_{2,3}^{\g_3})^q&
&(q \in Q_\alpha ). 
\end{align*}
For $\alpha =(0,\alpha_2)\in \Lambda_2$, 
we define a $\bC$-linear map 
$\mathrm{I}^{\gn_3}_{\alpha}\colon V_\alpha^{(2)}\to U(\g_{3\bC})$ by 
\begin{align*}
&\mathrm{I}^{\gn_3}_{\alpha}(v_{\alpha,q})=
(-\widetilde{E}_{1,3}^{\g_3})^{q}
(\widetilde{E}_{2,3}^{\g_3})^{-\alpha_2-q}&
&(q \in Q_\alpha ). 
\end{align*}
Since 
\begin{align*}
&\adj (E_{i,j}^{{\gk_3}})E_{k,3}^{\g_3}
=\delta_{j,k}E_{i,3}^{\g_3},&
&\adj (E_{i,j}^{{\gk_3}})\widetilde{E}_{k,3}^{\g_3}
=-\delta_{i,k}\widetilde{E}_{j,3}^{\g_3}
\end{align*}
for $1\leq i,j,k\leq 2$, 
we know that the $\bC$-linear map 
$\mathrm{I}^{\gn_3}_{\alpha}$ 
is a $K_2$-homomorphism 
for $\alpha =(\alpha_1,\alpha_2)\in \Lambda_2$ such that 
$\alpha_1=0$ or $\alpha_2=0$. 
Moreover, 
for $\alpha =(\alpha_1,\alpha_2)\in \Lambda_2$ such that 
$\alpha_1>0>\alpha_2$, 
we define a $K_2$-homomorphism 
$\mathrm{I}^{\gn_3}_{\alpha}\colon V_{\alpha}^{(2)}\to U(\g_{3\bC})$ by 
\begin{align*}
\mathrm{I}^{\gn_3}_{\alpha}=\mathrm{P}_{\g_3}
\circ (\mathrm{I}^{\gn_3}_{(\alpha_1,0)}\otimes 
\mathrm{I}^{\gn_3}_{(0,\alpha_2)})
\circ \mathrm{I}_{\alpha }^{(\alpha_1,0) ,(0,\alpha_2)},
\end{align*}
where $\mathrm{I}_{\alpha }^{(\alpha_1,0) ,(0,\alpha_2)}
\colon V_{\alpha}^{(2)}\to 
V_{(\alpha_1,0)}^{(2)}\otimes_\bC V_{(0,\alpha_2)}^{(2)}$ 
is the $K_2$-homomorphism in Lemma \ref{lem:C2_tensor}, and 
$\mathrm{P}_{\g_3}\colon U(\g_{3\bC})\otimes_\bC U(\g_{3\bC})\to U(\g_{3\bC})$ 
is a natural $K_2$-homomorphism defined by $X_1\otimes X_2\mapsto X_1X_2$. 
Then, for $q\in Q_\alpha $ and $\alpha =(\alpha_1,\alpha_2)\in \Lambda_2$ 
such that $\alpha_1\geq 0\geq \alpha_2$, we have 
\begin{align*}
\mathrm{I}_{\alpha}^{\gn_3}(v_{\alpha,q})
=&\sum_{i=\max \{0,q+\alpha_2\}}^{\min \{q,\alpha_1\}}
\binom{q}{i}
\binom{\alpha_1-\alpha_2-q}{\alpha_1-i}\\
&\times (E_{1,3}^{\g_3})^{\alpha_1-i}(E_{2,3}^{\g_3})^i
(-\widetilde{E}_{1,3}^{\g_3})^{q-i}
(\widetilde{E}_{2,3}^{\g_3})^{-\alpha_2-q+i}
\end{align*}
and 
\begin{align*}
\bigl(R\bigl(\mathrm{I}_{\alpha}^{\gn_3}(v_{\alpha,q})\bigr)f\bigr)(y)
=\left\{\begin{array}{ll}
(2\pi \varepsilon \sI y_2)^{\alpha_1-\alpha_2}f(y)&
\text{if}\ q=\alpha_1,\\
0&\text{otherwise}
\end{array}
\right.&\\
(f\in C^\infty (N_3\backslash G_3;\psi_\varepsilon ),\ 
y=\diag (y_1y_2y_3,y_2y_3,y_3)\in A_3)&.
\end{align*}
We define a $K_2$-homomorphism 
$\mathrm{P}_{G_3}\colon U(\g_{3\bC})\otimes_\bC C^\infty (G_3)$ 
by $X\otimes f\mapsto R(X)f$.

Using the $K_2$-homomorphisms in Lemma \ref{lem:C32_K32_restriction}, 
Proposition \ref{prop:C32_tensor_K2} and above, 
for $\lambda =(\lambda_1,\lambda_2)\in \Lambda_2$ 
such that $d_1\geq \lambda_2$ and $\lambda_1\geq d_3$, 
we define a $K_2$-homomorphism 
$\varphi^{[\varepsilon ]}_{\chi ,\lambda}\colon  V_\lambda^{(2)}\to 
{\mathrm{Wh}}(\Pi_{\chi},\psi_{\varepsilon })^{\mathrm{mg}}$ by 
\begin{align*}
\varphi^{[\varepsilon ]}_{\chi ,\lambda}
=
\frac{(-1)^{\alpha_{12}}}{\binom{\alpha_{12}-\alpha_{21}}{\alpha_{12}}}
\frac{(\varepsilon \sI)^{-\beta_1+\beta_2}}
{(2\pi \varepsilon \sI )^{\alpha_1-\alpha_2}}
\mathrm{P}_{G_3}\circ 
\bigl(\mathrm{I}_{\alpha}^{\gn_3}
\otimes 
\bigl(\varphi^{[\varepsilon ]}_{\chi }\circ 
\iota_{\beta}^{(d_1,d_2,d_3)}\bigr)\bigr)
\circ \mathrm{I}_{\lambda }^{\alpha ,\beta },
\end{align*}
where $\alpha =(\alpha_1,\alpha_2)$ and $\beta =(\beta_1,\beta_2)$ 
are determined by 
\begin{align}
\label{eqn:C32_def_alpha_beta}
\begin{split}
&\alpha_i=\alpha_{i1}+\alpha_{i2},\hspace{22.5mm} 
\beta_j=\lambda_j-\alpha_{1j}-\alpha_{2j},\\
&\alpha_{1j}=\max \{0,\lambda_j-d_j\},\hspace{10mm}
\alpha_{2j}=\min \{0,\lambda_j-d_{j+1}\}
\end{split}
\end{align}
for $i,j\in \{1,2\}$. 
Then, for $y=\diag (y_1y_2y_3,y_2y_3,y_3)\in A_3$ and $q\in Q_\lambda$, 
we have 
\begin{align*}
&\varphi^{[\varepsilon ]}_{\chi ,\lambda}(v_{\lambda ,q})(y)\\
&=
\binom{q}{\alpha_{11}}\binom{\lambda_1-\lambda_2-q}{-\alpha_{22}}
\sum_{i=\max \{0,q-\lambda_1+d_2\}
}^{\min \{q-\lambda_1+\beta_1,d_2-\beta_2\}}
\binom{q-\lambda_1+\beta_1}{i}
\binom{\lambda_1-\beta_2-q}{d_2-\beta_2-i}\\
&\phantom{=.}
\times (-1)^i(\varepsilon \sI)^{-\beta_1+\beta_2}
y_2^{\alpha_{1}-\alpha_{2}}\varphi^{[\varepsilon ]}_{\chi }\bigl(
u_{l(\beta ;q-\lambda_1+\beta_1-i,i)}\bigr)(y)\\
&=(\varepsilon \sI)^{\lambda_1-q}
y_1^2y_2^{\alpha_{1}-\alpha_{2}+2}(y_2y_3)^{2\nu_1+2\nu_2+2\nu_3}\\
&\phantom{=.}
\times 
\binom{q}{\alpha_{11}}\binom{\lambda_1-\lambda_2-q}{-\alpha_{22}}
\sum_{i=\max \{0,q-\lambda_1+d_2\}
}^{\min \{q-\lambda_1+\beta_1,d_2-\beta_2\}}
\binom{q-\lambda_1+\beta_1}{i}
\binom{\lambda_1-\beta_2-q}{d_2-\beta_2-i}\\
&\phantom{=.}\times \frac{1}{(2\pi \sqrt{-1})^2} \int_{t_2}\int_{t_1}
\frac{\Gamma_{\bC}\bigl(t_1+\nu_2+\tfrac{\lambda_1-d_2-q}{2}+i\bigr)
\Gamma_{\bC}\bigl(t_2-\nu_2+\tfrac{d_1-d_3-\beta_1+\beta_2}{2}\bigr)}
{ \Gamma_{\bC}\bigl(t_1+t_2+
\tfrac{-q+\lambda_1+d_1-d_2-d_3-\beta_1+\beta_2}{2}+i\bigr)}\\
&\phantom{=.}\times 
\Gamma_{\bC}\bigl(t_1+\nu_1+\tfrac{q-\lambda_1+d_1}{2}\bigr)
\Gamma_{\bC}\bigl(t_1+\nu_3+\tfrac{\lambda_1-d_3-q}{2}\bigr)\\
&\phantom{=.}\times 
\Gamma_{\bC}\bigl(t_2-\nu_1+\tfrac{\beta_1+\beta_2-d_2-d_3}{2}\bigr)
\Gamma_{\bC}\bigl(t_2-\nu_3+\tfrac{d_1+d_2-\beta_1-\beta_2}{2}\bigr)
\,y_1^{-2t_1} y_2^{-2t_2} \,dt_1dt_2
\end{align*}
if $\alpha_{11}\leq q\leq \lambda_1-\lambda_2+\alpha_{22}$, 
and $\varphi^{[\varepsilon ]}_{\chi ,\lambda}(v_{\lambda ,q})(y)=0$ 
otherwise.

\section{The local zeta integrals for $GL(3,\bC)\times GL(2,\bC)$}

Let $\varepsilon \in \{\pm 1\}$. 
Let $\Pi_\chi$ and $\Pi_{\chi'}$ be irreducible principal series 
representations of $G_3$ and $G_2$, respectively, with 
$\chi =\chi_{(\nu_1,d_1)}\boxtimes 
\chi_{(\nu_2,d_2)}\boxtimes \chi_{(\nu_3,d_3)}$ 
$(d_1\geq d_2\geq d_3)$ and $\chi' = \chi_{(\nu_1',d_1')}
\boxtimes \chi_{(\nu_2',d_2')}$ $(d_1'\geq d_2')$. 
We consider the local zeta integral 
\begin{align*}
Z(s,W,W')=&\int_{N_{2}\backslash G_{2}}
W\!\left(
\begin{array}{cc}
g& \\
 &1
\end{array}
\right) 
W'(g)|\det g|^{2s-1}
d\dot{g}
\end{align*}
defined for $W\in \mathrm{Wh}(\Pi_{\chi},\psi_{\varepsilon })^{\mathrm{mg}}$ 
and $W'\in \mathrm{Wh}(\Pi_{\chi'},\psi_{-\varepsilon })^{\mathrm{mg}}$. 
Here the right $G_2$-invariant measure 
$d\dot{g}$ on $N_2\backslash G_2$ is normalized so that, 
for any compactly supported continuous function $f$ on $N_2\backslash G_2$, 
\begin{align}
\label{eqn:C32_normalize_N2_G2}
\int_{N_2\backslash G_2}f(g)\,d\dot{g}
=\int_{0}^\infty \int_{0}^\infty 
\left(\int_{K_2}f(yk)\,dk\right)
y_1^{-2}
\frac{2dy_1}{y_1}
\frac{2dy_2}{y_2}
\end{align} 
with $y=\diag (y_1y_2,y_2)\in A_2$ and 
the normalized Haar measure $dk$ on $K_2$ such that 
$\int_{K_2}dk=1$. 

We regard $K_2$ as a subgroup of $K_3$ via the embedding 
(\ref{eqn:C32_K2_to_K3}). 
Let $\varphi \colon V_{\lambda}^{(2)}\to 
\mathrm{Wh}(\Pi_{\chi},\psi_{\varepsilon })^{\mathrm{mg}}$ and 
$\varphi' \colon V_{\lambda'}^{(2)}\to 
\mathrm{Wh}(\Pi_{\chi'},\psi_{-\varepsilon })^{\mathrm{mg}}$ 
be $K_2$-homomorphisms with 
$\lambda =(\lambda_1,\lambda_2),\lambda'=(\lambda_1',\lambda_2')
\in \Lambda_2$. Then the following lemma holds. 

\begin{lem}
\label{lem:C32_zeta32_schur}
Retain the notation. 
For $v\in V_{\lambda}^{(2)}$ and $v'\in V_{\lambda'}^{(2)}$, 
it holds that  
\begin{align}
\label{eqn:C32_zeta_schur}
\begin{split}
&Z(s,\varphi (v),\varphi'(v'))
=\frac{\langle v,v'\rangle}{\dim_\bC  V_\lambda^{(2)}}
\sum_{q\in Q_\lambda }
(-1)^{\lambda_1-q}
\binom{\lambda_1-\lambda_2}{q}\\
&\hspace{10mm}
\times \int_{0}^\infty \int_{0}^\infty   
\varphi (v_{\lambda,q})(\hat{y})
\varphi'(v_{\widetilde{\lambda},\lambda_1-\lambda_2-q})(y)
y_1^{2s-3}y_2^{4s-2}
\frac{2dy_1}{y_1}\frac{2dy_2}{y_2}
\end{split}
\end{align}
if $\lambda '=\widetilde{\lambda}$, 
and $Z(s,\varphi (v),\varphi'(v'))=0$ otherwise. Here 
$y=\diag (y_1y_2,y_2)\in A_2$ and $\hat{y}=\diag (y_1y_2,y_2,1)\in A_3$. 
\end{lem}
\begin{proof}
Let $v\in V_{\lambda}^{(2)}$ and $v'\in V_{\lambda'}^{(2)}$. 
By definition, we have 
\begin{align*}
&Z(s,\varphi (v),\varphi'(v'))
=\int_{N_2\backslash G_2}
\varphi (v)\!\left(
\begin{array}{cc}
g& \\
 &1
\end{array}
\right) 
\varphi'(v')(g)|\det g|^{2s-1}
d\dot{g}\\
&=\int_{0}^\infty \int_{0}^\infty   
\left(\int_{K_2}\varphi (\tau_{\lambda}(k)v)(\hat{y})
\varphi'(\tau_{\lambda'}(k)v')(y)\,dk\right)
y_1^{2s-3}y_2^{4s-2}
\frac{2dy_1}{y_1}\frac{2dy_2}{y_2}.
\end{align*}
By the decompositions 
\begin{align*}
&\tau_{\lambda}(k)v
=\sum_{q\in Q_\lambda }
(-1)^{\lambda_1-q}\binom{\lambda_1-\lambda_2}{q}
\langle \tau_{\lambda}(k)v,
v_{\widetilde{\lambda},\lambda_1-\lambda_2-q}\rangle 
v_{\lambda,q},\\
&\tau_{\lambda'}(k)v'
=\sum_{q'\in Q_{\lambda'} }
(-1)^{\lambda_1'-q'}\binom{\lambda_1'-\lambda_2'}{q'}
\langle \tau_{\lambda'}(k)v',
v_{\widetilde{\lambda'},\lambda_1'-\lambda_2'-q'}\rangle 
v_{\lambda',q'},&
\end{align*}
we have 
\begin{align*}
Z(s,\varphi (v),\varphi'(v'))=
&\sum_{q\in Q_\lambda }\sum_{q'\in Q_{\lambda'} }
(-1)^{\lambda_1+\lambda_1'-q'-q}
\binom{\lambda_1-\lambda_2}{q}\binom{\lambda_1'-\lambda_2'}{q'}\\
&\times \left(\int_{K_2}\langle \tau_{\lambda}(k)v,
v_{\widetilde{\lambda},\lambda_1-\lambda_2-q}\rangle 
\langle \tau_{\lambda'}(k)v',
v_{\widetilde{\lambda'},\lambda_1'-\lambda_2'-q'}\rangle 
\,dk\right)\\
&\times \int_{0}^\infty \int_{0}^\infty   
\varphi (v_{\lambda,q})(\hat{y})
\varphi'(v_{\lambda',q'})(y)
y_1^{2s-3}y_2^{4s-2}
\frac{2dy_1}{y_1}\frac{2dy_2}{y_2}.
\end{align*}
Applying Schur's orthogonality relation 
(\cite[Proposition 4.4]{Brocker_001}) 
\begin{align}
\begin{split}
\label{eqn:C32_schur}
\int_{K_2}\langle \tau_{\lambda }(k)v,w\rangle 
\langle \tau_{\lambda'}(k)v',w' \rangle \,dk
=\left\{\begin{array}{ll}
\displaystyle 
\frac{\langle v,v'\rangle 
\langle w,w'\rangle }{\dim_\bC  V_{\lambda}^{(2)}}
&\text{ if }\lambda'=\widetilde{\lambda },\\[3mm]
0&\text{ otherwise}
\end{array}\right.&\\
(v\in V_{\lambda}^{(2)},\ 
w\in V_{\widetilde{\lambda}}^{(2)},\ 
v'\in V_{\lambda'}^{(2)},\ 
w'\in V_{\widetilde{\lambda'}}^{(2)})&
\end{split}
\end{align}
to the right hand side of this equality, 
we obtain the assertion. 
\end{proof}

By this lemma, it suffices to consider 
the right handside of (\ref{eqn:C32_zeta_schur}) in the case of 
$\lambda'=\widetilde{\lambda}$. 
By the arguments in \S \ref{subsec:C2_ps}, we note that 
$\tau_{\lambda'}^{(2)}$ is the $K_2$-type of $\Pi_{\chi'}$ 
if and only if 
$\lambda'$ is of the form $\lambda'=(d_1'+l,d_2'-l)$ 
($l\in \bZ_{\geq 0}$). 
In \S \ref{subsec:C32_Wh_GL3}, 
for $\lambda =(\lambda_1,\lambda_2)\in \Lambda_2$ satisfying 
$\lambda_1\geq d_3$ and $d_1\geq \lambda_2$, 
we construct a $K_2$-homomorphism from $V_{\lambda}^{(2)}$ 
to $\mathrm{Wh}(\Pi_{\chi},\psi_{\varepsilon })^{\mathrm{mg}}$. 
When $\widetilde{\lambda}=\lambda'=(d_1'+l,d_2'-l)$ ($l\in \bZ_{\geq 0}$), 
it is easy to see that the inequalities 
$\lambda_1\geq d_3$ and $d_1\geq \lambda_2$ hold 
if and only if $l\geq l_0$ with 
$l_0=\max \{0,-d_1-d_1',d_3+d_2'\}$.

\begin{thm}
\label{thm:C32_zeta_L_coincide}
Let $\Pi_\chi$ and $\Pi_{\chi'}$ be irreducible principal series 
representations of $G_3=GL(3,\bC)$ and $G_2=GL(2,\bC)$, respectively, with 
$\chi =\chi_{(\nu_1,d_1)}\boxtimes 
\chi_{(\nu_2,d_2)}\boxtimes \chi_{(\nu_3,d_3)}$ 
$(d_1\geq d_2\geq d_3)$ and $\chi' = \chi_{(\nu_1',d_1')}
\boxtimes \chi_{(\nu_2',d_2')}$ $(d_1'\geq d_2')$. 
Set $\lambda =(\lambda_1,\lambda_2)=(-d_2'+l_0,-d_1'-l_0)$ with 
$l_0=\max \{0,-d_1-d_1',d_3+d_2'\}$. Let $\varepsilon \in \{\pm 1\}$. 
Then, for $v\in V_{\lambda}^{(2)}$, $v'\in V_{\widetilde{\lambda}}^{(2)}$ and 
$s\in \bC$ with sufficiently large real part, it holds that 
\begin{align}
\label{eqn:C32_zeta_L_coincide}
Z\bigl(s,\varphi^{[\varepsilon ]}_{\chi ,\lambda}(v),
\varphi^{[-\varepsilon ]}_{\chi',\widetilde{\lambda}}(v')\bigr)
=C(\chi,\chi')\langle v,v'\rangle L(s,\Pi_{\chi}\times \Pi_{\chi'}),
\end{align}
where $\varphi_{\chi ,\lambda}^{[\varepsilon ]}$, 
$\varphi_{\chi',\widetilde{\lambda}}^{[-\varepsilon ]}$ are 
the $K_2$-homomorphisms in \S \ref{subsec:C32_Wh_GL3}, 
\S \ref{subsec:C32_Wh_GL2}, respectively, and 
\begin{align*}
&C(\chi,\chi')=\frac{2(\lambda_1-\lambda_2)!}
{(\lambda_1-\lambda_2+1)\alpha_{11}!(-\alpha_{22})!(\beta_1-d_2-\alpha_{12})!
(d_2-\beta_2+\alpha_{21})!}
\end{align*}
with $\alpha_{ij}$, $\beta_i$ $(1\leq i,j\leq 2)$ 
in (\ref{eqn:C32_def_alpha_beta}). 
\end{thm}

\section{The calculation in the case $d_2>-d_2'$}
\label{subsec:C32_pf_Main_1}

Here we give a proof of Theorem \ref{thm:C32_zeta_L_coincide} 
in the case $d_2>-d_2'$. 
In this case, we note that 
$\alpha_1=\alpha_{11}=\alpha_{12}=0$ and 
\begin{align*}
&l_0=\max \{0,d_3+d_2'\},&
&\alpha_{21}=\lambda_1-d_2,&
&\alpha_{22}=\min \{0,\lambda_2-d_3\},\\
&\alpha_2=\lambda_1-d_2+\alpha_{22},&
&\beta_1=d_2,&
&\beta_2=\lambda_2-\alpha_{22}.
\end{align*}
For $0\leq q\leq \lambda_1-\lambda_2$ 
and $\hat{y}=\diag (y_1y_2,y_2,1)\in A_3$, 
we have 
\begin{align*}
&\varphi^{[\varepsilon ]}_{\chi ,\lambda}(v_{\lambda ,q})(\hat{y})
=(\varepsilon \sI)^{\lambda_1-q}
y_1^2y_2^{2\nu_1+2\nu_2+2\nu_3-\lambda_1+d_2-\alpha_{22}+2}\\
&\times 
\binom{\lambda_1-\lambda_2-q}{-\alpha_{22}}
\frac{1}{(2\pi \sqrt{-1})^2} \int_{t_2}\int_{t_1}
\frac{\Gamma_{\bC}\bigl(t_1+\nu_2+\tfrac{q-\lambda_1+d_2}{2}\bigr)}
{\Gamma_{\bC}\bigl(t_1+t_2+
\tfrac{q-\lambda_1+\lambda_2+d_1-d_3-\alpha_{22}}{2}\bigr)}\\
&\times 
\Gamma_{\bC}\bigl(t_2-\nu_2+\tfrac{\lambda_2+d_1-d_2-d_3-\alpha_{22}}{2}\bigr)
\Gamma_{\bC}\bigl(t_1+\nu_1+\tfrac{q-\lambda_1+d_1}{2}\bigr)
\Gamma_{\bC}\bigl(t_1+\nu_3+\tfrac{\lambda_1-d_3-q}{2}\bigr)\\
&\times 
\Gamma_{\bC}\bigl(t_2-\nu_1+\tfrac{\lambda_2-d_3-\alpha_{22}}{2}\bigr)
\Gamma_{\bC}\bigl(t_2-\nu_3+\tfrac{-\lambda_2+d_1+\alpha_{22}}{2}\bigr)
\,y_1^{-2t_1} y_2^{-2t_2} \,dt_1dt_2
\end{align*}
if $0\leq q\leq \lambda_1-\lambda_2+\alpha_{22}$, 
and $\varphi^{[\varepsilon ]}_{\chi ,\lambda}(v_{\lambda ,q})(y)=0$ otherwise. 
Hence, by Lemma \ref{lem:C32_zeta32_schur}, we have 
\begin{align*}
&Z\bigl(s,\varphi^{[\varepsilon ]}_{\chi ,\lambda}(v),
\varphi^{[-\varepsilon ]}_{\chi',\widetilde{\lambda}}(v')\bigr)\\
&=\frac{\langle v,v'\rangle}{\lambda_1-\lambda_2+1}
\sum_{q=0}^{\lambda_1-\lambda_2}
(-1)^{\lambda_1-q}
\binom{\lambda_1-\lambda_2}{q}\int_{0}^\infty \int_{0}^\infty   
\varphi^{[\varepsilon ]}_{\chi ,\lambda}(v_{\lambda ,q})(\hat{y})\\
&\phantom{=}\times 
\varphi^{[-\varepsilon ]}_{\chi',\widetilde{\lambda}}
(v_{\widetilde{\lambda},\lambda_1-\lambda_2-q})(y)
y_1^{2s-3}y_2^{4s-2}
\frac{2dy_1}{y_1}\frac{2dy_2}{y_2}\\
&=\frac{\langle v,v'\rangle}{\lambda_1-\lambda_2+1}
\sum_{q=0}^{\lambda_1-\lambda_2+\alpha_{22}}
(-\varepsilon \sI)^{\lambda_1-q}
\binom{\lambda_1-\lambda_2}{q}
\binom{\lambda_1-\lambda_2-q}{-\alpha_{22}}
\frac{1}{(2\pi \sqrt{-1})^2}\\
&\phantom{=}\times 
\int_{0}^\infty \int_{0}^\infty 
\biggl\{\int_{t_2}\int_{t_1}
\frac{\Gamma_{\bC}\bigl(t_1+\nu_2+\tfrac{q-\lambda_1+d_2}{2}\bigr)
\Gamma_{\bC}\bigl(t_2-\nu_2+\tfrac{\lambda_2+d_1-d_2-d_3-\alpha_{22}}{2}\bigr)}
{\Gamma_{\bC}\bigl(t_1+t_2+
\tfrac{q-\lambda_1+\lambda_2+d_1-d_3-\alpha_{22}}{2}\bigr)}\\
&\phantom{=}\times 
\Gamma_{\bC}\bigl(t_1+\nu_1+\tfrac{q-\lambda_1+d_1}{2}\bigr)
\Gamma_{\bC}\bigl(t_1+\nu_3+\tfrac{\lambda_1-d_3-q}{2}\bigr)\\
&\phantom{=}\times 
\Gamma_{\bC}\bigl(t_2-\nu_1+\tfrac{\lambda_2-d_3-\alpha_{22}}{2}\bigr)
\Gamma_{\bC}\bigl(t_2-\nu_3+\tfrac{-\lambda_2+d_1+\alpha_{22}}{2}\bigr)
\varphi^{[-\varepsilon ]}_{\chi',\widetilde{\lambda}}
(v_{\widetilde{\lambda},\lambda_1-\lambda_2-q})(y)\\
&\phantom{=}\times 
y_1^{-2t_1} y_2^{-2t_2} \,dt_1dt_2\biggr\}
y_1^{2s-1}y_2^{4s+2\nu_1+2\nu_2+2\nu_3-\lambda_1+d_2-\alpha_{22}}
\frac{2dy_1}{y_1}\frac{2dy_2}{y_2}.
\end{align*}

\vspace{2mm}

\underline{\bf Case 4-1:} $d_3>-d_2'$. 
Since 
$l_0=d_3+d_2'$, 
$\alpha_{22}=-2d_3-d_1'-d_2'$, 
$\lambda_1=d_3$, 
$\lambda_2=-d_3-d_1'-d_2'$ and 
\begin{align*}
&\varphi_{\chi',\widetilde{\lambda}}^{[-\varepsilon ]}
(v_{\widetilde{\lambda},\lambda_1-\lambda_2})(y)
=(-\varepsilon \sI )^{-d_3}y_1y_2^{2\nu_1'+2\nu_2'}\\
&\hspace{2cm}
\times 
\frac{1}{2\pi \sI}\int_t
\Gamma_{\bC}\bigl(t+\nu_1'+\tfrac{d_3+d_1'}{2}\bigr)
\Gamma_{\bC}\bigl(t+\nu_2'+\tfrac{d_3+d_2'}{2}\bigr)
y_1^{-2t}dt
\end{align*}
for $y=\diag (y_1y_2,y_2)\in A_2$, we have 
\begin{align*}
&Z\bigl(s,\varphi^{[\varepsilon ]}_{\chi ,\lambda}(v),
\varphi^{[-\varepsilon ]}_{\chi',\widetilde{\lambda}}(v')\bigr)\\
&=\frac{\langle v,v'\rangle}{2d_3+d_1'+d_2'+1}
\frac{1}{(2\pi \sqrt{-1})^3}
\int_{0}^\infty \int_{0}^\infty 
\biggl\{\int_t\int_{t_2}\int_{t_1}
\frac{\Gamma_{\bC}\bigl(t_1+\nu_2+\tfrac{d_2-d_3}{2}\bigr)}
{\Gamma_{\bC}\bigl(t_1+t_2+\tfrac{d_1-d_3}{2}\bigr)}\\
&\phantom{=}\times 
\Gamma_{\bC}\bigl(t_2-\nu_2+\tfrac{d_1-d_2}{2}\bigr)
\Gamma_{\bC}\bigl(t_1+\nu_1+\tfrac{d_1-d_3}{2}\bigr)
\Gamma_{\bC}(t_1+\nu_3)\Gamma_{\bC}(t_2-\nu_1)\\
&\phantom{=}\times 
\Gamma_{\bC}\bigl(t_2-\nu_3+\tfrac{d_1-d_3}{2}\bigr)
\Gamma_{\bC}\bigl(t+\nu_1'+\tfrac{d_3+d_1'}{2}\bigr)
\Gamma_{\bC}\bigl(t+\nu_2'+\tfrac{d_3+d_2'}{2}\bigr)\\
&\phantom{=}\times 
y_1^{-2t_1-2t} y_2^{-2t_2} \,dt_1dt_2dt\biggr\}
y_1^{2s}y_2^{4s+2\nu_1+2\nu_2+2\nu_3+2\nu_1'+2\nu_2'+d_2+d_3+d_1'+d_2'}
\frac{2dy_1}{y_1}\frac{2dy_2}{y_2}.
\end{align*}
Substituting $t_1\to t_1-t$ and applying Lemma \ref{lem:F32_Mellin} twice, 
we have 
\begin{align*}
&Z\bigl(s,\varphi^{[\varepsilon ]}_{\chi ,\lambda}(v),
\varphi^{[-\varepsilon ]}_{\chi',\widetilde{\lambda}}(v')\bigr)\\
&=\frac{2\langle v,v'\rangle}{2d_3+d_1'+d_2'+1}
\Gamma_{\bC}\bigl(2s+\nu_1+\nu_2+\nu_1'+\nu_2'
+\tfrac{d_1+d_2+d_1'+d_2'}{2}\bigr)\\
&\phantom{=}\times 
\Gamma_{\bC}\bigl(2s+\nu_1+\nu_3+\nu_1'+\nu_2'
+\tfrac{d_1+d_3+d_1'+d_2'}{2}\bigr)\\
&\phantom{=}\times 
\Gamma_{\bC}\bigl(2s+\nu_2+\nu_3+\nu_1'+\nu_2'
+\tfrac{d_2+d_3+d_1'+d_2'}{2}\bigr)\\
&\phantom{=}\times 
\frac{1}{4\pi \sqrt{-1}}
\int_t
\frac{\Gamma_{\bC}\bigl(t+\nu_1'+\tfrac{d_3+d_1'}{2}\bigr)
\Gamma_{\bC}\bigl(t+\nu_2'+\tfrac{d_3+d_2'}{2}\bigr)
\Gamma_{\bC}\bigl(-t+s+\nu_1+\tfrac{d_1-d_3}{2}\bigr)}
{\Gamma_{\bC}\bigl(-t+3s+\nu_1+\nu_2+\nu_3+\nu_1'+\nu_2'
+\tfrac{d_1+d_2+d_1'+d_2'}{2}\bigr)}\\
&\phantom{=}\times 
\Gamma_{\bC}\bigl(-t+s+\nu_2+\tfrac{d_2-d_3}{2}\bigr)
\Gamma_{\bC}(-t+s+\nu_3)dt.
\end{align*}
Applying Lemma \ref{lem:F32_Barnes2nd}, we have 
\begin{align*}
&Z\bigl(s,\varphi^{[\varepsilon ]}_{\chi ,\lambda}(v),
\varphi^{[-\varepsilon ]}_{\chi',\widetilde{\lambda}}(v')\bigr)
=\frac{2\langle v,v'\rangle}{2d_3+d_1'+d_2'+1}
\prod_{i=1}^3\prod_{j=1}^2
\Gamma_{\bC}\bigl(s+\nu_i+\nu_j'+\tfrac{d_i+d_j'}{2}\bigr).
\end{align*}
Hence we obtain the assertion in this case. \vspace{2mm}

\underline{\bf Case 4-2:} $d_2>-d_2'\geq d_3$. 
Since $l_0=0$, $\alpha_{22}=\min \{0,-d_3-d_1'\}$, 
$\lambda_1=-d_2'$, $\lambda_2=-d_1'$ and 
\begin{align*}
&\binom{d_1'-d_2'}{q}
\binom{d_1'-d_2'-q}{-\alpha_{22}}
=\binom{d_1'-d_2'}{-\alpha_{22}}
\binom{d_1'-d_2'+\alpha_{22}}{q},\\
&\varphi_{\chi',\widetilde{\lambda}}^{[-\varepsilon ]}
(v_{\widetilde{\lambda},d_1'-d_2'-q})(y)
=(-\varepsilon \sI )^{d_2'+q}y_1y_2^{2\nu_1'+2\nu_2'}\\
&\hspace{2cm}\times \frac{1}{2\pi \sI}\int_t
\Gamma_{\bC}\bigl(t+\nu_1'+\tfrac{d_1'-d_2'-q}{2}\bigr)
\Gamma_{\bC}\bigl(t+\nu_2'+\tfrac{q}{2}\bigr)
y_1^{-2t}dt
\end{align*}
for $0\leq q\leq d_1'-d_2'$ and $y=\diag (y_1y_2,y_2)\in A_2$, 
we have 
\begin{align*}
&Z\bigl(s,\varphi^{[\varepsilon ]}_{\chi ,\lambda }(v),
\varphi^{[-\varepsilon ]}_{\chi',\widetilde{\lambda}}(v')\bigr)\\
&=\frac{\langle v,v'\rangle}{d_1'-d_2'+1}
\binom{d_1'-d_2'}{-\alpha_{22}}
\sum_{q=0}^{d_1'-d_2'+\alpha_{22}}
\binom{d_1'-d_2'+\alpha_{22}}{q}
\frac{1}{(2\pi \sqrt{-1})^3}\\
&\phantom{=}\times 
\int_{0}^\infty \int_{0}^\infty 
\biggl\{\int_t\int_{t_2}\int_{t_1}
\frac{\Gamma_{\bC}\bigl(t_1+\nu_2+\tfrac{q+d_2+d_2'}{2}\bigr)
\Gamma_{\bC}\bigl(t_2-\nu_2+\tfrac{d_1-d_2-d_3-d_1'-\alpha_{22}}{2}\bigr)}
{\Gamma_{\bC}\bigl(t_1+t_2+
\tfrac{q+d_1-d_3-d_1'+d_2'-\alpha_{22}}{2}\bigr)}\\
&\phantom{=}\times 
\Gamma_{\bC}\bigl(t_1+\nu_1+\tfrac{q+d_1+d_2'}{2}\bigr)
\Gamma_{\bC}\bigl(t_1+\nu_3+\tfrac{-d_3-d_2'-q}{2}\bigr)
\Gamma_{\bC}\bigl(t_2-\nu_1+\tfrac{-d_3-d_1'-\alpha_{22}}{2}\bigr)\\
&\phantom{=}\times 
\Gamma_{\bC}\bigl(t_2-\nu_3+\tfrac{d_1+d_1'+\alpha_{22}}{2}\bigr)
\Gamma_{\bC}\bigl(t+\nu_1'+\tfrac{d_1'-d_2'-q}{2}\bigr)
\Gamma_{\bC}\bigl(t+\nu_2'+\tfrac{q}{2}\bigr)\\
&\phantom{=}\times 
y_1^{-2t_1-2t} y_2^{-2t_2} \,dt_1dt_2dt\biggr\}
y_1^{2s}y_2^{4s+2\nu_1+2\nu_2+2\nu_3+2\nu_1'+2\nu_2'+d_2+d_2'-\alpha_{22}}
\frac{2dy_1}{y_1}\frac{2dy_2}{y_2}.
\end{align*}
Substituting $t_1\to t_1-t$ and applying Lemma \ref{lem:F32_Mellin} twice, 
we have 
\begin{align*}
&Z\bigl(s,\varphi^{[\varepsilon ]}_{\chi ,\lambda }(v),
\varphi^{[-\varepsilon ]}_{\chi',\widetilde{\lambda}}(v')\bigr)\\
&=
\Gamma_{\bC}\bigl(2s+\nu_1+\nu_2+\nu_1'+\nu_2'
+\tfrac{d_1+d_2+d_1'+d_2'}{2}\bigr)\\
&\phantom{=}\times 
\Gamma_{\bC}\bigl(2s+\nu_1+\nu_3+\nu_1'+\nu_2'
+\tfrac{d_1-d_3-d_1'+d_2'}{2}-\alpha_{22}\bigr)\\
&\phantom{=}\times 
\Gamma_{\bC}\bigl(2s+\nu_2+\nu_3+\nu_1'+\nu_2'
+\tfrac{d_2-d_3-d_1'+d_2'}{2}-\alpha_{22}\bigr)\\
&\phantom{=}\times 
\frac{\langle v,v'\rangle}{d_1'-d_2'+1}
\binom{d_1'-d_2'}{-\alpha_{22}}
\sum_{q=0}^{d_1'-d_2'+\alpha_{22}}
\binom{d_1'-d_2'+\alpha_{22}}{q}
\frac{1}{2\pi \sqrt{-1}}\\
&\phantom{=}\times 
\int_t
\frac{\Gamma_{\bC}\bigl(t+\nu_1'+\tfrac{d_1'-d_2'-q}{2}\bigr)
\Gamma_{\bC}\bigl(t+\nu_2'+\tfrac{q}{2}\bigr)
\Gamma_{\bC}\bigl(-t+s+\nu_1+\tfrac{q+d_1+d_2'}{2}\bigr)}
{\Gamma_{\bC}\bigl(-t+3s+\nu_1+\nu_2+\nu_3+\nu_1'+\nu_2'
+\tfrac{d_1+d_2-d_3-d_1'+q}{2}+d_2'-\alpha_{22}\bigr)}\\
&\phantom{=}\times 
\Gamma_{\bC}\bigl(-t+s+\nu_2+\tfrac{q+d_2+d_2'}{2}\bigr)
\Gamma_{\bC}\bigl(-t+s+\nu_3+\tfrac{-d_3-d_2'-q}{2}\bigr)
dt.
\end{align*}
Substituting $t\to t+\tfrac{q}{2}$ and applying 
Lemma \ref{lem:F32_gauss_sum}, we have 
\begin{align*}
&Z\bigl(s,\varphi^{[\varepsilon ]}_{\chi ,\lambda }(v),
\varphi^{[-\varepsilon ]}_{\chi',\widetilde{\lambda}}(v')\bigr)\\
&=
\Gamma_{\bC}\bigl(2s+\nu_1+\nu_2+\nu_1'+\nu_2'
+\tfrac{d_1+d_2+d_1'+d_2'}{2}\bigr)\\
&\phantom{=}\times 
\Gamma_{\bC}\bigl(2s+\nu_1+\nu_3+\nu_1'+\nu_2'
+\tfrac{d_1-d_3-d_1'+d_2'}{2}-\alpha_{22}\bigr)\\
&\phantom{=}\times 
\Gamma_{\bC}\bigl(2s+\nu_2+\nu_3+\nu_1'+\nu_2'
+\tfrac{d_2-d_3-d_1'+d_2'}{2}-\alpha_{22}\bigr)\\
&\phantom{=}\times 
\frac{2\langle v,v'\rangle}{d_1'-d_2'+1}
\binom{d_1'-d_2'}{-\alpha_{22}}
\frac{\Gamma_{\bC}\bigl(s+\nu_3+\nu_2'+\tfrac{-d_3-d_2'}{2}\bigr)}
{\Gamma_{\bC}\bigl(s+\nu_3+\nu_2'+\tfrac{-d_3+d_2'}{2}-d_1'-\alpha_{22}\bigr)}
\frac{1}{4\pi \sqrt{-1}}\\
&\phantom{=}\times 
\int_t
\frac{\Gamma_{\bC}\bigl(t+\nu_1'+\tfrac{d_1'-d_2'}{2}\bigr)
\Gamma_{\bC}(t+\nu_2')
\Gamma_{\bC}\bigl(-t+s+\nu_1+\tfrac{d_1+d_2'}{2}\bigr)}
{\Gamma_{\bC}\bigl(-t+3s+\nu_1+\nu_2+\nu_3+\nu_1'+\nu_2'
+\tfrac{d_1+d_2-d_3-d_1'}{2}+d_2'-\alpha_{22}\bigr)}\\
&\phantom{=}\times 
\Gamma_{\bC}\bigl(-t+s+\nu_2+\tfrac{d_2+d_2'}{2}\bigr)
\Gamma_{\bC}\bigl(-t+s+\nu_3+\tfrac{-d_3+d_2'}{2}-d_1'-\alpha_{22}\bigr)dt.
\end{align*}
Applying Lemma \ref{lem:F32_Barnes2nd}, we have 
\begin{align*}
&Z\bigl(s,\varphi^{[\varepsilon ]}_{\chi ,\lambda }(v),
\varphi^{[-\varepsilon ]}_{\chi',\widetilde{\lambda}}(v')\bigr)\\
&=\frac{2\langle v,v'\rangle}{d_1'-d_2'+1}
\binom{d_1'-d_2'}{-\alpha_{22}}
\Gamma_{\bC}\bigl(s+\nu_1+\nu_1'+\tfrac{d_1+d_1'}{2}\bigr)
\Gamma_{\bC}\bigl(s+\nu_1+\nu_2'+\tfrac{d_1+d_2'}{2}\bigr)\\
&\phantom{=}\times 
\Gamma_{\bC}\bigl(s+\nu_2+\nu_1'+\tfrac{d_2+d_1'}{2}\bigr)
\Gamma_{\bC}\bigl(s+\nu_2+\nu_2'+\tfrac{d_2+d_2'}{2}\bigr)
\\
&\phantom{=}\times 
\Gamma_{\bC}\bigl(s+\nu_3+\nu_1'+\tfrac{-d_3-d_1'}{2}-\alpha_{22}\bigr)
\Gamma_{\bC}(s+\nu_3+\nu_2'+\tfrac{-d_3-d_2'}{2}).
\end{align*}
Hence we obtain the assertion in this case.

\section{The calculation in the case $-d_1'>d_2$}

Here we give a proof of Theorem \ref{thm:C32_zeta_L_coincide} 
in the case $-d_1'>d_2$. 
Our proof in this section is similar to that 
in \S \ref{subsec:C32_pf_Main_1}. 

In the case $-d_1'>d_2$, we note that 
$\alpha_2=\alpha_{21}=\alpha_{22}=0$ and 
\begin{align*}
&l_0=\max \{0,-d_1-d_1'\},&
&\alpha_{11}=\max \{0,\lambda_1-d_1\},&
&\alpha_{12}=\lambda_2-d_2,\\
&\alpha_1=\lambda_2-d_2+\alpha_{11},&
&\beta_1=\lambda_1-\alpha_{11},&
&\beta_2=d_2.
\end{align*}
For $0\leq q\leq \lambda_1-\lambda_2$ 
and $\hat{y}=\diag (y_1y_2,y_2,1)\in A_3$, 
we have 
\begin{align*}
&\varphi^{[\varepsilon ]}_{\chi ,\lambda}(v_{\lambda ,\lambda_1-\lambda_2-q})
(\hat{y})=(\varepsilon \sI)^{\lambda_2+q}
y_1^2y_2^{2\nu_1+2\nu_2+2\nu_3+\lambda_2-d_2+\alpha_{11}+2}\\
&\phantom{==}\times 
\binom{\lambda_1-\lambda_2-q}{\alpha_{11}}
\frac{1}{(2\pi \sqrt{-1})^2} \int_{t_2}\int_{t_1}
\frac{\Gamma_{\bC}\bigl(t_1+\nu_2+\tfrac{\lambda_2-d_2+q}{2}\bigr)}
{ \Gamma_{\bC}\bigl(t_1+t_2+
\tfrac{q-\lambda_1+\lambda_2+d_1-d_3+\alpha_{11}}{2}\bigr)}\\
&\phantom{==}\times 
\Gamma_{\bC}\bigl(t_2-\nu_2+\tfrac{d_1+d_2-d_3-\lambda_1+\alpha_{11}}{2}\bigr)
\Gamma_{\bC}\bigl(t_1+\nu_1+\tfrac{-\lambda_2+d_1-q}{2}\bigr)\\
&\phantom{==}\times 
\Gamma_{\bC}\bigl(t_1+\nu_3+\tfrac{\lambda_2-d_3+q}{2}\bigr)
\Gamma_{\bC}\bigl(t_2-\nu_1+\tfrac{\lambda_1-d_3-\alpha_{11}}{2}\bigr)
\\
&\phantom{==}\times 
\Gamma_{\bC}\bigl(t_2-\nu_3+\tfrac{d_1-\lambda_1+\alpha_{11}}{2}\bigr)
\,y_1^{-2t_1} y_2^{-2t_2} \,dt_1dt_2
\end{align*}
if $0\leq q\leq \lambda_1-\lambda_2-\alpha_{11}$, 
and $\varphi^{[\varepsilon ]}_{\chi ,\lambda}(v_{\lambda ,q})(y)=0$ otherwise. 
Hence, by Lemma \ref{lem:C32_zeta32_schur}, we have 
\begin{align*}
&Z\bigl(s,\varphi^{[\varepsilon ]}_{\chi ,\lambda}(v),
\varphi^{[-\varepsilon ]}_{\chi',\widetilde{\lambda}}(v')\bigr)\\
&=\frac{\langle v,v'\rangle}{\lambda_1-\lambda_2+1}
\sum_{q=0}^{\lambda_1-\lambda_2}
(-1)^{\lambda_2+q}\binom{\lambda_1-\lambda_2}{q}
\int_{0}^\infty \int_{0}^\infty   
\varphi^{[\varepsilon ]}_{\chi ,\lambda}
(v_{\lambda ,\lambda_1-\lambda_2-q})(\hat{y})\\
&\phantom{=}\times 
\varphi^{[-\varepsilon ]}_{\chi',\widetilde{\lambda}}
(v_{\widetilde{\lambda},q})(y)
y_1^{2s-3}y_2^{4s-2}
\frac{2dy_1}{y_1}\frac{2dy_2}{y_2}\\
&=\frac{\langle v,v'\rangle}{\lambda_1-\lambda_2+1}
\sum_{q=0}^{\lambda_1-\lambda_2-\alpha_{11}}
(-\varepsilon \sI)^{\lambda_2+q}
\binom{\lambda_1-\lambda_2}{q}
\binom{\lambda_1-\lambda_2-q}{\alpha_{11}}
\frac{1}{(2\pi \sqrt{-1})^2}\\
&\phantom{=}\times 
\int_{0}^\infty \int_{0}^\infty   
\biggl\{\int_{t_2}\int_{t_1}
\frac{\Gamma_{\bC}\bigl(t_1+\nu_2+\tfrac{\lambda_2-d_2+q}{2}\bigr)
\Gamma_{\bC}\bigl(t_2-\nu_2+\tfrac{d_1+d_2-d_3-\lambda_1+\alpha_{11}}{2}\bigr)}
{ \Gamma_{\bC}\bigl(t_1+t_2+
\tfrac{q-\lambda_1+\lambda_2+d_1-d_3+\alpha_{11}}{2}\bigr)}\\
&\phantom{==}\times 
\Gamma_{\bC}\bigl(t_1+\nu_1+\tfrac{-\lambda_2+d_1-q}{2}\bigr) 
\Gamma_{\bC}\bigl(t_1+\nu_3+\tfrac{\lambda_2-d_3+q}{2}\bigr)\\
&\phantom{==}\times 
\Gamma_{\bC}\bigl(t_2-\nu_1+\tfrac{\lambda_1-d_3-\alpha_{11}}{2}\bigr)
\Gamma_{\bC}\bigl(t_2-\nu_3+\tfrac{d_1-\lambda_1+\alpha_{11}}{2}\bigr)
\varphi^{[-\varepsilon ]}_{\chi',\widetilde{\lambda}}
(v_{\widetilde{\lambda},q})(y)\\
&\phantom{==}\times 
y_1^{-2t_1} y_2^{-2t_2} \,dt_1dt_2\biggr\}
y_1^{2s-1}y_2^{4s+2\nu_1+2\nu_2+2\nu_3+\lambda_2-d_2+\alpha_{11}}
\frac{2dy_1}{y_1}\frac{2dy_2}{y_2}.
\end{align*}

\vspace{2mm}

\underline{\bf Case 5-1:} $-d_1'>d_1$. 
Our proof in this case is similar to that in Case 4-1. 
Since $l_0=-d_1-d_1'$, 
$\alpha_{11}=-2d_1-d_1'-d_2'$, 
$\lambda_1=-d_1-d_1'-d_2'$, 
$\lambda_2=d_1$ and 
\begin{align*}
&\varphi_{\chi',\widetilde{\lambda}}^{[-\varepsilon ]}(v_{\widetilde{\lambda},0})(y)
=(-\varepsilon \sI )^{-d_1}y_1y_2^{2\nu_1'+2\nu_2'}\\
&\hspace{2cm}
\times \frac{1}{2\pi \sI}\int_t
\Gamma_{\bC}\bigl(t+\nu_1'+\tfrac{-d_1-d_1'}{2}\bigr)
\Gamma_{\bC}\bigl(t+\nu_2'+\tfrac{-d_1-d_2'}{2}\bigr)
y_1^{-2t}dt
\end{align*}
for $y=\diag (y_1y_2,y_2)\in A_2$, we have 
\begin{align*}
&Z\bigl(s,\varphi^{[\varepsilon ]}_{\chi ,\lambda}(v),
\varphi^{[-\varepsilon ]}_{\chi',\widetilde{\lambda}}(v')\bigr)
=\frac{2\langle v,v'\rangle}{-2d_1-d_1'-d_2'+1}\prod_{i=1}^3\prod_{j=1}^2
\Gamma_{\bC}\bigl(s+\nu_i+\nu_j'+\tfrac{-d_i-d_j'}{2}\bigr)
\end{align*}
by Lemmas \ref{lem:F32_Mellin} and \ref{lem:F32_Barnes2nd}.
Hence we obtain the assertion in this case. \vspace{2mm}

\underline{\bf Case 5-2:} $d_1\geq -d_1'>d_2$. 
Our proof in this case is similar to that in Case 4-2. 
Since $l_0=0$, $\alpha_{11}=\max \{0,-d_1-d_2'\}$, 
$\lambda_1=-d_2'$, $\lambda_2=-d_1'$ and 
\begin{align*}
&\binom{d_1'-d_2'}{q}
\binom{d_1'-d_2'-q}{\alpha_{11}}
=
\binom{d_1'-d_2'}{\alpha_{11}}
\binom{d_1'-d_2'-\alpha_{11}}{q},\\
&\varphi_{\chi',\widetilde{\lambda}}^{[-\varepsilon ]}
(v_{\widetilde{\lambda},q})(y)
=(-\varepsilon \sI )^{d_1'-q}y_1y_2^{2\nu_1'+2\nu_2'}\\
&\hspace{2cm}
\times \frac{1}{2\pi \sI}\int_t
\Gamma_{\bC}\bigl(t+\nu_1'+\tfrac{q}{2}\bigr)
\Gamma_{\bC}\bigl(t+\nu_2'+\tfrac{d_1'-d_2'-q}{2}\bigr)
y_1^{-2t}dt
\end{align*}
for $0\leq q\leq d_1'-d_2'$ and $y=\diag (y_1y_2,y_2)\in A_2$, 
we have 
\begin{align*}
&Z\bigl(s,\varphi^{[\varepsilon ]}_{\chi ,\lambda }(v),
\varphi^{[-\varepsilon ]}_{\chi',\widetilde{\lambda}}(v')\bigr)
=\frac{2\langle v,v'\rangle}{d_1'-d_2'+1}
\binom{d_1'-d_2'}{\alpha_{11}}\\
&\hspace{2cm}\times 
\Gamma_{\bC}(s+\nu_1+\nu_1'+\tfrac{d_1+d_1'}{2})
\Gamma_{\bC}\bigl(s+\nu_1+\nu_2'+\tfrac{d_1+d_2'}{2}+\alpha_{11}\bigr)\\
&\hspace{2cm}\times 
\Gamma_{\bC}\bigl(s+\nu_2+\nu_1'+\tfrac{-d_2-d_1'}{2}\bigr)
\Gamma_{\bC}\bigl(s+\nu_2+\nu_2'+\tfrac{-d_2-d_2'}{2}\bigr)\\
&\hspace{2cm}\times 
\Gamma_{\bC}\bigl(s+\nu_3+\nu_1'+\tfrac{-d_3-d_1'}{2}\bigr)
\Gamma_{\bC}\bigl(s+\nu_3+\nu_2'+\tfrac{-d_3-d_2'}{2}\bigr)
\end{align*}
by Lemmas \ref{lem:F32_Mellin}, \ref{lem:F32_Barnes2nd} 
and \ref{lem:F32_gauss_sum}. 
Hence we obtain the assertion in this case. \vspace{2mm}

\section{The calculation in the case $-d_2'\geq d_2\geq -d_1'$}

Here we give a proof of Theorem \ref{thm:C32_zeta_L_coincide} 
in the case 
$-d_2'\geq d_2\geq -d_1'$. 
In this case, we note that $l_0=\alpha_{21}=\alpha_{12}=0$ and 
\begin{align*}
&\lambda_1=-d_2',&
&\alpha_1=\alpha_{11}=\max \{0,-d_2'-d_1\},&
&\beta_1=-d_2'-\alpha_1,\\
&\lambda_2=-d_1',&
&\alpha_2=\alpha_{22}=\min \{0,-d_1'-d_3\},&
&\beta_2=-d_1'-\alpha_2.
\end{align*}
For $0\leq q\leq d_1'-d_2'$ and $\hat{y}=\diag (y_1y_2,y_2,1)\in A_3$, 
we have 
\begin{align*}
&\varphi^{[\varepsilon ]}_{\chi ,\lambda }(v_{\lambda ,q})(\hat{y})
=(\varepsilon \sI)^{-d_2'-q}
y_1^2y_2^{2\nu_1+2\nu_2+2\nu_3+\alpha_1-\alpha_2+2}\\
&\phantom{=.}
\times 
\binom{q}{\alpha_1}\binom{d_1'-d_2'-q}{-\alpha_2}
\sum_{i=\max \{0,q+d_2+d_2'\}
}^{\min \{q-\alpha_1,d_2+d_1'+\alpha_2\}}
\binom{q-\alpha_1}{i}
\binom{d_1'-d_2'+\alpha_2-q}{d_2+d_1'+\alpha_2-i}\\
&\phantom{=.}\times \frac{1}{(2\pi \sqrt{-1})^2} \int_{t_2}\int_{t_1}
\frac{\Gamma_{\bC}\bigl(t_1+\nu_2+\tfrac{-d_2-d_2'-q}{2}+i\bigr)}
{ \Gamma_{\bC}\bigl(t_1+t_2+
\tfrac{-q+d_1-d_2-d_3-d_1'+\alpha_1-\alpha_2}{2}+i\bigr)}\\
&\phantom{=.}\times 
\Gamma_{\bC}\bigl(t_2-\nu_2
+\tfrac{d_1-d_3-d_1'+d_2'+\alpha_1-\alpha_2}{2}\bigr)
\Gamma_{\bC}\bigl(t_1+\nu_1+\tfrac{q+d_1+d_2'}{2}\bigr)\\
&\phantom{=.}\times 
\Gamma_{\bC}\bigl(t_1+\nu_3+\tfrac{-d_3-d_2'-q}{2}\bigr)
\Gamma_{\bC}\bigl(t_2-\nu_1
+\tfrac{-d_2-d_3-d_1'-d_2'-\alpha_1-\alpha_2}{2}\bigr)\\
&\phantom{=.}\times 
\Gamma_{\bC}\bigl(t_2-\nu_3
+\tfrac{d_1+d_2+d_1'+d_2'+\alpha_1+\alpha_2}{2}\bigr)
\,y_1^{-2t_1} y_2^{-2t_2} \,dt_1dt_2
\end{align*}
if $\alpha_1\leq q\leq d_1'-d_2'+\alpha_2$, 
and $\varphi^{[\varepsilon ]}_{\chi ,\lambda}
(v_{\lambda ,q})(\hat{y})=0$ otherwise. 
For $0\leq q\leq d_1'-d_2'$ and $y=\diag (y_1y_2,y_2)\in A_2$, 
we have 
\begin{align*}
&\varphi_{\chi',\widetilde{\lambda}}^{[-\varepsilon ]}
(v_{\widetilde{\lambda},d_1'-d_2'-q})
(y)\\
&=(-\varepsilon \sI )^{d_2'+q}
\frac{y_1y_2^{2\nu_1'+2\nu_2'}}{2\pi \sI}\int_t
\Gamma_{\bC}\bigl(t+\nu_1'+\tfrac{d_1'-d_2'-q}{2}\bigr)
\Gamma_{\bC}\bigl(t+\nu_2'+\tfrac{q}{2}\bigr)
y_1^{-2t}dt. 
\end{align*}
By Lemma \ref{lem:C32_zeta32_schur} and the equality 
\begin{align*}
&\binom{d_1'-d_2'}{q}
\binom{q}{\alpha_1}\binom{d_1'-d_2'-q}{-\alpha_2}
\binom{q-\alpha_1}{i}
\binom{d_1'-d_2'+\alpha_2-q}{d_2+d_1'+\alpha_2-i}\\
&=
\binom{d_1'-d_2'}{d_2+d_1'}
\binom{-d_2-d_2'}{\alpha_1}
\binom{d_2+d_1'}{-\alpha_2}
\binom{-d_2-d_2'-\alpha_1}{q-\alpha_1-i}
\binom{d_2+d_1'+\alpha_2}{i},
\end{align*}
we have 
\begin{align*}
&Z\bigl(s,\varphi_{\chi ,\lambda}^{[\varepsilon ]}(v),
\varphi_{\chi',\widetilde{\lambda}}^{[-\varepsilon ]}(v')\bigr)\\
&=\frac{\langle v,v'\rangle}{d_1'-d_2'+1}
\sum_{q=0}^{d_1'-d_2'}
(-1)^{-d_2'-q}
\binom{d_1'-d_2'}{q}\int_{0}^\infty \int_{0}^\infty   
\varphi^{[\varepsilon ]}_{\chi ,\lambda }(v_{\lambda ,q})(\hat{y})\\
&\phantom{=}\times 
\varphi_{\chi',\widetilde{\lambda}}^{[-\varepsilon ]}
(v_{\widetilde{\lambda},d_1'-d_2'-q})(y)
y_1^{2s-3}y_2^{4s-2}
\frac{2dy_1}{y_1}\frac{2dy_2}{y_2}\\
&=
\frac{\langle v,v'\rangle }{d_1'-d_2'+1}
\binom{d_1'-d_2'}{d_2+d_1'}
\binom{-d_2-d_2'}{\alpha_1}
\binom{d_2+d_1'}{-\alpha_2}\\
&\phantom{=}\times 
\sum_{q=\alpha_1}^{d_1'-d_2'+\alpha_2}
\sum_{i=\max \{0,q+d_2+d_2'\}}^{\,\min \{q-\alpha_1,d_2+d_1'+\alpha_2\}}
\binom{-d_2-d_2'-\alpha_1}{q-\alpha_1-i}
\binom{d_2+d_1'+\alpha_2}{i}\\
&\phantom{=}\times 
\frac{1}{(2\pi \sqrt{-1})^3}
\int_{0}^\infty \int_{0}^\infty \biggl\{ 
\int_t\int_{t_2}\int_{t_1}
\frac{\Gamma_{\bC}\bigl(t_1+\nu_2+\tfrac{-d_2-d_2'-q}{2}+i\bigr)}
{\Gamma_{\bC}\bigl(t_1+t_2+
\tfrac{-q+d_1-d_2-d_3-d_1'+\alpha_1-\alpha_2}{2}+i\bigr)}\\
&\phantom{=}\times 
\Gamma_{\bC}\bigl(t_2-\nu_2
+\tfrac{d_1-d_3-d_1'+d_2'+\alpha_1-\alpha_2}{2}\bigr)
\Gamma_{\bC}\bigl(t_1+\nu_1+\tfrac{q+d_1+d_2'}{2}\bigr)\\
&\phantom{=}\times 
\Gamma_{\bC}\bigl(t_1+\nu_3+\tfrac{-d_3-d_2'-q}{2}\bigr)
\Gamma_{\bC}\bigl(t_2-\nu_1
+\tfrac{-d_2-d_3-d_1'-d_2'-\alpha_1-\alpha_2}{2}\bigr)\\
&\phantom{=}\times 
\Gamma_{\bC}\bigl(t_2-\nu_3
+\tfrac{d_1+d_2+d_1'+d_2'+\alpha_1+\alpha_2}{2}\bigr)
\Gamma_{\bC}\bigl(t+\nu_1'+\tfrac{d_1'-d_2'-q}{2}\bigr)
\Gamma_{\bC}\bigl(t+\nu_2'+\tfrac{q}{2}\bigr)\\
&\phantom{=}\times 
\,y_1^{-2t_1-2t} y_2^{-2t_2} \,dt_1dt_2dt
\biggr\}
y_1^{2s}y_2^{4s+2\nu_1+2\nu_2+2\nu_3+2\nu_1'+2\nu_2'+\alpha_1-\alpha_2}
\frac{2dy_1}{y_1}\frac{2dy_2}{y_2}.
\end{align*}
Substituting $t_1\to t_1-t$ and applying Lemma \ref{lem:F32_Mellin} twice, 
we have 
\begin{align*}
&Z\bigl(s,\varphi_{\chi ,\lambda }^{[\varepsilon ]}(v),
\varphi_{\chi',\widetilde{\lambda}}^{[-\varepsilon ]}(v')\bigr)\\
&=\frac{2\langle v,v'\rangle }{d_1'-d_2'+1}
\binom{d_1'-d_2'}{d_2+d_1'}
\binom{-d_2-d_2'}{\alpha_1}
\binom{d_2+d_1'}{-\alpha_2}\\
&\phantom{=}\times 
\Gamma_{\bC}\bigl(2s+\nu_1+\nu_2+\nu_1'+\nu_2'
+\tfrac{d_1+d_2+d_1'+d_2'}{2}+\alpha_1\bigr)\\
&\phantom{=}\times 
\Gamma_{\bC}\bigl(2s+\nu_1+\nu_3+\nu_1'+\nu_2'
+\tfrac{d_1-d_3-d_1'+d_2'}{2}+\alpha_1-\alpha_2\bigr)\\
&\phantom{=}\times 
\Gamma_{\bC}\bigl(2s+\nu_2+\nu_3+\nu_1'+\nu_2'
+\tfrac{-d_2-d_3-d_1'-d_2'}{2}-\alpha_2\bigr)\\
&\phantom{=}\times 
\sum_{q=\alpha_1}^{d_1'-d_2'+\alpha_2}
\sum_{i=\max \{0,q+d_2+d_2'\}}^{\,\min \{q-\alpha_1,d_2+d_1'+\alpha_2\}}
\binom{-d_2-d_2'-\alpha_1}{q-\alpha_1-i}
\binom{d_2+d_1'+\alpha_2}{i}
\frac{1}{4\pi \sqrt{-1}}\\
&\phantom{=}\times 
\int_t
\frac{\Gamma_{\bC}\bigl(t+\nu_1'+\tfrac{d_1'-d_2'-q}{2}\bigr)
\Gamma_{\bC}\bigl(t+\nu_2'+\tfrac{q}{2}\bigr)
\Gamma_{\bC}\bigl(-t+s+\nu_2+\tfrac{-d_2-d_2'-q}{2}+i\bigr)}
{\Gamma_{\bC}\bigl(-t+3s+\nu_1+\nu_2+\nu_3+\nu_1'+\nu_2'
+\tfrac{-q+d_1-d_2-d_3-d_1'}{2}+\alpha_1-\alpha_2+i\bigr)}\\
&\phantom{=}\times 
\Gamma_{\bC}\bigl(-t+s+\nu_1+\tfrac{q+d_1+d_2'}{2}\bigr)
\Gamma_{\bC}\bigl(-t+s+\nu_3+\tfrac{-d_3-d_2'-q}{2}\bigr)
dt.
\end{align*}
Replacing $q-\alpha_1-i\to p$ and substituting 
$t\to t+\tfrac{i-p}{2}$, we have 
\begin{align*}
&Z\bigl(s,\varphi_{\chi ,\lambda}^{[\varepsilon ]}(v),
\varphi_{\chi',\widetilde{\lambda}}^{[-\varepsilon ]}(v')\bigr)\\
&=\frac{2\langle v,v'\rangle }{d_1'-d_2'+1}
\binom{d_1'-d_2'}{d_2+d_1'}
\binom{-d_2-d_2'}{\alpha_1}
\binom{d_2+d_1'}{-\alpha_2}\\
&\phantom{=}\times 
\Gamma_{\bC}\bigl(2s+\nu_1+\nu_3+\nu_1'+\nu_2'
+\tfrac{d_1-d_3-d_1'+d_2'}{2}+\alpha_1-\alpha_2\bigr)\\
&\phantom{=}\times 
\Gamma_{\bC}\bigl(2s+\nu_2+\nu_3+\nu_1'+\nu_2'
+\tfrac{-d_2-d_3-d_1'-d_2'}{2}-\alpha_2\bigr)\\
&\phantom{=}\times 
\Gamma_{\bC}\bigl(2s+\nu_1+\nu_2+\nu_1'+\nu_2'
+\tfrac{d_1+d_2+d_1'+d_2'}{2}+\alpha_1\bigr)\\
&\phantom{=}\times 
\sum_{p=0}^{-d_2-d_2'-\alpha_1}
\sum_{i=0}^{d_2+d_1'+\alpha_2}
\binom{-d_2-d_2'-\alpha_1}{p}
\binom{d_2+d_1'+\alpha_2}{i}
\frac{1}{4\pi \sqrt{-1}}\\
&\phantom{=}\times 
\int_t
\frac{
\Gamma_{\bC}\bigl(-t+s+\nu_1+\tfrac{d_1+d_2'+\alpha_1}{2}+p\bigr)
\Gamma_{\bC}\bigl(t+\nu_1'+\tfrac{d_1'-d_2'-\alpha_1}{2}-p\bigr)}
{\Gamma_{\bC}\bigl(-t+3s+\nu_1+\nu_2+\nu_3+\nu_1'+\nu_2'
+\tfrac{\alpha_1+d_1-d_2-d_3-d_1'}{2}-\alpha_2\bigr)}\\
&\phantom{=}\times 
\Gamma_{\bC}\bigl(t+\nu_2'+\tfrac{\alpha_1}{2}+i\bigr)
\Gamma_{\bC}\bigl(-t+s+\nu_3+\tfrac{-d_3-d_2'-\alpha_1}{2}-i\bigr)\\
&\phantom{=}\times 
\Gamma_{\bC}\bigl(-t+s+\nu_2+\tfrac{-d_2-d_2'-\alpha_1}{2}\bigr)dt.
\end{align*}
Applying Lemma \ref{lem:F32_gauss_sum} twice, and 
Lemma \ref{lem:F32_Barnes2nd}, successively, we have 
\begin{align*}
&Z\bigl(s,\varphi_{\chi ,\lambda}^{[\varepsilon ]}(v),
\varphi_{\chi',\widetilde{\lambda}}^{[-\varepsilon ]}(v')\bigr)\\
&=\frac{2\langle v,v'\rangle }{d_1'-d_2'+1}
\binom{d_1'-d_2'}{d_2+d_1'}
\binom{-d_2-d_2'}{\alpha_1}
\binom{d_2+d_1'}{-\alpha_2}\\
&\phantom{=}\times 
\Gamma_{\bC}\bigl(s+\nu_1+\nu_1'+\tfrac{d_1+d_1'}{2}\bigr)
\Gamma_{\bC}\bigl(s+\nu_1+\nu_2'+\tfrac{d_1+d_2'}{2}+\alpha_1\bigr)\\
&\phantom{=}\times 
\Gamma_{\bC}\bigl(s+\nu_2+\nu_1'+\tfrac{d_2+d_1'}{2}\bigr)
\Gamma_{\bC}\bigl(s+\nu_2+\nu_2'+\tfrac{-d_2-d_2'}{2}\bigr)\\
&\phantom{=}\times 
\Gamma_{\bC}\bigl(s+\nu_3+\nu_1'+\tfrac{-d_3-d_1'}{2}-\alpha_2\bigr)
\Gamma_{\bC}\bigl(s+\nu_3+\nu_2'+\tfrac{-d_3-d_2'}{2}\bigr).
\end{align*}
Hence we obtain the assertion in this case.

%% file: appendix_GL32zeta_002.tex
\chapter{Archimedean zeta integrals for $GL(2)\times GL(m)$ ($m=1,2$)}

The archimedean zeta integrals for $GL(2)\times GL(1)$ and $GL(2)\times GL(2)$ 
were fully computed by Jacquet \cite{Jacquet_003}, 
Popa \cite{Popa_001} and the third author \cite{Miyazaki_003}. 
In the appendix, we rewrite these results in our fashion.

\section{The local zeta integrals for $GL(2,\bR)\times GL(1,\bR )$}
\label{sec:R21_arch_zeta}

In this section, we set $F=\bR$. 
Let $\Pi_\sigma $ be an 
irreducible generalized principal series representation of $G_2$. 
Let $\chi_{(\nu',\delta')}$ be 
a one dimensional representation of $G_1=\bR^\times$ 
in \S \ref{subsec:Rn_def_gps}. 
By (\ref{eqn:Rmn_weil_tensor}), we obtain the explicit form of 
the local $L$-factor for $\Pi_\sigma \times \chi_{(\nu',\delta')}$ as follows:
\begin{align*}
&L (s,\Pi_\sigma \times \chi_{(\nu',\delta')})\\
&=\left\{\begin{array}{ll}
\displaystyle 
\prod_{i=1}^2\Gamma_\bR (s+\nu_i+\nu'+|\delta_i-\delta'|)
&\text{if }\ \sigma = 
\chi_{(\nu_1 ,\delta_1 )}\boxtimes \chi_{(\nu_2 ,\delta_2 )},\\[4mm]
\displaystyle 
\Gamma_\bC \bigl(s+\nu +\nu'+\tfrac{\kappa -1}{2}\bigr)
&\text{if }\ \sigma =D_{(\nu ,\kappa )}.
\end{array}\right.
\end{align*}

Let $\varepsilon \in \{\pm 1\}$. 
For $W\in \mathrm{Wh}(\Pi_\sigma ,\psi_{\varepsilon })^{\mathrm{mg}}$, 
we consider the local zeta integral 
$Z(s,W,\chi_{(\nu',\delta')})$ for $G_{2}\times G_1$, 
which is defined in \S \ref{subsec:Fmn_main_result}, that is, 
\begin{align*}
&Z(s,W,\chi_{(\nu',\delta')})=
\int_{\bR^\times}
W\!\left(\begin{array}{cc}
h& \\
 &1
\end{array}\right) 
\chi_{(\nu',\delta')}(h)|h|^{s-\frac{1}{2}}
\frac{dh}{|h|}\\
&=
\int_0^\infty 
\left\{
W\!\left(\begin{array}{cc}
y_1& \\
 &1
\end{array}\right)
+(-1)^{\delta'}W\!\left(\begin{array}{cc}
-y_1& \\
 &1
\end{array}\right)
\right\} 
y_1^{s+\nu'-\frac{1}{2}}
\frac{dy_1}{y_1}.
\end{align*}

Jacquet--Langlands \cite[Theorem 5.1.5 (ii)]{Jacquet_Langlands_001} and 
Popa \cite[Theorem 1]{Popa_001} show the following proposition. 

\begin{prop}
\label{prop:R21_zeta_L_coincide}
Let $\Pi_\sigma$ be an irreducible generalized principal 
series representation of $G_2=GL(2,\bR )$. 
Let $\nu'\in \bC$, $\delta'\in \{0,1\}$ and $\varepsilon \in \{\pm 1\}$. 
Let $\varphi_{\sigma}^{[\varepsilon ]}$ be 
the $K_2$-homomorphism in \S \ref{subsec:R32_Wh_GL2}. 
Then, for $s\in \bC$ with sufficiently large real part, it holds that 
$Z(s,W,\chi_{(\nu',\delta')})
=L(s,\Pi_\sigma \times \chi_{(\nu',\delta')})$ 
with 
\begin{align*}
&W=\left\{\begin{array}{ll}
\displaystyle 
\varphi_{\sigma}^{[\varepsilon ]}(v_{(0,\delta'),0})
&\text{if }\ \sigma = 
\chi_{(\nu_1 ,\delta' )}\boxtimes \chi_{(\nu_2 ,\delta' )},\\[2mm]
\displaystyle 
(2\pi \varepsilon \sI)^{-1}R(E_{1,2}^{{\g_2}})
\varphi_{\sigma}^{[\varepsilon ]}(v_{(0,1-\delta'),0})
&\text{if }\ \sigma = 
\chi_{(\nu_1 ,1-\delta')}\boxtimes \chi_{(\nu_2 ,1-\delta')},\\[2mm]
\displaystyle 
\varphi_{\sigma}^{[\varepsilon ]}(v_{(1,0),\varepsilon })
&\text{if }\ \sigma = \chi_{(\nu_1 ,1)}\boxtimes \chi_{(\nu_2 ,0)},\\[2mm]
\displaystyle 
\varphi_{\sigma}^{[\varepsilon ]}(v_{(\kappa ,0),\varepsilon \kappa })
&\text{if }\ \sigma =D_{(\nu ,\kappa )}.
\end{array}\right.
\end{align*}
\end{prop}

This proposition follows immediately from the explicit formula 
in \S \ref{subsec:R32_Wh_GL2} with (\ref{eqn:R2_Kact}), 
Lemmas \ref{lem:Rn_g_act_Cpsi}, \ref{lem:F32_Mellin} 
and the equality 
$\Ad \bigl(k_{(-1,1)}^{(1,1)}\bigr)E_{1,2}^{{\g_2}}
=-E_{1,2}^{{\g_2}}$.

\section{The local zeta integrals for $GL(2,\bR)\times GL(2,\bR )$}
\label{sec:R22_arch_zeta}

In this section, we set $F=\bR$. 
Let $\Pi_\sigma $ and $\Pi_{\sigma'}$ be 
irreducible generalized principal series representations of $G_2$. 
By (\ref{eqn:Rmn_weil_tensor}), we obtain the explicit form of 
the local $L$-factor for $\Pi_\sigma \times \Pi_{\sigma'}$ as follows:
\begin{itemize}
\item When $\sigma =
\chi_{(\nu_1 ,\delta_1 )}\boxtimes \chi_{(\nu_2 ,\delta_2 )}$ 
and $\sigma' =\chi_{(\nu_1',\delta_1')}\boxtimes \chi_{(\nu_2',\delta_2')}$, 
we have 
\begin{align*}
&L (s,\Pi_\sigma \times \Pi_{\sigma'})=
\prod_{i=1}^2\prod_{j=1}^2
\Gamma_\bR (s+\nu_i+\nu_j'+|\delta_i-\delta_j'|).
\end{align*}

\item When $\sigma =
\chi_{(\nu_1 ,\delta_1 )}\boxtimes \chi_{(\nu_2 ,\delta_2 )}$ 
and $\sigma' =D_{(\nu',\kappa')}$, we have 
\begin{align*}
&L (s,\Pi_\sigma \times \Pi_{\sigma'})=
\prod_{i=1}^2
\Gamma_\bC \bigl(s+\nu_i+\nu '+\tfrac{\kappa '-1}{2}\bigr).
\end{align*}


\item When $\sigma =D_{(\nu ,\kappa )}$ 
and $\sigma' =D_{(\nu',\kappa')}$, 
we have 
\begin{align*}
L (s,\Pi_\sigma \times \Pi_{\sigma'})=\,&
\Gamma_\bC \bigl(s+\nu +\nu '+\tfrac{\kappa +\kappa '-2}{2}\bigr)
\Gamma_\bC \bigl(s+\nu +\nu '+\tfrac{|\kappa -\kappa '|}{2}\bigr).
\end{align*}

\end{itemize}

Let $\cS (\bR^2)$ be the space of Schwartz functions on $\bR^2$. 
Let $\cS (\bR^2)^{\mathrm{std}}$ be the subspace of $\cS (\bR^2)$ 
consisting of all functions $f$ of the form 
\begin{align*}
&f(z_1,z_2)=p(z_1,z_2)
\exp (-\pi (z_1^2+z_2^2))&
&(z_1,z_2\in \bR )
\end{align*}
with polynomial functions $p$ on $\bR^2$. 
We call functions in $\cS (\bR^2)^{\mathrm{std}}$ 
standard Schwartz functions on $\bR^2$. 
For non-negative integers $a$ and $b$, 
we define a standard Schwartz function $f_{(a,b)} $ on $ \bR^2 $ by 
\begin{equation}
\label{eqn:R22_def_fab}
f_{(a,b)}(z_1,z_2)=(-\sI z_1+z_2)^a
(\sqrt{-1}z_1+z_2)^b \exp (-\pi (z_1^2+z_2^2)).
\end{equation}

Let $\varepsilon \in \{\pm 1\}$. 
For $W\in \mathrm{Wh}(\Pi_\sigma ,\psi_{\varepsilon })^{\mathrm{mg}}$, 
$W'\in \mathrm{Wh}(\Pi_{\sigma'} ,\psi_{-\varepsilon })^{\mathrm{mg}}$ 
and $f\in \cS (\bC^2)$, 
we define the local zeta integral $Z(s,W,W',f)$ for $G_2\times G_2$ by 
\begin{align*}
Z(s,W,W',f)
=&
\int_{N_2\backslash G_2}W(g)W'(g)f((0,1)g)|\det g|^{s}\,d\dot{g}\\
=&\sum_{\varepsilon_1 \in \{\pm 1\}}
\int_0^\infty \!\!\int_0^\infty  
\biggl(\int_{0}^{2\pi}
W(yk_{(\varepsilon_1,1)}^{(1,1)}k_{\theta}^{(2)})
W'(yk_{(\varepsilon_1,1)}^{(1,1)}k_{\theta}^{(2)})\\
&\times f(-y_2\sin \theta, y_2 \cos \theta)
\,\frac{d\theta}{\pi}\biggr)
(y_1y_2^2)^s\frac{dy_1}{y_1^2}\frac{dy_2}{y_2}
\end{align*} 
with $y=\diag (y_1y_2,y_2)\in A_2$. 
Here $d\dot{g}$ is the right invariant measure on $N_2\backslash G_2$ 
normalized by (\ref{eqn:R32_normalize_N2_G2}). 
The local zeta integral $Z(s,W,W',f)$ converges 
for $\mathrm{Re}(s)\gg 0$. Jacquet proves that the local zeta integral 
has the meromorphic continuation and satisfies the local functional equation 
(\cite{Jacquet_003}). 

In Theorems 17.2 (3) of \cite{Jacquet_003}, 
he asserts only that the associated 
local $L$-factor $L(s,\Pi_\sigma \times \Pi_{\sigma'})$ 
can be expressed as a finite sum of 
some local zeta integrals for $G_2\times G_2$. 
However, in the proof, he shows the following stronger result.

\begin{prop}
\label{prop:R22_zeta}
Let $\Pi_\sigma$ and $\Pi_{\sigma'}$ be 
irreducible generalized principal series representations of 
$G_2=GL(2,\bR )$. 
Let $\varepsilon \in \{\pm 1\}$. 
Let $\varphi_{\sigma}^{[\varepsilon ]}$ and 
$\varphi_{\sigma'}^{[-\varepsilon ]}$ be 
the $K_2$-homomorphisms in \S \ref{subsec:R32_Wh_GL2}. 
Then, for $s\in \bC$ with sufficiently large real part, it holds that 
\[
Z(s,W,W',f)
=L(s,\Pi_\sigma \times \Pi_{\sigma'}),
\]
where $W$, $W'$ and $f$ are given as follows:

\medskip

\noindent 
{\bf (Case 1)} 
$\sigma = \chi_{(\nu_1,\delta_1)} \boxtimes \chi_{(\nu_2,\delta_2)}$ 
and $\sigma' = \chi_{(\nu_1',\delta_1')} 
\boxtimes \chi_{(\nu_2',\delta_2')}$. \vspace{1mm}
\begin{description}
\item[(Case 1-1)] The case 
$\delta_1=\delta_2=\delta_1'=\delta_2'$: 

\noindent 
$W=\varphi_{\sigma}^{[\varepsilon ]}(v_{(0,\delta_2),0})$, 
$W'=\varphi_{\sigma'}^{[-\varepsilon ]}(v_{(0,\delta_2'),0})$ and 
$f= f_{(0,0)}$. \vspace{1mm}

\item[(Case 1-2)] The case 
$\delta_1=\delta_2=1$ and $\delta_1'=\delta_2'=0$: 

\noindent 
$W=(-8\pi \varepsilon )^{-1}
R(X_{2}^{\gp_2})\varphi_{\sigma}^{[\varepsilon ]}(v_{(0,1),0})$, 
$W'=\varphi_{\sigma'}^{[-\varepsilon ]}(v_{(0,0),0})$ and 
$f= f_{(0,2)}$. 

\noindent 
Here $X_{2}^{\gp_2}$ is defined by (\ref{eqn:R2_def_Xr}). 
\vspace{1mm}

\item[(Case 1-3)] The case 
$(\delta_1,\delta_2)=(1,0)$ and 
$\delta_1'=\delta_2'$: 

\noindent 
$W=\varphi_{\sigma}^{[\varepsilon ]}(v_{(1,0),1})$, 
$W'=\varepsilon^{\delta_2'}
\varphi_{\sigma'}^{[-\varepsilon ]}(v_{(0,\delta_2'),0})$ and 
$f= f_{(0,1)}$. \vspace{1mm}

\item[(Case 1-4)] The case 
$(\delta_1,\delta_2)=(\delta_1',\delta_2')=(1,0)$:

\noindent 
$W=\varphi_{\sigma}^{[\varepsilon ]}(v_{(1,0),1})$, 
$W'=\varphi_{\sigma'}^{[-\varepsilon ]}(v_{(1,0),-1})$ and 
$f= f_{(0,0)}$. \vspace{1mm}
\end{description}

\noindent
{\bf (Case 2)} 
$ \sigma = \chi_{(\nu_1,\delta_1)} \boxtimes \chi_{(\nu_2,\delta_2)}$ 
and $ \Pi \cong D_{(\nu',\kappa')} $.\vspace{1mm}
\begin{description}
\item[(Case 2-1)] The case $\delta_1=\delta_2$:

\noindent 
$W=\varepsilon^{\delta_2}
\varphi_{\sigma}^{[\varepsilon ]}(v_{(0,\delta_2),0})$, 
$W'=\varphi_{\sigma'}^{[-\varepsilon ]}(v_{(\kappa',0),-\kappa'})$ and 
$f= f_{(\kappa' ,0)}$. \vspace{1mm}

\item[(Case 2-2)] The case $ (\delta_1,\delta_2)=(1,0) $:

\noindent 
$W=\varphi_{\sigma}^{[\varepsilon ]}(v_{(1,0),1})$, 
$W'=\varphi_{\sigma'}^{[-\varepsilon ]}(v_{(\kappa',0),-\kappa'})$ and 
$f= f_{(\kappa'-1,0)}$. \vspace{1mm}

\end{description}

\noindent
{\bf (Case 3)}
$ \Pi \cong D_{(\nu,\kappa)} $ and $ \Pi' \cong D_{(\nu',\kappa')} $ 
$(\kappa \geq \kappa')$: 

$W=\varphi_{\sigma}^{[\varepsilon ]}(v_{(\kappa ,0),\kappa})$, 
$W'=\varphi_{\sigma'}^{[-\varepsilon ]}(v_{(\kappa',0),-\kappa'})$ and 
$f= 2^{\kappa'-1}f_{(0,\kappa -\kappa')}$. 

\end{prop}
\begin{proof}
From the definition of $f_{(a,b)} $ and (\ref{eqn:R2_Kact}), 
we know that
\begin{align}
\label{eqn:R22_fab_yk}
&f_{(a,b)}(-y_2\sin \theta, y_2 \cos \theta) 
 = e^{ \sqrt{-1}(a-b)\theta } 
y_2^{a+b} \exp (-\pi y_2^2)&
&(y_2>0,\ \theta \in \bR ).
\end{align}

\vspace{2mm}

\underline{\bf Case 1-1:} 
From (\ref{eqn:R2_Kact}) and (\ref{eqn:R22_fab_yk}), we have
\begin{align*}
&Z(s,W,W',f)\\
&=4\int_0^{\infty} \!\! \int_0^{\infty} 
\varphi_{\sigma}^{[\varepsilon ]}(v_{(0,\delta_2),0})(y) 
\varphi_{\sigma'}^{[-\varepsilon ]}(v_{(0,\delta_2'),0})(y)
\exp(-\pi y_2^2) 
y_1^sy_2^{2s} \frac{dy_1}{y_1^2} \frac{dy_2}{y_2}. 
\end{align*}
In view of the explicit formulas 
\begin{align*}
\varphi_{\sigma}^{[\varepsilon ]}
(v_{(0,\delta_2),0})(y)\,&=
\frac{y_1^{1/2}y_2^{\nu_1+\nu_2}}{4\pi \sI}
\int_t
\Gamma_\bR (t+\nu_1)
\Gamma_\bR (t+\nu_2)y_1^{-t}dt,\\
\varphi_{\sigma'}^{[-\varepsilon ]}
(v_{(0,\delta_2'),0})(y)\,&=
\frac{y_1^{1/2}y_2^{\nu_1'+\nu_2'}}{4\pi \sI}
\int_t
\Gamma_\bR (t+\nu_1')
\Gamma_\bR (t+\nu_2')y_1^{-t}dt
\end{align*}
in \S \ref{subsec:R32_Wh_GL2}, 
we can integrate with respect to $y_2$:  
\[
\int_0^{\infty} \exp(-\pi y_2^2) y_2^{2s+\nu_1+\nu_2+\nu_1'+\nu_2'}
\frac{dy_2}{y_2} = 2^{-1} \Gamma_{\bR}(2s+\nu_1+\nu_2+\nu_1'+\nu_2'). 
\] 
Further, Lemma \ref{lem:F32_Mellin} implies that 
\begin{align*}
& Z(s,W,W',f)
= \Gamma_{\bR}(2s+\nu_1+\nu_2+\nu_1'+\nu_2') \\
& \hspace{10mm}\times 
\frac{1}{4\pi \sqrt{-1}} \int_t \Gamma_{\bR}(t+\nu_1)\Gamma_{\bR}(t+\nu_2)
 \Gamma_{\bR}(s-t+\nu_1') \Gamma_{\bR}(s-t+\nu_2') dt.
\end{align*}
Thus, by Lemma \ref{lem:F32_Barnes_1st}, we get 
\[
Z(s,W,W',f) = \prod_{i=1}^2 \prod_{j=1}^2 \Gamma_{\bR}(s+\nu_i+\nu_j'). 
\]

\medskip

\underline{\bf Case 1-2:} 
By (\ref{eqn:R2_Kact}) and (\ref{eqn:R2_Kact_p2}), we have 
\begin{align*}
\bigl(R(X_{2}^{\gp_2})\varphi_{\sigma}^{[\varepsilon ]}(v_{(0,1),0})\bigr)
\bigl(yk_{(\varepsilon_1,1)}^{(1,1)}k_{\theta}^{(2)}\bigr)
&=
\bigl(R(k_{(\varepsilon_1,1)}^{(1,1)}k_{\theta}^{(2)})
R(X_{2}^{\gp_2})\varphi_{\sigma}^{[\varepsilon ]}(v_{(0,1),0})\bigr)(y)\\ 
&=
\varepsilon_1e^{2\sI \theta }\bigl(R(X_{2\varepsilon_1}^{\gp_2})
\varphi_{\sigma}^{[\varepsilon ]}(v_{(0,1),0})\bigr)(y),\\
\varphi_{\sigma'}^{[-\varepsilon ]}(v_{(0,0),0})
\bigl(yk_{(\varepsilon_1,1)}^{(1,1)}k_{\theta}^{(2)}\bigr)
&=\varphi_{\sigma'}^{[-\varepsilon ]}(v_{(0,0),0})(y)
\end{align*}
for $y=\diag (y_1y_2,y_2)\in A_2$, $\varepsilon_1\in \{\pm 1\}$ and 
$\theta \in \bR$. 
These equalities and (\ref{eqn:R22_fab_yk}) imply that 
\begin{align*}
Z(s,W,W',f)
=\,&
2\int_0^\infty \!\!\int_0^\infty 
(-8\pi \varepsilon )^{-1}
\bigl(R(X_{2}^{\gp_2}-X_{-2}^{\gp_2})
\varphi_{\sigma}^{[\varepsilon ]}(v_{(0,1),0})\bigr)(y)\\
&\times 
\varphi_{\sigma'}^{[-\varepsilon ]}(v_{(0,0),0})(y)
\exp(-\pi y_2^2)y_1^sy_2^{2s+2}\frac{dy_1}{y_1^2}\frac{dy_2}{y_2}
\end{align*} 
By Lemma \ref{lem:Rn_g_act_Cpsi}, 
(\ref{eqn:R2_gkact}) and the explicit formulas in \S \ref{subsec:R32_Wh_GL2}, 
we find that 
\begin{align*}
&(-8\pi \varepsilon )^{-1}
\bigl(R(X_{2}^{\gp_2}-X_{-2}^{\gp_2})\varphi_{\sigma}^{[\varepsilon ]}
(v_{(0,1),0})\bigr)(y)\\
&\hspace{15mm}=
\frac{1}{2\pi \varepsilon \sI }
\bigl(R(2E_{1,2}^{{\g}}-E_{1,2}^{{\gk}})
\varphi_{\sigma}^{[\varepsilon ]}
(v_{(0,1),0})\bigr)(y)
=2y_1\varphi_{\sigma}^{[\varepsilon ]}(v_{(0,1),0})(y)\\
&\hspace{15mm}=\frac{y_1^{1/2}y_2^{\nu_1+\nu_2}}{2\pi \sI}
\int_t\Gamma_\bR (t+\nu_1+1)
\Gamma_\bR (t+\nu_2+1)y_1^{-t}dt.
\end{align*}
Hence, as is Case 1-1, 
by Lemmas \ref{lem:F32_Mellin} and \ref{lem:F32_Barnes_1st}, we obtain 
\[
Z(s,W,W',f)= \prod_{i=1}^2 \prod_{j=1}^2 \Gamma_{\bR}(s+\nu_i+\nu_j'+1).
\]

\medskip

\underline{\bf Case 1-3:} 
The equalities (\ref{eqn:R2_Kact}) and (\ref{eqn:R22_fab_yk})
imply that 
\begin{align*}
Z(s,W,W',f)
=&2\int_0^\infty \!\!\int_0^\infty  \varphi_{\sigma}^{[\varepsilon ]}
(v_{(1,0),1}+(-1)^{\delta_2'}v_{(1,0),-1})(y)
\\
&\times 
\varepsilon^{\delta_2'}
\varphi_{\sigma'}^{[-\varepsilon ]}(v_{(0,\delta_2'),0})(y)
\exp(-\pi y_2^2)y_1^sy_2^{2s+1}\frac{dy_1}{y_1^2}\frac{dy_2}{y_2}.
\end{align*} 
According to the explicit formulas in \S \ref{subsec:R32_Wh_GL2}, 
we can see that
\begin{align*}
&\varphi_{\sigma}^{[\varepsilon ]}
(v_{(1,0),1}+(-1)^{\delta_2'}v_{(1,0),-1})(y) \\
& = 
\varepsilon^{\delta_2'} 
\frac{y_1^{1/2}y_2^{\nu_1+\nu_2}}{2\pi \sI}
\int_t
\Gamma_\bR (t+\nu_1+1-\delta_2')
\Gamma_\bR (t+\nu_2+\delta_2')
y_1^{-t}dt.
\end{align*}
Hence, as is Case 1-1, 
by Lemmas \ref{lem:F32_Mellin} and \ref{lem:F32_Barnes_1st}, we obtain 
\[
Z(s,W,W',f) = \prod_{j=1}^2
\Gamma_{\bR}(s+\nu_1+\nu_j'+1-\delta_2')
\Gamma_{\bR}(s+\nu_2+\nu_j'+\delta_2').
\]

\medskip

\underline{\bf Case 1-4:} 
The equalities (\ref{eqn:R2_Kact}) and (\ref{eqn:R22_fab_yk})
imply that 
\begin{align*}
&Z(s,W,W',f)\\
&=
2\int_0^\infty \!\!\int_0^\infty  
\bigl\{\varphi_{\sigma}^{[\varepsilon ]}(v_{(1,0),1})(y)
\varphi_{\sigma'}^{[-\varepsilon ]}(v_{(1,0),-1})(y)\\
&\phantom{=.}
+\varphi_{\sigma}^{[\varepsilon ]}(v_{(1,0),-1})(y)
\varphi_{\sigma'}^{[-\varepsilon ]}(v_{(1,0),1})(y)\bigr\}
\exp(-\pi y_2^2) y_1^sy_2^{2s} 
\frac{dy_1}{y_1^2}\frac{dy_2}{y_2}\\
&=
\int_0^\infty \!\!\int_0^\infty  
\bigl\{\varphi_{\sigma}^{[\varepsilon ]}(v_{(1,0),1}+v_{(1,0),-1})(y)
\varphi_{\sigma'}^{[-\varepsilon ]}(v_{(1,0),1}+v_{(1,0),-1})(y)\\
&\phantom{=.}
-\varphi_{\sigma}^{[\varepsilon ]}(v_{(1,0),1}-v_{(1,0),-1})(y)
\varphi_{\sigma'}^{[-\varepsilon ]}(v_{(1,0),1}-v_{(1,0),-1})(y)\bigr\}\\
&\phantom{=.}
\times \exp(-\pi y_2^2) y_1^sy_2^{2s} 
\frac{dy_1}{y_1^2}\frac{dy_2}{y_2}.
\end{align*} 
By the explicit formulas in \S \ref{subsec:R32_Wh_GL2} 
and Lemma \ref{lem:F32_Mellin}, we have 
\begin{align*}
&Z(s,W,W',f) \\
& =  \Gamma_{\bR}(2s+\nu_1+\nu_2+\nu_1'+\nu_2') \\
&\phantom{=} \times \frac{1}{4\pi \sqrt{-1}} \int_t 
    \{ \Gamma_{\bR}(t+\nu_1+1)\Gamma_{\bR}(t+\nu_2) 
       \Gamma_{\bR}(s-t+\nu_1'+1)\Gamma_{\bR}(s-t+\nu_2')
\\
&\phantom{=} + \Gamma_{\bR}(t+\nu_1)\Gamma_{\bR}(t+\nu_2+1) 
       \Gamma_{\bR}(s-t+\nu_1')\Gamma_{\bR}(s-t+\nu_2'+1) \} dt.
\end{align*}
By Lemma \ref{lem:F32_Barnes_1st} and (\ref{eqn:Fn_FE_GammaRC}), 
we obtain 
\begin{align*}
&Z(s,W,W',f) \\
& = \frac{\Gamma_{\bR}(2s+\nu_1+\nu_2+\nu_1'+\nu_2')}
     { \Gamma_{\bR}(2s+\nu_1+\nu_2+\nu_1'+\nu_2'+2)}
\\
& \phantom{..}
\times\{  \Gamma_{\bR}(s+\nu_1+\nu_1'+2)\Gamma_{\bR}(s+\nu_1+\nu_2'+1)
   \Gamma_{\bR}(s+\nu_2+\nu_1'+1) \Gamma_{\bR}(s+\nu_2+\nu_2')
\\
& \phantom{..}
+ \Gamma_{\bR}(s+\nu_1+\nu_1')\Gamma_{\bR}(s+\nu_1+\nu_2'+1)
   \Gamma_{\bR}(s+\nu_2+\nu_1'+1) \Gamma_{\bR}(s+\nu_2+\nu_2'+2)\}\\
& =\Gamma_{\bR}(s+\nu_1+\nu_1')\Gamma_{\bR}(s+\nu_1+\nu_2'+1)
   \Gamma_{\bR}(s+\nu_2+\nu_1'+1) \Gamma_{\bR}(s+\nu_2+\nu_2').
\end{align*}

\medskip

\underline{\bf Case 2-1:}
The equalities (\ref{eqn:R2_Kact}) and (\ref{eqn:R22_fab_yk})
imply that 
\begin{align*}
Z(s,W,W',f)\,
&=
2\int_0^\infty \!\!\int_0^\infty  
\varepsilon^{\delta_2}\varphi_{\sigma}^{[\varepsilon ]}(v_{(0,\delta_2),0})(y)
\bigl\{\varphi_{\sigma'}^{[-\varepsilon ]}(v_{(\kappa',0),-\kappa'})(y)\\
&\phantom{=.}
+(-1)^{\delta_2}
\varphi_{\sigma'}^{[-\varepsilon ]}(v_{(\kappa',0),\kappa'})(y)\bigr\}
\exp(-\pi y_2^2) y_1^sy_2^{2s+\kappa'} 
\frac{dy_1}{y_1^2}\frac{dy_2}{y_2}.
\end{align*} 
By the explicit formulas in \S \ref{subsec:R32_Wh_GL2} 
and Lemma \ref{lem:F32_Mellin}, we have 
\begin{align*}
&Z(s,W,W',f)
=
\Gamma_{\bR}(2s+\nu_1+\nu_2+2\nu'+\kappa')  \\
& \hspace{1.5cm}
\times \frac{1}{4\pi \sqrt{-1}} \int_t 
       \Gamma_{\bR}(t+\nu_1) \Gamma_{\bR}(t+\nu_2) 
       \Gamma_{\bC} \bigl(s-t+\nu'+\tfrac{\kappa'-1}{2}\bigr) dt.
\end{align*} 
By Lemma \ref{lem:F32_Barnes_1st} and 
(\ref{eqn:Fn_gammaRC_duplication}), 
we find that  
\begin{align}
\label{eqn:R22_CR_barnes}
&\frac{1}{4\pi \sI}\int_z
\Gamma_{\bR} (z+a_1)\Gamma_{\bR} (z+a_2)\Gamma_{\bC} (-z+b_1)dz
=\frac{\Gamma_{\bC} (a_1+b_1)
\Gamma_{\bC} (a_2+b_1)}{\Gamma_{\bR} (a_1+a_2+2b_1+1)}
\end{align}
for $a_1,a_2,b_1\in \bC $ such that 
$\mathrm{Re}(a_i+b_1)>0$ $(i=1,2)$, 
where the path of integration $\int_{z}$ is the vertical line 
from $\mathrm{Re}(z)-\sI \infty$ to $\mathrm{Re}(z)+\sI \infty$ 
with the real part 
$\max \{-\mathrm{Re}(a_1),-\mathrm{Re}(a_2)\}<\mathrm{Re}(z)
<\mathrm{Re}(b_1)$. 
Thus we obtain
\[
Z(s,W,W',f) = 
\Gamma_{\bC}\bigl(s+\nu_1+\nu'+\tfrac{\kappa'-1}{2}\bigr)
\Gamma_{\bC}\bigl(s+\nu_2+\nu'+\tfrac{\kappa'-1}{2}\bigr). 
\]

\medskip

\underline{\bf Case 2-2:}
The equalities (\ref{eqn:R2_Kact}) and (\ref{eqn:R22_fab_yk})
imply that 
\begin{align*}
&Z(s,W,W',f)\\
&=
2\int_0^\infty \!\!\int_0^\infty  
\bigl\{\varphi_{\sigma}^{[\varepsilon ]}(v_{(1,0),1})(y)
\varphi_{\sigma'}^{[-\varepsilon ]}(v_{(\kappa',0),-\kappa'})(y)\\
&\phantom{=.}
+\varphi_{\sigma}^{[\varepsilon ]}(v_{(1,0),-1})(y)
\varphi_{\sigma'}^{[-\varepsilon ]}(v_{(\kappa',0),\kappa'})(y)\bigr\}
\exp(-\pi y_2^2) y_1^sy_2^{2s+\kappa'-1} 
\frac{dy_1}{y_1^2}\frac{dy_2}{y_2}.
\end{align*} 
By the explicit formulas in \S \ref{subsec:R32_Wh_GL2} 
and Lemma \ref{lem:F32_Mellin}, we have 
\begin{align*}
&Z(s,W,W',f)= 
\Gamma_{\bR}(2s+\nu_1+\nu_2+2\nu'+\kappa'-1)  \\
& \hspace{15mm}
\times \frac{1}{4\pi \sqrt{-1}} \int_t 
       \{\Gamma_{\bR}(t+\nu_1+1) \Gamma_{\bR}(t+\nu_2) 
     +\Gamma_{\bR}(t+\nu_1) \Gamma_{\bR}(t+\nu_2+1) \} \\
& \hspace{15mm}\times 
       \Gamma_{\bC} \bigl(s-t+\nu'+\tfrac{\kappa'-1}{2}\bigr) dt.
\end{align*}
By (\ref{eqn:R22_CR_barnes}) and (\ref{eqn:Fn_FE_GammaRC}), 
we obtain 
\begin{align*}
 Z(s,W,W',f)
 = \,&\frac{\Gamma_{\bR}(2s+\nu_1+\nu_2+2\nu'+\kappa'-1)}
{\Gamma_{\bR}(2s+\nu_1+\nu_2+2\nu'+\kappa'+1)}
\\
&\times 
  \left\{ \Gamma_{\bC}\bigl(s+\nu_1+\nu'+\tfrac{\kappa'+1}{2}\bigr)
  \Gamma_{\bC}\bigl(s+\nu_2+\nu'+\tfrac{\kappa'-1}{2}\bigr) \right. \\
& \left. + 
\Gamma_{\bC}\bigl(s+\nu_1+\nu'+\tfrac{\kappa'-1}{2}\bigr)
  \Gamma_{\bC}\bigl(s+\nu_2+\nu'+\tfrac{\kappa'+1}{2} \bigr)\right\}\\
=&
\Gamma_{\bC}\bigl(s+\nu_1+\nu'+\tfrac{\kappa'-1}{2}\bigr)
\Gamma_{\bC}\bigl(s+\nu_2+\nu'+\tfrac{\kappa'-1}{2}\bigr).
\end{align*}

\medskip

\underline{\bf Case 3:}
The equalities (\ref{eqn:R2_Kact}) and (\ref{eqn:R22_fab_yk})
imply that 
\begin{align*}
&Z(s,W,W',f)\\
&=
2^{\kappa'}
\int_0^\infty \!\!\int_0^\infty  
\bigl\{\varphi_{\sigma}^{[\varepsilon ]}(v_{(\kappa,0),\kappa})(y)
\varphi_{\sigma'}^{[-\varepsilon ]}(v_{(\kappa',0),-\kappa'})(y)\\
&\phantom{=.}
+\varphi_{\sigma}^{[\varepsilon ]}(v_{(\kappa ,0),-\kappa })(y)
\varphi_{\sigma'}^{[-\varepsilon ]}(v_{(\kappa',0),\kappa'})(y)\bigr\}
\exp(-\pi y_2^2) y_1^sy_2^{2s+\kappa -\kappa'} 
\frac{dy_1}{y_1^2}\frac{dy_2}{y_2}.
\end{align*} 
By the explicit formulas in \S \ref{subsec:R32_Wh_GL2}, we have 
\begin{align*}
Z(s,W,W',f)
=\,&2^{\kappa'+2}
\left(
\int_0^\infty 
\exp (-4\pi y_1)
y_1^{s+\nu  +\nu' +\tfrac{\kappa +\kappa'-2}{2}}
\frac{dy_1}{y_1}
\right)\\
&\times \left(
\int_0^\infty  
\exp(-\pi y_2^2)y_2^{2s+2\nu +2\nu'+\kappa -\kappa'}\frac{dy_2}{y_2}
\right)\\
=\,&
\Gamma_\bC \bigl(s+\nu  +\nu' +\tfrac{\kappa +\kappa'-2}{2}\bigr)
\Gamma_\bC \bigl(s+\nu +\nu'+\tfrac{\kappa -\kappa'}{2}\bigr).
\end{align*} 
Thus we complete a proof of Proposition \ref{prop:R22_zeta}.
\end{proof}

\section{The local zeta integrals for $GL(2,\bC)\times GL(1,\bC )$}
\label{sec:C21_arch_zeta}

In this section, we set $F=\bC$. 
Let $\Pi_\chi $ be an 
irreducible principal series representation of $G_2$ 
with $\chi =\chi_{(\nu_1,d_1)}\boxtimes \chi_{(\nu_2,d_2)}$ 
$(d_1\geq d_2)$. 
Let $\chi_{(\nu',d')}$ be a character of $G_1=\bC^\times$ 
in \S \ref{subsec:Cn_def_ps}. 
Then the explicit form of the local $L$-factor for 
$\Pi_\chi \times \chi_{(\nu',d')}$ is given by 
\begin{align*}
&L (s,\Pi_{\chi}\times \chi_{(\nu',d')})=
\prod_{i=1}^2
\Gamma_\bC \Bigl(s+\nu_i+\nu'+\tfrac{|d_i+d'|}{2}\Bigr).
\end{align*}

Let $\varepsilon \in \{\pm 1\}$. 
For $W\in \mathrm{Wh}(\Pi_\chi ,\psi_{\varepsilon })^{\mathrm{mg}}$, 
we consider the local zeta integral 
$Z(s,W,\chi_{(\nu',d')})$ for $G_{2}\times G_1$, 
which is defined in \S \ref{subsec:Fmn_main_result}, that is, 
\begin{align*}
&Z(s,W,\chi_{(\nu',d')})=
\int_{\bC^\times}W\!\left(
\begin{array}{cc}
h& \\
 &1
\end{array}
\right) \chi_{(\nu',d')}(h)|h|^{2s-1}
\frac{dh}{\pi |h|^2}\\
&=
\int_{0}^\infty
\left\{\int_{0}^{2\pi}
e^{\sI d'\theta}W\!\left(\begin{array}{cc}
y_1e^{\sI \theta}& \\
 &1
\end{array}\right)\frac{d\theta }{2\pi}
\right\} 
y_1^{2s+2\nu'-1}
\frac{2dy_1}{y_1}.
\end{align*}

Jacquet--Langlands \cite[Theorem 6.4]{Jacquet_Langlands_001} and 
Popa \cite[Theorem 1]{Popa_001} show the following proposition.

\begin{prop}
\label{prop:C21_zeta_L_coincide}
Let $\Pi_\chi$ be an irreducible principal 
series representation of $G_2=GL(2,\bC )$ with 
$\chi =\chi_{(\nu_1,d_1)}\boxtimes \chi_{(\nu_2,d_2)}$ $(d_1\geq d_2)$. 
Let $\nu'\in \bC$, $d'\in \bZ$ and $\varepsilon \in \{\pm 1\}$. 
Set $\lambda =(\lambda_1,\lambda_2)=(d_1+l_0,d_2-l_0)$ with 
$l_0=\max \{0,-d_1-d',d_2+d'\}$. 
Then, for $s\in \bC$ with sufficiently large real part, it holds that 
\[
Z\bigl(s,\varphi_{\chi ,\lambda }^{[\varepsilon ]}
(v_{\lambda ,\lambda_1+d'}),\chi_{(\nu',d')}\bigr)
=(\varepsilon \sI )^{-d'}L(s,\Pi_\chi \times \chi_{(\nu',d')}),
\] 
where $\varphi_{\chi ,\lambda}^{[\varepsilon ]}$ is 
the $K_2$-homomorphism in \S \ref{subsec:C32_Wh_GL2}. 
\end{prop}

This proposition follows immediately from 
the explicit formulas in \S \ref{subsec:C32_Wh_GL2} 
with (\ref{eqn:C2_Mact}) and Lemma \ref{lem:F32_Mellin}.

\section{The local zeta integrals for $GL(2,\bC)\times GL(2,\bC )$}
\label{sec:C22_arch_zeta}

In this section, we set $F=\bC$. 
Let $\Pi_\chi $ and $\Pi_{\chi'}$ be 
irreducible principal series representations of $G_2$ 
where 
\begin{align*}
&\chi =\chi_{(\nu_1,d_1)}\boxtimes \chi_{(\nu_2,d_2)}\quad 
(d_1\geq d_2),&
&\chi'=\chi_{(\nu_1',d_1')}\boxtimes \chi_{(\nu_2',d_2')}\quad 
(d_1'\geq d_2').
\end{align*}
Then the explicit form of the local $L$-factor for 
$\Pi_\chi \times \Pi_{\chi'}$ is given by 
\begin{align*}
&L (s,\Pi_{\chi}\times \Pi_{\chi'})=
\prod_{i=1}^2\prod_{j=1}^2
\Gamma_\bC \Bigl(s+\nu_i+\nu_j'+\tfrac{|d_i+d_j'|}{2}\Bigr).
\end{align*}

Let $\cS (\bC^2)$ be the space of Schwartz functions on $\bC^2$. 
Let $\cS (\bC^2)^{\mathrm{std}}$ be the subspace of $\cS (\bC^2)$ 
consisting of all functions $f$ of the form 
\begin{align*}
&f(z_1,z_2)=p(z_1,z_2,\overline{z_1},\overline{z_2})
\exp \bigl(-2\pi (|z_1|^2+|z_2|^2)\bigr)&
&(z_1,z_2\in \bC )
\end{align*}
with polynomial functions $p$ on $\bC^4$. 
We call functions in $\cS (\bC^2)^{\mathrm{std}}$ 
standard Schwartz functions on $\bC^2$. 
For non-negative integers $a_1$, $a_2$, $b_1$ and $b_2$, 
we define a standard Schwartz function 
$f_{(a_1,a_2,b_1,b_2)}$ on $\bC^2$ by 
\begin{align*}
&f_{(a_1,a_2,b_1,b_2)}(z_1,z_2)=
z_1^{a_1}z_2^{a_2}
\overline{z_1}^{\,b_1}
\overline{z_2}^{\,b_2}\exp \bigl(-2\pi (|z_1|^2+|z_2|^2)\bigr).
\end{align*}

Let $\varepsilon \in \{\pm 1\}$. 
For $W\in \mathrm{Wh}(\Pi_\chi ,\psi_{\varepsilon })^{\mathrm{mg}}$, 
$W'\in \mathrm{Wh}(\Pi_{\chi'} ,\psi_{-\varepsilon })^{\mathrm{mg}}$ 
and $f\in \cS (\bC^2)$, 
we define the local zeta integral $Z(s,W,W',f)$ for $G_2\times G_2$ by 
\begin{align*}
&Z(s,W,W',f)=
\int_{N_2\backslash G_2}W(g)W'(g)f((0,1)g)|\det g|^{2s}\,d\dot{g}\\
&\hspace{10mm}=\int_0^\infty \!\!\int_0^\infty  
\left(\int_{K_2}
W(yk)W'(yk)f((0,1)yk)\,dk\right)
y_1^{2s}y_2^{4s}
\frac{2dy_1}{y_1^3}\frac{2dy_2}{y_2}
\end{align*} 
with $y=\diag (y_1y_2,y_2)\in A_2$. 
Here $d\dot{g}$ is the right invariant measure on $N_2\backslash G_2$ 
normalize by  (\ref{eqn:C32_normalize_N2_G2}), 
and $dk$ is the Haar measure on $K_2$ such that $\int_{K_2}dk=1$. 
The local zeta integral $Z(s,W,W',f)$ converges 
for $\mathrm{Re}(s)\gg 0$. Jacquet proves that the local zeta integral 
has the meromorphic continuation and satisfies the local functional equation 
(\cite{Jacquet_003}). 
In \cite{Miyazaki_003}, the third author proves the following.

\begin{prop}
\label{prop:C22_zeta_L_coincide}
Let $\Pi_\chi$ and $\Pi_{\chi'}$ be irreducible principal 
series representations of $G_2=GL(2,\bC )$ with 
$\chi =\chi_{(\nu_1,d_1)}\boxtimes \chi_{(\nu_2,d_2)}$ $(d_1\geq d_2)$ 
and  $\chi'=\chi_{(\nu_1',d_1')}\boxtimes \chi_{(\nu_2',d_2')}$ 
$(d_1'\geq d_2')$. 
Set $\lambda =(\lambda_1,\lambda_2)=(d_1+l_0,d_2-l_0)$ with 
$l_0=\max \{0,-d_1-d_1',d_2+d_2'\}$, 
and $\lambda'=(d_1',d_2')$. 
Then, for $s\in \bC$ with sufficiently large real part, it holds that 
\[
Z\bigl(s,\varphi_{\chi ,\lambda }^{[\varepsilon ]}(v_{\lambda ,0}),
\varphi_{\chi',\lambda'}^{[-\varepsilon ]}(v_{\lambda',d_1'-d_2'}),
f_{(a_1,a_2,b_1,b_2)}\bigr)
=C(\chi,\chi')L(s,\Pi_\chi \times \Pi_{\chi'}),
\] 
where $\varphi_{\chi ,\lambda}^{[\varepsilon ]}$, 
$\varphi_{\chi',\lambda'}^{[-\varepsilon ]}$ are 
the $K_2$-homomorphisms in \S \ref{subsec:C32_Wh_GL2}, and  
\begin{align}
\label{eqn:def_pi_qi}
&\left\{\begin{array}{lll}
a_i=-\lambda_i-d_{3-i}',&b_i=0&
\text{ if }\, \lambda_i+d_{3-i}'\leq 0,\\[1mm]
a_i=0,&b_i=\lambda_i+d_{3-i}'&
\text{ if }\, \lambda_i+d_{3-i}'\geq 0.
\end{array}\right.\quad (i=1,2),\\
&C(\chi,\chi')=
\frac{(-1)^{d_2'}(\lambda_1-\lambda_2)!(d_1'-d_2')!}
{(\lambda_1+d_1'+a_1+a_2+1)!(\lambda_1+d_1'-b_1-b_2)!}.
\end{align}
\end{prop}

\begin{lem}
\label{lem:C22_schur}
Let $y=\diag (y_1y_2,y_2)\in A_2$. 
For $\lambda =(\lambda_1,\lambda_2),\, 
\lambda' =(\lambda_1',\lambda_2')\in \Lambda_2$, 
$0\leq q\leq \lambda_1-\lambda_2$, 
$0\leq q'\leq \lambda_1'-\lambda_2'$ and 
$a_1,a_2,b_1,b_2\in \bZ_{\geq 0}$, the integral 
\begin{align*}
&\int_{K_2}
\bigl\langle 
\tau^{(2)}_{\lambda}(k)v_{\lambda ,0},
v_{\widetilde{\lambda},\widetilde{q}}\bigr\rangle 
\bigl\langle 
\tau^{(2)}_{\lambda'}(k)v_{\lambda',\lambda_1'-\lambda_2'}, 
v_{\widetilde{\lambda'},\widetilde{q'}}\bigr\rangle 
f_{(a_1,a_2,b_1,b_2)}((0,1)yk)dk
\end{align*}
is equal to 
\[
\frac{(-1)^{a_1+b_1-q}(q+a_1)!(q'+a_2)!}
{(\lambda_1+\lambda_1'+a_1+a_2+1)!}
y_2^{a_1+a_2+b_1+b_2}\exp (-2\pi y_2^2)
\]
if $\lambda_1+\lambda_2'+a_1-b_1
=\lambda_2+\lambda_1'+a_2-b_2
=\lambda_1+\lambda_1'-q-q'=0$, 
and is equal to $0$ otherwise. 
Here we set $\widetilde{q}=\lambda_1-\lambda_2-q$ 
and $\widetilde{q'}=\lambda_1'-\lambda_2'-q'$. 
\end{lem}
\begin{proof}
By direct computation, we have  
\begin{align*}
&\bigl\langle 
\tau^{(2)}_{\lambda}(k)v_{\lambda ,0},
v_{\widetilde{\lambda},\widetilde{q}}\bigr\rangle 
=(-1)^{\lambda_1-q}(\det k)^{\lambda_2}k_{11}^{\widetilde{q}}k_{21}^{q},\\
&\bigl\langle 
\tau^{(2)}_{\lambda'}(k)v_{\lambda_1'-\lambda_2'}^{\lambda'},
v_{\widetilde{\lambda'},\widetilde{q'}}\bigr\rangle 
=(-1)^{\lambda_1'-q'}(\det k)^{\lambda_2'}
k_{12}^{\widetilde{q'}}k_{22}^{q'},\\
&f_{(a_1,a_2,b_1,b_2)}((0,1)yk)=(-1)^{b_1}(\det k)^{-b_1-b_2}
k_{21}^{a_1}k_{22}^{a_2}
k_{12}^{b_1}k_{11}^{b_2}\\
&\phantom{f_{(a_1,a_2,b_1,b_2)}((0,1)yk)=}
\times y_2^{a_1+a_2+b_1+b_2}\exp (-2\pi y_2^2)
\end{align*}
for $\displaystyle k=\left(\begin{array}{cc}
k_{11}&k_{12}\\
k_{21}&k_{22}
\end{array}\right)\in K_2$.
Hence, we have 
\begin{align*}
&\int_{K_2}
\bigl\langle 
\tau^{(2)}_{\lambda}(k)v_{\lambda ,0},
v_{\widetilde{\lambda},\widetilde{q}}\bigr\rangle 
\bigl\langle 
\tau^{(2)}_{\lambda'}(k)v_{\lambda',\lambda_1'-\lambda_2'},
 v_{\widetilde{\lambda'},\widetilde{q'}}\bigr\rangle 
f_{(a_1,a_2,b_1,b_2)}((0,1)yk)dk\\
&=(-1)^{b_1}y_2^{a_1+a_2+b_1+b_2}\exp (-2\pi y_2^2)\\
&\phantom{=}\times \int_{K_2}
\bigl\langle 
\tau^{(2)}_{\lambda +(a_1,-b_2)}(k)v_{\lambda +(a_1,-b_2),\,0},
v_{\widetilde{\lambda}+(b_2,-a_1),\,\widetilde{q}+b_2}
\bigr\rangle \\
&\phantom{=}\times 
\bigl\langle 
\tau^{(2)}_{\lambda'+(a_2,-b_1)}(k)
v_{\lambda'+(a_2,-b_1),\,\lambda_1'-\lambda_2'+a_2+b_1},
v_{\widetilde{\lambda'}+(b_1,-a_2),\,\widetilde{q'}+b_1}
\bigr\rangle dk.
\end{align*}
Applying Schur's orthogonality relation (\ref{eqn:C32_schur}) 
to the right hand side of this equality, we obtain the assertion. 
\end{proof}

\begin{proof}[Proof of Proposition \ref{prop:C22_zeta_L_coincide}]
In order to simplify the notation, we set 
\begin{align*}
&\varphi =\varphi_{\chi ,\lambda }^{[\varepsilon ]},&
&\varphi'=\varphi_{\chi',\lambda'}^{[-\varepsilon ]},&
&f=f_{(a_1,a_2,b_1,b_2)}.& 
\end{align*}
For $s\in \bC$ with sufficiently large real part, we have 
\begin{align*}
&Z(s,\varphi (v_{\lambda,0}),\varphi'(v_{\lambda',d_1'-d_2'}),f)\\
&=\int_0^\infty \!\!\int_0^\infty \biggl(\int_{K_2}
\varphi (\tau^{(2)}_\lambda (k)v_{\lambda,0})(y)
\varphi'(\tau^{(2)}_{\lambda'} (k)v_{\lambda',d_1'-d_2'})(y)\\
&\phantom{=.}
\times f((0,1)yk)\,dk\biggr)y_1^{2s}y_2^{4s}
\frac{2dy_1}{y_1^3}\frac{2dy_2}{y_2}
\end{align*}
with $y=\diag (y_1y_2,y_2)\in A_2$. 
Since  
\begin{align*}
&\tau^{(2)}_{\lambda}(k)v_{\lambda,0}=\sum_{q=0}^{\lambda_1-\lambda_2}
(-1)^{\lambda_1-q}\binom{\lambda_1-\lambda_2}{q}
\bigl\langle 
\tau^{(2)}_{\lambda}(k)v_{\lambda,0},
v_{\widetilde{\lambda},\tilde{q}}\bigr\rangle v_{\lambda ,q},\\
&\tau^{(2)}_{\lambda'}(k)v_{\lambda',d_1'-d_2'}=\sum_{q'=0}^{d_1'-d_2'}
(-1)^{d_1'-q'}\binom{d_1'-d_2'}{q'}
\bigl\langle 
\tau^{(2)}_{\lambda'}(k)v_{\lambda',d_1'-d_2'},
v_{\widetilde{\lambda'},\tilde{q'}}\bigr\rangle v_{\lambda',q'}
\end{align*}
with $\tilde{q}=\lambda_1-\lambda_2-q$ and 
$\tilde{q'}=d_1'-d_2'-q'$, we have 
\begin{align*}
&Z(s,\varphi (v_{\lambda,0}),\varphi'(v_{\lambda',d_1'-d_2'}),f)\\
&=\sum_{q=0}^{\lambda_1-\lambda_2}\sum_{q'=0}^{d_1'-d_2'}
(-1)^{\lambda_1+d_1'-q-q'}\binom{\lambda_1-\lambda_2}{q}\binom{d_1'-d_2'}{q'}
\\
&\phantom{=}\times 
\int_{0}^\infty \int_{0}^\infty 
\biggl(\int_{K_2}
\bigl\langle 
\tau^{(2)}_{\lambda}(k)v_{\lambda,0},
v_{\widetilde{\lambda},\tilde{q}}\bigr\rangle 
\bigl\langle 
\tau^{(2)}_{\lambda'}(k)v_{\lambda',d_1'-d_2'},
v_{\widetilde{\lambda'},\tilde{q'}}\bigr\rangle \\
&\phantom{=}\times f((0,1)yk)\,dk\biggr)
\varphi (v_{q}^{\lambda})(y)\varphi'(v_{\lambda',q'})(y)y_1^{2s}y_2^{4s}
\frac{2dy_1}{y_1^3}\frac{2dy_2}{y_2}.
\end{align*}
Applying Lemma \ref{lem:C22_schur} and 
integrating with respect to $y_2$:  
\begin{align*}
&\int_{0}^\infty 
\exp (-2\pi y_2^2)
y_2^{4s+2\nu_1+2\nu_2+2\nu_1'+2\nu_2'+a_1+a_2+b_1+b_2}\frac{4dy_2}{y_2}\\
&=\Gamma_{\bC}
\bigl(2s+\nu_1+\nu_2+\nu_1'+\nu_2'+\tfrac{a_1+a_2+b_1+b_2}{2}\bigr),
\end{align*}
we find that 
\begin{align*}
&Z(s,\varphi (v_{\lambda,0}),\varphi'(v_{\lambda',d_1'-d_2'}),f)
=\Gamma_{\bC}
\bigl(2s+\nu_1+\nu_2+\nu_1'+\nu_2'+\tfrac{a_1+a_2+b_1+b_2}{2}\bigr)\\
&\hspace{2cm}\times 
\sum_{q=b_1}^{\lambda_1+d_1'-b_2}
\frac{(-1)^{\lambda_1+d_2'-q}(\lambda_1-\lambda_2)!(d_1'-d_2')!}
{(\lambda_1+d_1'+a_1+a_2+1)!(q-b_1)!(\lambda_1+d_1'-b_2-q)!}\\
&\hspace{2cm}\times 
\int_{0}^\infty 
\varphi (v_{\lambda,q})(y')
\varphi'(v_{\lambda',\lambda_1+d_1'-q})(y')y_1^{2s}\frac{dy_1}{y_1^3}.
\end{align*}
with $y'=\diag (y_1,1)\in A_2$.

Let us consider the case of $l_0=0$, that is, 
$\lambda =(d_1,d_2)$. 
By Lemmas \ref{lem:F32_Mellin}, \ref{lem:F32_Barnes_1st} 
and the formula (\ref{eqn:C32_Wh_GL2_min}), we have 
\begin{align*}
&Z(s,\varphi (v_{\lambda,0}),\varphi'(v_{\lambda',d_1'-d_2'}),f)\\
&=
\frac{(-1)^{d_2'}(d_1-d_2)!(d_1'-d_2')!}
{(d_1+d_1'+a_1+a_2+1)!(d_1+d_1'-b_1-b_2)!}\\
&\phantom{=}
\times \Gamma_{\bC}\bigl(s+\nu_1+\nu_1'+\tfrac{d_1+d_1'}{2}\bigr)
\Gamma_{\bC}\bigl(s+\nu_2+\nu_2'+\tfrac{-d_2-d_2'}{2}\bigr)\\
&\phantom{=}\times 
\frac{\Gamma_{\bC}\bigl(2s+\nu_1+\nu_2+\nu_1'+\nu_2'
+\tfrac{a_1+a_2+b_1+b_2}{2}\bigr)}
{\Gamma_{\bC}\bigl(2s+\nu_1+\nu_2+\nu_1'+\nu_2'
+\tfrac{d_1-d_2+d_1'-d_2'}{2}\bigr)}
\sum_{q=b_1}^{d_1+d_1'-b_2}
\binom{d_1+d_1'-b_1-b_2}{q-b_1}\\
&\phantom{=}\times 
\Gamma_{\bC}
\bigl(s+\nu_1+\nu_2'+\tfrac{-d_1-d_2'}{2}+q\bigr)
\Gamma_{\bC}
\bigl(s+\nu_2+\nu_1'+\tfrac{2d_1-d_2+d_1'}{2}-q\bigr).
\end{align*}
Replacing $q\to j+b_1$  
and applying Lemma \ref{lem:F32_gauss_sum}, 
we obtain the assertion in this case.  

Let us consider the case of $l_0=-d_1-d_1'$. 
In this case, we note that 
\begin{align*}
&a_1=-\lambda_1-d_2',&
&a_2=-\lambda_2-d_1',&
&b_1=b_2=0,&
&\lambda_1=-d_1'.
\end{align*}
Hence, we have 
\begin{align*}
\begin{split}
&Z(s,\varphi (v_{\lambda,0}),\varphi'(v_{\lambda',d_1'-d_2'}),f)
=\Gamma_{\bC}
\bigl(2s+\nu_1+\nu_2+\nu_1'+\nu_2'+\tfrac{-\lambda_2-d_2'}{2}\bigr)\\
&\hspace{10mm}\times (-1)^{\lambda_1+d_2'}
\frac{(\lambda_1-\lambda_2)!(d_1'-d_2')!}
{(-\lambda_2-d_2'+1)!}\int_{0}^\infty 
\varphi (v_{\lambda,0})(y')
\varphi'(v_{\lambda',0})(y')y_1^{2s}
\frac{dy_1}{y_1^3}.
\end{split}
\end{align*}
By Lemmas \ref{lem:F32_Mellin}, \ref{lem:F32_Barnes_1st} 
and the formula (\ref{eqn:C32_Wh_GL2_1st_simple}), 
we obtain the assertion in this case.

Let us consider the case of $l_0=d_2+d_2'$. 
In this case, we note that 
\begin{align*}
&a_1=a_2=0,&
&b_1=\lambda_1+d_2',&
&b_2=\lambda_2+d_1',&
&\lambda_2=-d_2'.
\end{align*}
Hence, we have 
\begin{align*}
&Z(s,\varphi (v_{\lambda,0}),\varphi'(v_{\lambda',d_1'-d_2'}),f)
=\Gamma_{\bC}
\bigl(2s+\nu_1+\nu_2+\nu_1'+\nu_2'+\tfrac{\lambda_1+d_1'}{2}\bigr)\\
&\hspace{10mm}\times \frac{(\lambda_1-\lambda_2)!(d_1'-d_2')!}
{(\lambda_1+d_1'+1)!}
\int_{0}^\infty 
\varphi (v_{\lambda,\lambda_1-\lambda_2})(y')
\varphi'(v_{\lambda',d_1'-d_2'})(y')y_1^{2s}\frac{dy_1}{y_1^3}.
\end{align*}
By Lemmas \ref{lem:F32_Mellin}, \ref{lem:F32_Barnes_1st} 
and the formula (\ref{eqn:C32_Wh_GL2_2nd_simple}), 
we obtain the assertion in this case. 
\end{proof}